\documentclass[english]{article} 
\usepackage[absolute,overlay]{textpos} 
\usepackage{amsthm}
\usepackage{amsmath}
\usepackage{amssymb}
\usepackage[utf8]{inputenc}
\usepackage[english]{babel}
\usepackage{graphicx}
\usepackage{geometry}
\usepackage{caption}
\usepackage{subcaption}
\usepackage{pdfpages}
\usepackage{comment}
\usepackage{bm}
\usepackage{etoolbox}
\usepackage{mathtools}
\usepackage{tikz}

\usepackage{pgfplots}
\usepackage{pgfplotstable}
\usepackage{stmaryrd} 

\usepackage{hyperref}

\usepackage[noabbrev,capitalise]{cleveref}
\crefformat{appendix}{#2#1#3}

\usetikzlibrary{fillbetween}
\usepgfplotslibrary{fillbetween}

\usetikzlibrary{shapes.misc}
\usetikzlibrary{math} 

\usetikzlibrary{arrows}

\definecolor{darkspringgreen}{rgb}{0., 0.55, 0.3}
\definecolor{dartmouthgreen}{rgb}{0.05, 0.5, 0.06}
\definecolor{etonblue}{rgb}{0.59, 0.78, 0.64}
\definecolor{airforceblue}{rgb}{0., 0.4, 0.66}
\definecolor{arylideyellow}{rgb}{0.91, 0.84, 0.42}
\definecolor{emerald}{rgb}{0.31, 0.78, 0.47}
\definecolor{uclagold}{rgb}{1.0, 0.7, 0.0}
\definecolor{cadmiumorange}{rgb}{0.93, 0.53, 0.18}

\newtheorem{theorem}{Theorem}
\numberwithin{theorem}{section}

\numberwithin{definition}{section}
\newtheorem{remark}{Remark}
\numberwithin{remark}{section}

\newcommand{\lopd}[0]{\mathcal{L}_\Delta}
\newcommand{\lopdt}[0]{\mathcal{L}_{\Delta}}
\newcommand{\LL}{\lopdt}
\newcommand{\usol}[0]{\underline{\uvec{u}}_\Delta}

\newcommand{\uapp}[0]{\uvec{u}_h}

\newcommand{\tess}[0]{\mathcal{T}_h}

\newcommand{\uvec}[2][3]{\boldsymbol{#2\mkern-#1mu}\mkern#1mu}

\newcommand\norm[1]{\left\lVert#1\right\rVert}

\newcommand{\spacestuff}[0]{\bm{\phi}_i}

\newcommand{\cund}[0]{\underline{\uvec{c}}}

\newcommand{\lopdi}[0]{\mathcal{L}_{\Delta,i}}

\newcommand{\R}{\mathbb{R}}
\newcommand{\dt}{\Delta t}

\newcommand{\bu}{\uvec{u}}
\newcommand{\bbu}{\underline{\uvec{u}}}

\newcommand{\bbU}{\underline{\uvec{U}}}
\newcommand{\bbr}{\underline{\uvec{r}}}
\newcommand{\by}{\uvec{y}}

\newcommand{\bG}{{\uvec{G}}}

\newcommand{\vecz}{\underline{0}}
\newcommand{\matz}{\underline{\underline{0}}}
\newcommand{\vecbeta}{\underline{\beta}}
\newcommand{\CIP}{\text{CIP}}
\newcommand{\OSS}{\text{OSS}}
\newcommand{\CFL}{\text{CFL}}

\newcommand{\undu}[0]{\underline{\uvec{u}}}
\newcommand{\GL}[0]{GL}

\begin{document}
\author{L. Micalizzi\footnote{Affiliation: Institute of Mathematics, University of Zurich, Winterthurerstrasse 190, Zurich, 8057, Switzerland. Email: lorenzo.micalizzi@math.uzh.ch.} and D.Torlo\footnote{Affiliation: SISSA mathLab, SISSA, via Bonomea 265, Trieste, 34136, Italy. Email: davide.torlo@sissa.it.}}
\title{A new efficient explicit Deferred Correction framework: \\
analysis and applications to hyperbolic PDEs and adaptivity}

\maketitle

\abstract{	The Deferred Correction (DeC) is an iterative procedure, characterized by increasing accuracy at each iteration, which can be used to design numerical methods for systems of ODEs. %
	The main advantage of such framework is the automatic way of getting arbitrarily high order methods, which can be put in Runge--Kutta (RK) form. 
	The drawback is the larger computational cost with respect to the most used RK methods.
	To reduce such cost, in an explicit setting, we propose an efficient modification: we introduce interpolation processes between the DeC iterations, decreasing the computational cost associated to the low order ones.
	We provide the Butcher tableaux of the new modified methods and we study their stability, showing that in some cases the computational advantage does not affect the stability. 
	The flexibility of the novel modification allows nontrivial applications to PDEs and construction of adaptive methods.
	The good performances of the introduced methods are broadly tested on several benchmarks both in ODE and PDE contexts.}

\section{Introduction}
A huge amount of phenomena in many different fields can be modeled through ODEs and PDEs, whose analytical solutions are usually not available,  
hence, many numerical methods have been developed to approximate such solutions. 
Indeed, the higher is the accuracy needed in the approximation, the more expensive the associated numerical simulations are in terms of computational time and resources employed.
If, on the one hand, with modern computers the speed of the simulations has drastically improved, on the other hand, the always stricter tolerances required by modern applications have lead to massive simulations only accessible to supercomputers and, still, characterized by very long computational times. 
That is why any effort in reducing the computational costs of numerical simulations is of paramount importance. 
A classical way of reducing them is the adoption of high order methods, which allow to reach lower errors within coarse discretizations. 

A wide series of arbitrarily high order methods is based on the DeC approach. Its original formulation has been firstly introduced in 1949 in \cite{fox1949some} in a simple prediction-correction time integrator framework. 
A more elegant version based on spectral integration in time was introduced in 2000 \cite{Decoriginal}, characterized by an iterative procedure allowing to increase the order of accuracy by one at each iteration.
In 2003 \cite{minion2003semi}, Minion generalized the DeC  framework to obtain an implicit-explicit arbitrarily high order method, with various applications to ODEs and PDEs \cite{minion2004semi,layton2005implications,huang2006accelerating,minion2011hybrid,speck2015multi}. 
Later on, the DeC approach has been generalized by Abgrall \cite{Decremi} to solve hyperbolic PDEs with high order continuous Galerkin (CG) spatial discretizations, overcoming the burden related to the mass matrix %
leading to numerous applications in the hyperbolic field \cite{abgrall2019high,abgrall2020high,michel2022spectral,CIALLELLA2022105630,Bacigaluppi2023}. 
The DeC has been also modified in order to preserve physical structures (positivity, entropy, moving equilibria, conservation) \cite{offner2020arbitrary,abgrall2021relaxation,CIALLELLA2022105630,abgrall2022staggered}. 
Finally, in \cite{han2021dec} it has been pointed out that DeC and ADER methods are very similar iterative time integrators and, when restricted to ODEs, they can be written as RK schemes, see also \cite{ketcheson2014comparison,torlo2020hyperbolic}.

The clear advantage of the DeC framework is the possibility to easily increase the order of accuracy, the drawback is the expensive computational cost, due to the iterations and to the high degree of the polynomial reconstruction of the numerical solution considered in each of them.
To alleviate such cost, the \textit{ladder }strategy was proposed in implicit DeC algorithms \cite{minion2003semi,layton2005implications,speck2015multi}, where the reconstruction in time increases the degree at each iteration. 
Between the iterations, an interpolation procedure links the different reconstructions. 
Though being the idea used in some works, it has never been deeply studied and analyzed, in particular, for the purely explicit DeC. 

Inspired by this idea, in this work, we provide a detailed description of two novel families of efficient explicit DeC methods, based on easy modifications of existing DeC schemes. 
By explicitly constructing their Butcher tableaux and studying their stability, we show that in some cases the new efficient versions and the classical one have the same stability functions.
Moreover, we exploit the modification to build adaptive methods that, given a certain tolerance, automatically choose the order of accuracy to reach such error in the most efficient way.
We also apply the efficient modification in the context of mass matrix-free CG--DeC methods \cite{Decremi} for hyperbolic PDEs.

The structure of this work is the following.
We start by introducing the DeC procedure in an abstract framework in Section~\ref{chapDECAbstract} and as a tool for the numerical solution of ODEs systems in Section~\ref{chapDECODE}.
In Section~\ref{chapDECSTAG}, we introduce the new families of efficient DeC methods. Then, we give their Butcher tableaux in Section~\ref{chapRK} and in Section~\ref{chapSTAB} we study in detail their linear stability.
In Section~\ref{chapDECRD}, we describe the application to the numerical solution of hyperbolic problems with CG spatial discretizations avoiding mass matrices.
We propose an adaptive and efficient version of the methods in Section~\ref{chapADAPT}.
In Section~\ref{chapNum}, we present numerical results for ODEs and hyperbolic PDEs with various comparisons with the classical DeC methods.
Section~\ref{chapConc} is dedicated to the conclusions.

\section{Abstract DeC formulation}\label{chapDECAbstract}
We will first introduce the DeC abstract formulation proposed by Abgrall in \cite{Decremi}.
Let us assume that we have two operators between two normed vector spaces $\left(X, \norm{\cdot}_X \right)$ and $\left(Y, \norm{\cdot}_Y \right)$, namely $\lopd^1,\lopd^2:X \longrightarrow Y$, associated to two discretizations of the same problem and dependent on a same discretization parameter $\Delta$. 
In particular, assume that $\lopd^2$ corresponds to a high order implicit discretization, while, $\lopd^1$ corresponds to a low order explicit one.
We would like to solve $\lopd^2$, i.e., finding $\underline{\uvec{u}}_\Delta \in X$ such that $\lopd^2(\underline{\uvec{u}}_\Delta)=\uvec{0}_Y$, to get a high order approximation of the solution to the original problem, but this is not easy because of its implicit character. 
Instead, the low order explicit operator $\lopd^1$ is very easy to solve and, more in general, we assume that it is easy to solve $\lopd^1(\underline{\uvec{u}})=\underline{\uvec{r}}$ with $\underline{\uvec{r}} \in Y$ given, but the associated accuracy is not sufficient for our intended goals.
In the next theorem, we will provide a simple recipe to get an arbitrary high order approximation of the solution of $\lopd^2$ by combining the operators $\lopd^1$ and $\lopd^2$ in an easy iterative procedure.
\begin{theorem}[DeC accuracy]\label{th:DeC} 
Let the following hypotheses hold
\begin{enumerate}
\item \textbf{Existence of a unique solution to $\lopd^2$} \\
$\exists ! \,\usol \in X$ solution of $\lopd^2$ such that $\lopd^2(\usol)=\uvec{0}_Y$;
\item \textbf{Coercivity-like property of $\lopd^1$} \\
$\exists \,\alpha_1 \geq 0$ independent of $\Delta$ such that 
\begin{equation}
\norm{\lopd^1(\underline{\uvec{v}})-\lopd^1(\underline{\uvec{w}})}_Y\geq \alpha_1\norm{\underline{\uvec{v}}-\underline{\uvec{w}}}_X, ~ \forall \underline{\uvec{v}},\underline{\uvec{w}}\in X;
\label{eq:coercivitylike}
\end{equation}
\item \textbf{Lipschitz-continuity-like property of $\lopd^1-\lopd^2$} \\
$\exists\, \alpha_2 \geq 0$ independent of $\Delta$ such that
\begin{equation}
	\norm{\left(\lopd^1(\underline{\uvec{v}})\!-\!\lopd^2(\underline{\uvec{v}})\right)\!-\!\left(\lopd^1(\underline{\uvec{w}})\!-\!\lopd^2(\underline{\uvec{w}})\right)}_Y\!\leq \!\alpha_2 \Delta \!\norm{\underline{\uvec{v}}-\underline{\uvec{w}}}_X, ~ \forall \underline{\uvec{v}},\underline{\uvec{w}}\in X.
\label{eq:lipschitzlike}
\end{equation}
\end{enumerate}
Then, if we iteratively define $\underline{\uvec{u}}^{(p)}$ as the solution of
\begin{equation}
\label{DeC}
\lopd^1(\underline{\uvec{u}}^{(p)})=\lopd^1(\underline{\uvec{u}}^{(p-1)})-\lopd^2(\underline{\uvec{u}}^{(p-1)}), \quad p=1,\dots,P,
\end{equation}
we have that
\begin{equation}
\label{accuracyestimate}
\norm{\underline{\uvec{u}}^{(P)}-\usol}_X \leq \left( \Delta \frac{\alpha_2}{\alpha_1} \right)^P\norm{\underline{\uvec{u}}^{(0)}-\usol}_X.
\end{equation}
\end{theorem}
\begin{proof}
The proof relies on a direct use of the hypotheses. In particular, we have
\begin{subequations}
\begin{align}
\norm{\underline{\uvec{u}}^{(P)}-\usol}_X &\leq \frac{1}{\alpha_1}  \norm{\lopd^1(\underline{\uvec{u}}^{(P)})-\lopd^1(\usol)}_Y \label{eq:using_coercivity}\\
&= \frac{1}{\alpha_1}  \norm{\lopd^1(\underline{\uvec{u}}^{(P-1)})-\lopd^2(\underline{\uvec{u}}^{(P-1)})-\lopd^1(\usol)}_Y \label{eq:using_DeC}\\ 
&=\frac{1}{\alpha_1}  \norm{\lopd^1(\underline{\uvec{u}}^{(P-1)})-\lopd^2(\underline{\uvec{u}}^{(P-1)})-\lopd^1(\usol)+\lopd^2(\usol)}_Y\label{eq:using_solution}\\ 
&\leq \Delta \frac{\alpha_2}{\alpha_1}\norm{\underline{\uvec{u}}^{(P-1)}-\usol}_X,
\label{eq:using_Lipschitz}
\end{align}
\end{subequations}
where in \eqref{eq:using_coercivity} we have used \eqref{eq:coercivitylike}, in \eqref{eq:using_DeC} the definition of the DeC iteration \eqref{DeC}, in \eqref{eq:using_solution} the fact that $\lopd^2(\usol)=\uvec{0}_Y$ and, finally, in \eqref{eq:using_Lipschitz} we have used \eqref{eq:lipschitzlike}.
By repeating these calculations recursively we get the desired result.
\end{proof}
Let us remark that, due to our assumption on the operator $\lopd^1$, the updating formula \eqref{DeC} represents a simple explicit recipe to approximate arbitrarily well the solution $\usol$ of $\lopd^2$.
The convergence for $P \rightarrow +\infty$ is ensured independently of the starting vector $\underline{\uvec{u}}^{(0)}$ provided that $\Delta \frac{\alpha_2}{\alpha_1}<1$. 
The coefficients $\alpha_1$ and $\alpha_2$ can be computed once the operators $\lopd^1$ and $\lopd^2$ are defined. 
In the next sections, we will provide such definitions for different DeC ODE solvers, and the convergence constraint imposed by $\Delta \frac{\alpha_2}{\alpha_1}<1$ will sum up to a classical timestep restriction for explicit methods.

If the solution $\usol$ of $\lopd^2$ is an $R$-th order accurate approximation of the exact solution $\underline{\uvec{u}}^{ex}$ of the original problem to which the operators are associated, it does not make sense to approximate $\usol$ with accuracy higher than $R$, as we are actually interested in $\underline{\uvec{u}}^{ex}$. In particular, thanks to the accuracy estimate \eqref{accuracyestimate}, if $\undu^{(0)}$ is an $O(\Delta)$-approximation of $\underline{\uvec{u}}^{ex}$, the optimal choice is $P=R$, i.e., the optimal number of iterations coincides with the accuracy of the operator $\lopd^2$. Any further iteration results in a waste of computational resources.

In the following, we will characterize the operators $\lopd^1$ and $\lopd^2$ for some DeC ODEs solvers, explicitly writing the associated updating formulas. In order to provide a clearer understanding of the methods, we also report their more classical formulation, in Appendix~\ref{app:residual_formulations}, in terms of residual and error functions \cite{Decoriginal}.
However, we will stick to Abgrall's formulation \cite{Decremi} for its compactness, the possibility to directly work on the solution and its flexibility, which allows for applications to more general contexts, such as structure preserving methods \cite{mPDeC,CIALLELLA2022105630,abgrall2021relaxation,abgrall2022staggered}, mass-matrix free finite element methods \cite{Decremi,abgrall2019high,abgrall2020high}, ADER-DG methods \cite{han2021dec,micalizzi2022efficient}. All these generalizations and the efficient  modifications that we present in this paper are straightforward in Abgrall's formulation, while they are more involved in the classical DeC framework.

\section[DeC for ODEs]{The DeC for systems of ODEs}\label{chapDECODE}
We want to solve the Cauchy problem
\begin{equation}
\label{ODE}
\begin{cases}
\frac{d}{dt}\uvec{u}(t) = \uvec{G}(t,\uvec{u}(t)),\quad t\in[0,T], \\
\uvec{u}(0)=\uvec{z},
\end{cases}
\end{equation}
with $\uvec{u}(t) \in \mathbb{R}^Q$, $\uvec{z} \in \R^Q$ and $\uvec{G}: \R^+_0 \times \R^Q \to \R^Q$ a continuous map Lipschitz continuous with respect to $\uvec{u}$ uniformly with respect to $t$ with a Lipschitz constant $L$, which ensures the existence of a unique solution. 
We will present two explicit DeC methods for the numerical solution of such problem, which are based on approximations of its integral form 
\begin{itemize}
\item bDeC, which was introduced originally in \cite{liu2008strong} in a more general family of schemes, but fully exploited for its simplicity only starting from \cite{Decremi} in the context of Galerkin solvers for hyperbolic PDEs without mass matrix. In this method, the integral form is approximated on \textit{``big'' }intervals, hence the name bDeC.
\item sDeC, which has a longer history \cite{Decoriginal} and more developments \cite{minion2003semi,layton2004conservative,huang2006accelerating,speck2015multi}. In this method, the integral form is approximated on \textit{``small'' }intervals, hence the name sDeC.
\end{itemize}
Then, we will consider a general family of DeC methods, $\alpha$DeC, depending on a parameter $\alpha$, which contains both the previously described formulations as particular cases, as described in \cite{liu2008strong}. 

We assume a one-step method setting: 
at each time interval $[t_n,t_{n+1}]$, we assume to know $\uvec{u}_{n}\approx \uvec{u}(t_{n})$ and we look for $\uvec{u}_{n+1}\approx \uvec{u}(t_{n+1})$. In particular, as in the context of a general consistency analysis, we assume $\uvec{u}_{n}=\uvec{u}(t_{n})$. In this context, the parameter $\Delta$ of the DeC is the step size $\Delta t=t_{n+1}-t_{n}$.
A more traditional but equivalent formulation of bDeC and sDeC in terms of error and residual functions \cite{christlieb2009comments,christlieb2010integral,boscarino2018implicit,boscarino2016error} is reported in Appendix~\ref{app:residual_formulations}.

\subsection{bDeC}
In the generic time step $[t_n,t_n+\Delta t]$, we introduce $M+1$ subtimenodes  $t_n=t^0<t^1<\dots<t^M=t_n+\Delta t.$ Several choices of subtimenodes are possible, but for the following discussion we will consider equispaced ones. In the numerical tests, we will also present results obtained with Gauss--Lobatto (GL) subtimenodes \cite{mPDeC,han2021dec,Decoriginal}, which can obtain higher accuracy for a fixed number of subtimenodes.
We will refer to $\uvec{u}(t^m)$ as the exact solution in the subtimenode $t^m$ and to $\uvec{u}^m$ as the approximation of the solution in the same subtimenode.
Just for the first subtimenode, we set $\uvec{u}^0:=\uvec{u}_n$. 

The bDeC method is based on the integral version of the ODE \eqref{ODE} in each interval $[t^0,t^m]$, which reads
\begin{equation}
	\label{exint}
	\uvec{u}(t^m)-\uvec{u}^0-\int_{t^0}^{t^m}\uvec{G}(t,\uvec{u}(t))dt=\uvec{0}, \quad m=1,\dots,M.
\end{equation}
Starting from this formulation, we define the high order operator $\LL^2$ and  the low order operator $\LL^1$.
We define $\LL^2:\R^{(M\times Q)}\to\R^{(M\times Q)}$ by approximating the function $\uvec{G}$ in \eqref{exint} with a high order interpolation via the Lagrange polynomials $\psi^\ell$ of degree $M$ associated to the $M+1$ subtimenodes and exact integration of such polynomials
\begin{equation}
	\label{l2ODE}
	\lopdt^2(\underline{\uvec{u}}) = \begin{pmatrix}
		\uvec{u}^1-\uvec{u}^0-\Delta t \sum_{\ell=0}^{M} \theta^1_\ell \uvec{G}(t^\ell,\uvec{u}^\ell)\\
		\vdots\\
		\uvec{u}^M-\uvec{u}^0-\Delta t \sum_{\ell=0}^{M} \theta^M_\ell \uvec{G}(t^\ell,\uvec{u}^\ell)\\
	\end{pmatrix},
	\text{ with }
	\underline{\uvec{u}}=\left(
	\begin{array}{ccc}
		\uvec{u}^1\\
		\vdots\\
		\uvec{u}^M
	\end{array}
	\right),
\end{equation}
where the normalized coefficients $\theta^m_\ell :=\frac{1}{\Delta t} \int_{t^0}^{t^m}\psi^\ell(t)dt$ do not depend on $\dt$. This leads to the definition of the spaces $X=Y:=\R^{M\times Q}$ of Section~\ref{chapDECAbstract}. 
Let us remark that $\lopdt^2$ is defined on the $M$ components $\uvec{u}^m\in \R^Q$ corresponding to the subtimenodes where the solution is unknown, while $\uvec{u}^0$ is an intrinsic datum of the operator.
The generic $m$-th component $\lopdt^{2,m}(\underline{\uvec{u}})=\uvec{0}$ of the global problem $\lopdt^{2}(\underline{\uvec{u}})=\uvec{0}$ corresponds to a high order discretization of \eqref{exint}.
In particular, for equispaced subtimenodes, we have that if $\uvec{u}^m$ is the $m$-th component of the solution of \eqref{l2ODE}, then, it is an $(M+1)$-th order accurate approximation of $\uvec{u}(t^m)$.
The proof is based on a fixed-point argument and can be found in the supplementary material.
It is worth noting that $\LL^2(\undu)=\uvec{0}$ coincides with an implicit RK method with $M$ stages, e.g., when choosing \GL~subtimenodes one obtains the LobattoIIIA methods.

The definition of the low order explicit operator $\lopdt^1: \mathbb{R}^{(M \times Q)} \rightarrow \mathbb{R}^{(M \times Q)}$ is based on a first order explicit Euler discretization of \eqref{exint} leading to
\begin{equation}
	\label{l1ODE}
	\lopdt^1(\underline{\uvec{u}}) = \begin{pmatrix}
		
		\uvec{u}^1-\uvec{u}^0-\Delta t \beta^1 \uvec{G}(t^0,\uvec{u}^0)\\
		\vdots\\
		\uvec{u}^M-\uvec{u}^0-\Delta t \beta^M \uvec{G}(t^0,\uvec{u}^0)\\
	\end{pmatrix},
\end{equation}
where the normalized coefficients $\beta^m=\frac{t^m-t^0}{\Delta t}$ are determined only by the distribution of the subtimenodes. 
The generic $m$-th component $\lopdt^{1,m}(\underline{\uvec{u}})=\uvec{0}$ of $\lopdt^{1}(\underline{\uvec{u}})=\uvec{0}$ corresponds to the explicit Euler discretization of \eqref{exint}, hence, it is first order accurate and any system $\lopdt^1(\bbu)=\bbr$ can be readily solved for a given $\bbr \in \R^{M\times Q}$. 
\begin{figure}
	\centering
	\includegraphics[width=0.8\textwidth]{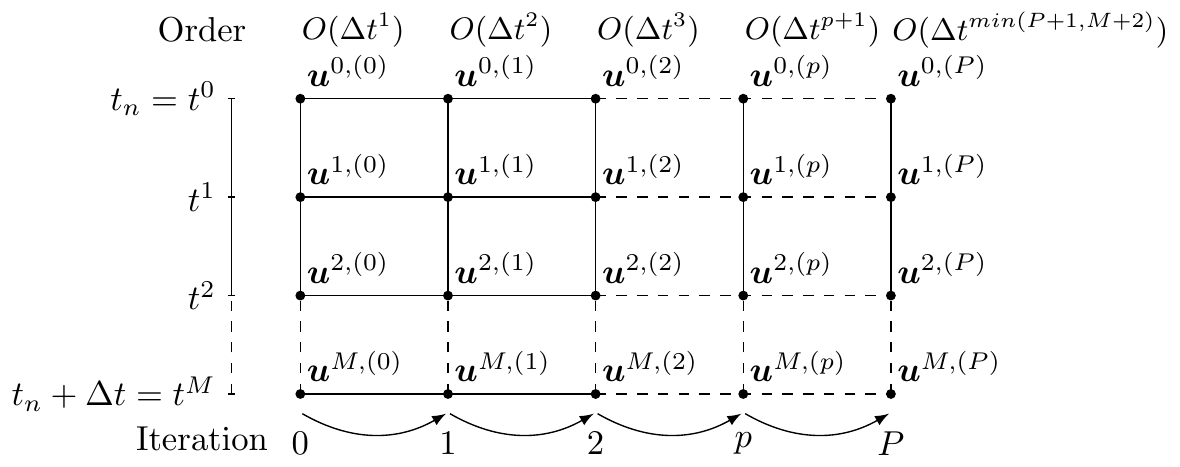}
	\caption{Sketch of the DeC iterative process for equispaced subtimenodes}\label{fig:grid}
\end{figure}
The operators $\lopdt^1$ and $\lopdt^2$ fulfill the hypotheses required to apply the DeC procedure, the proofs can be found in the supplementary material. In particular, we highlight that $\alpha_1=1$, while $\alpha_2=L \cdot \max_{m=1,\dots,M}\sum_{\ell=1}^M\vert\theta_\ell^m\vert$. 

Let us now characterize the updating formula \eqref{DeC} to this setting.
The vector $\underline{\uvec{u}}^{(p)} \in \R^{(M\times Q)}$ is, in this case, made by $M$ components $\uvec{u}^{m,(p)}\in \R^Q$, associated to the subtimenodes $t^m$ $m=1,\dots,M$ in which the solution is unknown, while we set $\uvec{u}^{0,(p)}:=\uvec{u}_n$ for all $p$. Then,  \eqref{DeC} gives 
\begin{align}\label{eq:DeCODE_Remi}
	\uvec{u}^{m,(p)} = \uvec{u}^0+\Delta t \sum_{\ell=0}^{M} \theta^m_\ell \uvec{G}(t^\ell,\uvec{u}^{\ell,(p-1)}), \quad m=1,\dots,M.
\end{align}
The starting vector $\underline{\uvec{u}}^{(0)}$ for our iterative procedure is chosen as $\uvec{u}^{m,(0)}:=\uvec{u}_n$ for all $m$. 
At the end of the iteration process, we set $\uvec{u}_{n+1}:= \uvec{u}^{M,(P)}$.
A graphical sketch of the updating process is shown in Figure~\ref{fig:grid}.
As said in Section~\ref{chapDECAbstract}, the optimal number of iterations depends on the accuracy of the operator $\LL^2$, i.e., $P=M+1$ for equispaced subtimenodes and $P=2M$ for \GL~ones.
Further iterations would not increase the order of accuracy of the method. 
On the other hand, to build a $P$-th order method, the optimal choice consists of $P$ iterations with $M=P-1$ for equispaced and $M=\left \lceil \frac{P}{2}\right \rceil $ for \GL~subtimenodes.

\subsection{sDeC} 
The sDeC operators differ from the bDeC ones by the ``smaller'' intervals considered to obtain the integral version of the ODE. In fact, adopting the previous definition of the subtimenodes, the sDeC method is based on the integral version of \eqref{ODE} over the intervals $[t^{m-1},t^m]$ for $m=1,\dots ,M$. 
This leads to the following definition of the operators $\lopdt^1, \lopdt^2: \mathbb{R}^{(M \times Q)} \rightarrow \mathbb{R}^{(M \times Q)}$ 
\begin{align}
	\LL^{1,m}(\bbu):=\uvec{u}^m-\uvec{u}^{m-1}-\Delta t \gamma^m \uvec{G}(t^{m-1},\uvec{u}^{m-1}), \qquad \text{for  }m=1,\dots,M, \label{l1ODE_2}\\
	\LL^{2,m}(\bbu):=\uvec{u}^m-\uvec{u}^{m-1}-\Delta t \sum_{\ell=0}^{M} \delta^m_\ell \uvec{G}(t^\ell,\uvec{u}^\ell), \qquad \text{for  }m=1,\dots,M, \label{l2ODE_2}
\end{align}
with $\gamma^m=\frac{t^m-t^{m-1}}{\Delta t}$ and $\delta^m_\ell:=\frac{1}{\Delta t}\int_{t^{m-1}}^{t^m}\psi^\ell(t)dt$ normalized coefficients. As before, $\LL^{1,m}(\bbu)=\uvec{0}$ is a first order explicit discretization, while, $\LL^{2,m}(\bbu)=\uvec{0}$ is a high order implicit one and, further, we have $\uvec{u}^0:=\uvec{u}_n$.

Differently from the previous formulation, in this case we cannot solve the operator $\lopdt^1$ in all its components at the same time but we have to do it component by component from $\uvec{u}^1$ to $\uvec{u}^M$. The same holds for the general problem $\lopdt^1(\underline{\uvec{u}})=\underline{\uvec{r}}$ for a fixed $\underline{\uvec{r}}\in \mathbb{R}^{(M \times Q)}$. However, still the computation of its solution can be performed explicitly.

Let us characterize the updating formula \eqref{DeC} to this context.
The explicit character of the operator $\lopdt^1$ leads to an explicit recipe for the computation of $\underline{\uvec{u}}^{(p)}$ whose components, in this case, must be computed in increasing order 
\begin{align}  
	\begin{split}
        \uvec{u}^{m,(p)}=\uvec{u}^{m-1,(p)}&+\Delta t \gamma^m \left( \uvec{G}(t^{m-1},\uvec{u}^{m-1,(p)})- \uvec{G}(t^{m-1},\uvec{u}^{m-1,(p-1)})\right)\\ &
        +\Delta t \sum_{\ell=0}^{M} \delta^m_\ell \uvec{G}(t^\ell,\uvec{u}^{\ell,(p-1)}).
	\end{split}
  \label{secondformulationintermediate}
\end{align}    
With recursive substitutions, \eqref{secondformulationintermediate} can be equivalently written as
\begin{equation}
	\begin{split}
\uvec{u}^{m,(p)}=\uvec{u}^{0}&+\Delta t \sum_{\ell=0}^{m-1} \gamma^{\ell+1} \left( \uvec{G}(t^{\ell},\uvec{u}^{\ell,(p)})-\uvec{G}(t^{\ell},\uvec{u}^{\ell,(p-1)}) \right)\\
&+\Delta t\sum_{r=1}^m \sum_{\ell=0}^{M} \delta^r_\ell \uvec{G}(t^\ell,\uvec{u}^{\ell,(p-1)}).
	\end{split}
\label{secondformulationalmost}
\end{equation}
Now, let us focus on the last term of \eqref{secondformulationalmost}. Exchanging the sums over $r$ and $\ell$, thanks to the fact that $\sum_{r=1}^m \delta^r_\ell=\theta^m_\ell$, we have
\begin{equation}
	\begin{split}
\uvec{u}^{m,(p)}=\uvec{u}^{0}+&\Delta t \sum_{\ell=0}^{m-1} \gamma^{\ell+1} \left( \uvec{G}(t^{\ell},\uvec{u}^{\ell,(p)})-\uvec{G}(t^{\ell},\uvec{u}^{\ell,(p-1)}) \right)\\
&+\Delta t \sum_{\ell=0}^{M} \theta^m_\ell \uvec{G}(t^\ell,\uvec{u}^{\ell,(p-1)}),
	\end{split}
\label{secondformulation}
\end{equation}
which allows to explicitly compute all the components $\uvec{u}^{m,(p)}$ in sequence from $m=1$ to $m=M$, in opposition to bDeC where a parallel strategy can be adopted.
For what concerns the accuracy of the method and the optimal number of iterations, one can refer to what already said in the context of the bDeC formulation. 

Let us observe that the sDeC method is equivalent to the DeC method presented in \cite{Decoriginal} in terms of residuals and error functions. We show the equivalence in Appendix~\ref{app:residual_formulations}.

\subsection{A general family of DeC methods, $\alpha$DeC}
Following \cite{ketcheson2014comparison}, we can construct a family of schemes dependent on a single parameter $\alpha\in [0,1]$ by a convex combination of the updating formulas of bDeC \eqref{eq:DeCODE_Remi} and sDeC \eqref{secondformulation}:
\begin{equation}
	\begin{split}
\uvec{u}^{m,(p)}=\uvec{u}^0&+\Delta t \sum_{\ell=0}^{M} \theta^m_\ell \uvec{G}(t^\ell,\uvec{u}^{\ell,(p-1)})\\
&+
\alpha\left[\Delta t \sum_{\ell=0}^{m-1} \gamma^{\ell+1} \left( \uvec{G}(t^{\ell},\uvec{u}^{\ell,(p)})-\uvec{G}(t^{\ell},\uvec{u}^{\ell,(p-1)}) \right)\right].
	\end{split}
\label{eq:generalformulation}
\end{equation}
Through \eqref{eq:generalformulation}, it is possible to explicitly compute iteration by iteration the different components $\uvec{u}^{m,(p)}$ starting from $m=1$ until $M$.
Of course, when $\alpha=0$ we retrieve the bDeC formulation, while for $\alpha=1$ we get the sDeC one.

\subsubsection{Matrix formulation}
\label{matrfor}
We will now introduce a compact matrix-formulation of the presented methods.
For convenience, we will now introduce the vectors containing as components the quantities related to all the subtimenodes including the initial one, even if $\bu^0=\bu_n$ is never changed along the iterations and it is not an input of the operators previously described.
In order to avoid confusion, we refer to the vectors not containing such component with the small letter and to the vectors containing it with the capital letter, i.e.,
\begin{equation}
\underline{\uvec{u}}^{(p)}=\left(
   \begin{array}{ccc}
   \uvec{u}^{1,(p)}\\
   \vdots\\
   \uvec{u}^{M,(p)}
   \end{array}
\right), \quad \underline{\uvec{U}}^{(p)}=\left(
   \begin{array}{ccc}
   \uvec{u}^{0}\\
   \underline{\uvec{u}}^{(p)}
   \end{array}
\right).
\end{equation}
We will also denote the component-wise application of $\uvec{G}$ to the vectors $\underline{\uvec{u}}^{(p)}$ and $\underline{\uvec{U}}^{(p)}$ by
\begin{equation}
\underline{\uvec{G}}(\underline{\uvec{u}}^{(p)})=\left(
   \begin{array}{ccc}
   \uvec{G}(t^{1},\uvec{u}^{1,(p)})\\
   \vdots\\
   \uvec{G}(t^{M},\uvec{u}^{M,(p)})
   \end{array}
\right), \quad \underline{\uvec{G}}(\underline{\uvec{U}}^{(p)})=\left(
   \begin{array}{ccc}
   \uvec{G}(t^{0},\uvec{u}^{0})\\
   \underline{\uvec{G}}(\underline{\uvec{u}}^{(p)})
   \end{array}
\right).
\end{equation}
With the previous definitions, it is possible to recast the general updating formula \eqref{eq:generalformulation} in the following compact form
\begin{align}
\begin{split}
\underline{\uvec{U}}^{(p)} & =  \underline{\uvec{U}}^{(0)}+\Delta t\Theta \underline{\uvec{G}}(\underline{\uvec{U}}^{(p-1)})+\Delta t \alpha \Gamma (\underline{\uvec{G}}(\underline{\uvec{U}}^{(p)}) -\underline{\uvec{G}}(\underline{\uvec{U}}^{(p-1)}))\\
& = \underline{\uvec{U}}^{(0)}+\Delta t(\Theta -\alpha \Gamma )\underline{\uvec{G}}(\underline{\uvec{U}}^{(p-1)})+\Delta t \alpha\Gamma \underline{\uvec{G}}(\underline{\uvec{U}}^{(p)}),
\end{split}
\label{eq:matricialformstandard}
\end{align}
where the vector $\underline{\uvec{U}}^{(0)}\in \mathbb{R}^{((M+1)\times Q)}$ and the matrices $\Theta,\Gamma \in \mathbb{R}^{(M+1)\times(M+1)}$ are defined as
\begin{align}
\underline{\uvec{U}}^{(0)}=\left(
   \begin{array}{ccc}
   \uvec{u}_{n}\\
   \vdots\\
   \uvec{u}_{n}
   \end{array}
\right), \,
\Theta=\begin{pmatrix}
0 & 0 & \dots & 0 \\
\theta^1_0 & \theta^1_1 & \dots & \theta^1_M \\
\theta^2_0 & \theta^2_1 & \dots & \theta^2_M \\
\vdots & \vdots & \ddots & \vdots \\
\theta^M_0 & \theta^M_1 & \dots & \theta^M_M \\
\end{pmatrix}, \,
\Gamma=\begin{pmatrix}
0 & 0 & \dots & 0 & 0 \\
\gamma^1 & 0 & \dots & 0 & 0 \\
\gamma^1 & \gamma^2 & \ddots & 0 & 0 \\
\vdots & \vdots & \ddots & \vdots & \vdots \\
\gamma^1 & \gamma^2 & \dots & \gamma^M  & 0\\
\end{pmatrix},
\label{eq:matricesstandard}
\end{align}
with the matrix $\Gamma$ being strictly lower-triangular, as the scheme is fully explicit.
Let us observe that the first component $\uvec{u}^{0}$ of $\underline{\uvec{U}}^{(p)}$ is never updated.
This is coherent with what we have said so far.
The matrices $\Theta$ and $\Gamma$ that we have defined are referred to a scalar ODE ($Q=1$). In case one wants to adapt them to a vectorial problem, they must be block-expanded. 

\section{Two novel families of DeC methods}\label{chapDECSTAG} 
In this section, we will show how to construct two novel families of efficient DeC methods by introducing a modification in the $\alpha$DeC methods, first focusing on equispaced subtimenodes and then extending the idea to \GL~ones. 
The modification is based on the following observation: at any iteration $p<M+1$, we get a solution $\underline{\uvec{u}}^{(p)}$ that is $p$-th order accurate using $M+1$ subtimenodes even though only $p$ would be formally sufficient to provide such accuracy. 
In other words, the number of subtimenodes is fixed a priori for all iterations in order to get the desired order of accuracy. These subtimenodes are used throughout the whole iterative process, although the formal order of accuracy, for which such nodes are required, is reached only in the final iteration. This represents indeed a waste of computational resources.

The proposed modification consists in starting with only two subtimenodes and increasing their number, iteration by iteration, matching the order of accuracy achieved in the specific iteration. In particular, we introduce intermediate interpolation processes between the iterations in order to retrieve the needed quantities in the new subtimenodes. 
The idea has been introduced in \cite{minion2003semi} for implicit methods, but without a systematic theory and related analytical study. We will present here two possible interpolation strategies which will lead to the definition of two general families of efficient DeC methods.

We will use the star symbol $*$ to refer to quantities obtained through the interpolation process.
The number of subtimenodes will change iteration by iteration, therefore, it is useful to define the vector $\underline{t}^{(p)}:=\left(t^{0,(p)}, \dots, t^{p,(p)} \right)^T$ of the subtimenodes in which we obtain the approximations of the solution at the $p$-th iteration, with $t^{0,(p)}=t_n$ and $t^{p,(p)}=t_{n+1}$. 

\subsection{$\alpha$DeCu}
The $\alpha$DeCu methods are obtained from the $\alpha$DeC methods by introducing an intermediate interpolation process on the solution $\uvec{u}(t)$ between the iterations.
For convenience, we will formulate the methods in terms of the vectors $\underline{\uvec{U}}^{(p)}$ containing the component $\uvec{u}^0=\uvec{u}_n$ associated to the initial subtimenode.
\begin{figure}
	\centering
	\includegraphics[width=0.8\textwidth]{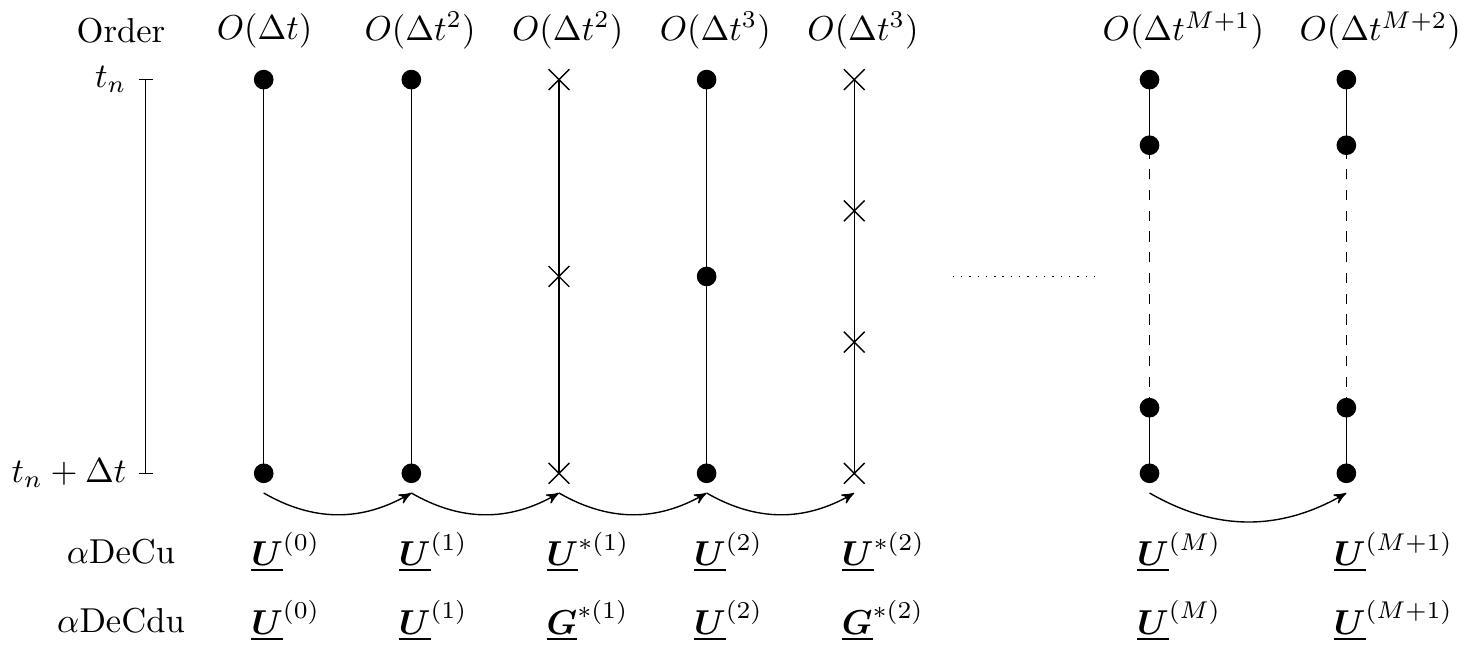}
    \caption{$\alpha$DeCu and $\alpha$DeCdu, sketches: dots for computed values, crosses for interpolated ones}\label{fig:sketchnewmethods}
\end{figure}

We start with $\underline{\uvec{U}}^{(0)}=(\uvec{u}_n,\uvec{u}_n)^T \in \mathbb{R}^{(2\times Q)}$ associated to two subtimenodes, $t_n$ and $t_n+\Delta t$, and we perform the first iteration
\begin{align}
\underline{\uvec{U}}^{(1)}  = \underline{\uvec{U}}^{(0)}+\Delta t(\Theta^{(1)} -\alpha \Gamma^{(1)} )\underline{\uvec{G}}(\underline{\uvec{U}}^{(0)})+\Delta t \alpha\Gamma^{(1)} \underline{\uvec{G}}(\underline{\uvec{U}}^{(1)}) \in \R^{(2\times Q)}.
\label{eq:first_iter1}
\end{align}
$\underline{\uvec{U}}^{(1)}$ is first order accurate and it yields an $O(\dt^2)$-accurate reconstruction on $[t_n,t_{n+1}]$. Here, $\Gamma^{(1)}$ and $\Theta^{(1)}$ are the operators associated to two subtimenodes.
Now, we perform the first interpolation, via a suitable interpolation matrix $H^{(1)}$, passing from two to three equispaced subtimenodes
\begin{align}
\begin{split}
\underline{\uvec{U}}^{*(1)} &= H^{(1)}\underline{\uvec{U}}^{(1)}\\
                            &=H^{(1)}\left[\underline{\uvec{U}}^{(0)}+\Delta t(\Theta^{(1)} -\alpha \Gamma^{(1)} )\underline{\uvec{G}}(\underline{\uvec{U}}^{(0)})+\Delta t \alpha\Gamma^{(1)} \underline{\uvec{G}}(\underline{\uvec{U}}^{(1)})\right]\\
                            &=\underline{\uvec{U}}^{(0)}_3+\Delta t H^{(1)}(\Theta^{(1)} -\alpha \Gamma^{(1)} )\underline{\uvec{G}}(\underline{\uvec{U}}^{(0)})+\Delta t \alpha H^{(1)}\Gamma^{(1)} \underline{\uvec{G}}(\underline{\uvec{U}}^{(1)}),
\end{split}
\label{eq:first_interp1}
\end{align}
where the last equality is due to the fact that, by consistency, the sum of the elements on the rows of the interpolation matrices $H^{(p)}$ is equal to $1$. The subscript $3$ has been added to $\underline{\uvec{U}}^{(0)}_3 \in \mathbb{R}^{3\times Q}$ to distinguish it from the initial $\underline{\uvec{U}}^{(0)}\in \mathbb{R}^{2\times Q}$. Now, we have  $\underline{\uvec{U}}^{*(1)}\in \R^{(3\times Q)}$, still first order accurate.
Then, we perform the second iteration
\begin{align}
\underline{\uvec{U}}^{(2)}  = \underline{\uvec{U}}^{(0)}_3+\Delta t(\Theta^{(2)} -\alpha \Gamma^{(2)} )\underline{\uvec{G}}(\underline{\uvec{U}}^{*(1)})+\Delta t \alpha\Gamma^{(2)} \underline{\uvec{G}}(\underline{\uvec{U}}^{(2)}),
\label{eq:first_iter2}
\end{align}
which gives a second order accurate approximation, i.e., an $O(\dt^3)$-accurate approximation. 
Thus, we continue with another interpolation
\begin{align}
\begin{split}
\underline{\uvec{U}}^{*(2)} &= H^{(2)}\underline{\uvec{U}}^{(2)}\\
                            &=H^{(2)}\left[\underline{\uvec{U}}^{(0)}_3+\Delta t(\Theta^{(2)} -\alpha \Gamma^{(2)} )\underline{\uvec{G}}(\underline{\uvec{U}}^{*(1)})+\Delta t \alpha\Gamma^{(2)} \underline{\uvec{G}}(\underline{\uvec{U}}^{(2)})\right]\\
                            &=\underline{\uvec{U}}^{(0)}_4+\Delta t H^{(2)}(\Theta^{(2)} -\alpha \Gamma^{(2)} )\underline{\uvec{G}}(\underline{\uvec{U}}^{*(1)})+\Delta t \alpha H^{(2)}\Gamma^{(2)} \underline{\uvec{G}}(\underline{\uvec{U}}^{(2)}),
\end{split}                     
\label{eq:first_interp2}
\end{align}
from which we can get $\underline{\uvec{U}}^{(3)}$ $O(\dt^4)$-accurate and so on.
Proceeding iteratively, at the $p$-th iteration we have
\begin{align}
	\begin{split}
\underline{\uvec{U}}^{*(p-1)}  = \underline{\uvec{U}}^{(0)}_{p+1}&+\Delta t H^{(p-1)}(\Theta^{(p-1)} -\alpha \Gamma^{(p-1)} )\underline{\uvec{G}}(\underline{\uvec{U}}^{*(p-2)})\\
&+\Delta t \alpha H^{(p-1)}\Gamma^{(p-1)} \underline{\uvec{G}}(\underline{\uvec{U}}^{(p-1)}),
	\end{split}
\label{eq:first_interpp}\\
\underline{\uvec{U}}^{(p)}  = \underline{\uvec{U}}^{(0)}_{p+1}&+\Delta t (\Theta^{(p)} -\alpha \Gamma^{(p)} )\underline{\uvec{G}}(\underline{\uvec{U}}^{*(p-1)})+\Delta t \alpha \Gamma^{(p)} \underline{\uvec{G}}(\underline{\uvec{U}}^{(p)}),
\label{eq:first_iterpp}
\end{align}
where $\underline{\uvec{U}}^{*(p-1)}  \in \R^{(p+1)\times Q}$ is $O(\dt^p)$-accurate and $\underline{\uvec{U}}^{(p)}\in  \R^{(p+1)\times Q}$ is $O(\dt^{p+1})$-accurate.
Clearly, the DeC operators $\Theta^{(p)}$ and  $\Gamma^{(p)}$, used in the $p$-th iteration, are chosen according to the dimension of involved variables.

Let us notice that $\underline{\uvec{U}}^{(p)}\in \mathbb{R}^{((p+1)\times Q)}$, got at the $p$-th iteration, is $O(\Delta t^{p+1})$-accurate and associated to $p+1$ subtimenodes but, actually, they would be enough to guarantee $O(\Delta t^{p+2})$-accuracy. 
For this reason, if the final number of subtimenodes is fixed to be $M+1$, the optimal choice is to perform $M$ iterations to reach such setting and a final $(M+1)$-th iteration without interpolation to saturate the $O(\Delta t^{M+2})$-accuracy associated to the subtimenodes. 
In this way, we have that the interpolation is performed at each iteration except the first and the last one. 
Thus, the last iteration reads
\begin{align}
\underline{\uvec{U}}^{(M+1)}\!\!  = \underline{\uvec{U}}^{(0)}_{M+1}\!+\Delta t (\Theta^{(M)}\!\! -\alpha \Gamma^{(M)} )\underline{\uvec{G}}(\underline{\uvec{U}}^{(M)})\!+\Delta t \alpha \Gamma^{(M)} \underline{\uvec{G}}(\underline{\uvec{U}}^{(M+1)}),
\label{eq:first_last}
\end{align}
where the matrices $\Theta^{(M)}$ and $\Gamma^{(M)}$ are the ones used also for the $M$-th iteration.
A useful sketch of the algorithm is represented in Figure~\ref{fig:sketchnewmethods}.

On the other hand, one could also not fix a priori the final number of subtimenodes and stop when certain conditions are met, see an example for adaptive methods in Section~\ref{chapADAPT}.

\subsection{$\alpha$DeCdu}
Like the $\alpha$DeCu methods, the $\alpha$DeCdu methods are based on the introduction of an interpolation process between the iterations. In this case, the interpolated quantity is the function $\uvec{G}(t,\uvec{u}(t))$.
The name is due to the fact that formally we interpolate $\frac{d}{dt}\uvec{u}(t)=\uvec{G}(t,\uvec{u}(t))$.

We start with two subtimenodes, associated to $t_n$ and $t_n+\Delta t$, and $\underline{\uvec{U}}^{(0)}\in \mathbb{R}^{(2\times Q)}$ and we perform the first iteration of the $\alpha$DeC method, as in \eqref{eq:first_iter1}, getting $\underline{\uvec{U}}^{(1)}\in \mathbb{R}^{(2\times Q)}$, which is $O(\Delta t^2)$-accurate.
Then, we can compute $\underline{\uvec{G}}(\underline{\uvec{U}}^{(1)})$, whose components allow to get an $O(\Delta t^2)$-accurate global reconstruction of $\uvec{G}(t,\uvec{u}(t))$ in the interval $[t_n,t_n+\Delta t]$ through Lagrange interpolation.
We thus perform an interpolation to retrieve the approximated values of $\uvec{G}(t,\uvec{u}(t))$ in three equispaced subtimenodes in the interval $[t_n,t_n+\Delta t]$, getting
$\underline{\uvec{G}}^{*(1)}=H^{(1)}\underline{\uvec{G}}(\underline{\uvec{U}}^{(1)})\in \mathbb{R}^{(3\times Q)}$. Then, we  compute
\begin{align}
	\begin{split}
		\underline{\uvec{U}}^{(2)}  &= \underline{\uvec{U}}^{(0)}_3+\Delta t(\Theta^{(2)} -\alpha \Gamma^{(2)} )\underline{\uvec{G}}^{*(1)} +\Delta t \alpha\Gamma^{(2)} \underline{\uvec{G}}(\underline{\uvec{U}}^{(2)})\\
		&= \underline{\uvec{U}}^{(0)}_3+\Delta t(\Theta^{(2)} -\alpha \Gamma^{(2)} )H^{(1)}\underline{\uvec{G}}(\underline{\uvec{U}}^{(1)})+\Delta t \alpha\Gamma^{(2)} \underline{\uvec{G}}(\underline{\uvec{U}}^{(2)}),
	\end{split}
	\label{eq:second_iter2}
\end{align}
which is in $\mathbb{R}^{(3\times Q)}$ and $O(\Delta t^3)$-accurate. 
We can iteratively continue with interpolations, $\underline{\uvec{G}}^{*(p-1)}=H^{(p-1)}\underline{\uvec{G}}(\underline{\uvec{U}}^{(p-1)})$, and iterations, obtaining the general updating formula
\begin{align}
	\underline{\uvec{U}}^{(p)}  = \underline{\uvec{U}}^{(0)}_{p+1}+\Delta t(\Theta^{(p)} -\alpha \Gamma^{(p)} )H^{(p-1)}\underline{\uvec{G}}(\underline{\uvec{U}}^{(p-1)})+\Delta t \alpha\Gamma^{(p)} \underline{\uvec{G}}(\underline{\uvec{U}}^{(p)}),
	\label{eq:second_iterp}
\end{align}
with $\underline{\uvec{U}}^{(p)}\in\mathbb{R}^{((p+1)\times Q)}$ and $O(\Delta t^{p+1})$-accurate.
Analogous considerations, as for the $\alpha$DeCu method, hold on the advantage of performing a final iteration with no interpolation when the final number of subtimenodes is fixed.
Also in this case, the reader is referred to Figure~\ref{fig:sketchnewmethods} for a better understanding of the method.

\subsection{$\alpha$DeCu and $\alpha$DeCdu with Gauss--Lobatto subtimenodes}
As already explained, $M+1$ \GL~subtimenodes can guarantee an accuracy equal to $2M$. 
In such a case, if the final number of subtimenodes is fixed, we start with two subtimenodes and we alternate iterations 
of the $\alpha$DeC method and interpolations as in the equispaced case, adding one subtimenode at 
each iteration until reaching the desired $M+1$ subtimenodes, 
then, we continue with normal iterations of the $\alpha$DeC until $P=2M$
to get the maximal order of accuracy associated to such choice. 
The updating formulas are identical to the ones already presented. The interpolation is not performed at the first iteration and from the $(M+1)$-th iteration on.
On the other hand, if the order $P$ is fixed, the most efficient choice is given by a final number of subtimenodes equal to $M+1$ with $M=\left \lceil \frac{P}{2}\right \rceil $  and $P$ iterations. 

Contrarily to what one might think, it is not possible to postpone an interpolation process after the saturation of the maximal accuracy associated to some intermediate number of \GL~subtimenodes adopted in the early iterations. The interpolation processes must mandatorily take place in the first iterations. This is due to the mismatch between the $O(\Delta t^{2p+1})$-accuracy of the operator $\lopdt^2$ associated to $p+1$ \GL~subtimenodes and the $O(\Delta t^{p+1})$-accuracy of the interpolation process with the same number of subtimenodes.


\section{The DeC as RK}\label{chapRK}
An explicit RK method with $S$ stages applied in the interval $[t_n,t_{n+1}]$  reads
\begin{equation}\label{eq:RK}
\begin{cases}
	\by^0=\uvec{u}_n,\\
	\by^s = \uvec{u}_n + \dt \sum_{r = 0}^{s-1} a_{s,r} \bG(t_n+c_r \dt, \by^r), \quad \text{for }s=1,\dots,S-1,\\
	\uvec{u}_{n+1}  = \uvec{u}_n + \dt \sum_{r = 0}^{S-1} b_{r} \bG(t_n+c_r \dt, \by^r) .
\end{cases}	 
\end{equation}
The coefficients $a_{sr}$, $c_r$ and $b_r$ uniquely characterize the RK method and can be stored, respectively, into the strictly lower triangular matrix $A$ and the vectors $\uvec{c}$ and $\uvec{b}$, often summarized in a Butcher tableau
$$\begin{array}{c|c}
	\uvec{c}& A  \\
	\hline
	& \uvec{b} 
\end{array}.
$$
It is well known, as presented in \cite{liu2008strong,ketcheson2014comparison,han2021dec}, that DeC methods can be written into RK form. This also holds for the new methods, $\alpha$DeCu and $\alpha$DeCdu. In this section, we will explicitly construct their Butcher tableaux.
We will adopt a zero-based numeration and the following convention for slicing.
If $\mathcal{M}\in \R^{D_0\times D_1}$, we denote by $\mathcal{M}_{i:j,k:\ell}$ its slice from the $i$-th row to the $j$-th row (included) and from the $k$-th column to the $\ell$-th column (included). 
We omit the last (first) index in case we want to include all the entries until the end (from the beginning), e.g., $\mathcal{M}_{1:,:6}=M_{1:D_0-1,0:6}$. %
The same notation is assumed for vectors.
%
%
We define also the vectors $\underline{\beta}^{(p)}:=\left(0,\frac{t^{1,(p)}-t_n}{\Delta t}, \dots, \frac{t^{p,(p)}-t_n}{\Delta t} \right)^T$ of the $\beta^m$ coefficients in the different iterations of the new methods and, for the original $\alpha$DeC method, the fixed vector $\underline{\beta}:=\left(0,\frac{t^{1}-t_n}{\Delta t}, \dots, \frac{t^{M}-t_n}{\Delta t} \right)^T$.
In order make the Butcher tableaux as compact as possible, the computation of the solution in the different subtimenodes at the first iteration will be always made through the explicit Euler method. This little modification has no impact on the formal accuracy, since the first iteration is meant to provide a first order approximation of the solution. 

We will focus on equispaced subtimenodes. The extension to the \GL~case is trivial: it suffices to repeat the block without interpolation, related to the final iteration of the standard method, for the needed number of times, $M-1$ in the optimal case.

\subsection{$\alpha$DeC}
We recall the general updating formula of the $\alpha$DeC methods in matricial form
\begin{align}
\underline{\uvec{U}}^{(p)}  = \underline{\uvec{U}}^{(0)}+\Delta t(\Theta -\alpha \Gamma )\underline{\uvec{G}}(\underline{\uvec{U}}^{(p-1)})+\Delta t \alpha\Gamma \underline{\uvec{G}}(\underline{\uvec{U}}^{(p)}).
\end{align}
If we align each iteration one after the other and we consider the approximation in each subtimenode of each iteration as a RK stage, we can pass to the RK formulation. 
Indeed, we do not repeat the redundant states, i.e., all the $\bu^{0,(p)}=\uvec{u}_n$, and we keep only $\bu^0$ as representative of all of them. 
This leads to the RK formulation \eqref{eq:RK} with Butcher tableau as in Table~\ref{tab:RKalphaDeC}, where we added on top and on the right side the references to the different iteration steps. 
The number of stages of this formulation amounts to $S=M P$ for any type of subtimenodes. 
If $\alpha=0$, the $\alpha$DeC method reduces to the bDeC method and the Butcher tableau simplifies to Table~\ref{tab:RKbDeC}. In such case, we observe that we do not need the whole vector $\bbu^{(P)}$, but we can just compute the component associated to the final subtimenode with the only $\bbu^{(P-1)}$, leading to a total number of RK stages equal to $S=M(P-1)+1$. 

\begin{table}
	\centering
\resizebox{0.9\columnwidth}{!}{
	\begin{tabular}{|c||ccccccc|l|}
		\hline
		$\uvec{c}$& $\bu^0$ & $\bbu^{(1)}$& $\bbu^{(2)}$&$\bbu^{(3)}$& $\cdots$& $\bbu^{(M)}$& $\bbu^{(M+1)}_{:M-1}$&A\\
		\hline
		$0$ & 0 & &&&&&&$\bu^0$\\
		$\vecbeta_{1:}$& $\vecbeta_{1:}$ & $\matz$ &&&&&&$\bbu^{(1)}$ \\
		$\vecbeta_{1:}$& $\Theta_{1:,0}$ & $\!\!\!(\Theta-\alpha \Gamma)_{1:,1:}\!\!\!\!\!\!$ &$\alpha \Gamma_{1:,1:}$ & $\matz$&&&&$\bbu^{(2)}$\\
		$\vecbeta_{1:}$& $\Theta_{1:,0}$&  $\matz$& $\!\!\!(\Theta-\alpha \Gamma)_{1:,1:}\!\!\!\!$ &$\alpha \Gamma_{1:,1:}$ & $\matz$&&&$\bbu^{(3)}$\\
		&$\vdots$   &$\vdots$ &&$\ddots$& $\ddots$&&& $\vdots$\\
		&$\vdots$   &$\vdots$ &&&$\ddots$& $\ddots$&& $\vdots$\\
		$\vecbeta_{1:M-1}$ &$\!\!\Theta_{1:M-1,0}\!\!\!\!\!\!\!$ &  $\matz$&$\cdots$&$\cdots$&$\matz$&$\!\!\!\!\!\!(\Theta-\alpha \Gamma)_{1:M-1,1:}\!\!\!\!$&$\alpha \Gamma_{1:M-1,1:M-1}$ &$\bbu^{(M+1)}_{:M-1}$\\ \hline\hline
		$\uvec{b}$&$\Theta_{M,0}$ &  $\vecz$&$\cdots$&$\cdots$&$\vecz$&$(\Theta-\alpha \Gamma)_{M,1:}$&$\alpha \Gamma_{M,1:M-1}$ & $\bbu^{M,(M+1)}$\\ \hline
	\end{tabular} }
	\caption{RK structures for the original $\alpha$DeC with equispaced subtimenodes, $\uvec{c}$ at the left $\uvec{b}$ at the bottom, $A$ in the middle \label{tab:RKalphaDeC}}
\end{table}

\begin{table}
	\centering
	\begin{tabular}{|c||ccccccc|l|}
		\hline
		$\uvec{c}$& $\bu^0$ & $\bbu^{(1)}$& $\bbu^{(2)}$&$\bbu^{(3)}$& $\cdots$& $\bbu^{(M-1)}$& $\bbu^{(M)}$&A\\
		\hline
		$0$ & 0 & &&&&&&$\bu^0$\\
		$\vecbeta_{1:}$& $\vecbeta_{1:}$ & $\matz$ &&&&&&$\bbu^{(1)}$ \\
		$\vecbeta_{1:}$& $\Theta_{1:,0}$ & $\Theta_{1:,1:}$ & $\matz$&&&&&$\bbu^{(2)}$\\
		$\vecbeta_{1:}$& $\Theta_{1:,0}$&  $\matz$& $\Theta_{1:,1:}$ & $\matz$&&&&$\bbu^{(3)}$\\
		&$\vdots$   &$\vdots$ &&$\ddots$& $\ddots$&&& $\vdots$\\
		&$\vdots$   &$\vdots$ &&&$\ddots$& $\ddots$&& $\vdots$\\
		$\vecbeta_{1:}$ &$\Theta_{1:,0}$ &  $\matz$&$\cdots$&$\cdots$&$\matz$&$\Theta_{1:,1:}$&$\matz$&$\bbu^{(M)}$\\ \hline\hline
		$\uvec{b}$&$\Theta_{M,0}$ &  $\vecz$&$\cdots$&$\cdots$&$\cdots$&$\vecz$&$\Theta_{M,1:}$&$\bbu^{M,(M+1)}$\\ \hline
	\end{tabular}
	\caption{RK structures for the original bDeC with equispaced subtimenodes, $\uvec{c}$ at the left $\uvec{b}$ at the bottom, $A$ in the middle  \label{tab:RKbDeC}}
\end{table}
 

\subsection{bDeCu}
Let us recall the general updating formulas of the $\alpha$DeCu methods 
\begin{align}
\begin{split}
\underline{\uvec{U}}^{*(p-1)}  = \underline{\uvec{U}}^{(0)}_{p+1}&+\Delta t H^{(p-1)}(\Theta^{(p-1)} -\alpha \Gamma^{(p-1)} )\underline{\uvec{G}}(\underline{\uvec{U}}^{*(p-2)})\\
&+\Delta t \alpha H^{(p-1)}\Gamma^{(p-1)} \underline{\uvec{G}}(\underline{\uvec{U}}^{(p-1)}),
\end{split}
\label{eq:alphaDeCu_ustar}\\
\underline{\uvec{U}}^{(p)}  = \underline{\uvec{U}}^{(0)}_{p+1}&+\Delta t (\Theta^{(p)} -\alpha \Gamma^{(p)} )\underline{\uvec{G}}(\underline{\uvec{U}}^{*(p-1)})+\Delta t \alpha \Gamma^{(p)} \underline{\uvec{G}}(\underline{\uvec{U}}^{(p)}),
\label{eq:alphaDeCu_u}
\end{align}
to which we need to add an initial iteration made with Euler and either a final iteration or, in the context of \GL~subtimenodes, some final iterations ($M$ in the optimal case) of the standard $\alpha$DeC method performed without interpolation. 
In this case, the stages of the RK method are given by all the components of the vectors $\underline{\uvec{U}}^{(p)}$ and $\underline{\uvec{U}}^{*(p)}$ (excluding the redundant states). 
From easy computations, one can see that for $\alpha \neq 0$ the number of stages of the $\alpha$DeCu method coincides with the number of stages of the $\alpha$DeC method without computational advantage under this point of view. 
For this reason, 
we focus on the bDeCu method ($\alpha = 0$), for which we have a substantial computational advantage. 
In such case, the updating formulas \eqref{eq:alphaDeCu_ustar} and \eqref{eq:alphaDeCu_u} reduce to
\begin{align}
&\underline{\uvec{U}}^{*(p-1)}  = \underline{\uvec{U}}^{(0)}_{p+1}+\Delta t H^{(p-1)}\Theta^{(p-1)} \underline{\uvec{G}}(\underline{\uvec{U}}^{*(p-2)}),
\label{eq:alphaDeCu_ustar_new}\\
&\underline{\uvec{U}}^{(p)}  = \underline{\uvec{U}}^{(0)}_{p+1}+\Delta t \Theta^{(p)} \underline{\uvec{G}}(\underline{\uvec{U}}^{*(p-1)}).
\label{eq:alphaDeCu_u_new}
\end{align}
The right-hand sides of the previous equations involve the computation of $\uvec{G}$ in interpolated states $\underline{\uvec{U}}^*$ only and, in particular, the update of $\underline{\uvec{U}}^{*(p-1)}$ only depends on $\underline{\uvec{U}}^{*(p-2)}$.  This means that the scheme can be rewritten in terms of the vectors $\underline{\uvec{U}}^{*(p)}$ only (plus $\bbU^{M,(P)}$), drastically reducing the number of stages. The RK coefficients are reported in Table~\ref{tab:RK_bDeCu}, in which we have
\begin{align}
	W^{(p)} : = \begin{cases}
		H^{(p)}\Theta^{(p)} \in \R^{(p+2)\times(p+1)}, &\text{if } p =2, \dots, M-1,\\
		\Theta^{(M)} \in \R^{(M+1)\times (M+1)}, & \text{if } p \geq M.
	\end{cases}
\end{align}
The total number of RK stages is given by $S=M (P-1) +1 -\frac{(M-1)(M-2)}{2} $, so $\frac{(M-1)(M-2)}{2}$ less with respect to the original method. 
The formula holds for both equispaced and \GL~subtimenodes.
\begin{table}
	\centering
\resizebox{0.9\columnwidth}{!}{
	\begin{tabular}{|c||cccccccc|l||r|}
		\hline
		$\uvec{c}$& $\bu^0$ & $\!\!\bbu^{*(1)}\!\!$& $\!\!\bbu^{*(2)}\!\!$&$\!\!\bbu^{*(3)}\!\!$& $\!\!\cdots\!\!$& $\!\!\bbu^{*(M-2)}\!\!$& $\!\!\bbu^{*(M-1)}\!\!$& $\!\!\bbu^{(M)}\!\!$&A&$\!\!$dim$\!\!$\\
		\hline
		$0$ & 0 & &&&&&&&$\bu^0$&1\\
		$\vecbeta^{(2)}_{1:}$& $\vecbeta^{(2)}_{1:}$ & $\matz$ &&&&&&&$\bbu^{*(1)}$&2 \\
		$\vecbeta^{(3)}_{1:}$& $\!\!W^{(2)}_{1:,0}\!\!$ & $\!\!\!\!W^{(2)}_{1:,1:}\!\!\!\!$ & $\matz$&&&&&&$\bbu^{*(2)}$&3\\
		$\vecbeta^{(4)}_{1:}$& $W^{(3)}_{1:,0}$&  $\matz$& $\!\!\!\!W^{(3)}_{1:,1:}\!\!\!\!$ & $\matz$&&&&&$\bbu^{*(3)}$&4\\
		&$\vdots$   &$\vdots$ &&$\ddots$& $\ddots$&&&& $\vdots$&$\vdots$\\
		&$\vdots$   &$\vdots$ &&&$\ddots$& $\ddots$&&& $\vdots$&$\vdots$\\
		$\vecbeta^{(M)}_{1:}\!\!$ &$\!\!W^{(M-1)}_{1:,0}\!\!\!\!$ &  $\matz$&$\cdots$&$\cdots$&$\matz$ &$\!\!\!\!W^{(M-1)}_{1:,1:}\!\!\!\!$&$\matz$&$\matz$&$\bbu^{*(M-1)}\!\!$&$M$\\
		$\vecbeta^{(M)}_{1:}\!\!$&$W^{(M)}_{1:,0}$ &   $\matz$&$\cdots$&$\cdots$&$\cdots$&$\matz$&$\!\!\!\!W^{(M)}_{1:,1:}\!\!\!\!$&$\matz$&$\bbu^{(M)}$&$M$\\ \hline\hline
		$\uvec{b}$&$\!\!W^{(M+1)}_{M,0}$ &  $\vecz$&$\cdots$&$\cdots$&$\cdots$&$\cdots$&$\vecz$&$\!\!\!\!W^{(M+1)}_{M,1:}\!\!$&$\bbu^{M,(M+1)}\!\!$&\\ \hline
	\end{tabular} }
	\caption{RK structures for the bDeCu method, $\uvec{c}$ at the left $\uvec{b}$ at the bottom, $A$ in the middle
	 \label{tab:RK_bDeCu}}
\end{table}
\begin{remark}[On the relation between stages and computational cost]\label{remark_on_alphadecu}
The number of stages is not completely explanatory of the computational costs of the new algorithms. In the context of the novel methods, the cost associated to the computation of the different stages is not homogeneous, especially in applications to PDEs, as some of them are ``properly'' computed through the updating formula \eqref{eq:generalformulation} of the original scheme, while the others are got through an interpolation process which is much cheaper. As an example, \eqref{eq:alphaDeCu_ustar} can be computed as $\underline{\uvec{U}}^{*(p-1)}=H^{(p-1)} \underline{\uvec{U}}^{(p-2)}$.
In particular, as already specified, the novel $\alpha$DeCu methods for $\alpha \neq 0$ are characterized by the same number of stages as the original $\alpha$DeC, nevertheless, roughly half of them is computed through interpolation. For this reason, they have been numerically investigated for $\alpha=1$.
\end{remark}

\subsection{$\alpha$DeCdu}
Again, we start by recalling the updating formulas of the method 
\begin{align}
\underline{\uvec{U}}^{(p)}  = \underline{\uvec{U}}^{(0)}_{p+1}+\Delta t(\Theta^{(p)} -\alpha \Gamma^{(p)} )H^{(p-1)}\underline{\uvec{G}}(\underline{\uvec{U}}^{(p-1)})+\Delta t \alpha\Gamma^{(p)} \underline{\uvec{G}}(\underline{\uvec{U}}^{(p)}),
\label{eq:alphaDeCdu}
\end{align}
supplemented with an initial Euler step and a final iteration or, for \GL~subtimenodes, at most $M$ final iterations of $\alpha$DeC without interpolation.
The usual identification of subtimenodes and RK stages leads to the Butcher tableau
in Table~\ref{tab:RK_alphaDeCdu}, in which we have
\begin{align}
	X^{(p)} : = \begin{cases}
		(\Theta^{(p)} - \alpha \Gamma^{(p)} ) H^{(p-1)} \in \R^{(p+1)\times p}, &\text{if } p =2, \dots, M,\\
		\Theta^{(M)} - \alpha \Gamma^{(M)} \in \R^{(M+1)\times (M+1)}, & \text{if } p >M,
	\end{cases}\\ 
	Y^{(p)} : = \begin{cases}
    \alpha \Gamma^{(p)}  \in \R^{(p+1)\times (p+1)}, &\text{if } p =2, \dots, M,\\
	\alpha \Gamma^{(M)} \in \R^{(M+1)\times (M+1)}, & \text{if } p >M.
\end{cases}
\end{align}
The number of stages in this case amounts to $S=MP-\frac{M(M-1)}{2}$, with a computational advantage of $\frac{M(M-1)}{2}$ with respect to the original method.
\begin{table}
	\centering
	\resizebox{0.9\columnwidth}{!}{
	\begin{tabular}{|c||ccccccccc|l||r|}
		\hline
		$\uvec{c}$& $\bu^0$ & $\bbu^{(1)}$& $\bbu^{(2)}$&$\bbu^{(3)}$& $\cdots$& $\!\!\bbu^{(M-2)}\!\!$& $\bbu^{(M-1)}\!\!$& $\bbu^{(M)}$&$\bbu^{(M+1)}_{:M-1}$&A&dim\\
		\hline
		$0$ & 0 & &&&&&&&&$\bu^0$&1\\
		$\vecbeta^{(1)}_{1:}$& $\vecbeta^{(1)}_{1:}$ & $\matz$ &&&&&&&&$\bbu^{(1)}$&1 \\
		$\vecbeta^{(2)}_{1:}$& $X^{(2)}_{1:,0}$ & $\!\!\!\!\!X^{(2)}_{1:,1:}\!\!$ & $\!\!Y^{(2)}_{1:,1:}\!\!$&&&&&&&$\bbu^{(2)}$&2\\
		$\vecbeta^{(3)}_{1:}$& $X^{(3)}_{1:,0}\!\!$&  $\matz$& $\!\!X^{(3)}_{1:,1:}\!\!$ & $\!\!Y^{(3)}_{1:,1:}\!\!$&&&&&&$\bbu^{(3)}$&3\\
		&$\vdots$   &$\vdots$ &&$\ddots$& $\ddots$&&&&& $\vdots$&$\vdots$\\
		&$\vdots$   &$\vdots$ &&&$\ddots$& $\ddots$&&&& $\vdots$&$\vdots$\\
		$\!\!\vecbeta^{(M-1)}_{1:}\!\!$ &$X^{(M-1)}_{1:,0}\!\!\!\!$ &  $\matz$&$\cdots$&$\cdots$&$\matz$ &$\!\!X^{(M-1)}_{1:,1:}\!\!\!$&$\!Y^{(M-1)}_{1:,1:}\!\!\!\!$&$\matz$&&$\bbu^{(M-1)}\!\!\!\!$&$\!\!M-1\!\!$\\
		$\!\!\vecbeta^{(M)}_{1:}\!\!$&$X^{(M)}_{1:,0}\!\!$ &   $\matz$&$\cdots$&$\cdots$&$\cdots$&$\matz$&$\!\!X^{(M)}_{1:,1:}\!\!$&$\!\!Y^{(M)}_{1:,1:}\!\!$&&$\bbu^{(M)}\!\!$&$M$\\ 
		$\!\!\vecbeta^{(M)}_{1:M-1}\!\!$&$\!\!X^{(M+1)}_{1:M-1,0}\!\!\!\!$ &   $\matz$&$\cdots$&$\cdots$&$\cdots$&$\cdots$&$\matz$&$\!\!\!\!\!\!X^{(M+1)}_{1:M-1,1:}\!\!\!\!$&$\!\!\!\!Y^{(M+1)}_{1:M-1,1:M-1}\!$&$\bbu^{(M+1)}_{1:M-1}\!\!\!\!$&$\!\!M-1\!\!$\\ \hline\hline
		$\uvec{b}$&$X^{(M+1)}_{M,0}\!\!\!$ &  $\vecz$&$\cdots$&$\cdots$&$\cdots$&$\cdots$&$\vecz$&$X^{(M+1)}_{M,1:}\!\!$&$\!\!Y^{(M+1)}_{M,1:M-1}\!\!$&$\bbu^{M,(M+1)}\!\!$&\\ \hline
	\end{tabular} } 
	\caption{RK structures for the $\alpha$DeCdu method with equispaced subtimenodes, $\uvec{c}$ at the left $\uvec{b}$ at the bottom, $A$ in the middle  \label{tab:RK_alphaDeCdu}}
\end{table}
Also in this case, it is worth giving a particular attention to the method given by $\alpha=0$. Again, the possibility to compute $\bbu^{M,(P)}$ without any need for the other components of $\bbu^{(P)}$ further reduces the number of stages to $S=M(P-1) +1 -\frac{M(M-1)}{2}$. 

We conclude this section with two tables, Table~\ref{tab:stagesNumbersEquispaced} and Table~\ref{tab:stagesNumbersGL}, 
containing the number of stages of the original methods and of the novel ones, respectively for equispaced and \GL~subtimenodes, up to order $13$ 
with associated theoretical speed up factors computed as the ratios between the stages of the original
methods and the stages of the modified methods.

\begin{table}
	\centering	
	\resizebox{0.95\columnwidth}{!}{
	\begin{tabular}{|c|c||c|c|c||c|c|c|c|c|}\hline
		\multicolumn{2}{|c||}{} & \multicolumn{3}{c||}{$\alpha$DeC} & \multicolumn{5}{c|}{bDeC}\\ \hline
		\multicolumn{2}{|c||}{} & \multicolumn{2}{c|}{RK stages}& speed up & \multicolumn{3}{c|}{RK stages} & \multicolumn{2}{c|}{speed up}\\ \hline
P & M & $\alpha$DeC/$\alpha$DeCu & $\alpha$DeCdu & $\alpha$DeCdu   & bDeC & bDeCu  & bDeCdu & bDeCu &  bDeCdu\\ \hline
2  &     1  &    2  &    2 & 1.000 &   2   &    2 &   2 & 1.000 & 1.000 \\ 
3  &     2  &    6  &    5 & 1.200 &   5   &    5 &   4 & 1.000 & 1.250 \\ 
4  &     3  &   12  &    9 & 1.333 &  10   &    9 &   7 & 1.111 & 1.429 \\ 
5  &     4  &   20  &   14 & 1.429 &  17   &   14 &  11 & 1.214 & 1.545 \\ 
6  &     5  &   30  &   20 & 1.500 &  26   &   20 &  16 & 1.300 & 1.625 \\ 
7  &     6  &   42  &   27 & 1.556 &  37   &   27 &  22 & 1.370 & 1.682 \\ 
8  &     7  &   56  &   35 & 1.600 &  50   &   35 &  29 & 1.429 & 1.724 \\ 
9  &     8  &   72  &   44 & 1.636 &  65   &   44 &  37 & 1.477 & 1.757 \\ 
10  &     9  &   90  &   54 & 1.667 &  82   &   54 &  46 & 1.519 & 1.783 \\ 
11  &    10  &  110  &   65 & 1.692 & 101   &   65 &  56 & 1.554 & 1.804 \\ 
12  &    11  &  132  &   77 & 1.714 & 122   &   77 &  67 & 1.584 & 1.821 \\ 
13  &    12  &  156  &   90 & 1.733 & 145   &   90 &  79 & 1.611 & 1.835 \\ 
 
 \hline		
	\end{tabular}
}
	\caption{Number of stages for the original ($\alpha$DeC, bDeC) and novel ($\alpha$DeCu, $\alpha$DeCdu, bDeCu, bDeCdu) methods with equispaced subtimenodes and speed up factor
	\label{tab:stagesNumbersEquispaced}}
\end{table}

\begin{table}
	\centering
	\resizebox{0.95\columnwidth}{!}{
	\begin{tabular}{|c|c||c|c|c||c|c|c|c|c|}\hline
	\multicolumn{2}{|c||}{} & \multicolumn{3}{c||}{$\alpha$DeC} & \multicolumn{5}{c|}{bDeC}\\ \hline
	\multicolumn{2}{|c||}{} & \multicolumn{2}{c|}{RK stages}& speed up & \multicolumn{3}{c|}{RK stages} & \multicolumn{2}{c|}{speed up}\\ \hline
	P & M & $\alpha$DeC/$\alpha$DeCu & $\alpha$DeCdu & $\alpha$DeCdu   & bDeC & bDeCu  & bDeCdu & bDeCu &  bDeCdu\\ \hline
  2  &     1  &    2  &    2 & 1.000 &   2   &    2 &   2 & 1.000 & 1.000 \\ 
3  &     2  &    6  &    5 & 1.200 &   5   &    5 &   4 & 1.000 & 1.250 \\ 
4  &     2  &    8  &    7 & 1.143 &   7   &    7 &   6 & 1.000 & 1.167 \\ 
5  &     3  &   15  &   12 & 1.250 &  13   &   12 &  10 & 1.083 & 1.300 \\ 
6  &     3  &   18  &   15 & 1.200 &  16   &   15 &  13 & 1.067 & 1.231 \\ 
7  &     4  &   28  &   22 & 1.273 &  25   &   22 &  19 & 1.136 & 1.316 \\ 
8  &     4  &   32  &   26 & 1.231 &  29   &   26 &  23 & 1.115 & 1.261 \\ 
9  &     5  &   45  &   35 & 1.286 &  41   &   35 &  31 & 1.171 & 1.323 \\ 
10  &     5  &   50  &   40 & 1.250 &  46   &   40 &  36 & 1.150 & 1.278 \\ 
11  &     6  &   66  &   51 & 1.294 &  61   &   51 &  46 & 1.196 & 1.326 \\ 
12  &     6  &   72  &   57 & 1.263 &  67   &   57 &  52 & 1.175 & 1.288 \\ 
13  &     7  &   91  &   70 & 1.300 &  85   &   70 &  64 & 1.214 & 1.328 \\ 
  \hline		
	\end{tabular}}
	\caption{Number of stages for the original ($\alpha$DeC, bDeC) and novel ($\alpha$DeCu, $\alpha$DeCdu, bDeCu, bDeCdu) methods with \GL~subtimenodes and speed up factor
	\label{tab:stagesNumbersGL}}
\end{table}

\section{Stability analysis}\label{chapSTAB}
In this section, we study the stability of the novel DeC schemes.
We will prove two original results. First, the stability functions of bDeCu and bDeCdu coincide with the bDeC ones and do not depend on the distribution of the subtimenodes but only on the order. Second, if we fix the subtimenodes distribution and the order, the $\alpha$DeCdu methods coincide with the $\alpha$DeCu methods on linear problems. For all the schemes, we will show the stability region using some symbolical and numerical tools. 

Let us start by reviewing some known results for RK methods \cite{butcher2016numerical,wanner1996solving}. The linear stability of a RK scheme is tested on Dahlquist's problem 
$u'=\lambda u,$
where $\lambda\in \mathbb{C}$ with $Re(\lambda)<0$. Being the RK schemes linear, we can write a general RK iteration as $u_{n+1}=R(\lambda \dt) u_n$, with $R(\cdot)$ the stability function of the method. The stability function is defined as
\begin{equation}\label{eq:stab_fun_RK}
	R(z) = 1+ z\uvec{b}^T (I-zA)^{-1} \mathbf{1},
\end{equation}
where $\mathbf{1}$ is a vector with all the entries equal to $1$.
The set of complex numbers $z$ such that $\lvert R(z)\rvert< 1$ is called stability region.
We remark that the stability function for explicit RK methods is a polynomial. In fact, the inverse of $(I-zA)$ can be written in Taylor expansion as
\begin{equation}
	(I-zA)^{-1} = \sum_{r=0}^{\infty} z^rA^r = I + zA + z^2A^2+\dots
\end{equation} 
and, since $A$ is strictly lower triangular, it is nilpotent, i.e., there exists an integer $r$ such that $A^r=\matz$ and the minimum of these natural numbers $\mathcal{N}$ is called degree of nilpotence. By definition of $\mathcal{N}$, it is clear that $A^{\mathcal{N}+r}=\matz$ for all $r\geq 0$. Moreover, it is also clear that $\mathcal{N}\leq S$, where $S$ is the number of stages of the explicit RK method and the dimension of the matrix $A$.
Hence, $R(z)$ is a polynomial in $z$ with degree at most equal to $S$.
We recall that \cite{wanner1996solving},
if a RK method is of order $P$, then
\begin{equation}\label{eq:RK_order}
	R(z) = 1 + z + \frac{z^2}{2!} +\dots +\frac{z^P}{P!}+O(z^{P+1}).
\end{equation}

So,  we know the first $P+1$ terms of the stability functions $R(\cdot)$ for all the DeCs of order $P$ presented above. Further, the following result holds.
\begin{theorem}\label{th:bDeC_equivalence} 
	The stability function of any bDeC, bDeCu and bDeCdu method of order $P$ is
	\begin{equation}\label{eq:stab_function_bDeC}
		R(z) = \sum_{r=0}^{P} \frac{z^r}{r!}= 1 + z + \frac{z^2}{2!} +\dots +\frac{z^P}{P!},
	\end{equation}
	and does not depend on the distribution of the subtimenodes.
\end{theorem}  
\begin{proof}
	The proof of this theorem relies only on the block structure of the matrix $A$ for such schemes. In all these cases, the matrix $A$ can be written as \begin{equation}
		A= \begin{pmatrix}
			0 & 0 & 0&  \dots &0 &0\\
			\star & 0 & 0&  \dots & 0 &0\\
			\star & \star & 0 & \dots & 0 &0\\
			\star & 0& \star  & \dots & 0 &0\\
			\dots & \dots & \dots &\dots &\dots&\dots \\
			\star & 0 & 0 &\cdots   &\star &0
		\end{pmatrix},
	\end{equation}
where $\star$ are some non-zero block matrices and the 0 are some zero block matrices of different sizes. 
The number of blocks in each row and column of $A$ is $P$, the order of the scheme.
By induction, we can prove that $A^k$ 
has zeros in the main block diagonal, and in all the $k-1$ block diagonals below the main diagonal, i.e., $(A^k)_{i,j} = 0$ if $i< j+k$, where the indices here refer to the blocks.
Indeed, it is true that $A_{i,j}=0$ if $i<j+1$.
Now, let us consider the entry $(A^{k+1})_{i,j}$ with $i<j+k+1$, i.e., $i-k<j+1$. Such entry is defined as $
	(A^{k+1})_{i,j} = \sum_{w} (A^{k})_{i,w} A_{w,j},
$
and we will prove that all the terms of the sum are 0. Let $w<j+1$, then $A_{w,j}=0$ because of the structure of $A$; while, if $w\geq j+1 >i-k$, we have that $i<w+k$, so $(A^k)_{i,w} = 0$ by induction. 

In particular, this means that $A^{P}=\matz$, because any block row index $i$ is smaller than $j+P$ 
for any  block column index $j$, 
as $P$ is the number of the blocks that we have in each row and column. 
Hence, 
\begin{equation}
	(I-zA)^{-1}=\sum_{r=0}^{\infty} z^rA^r =\sum_{r=0}^{P-1} z^rA^r = I + zA + z^2A^2+\dots + z^{P-1}A^{P-1}. 
\end{equation}
Plugging this result into \eqref{eq:stab_fun_RK}, we can state that the stability function $R(z)$ is a polynomial of degree $P$, the order of the scheme. Since all the terms of degree lower or equal to $P$ must agree with the expansion of the exponential function \eqref{eq:RK_order}, the stability function must be \eqref{eq:stab_function_bDeC}.
Finally, let us notice that no assumption has been made on the distribution of the subtimenodes, hence, the result is general for any distribution.
\end{proof}
In the following, we will show that, given a certain order $P$ and a distribution of subtimenodes, the $\alpha$DeCu and $\alpha$DeCdu methods are equivalent on linear problems and, as a consequence, they share the same stability functions. 
\begin{theorem}[Equivalence on linear problems]\label{th:equivalent_linear}
	Given an order $P$, a distribution of subtimenodes and $\alpha \in [0,1]$, the schemes $\alpha$DeCu and $\alpha$DeCdu applied to linear systems are equivalent.
\end{theorem}
\begin{proof}
Without loss of generality, we can focus on Dahlquist's equation $u'=\lambda u$. Since the schemes are linear, the same arguments would apply component-wise also on linear systems of equations. Let us start by explicitly writing down the general updating formula \eqref{eq:second_iterp} of the $\alpha$DeCdu method for Dahlquist's equation
\begin{align}\label{eq:second_linear}
	\underline{\uvec{U}}^{(p)}  = \underline{\uvec{U}}^{(0)}_{p+1}+\Delta t \lambda(\Theta^{(p)} -\alpha \Gamma^{(p)} )H^{(p-1)}\underline{\uvec{U}}^{(p-1)}+\Delta t \lambda \alpha\Gamma^{(p)}\underline{\uvec{U}}^{(p)}.
\end{align}
For the $\alpha$DeCu method, the updating formula \eqref{eq:first_iterpp} becomes
\begin{align}
	\underline{\uvec{U}}^{(p)}  = \underline{\uvec{U}}^{(0)}_{p+1}+\Delta t \lambda (\Theta^{(p)} -\alpha \Gamma^{(p)} )\underline{\uvec{U}}^{*(p-1)}+\Delta t \lambda\alpha \Gamma^{(p)} \underline{\uvec{U}}^{(p)},
\end{align}
now, using the definition of $\underline{\uvec{U}}^{*(p-1)} = H^{(p-1)}\underline{\uvec{U}}^{(p-1)}$, we obtain
\begin{equation}
		\underline{\uvec{U}}^{(p)}  = \underline{\uvec{U}}^{(0)}_{p+1}+\Delta t \lambda (\Theta^{(p)} -\alpha \Gamma^{(p)} )H^{(p-1)}\underline{\uvec{U}}^{(p-1)}+\Delta t \lambda \alpha \Gamma^{(p)} \underline{\uvec{U}}^{(p)},
\end{equation}
which coincides with \eqref{eq:second_linear}. 
This means that, at each iteration, the two modified schemes coincide. 
\end{proof}

\begin{figure}
	\begin{center}
		\includegraphics[height=0.25\textheight]{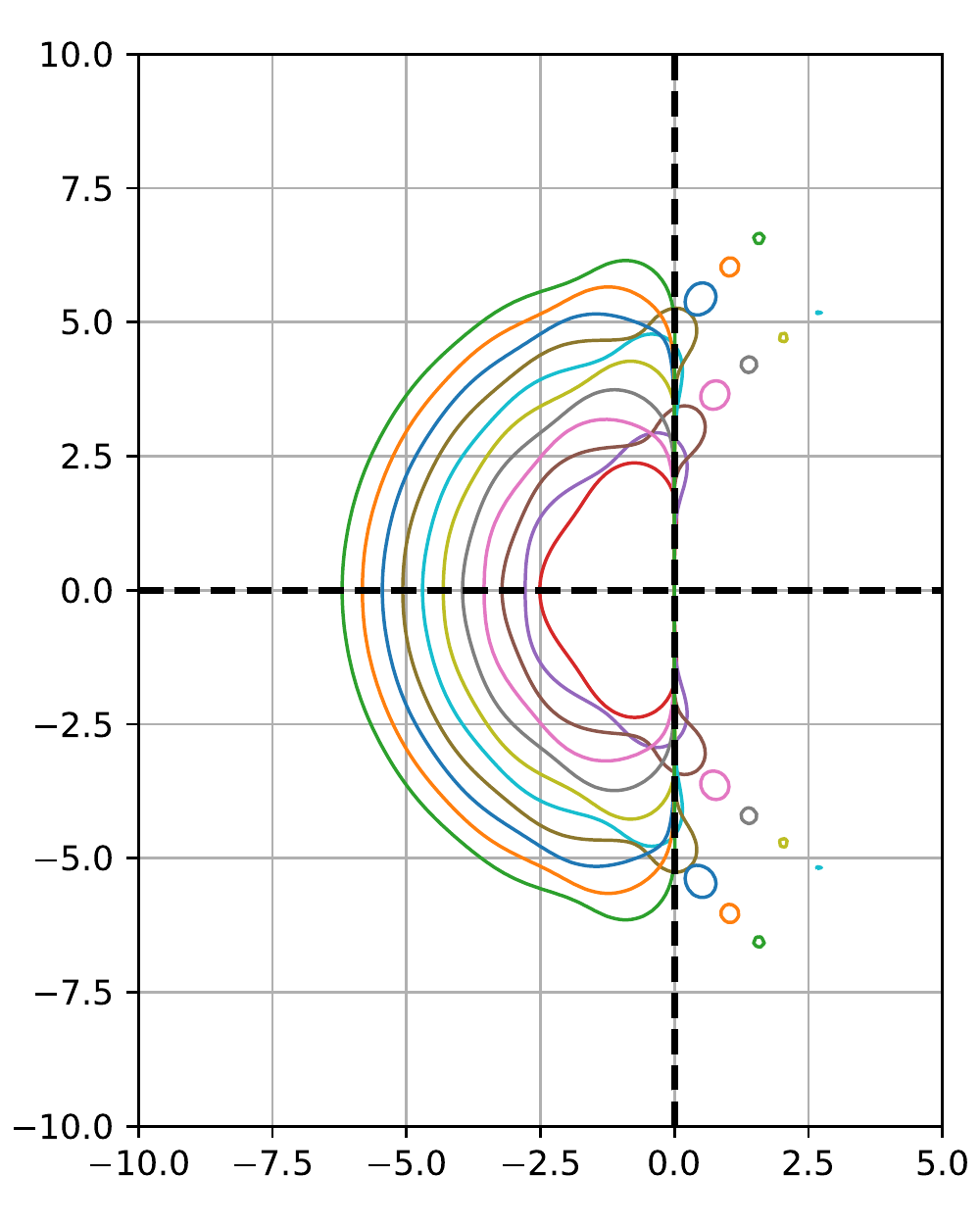}
		\includegraphics[height=0.25\textheight]{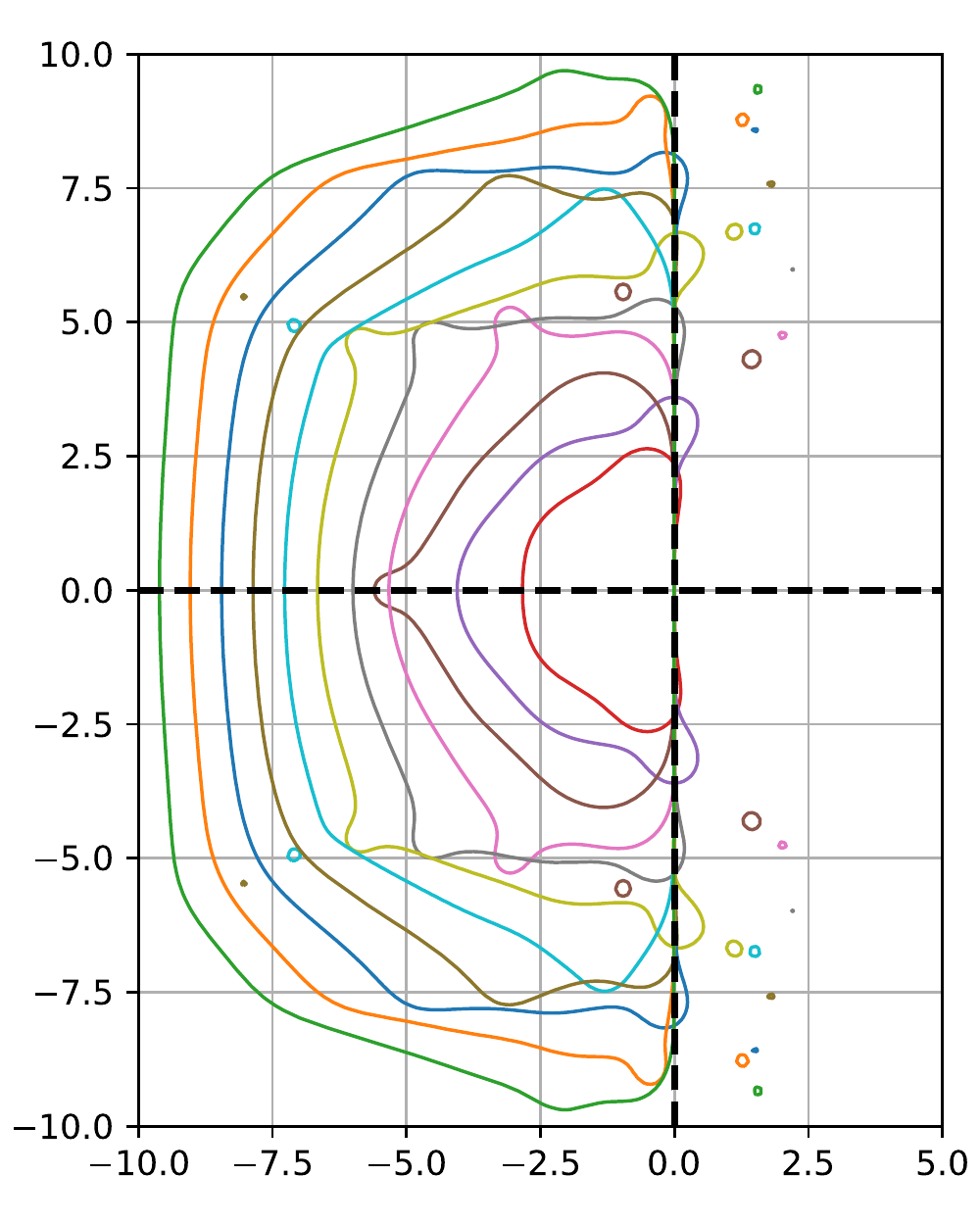} 
		\includegraphics[height=0.25\textheight]{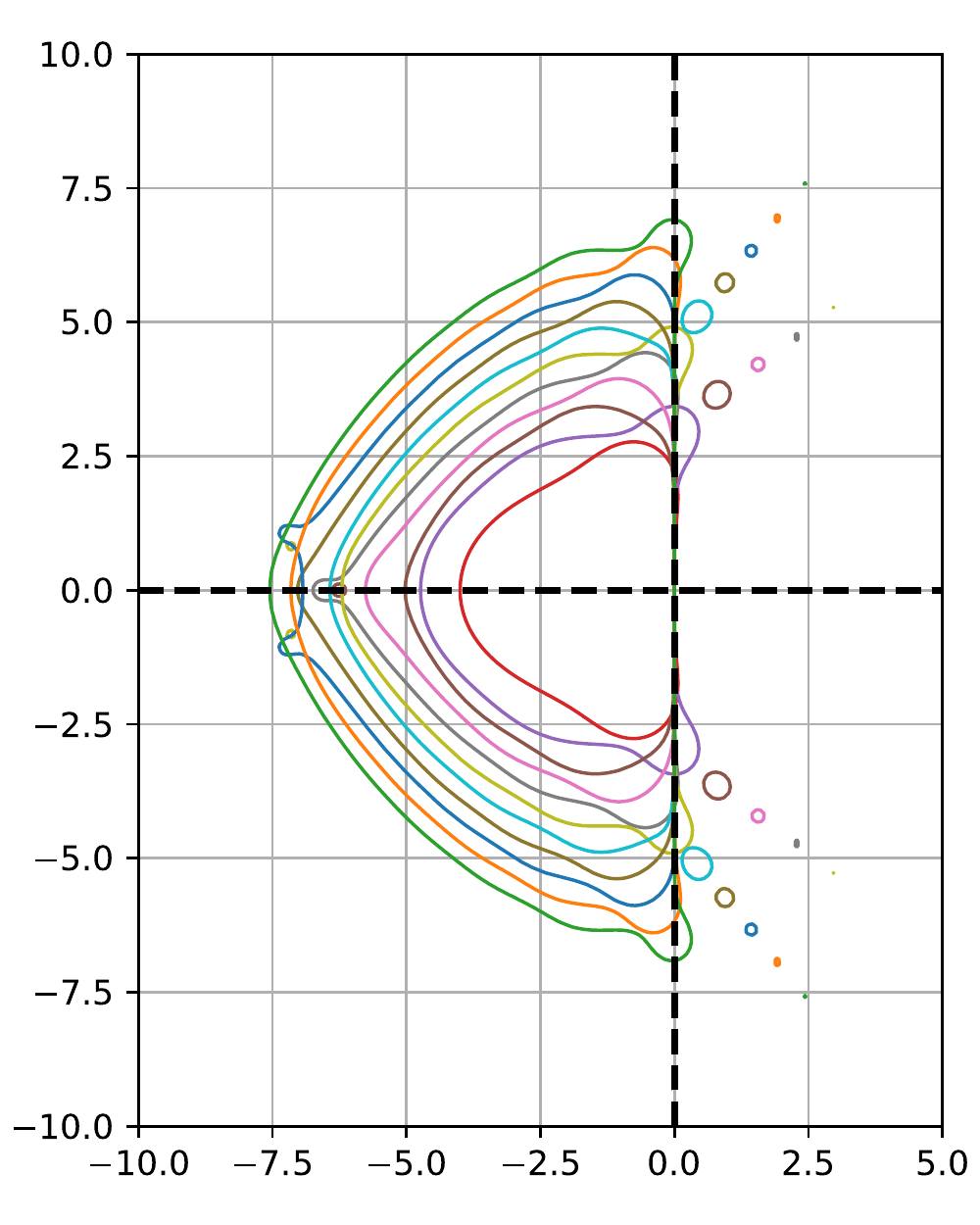}
		\includegraphics[height=0.25\textheight]{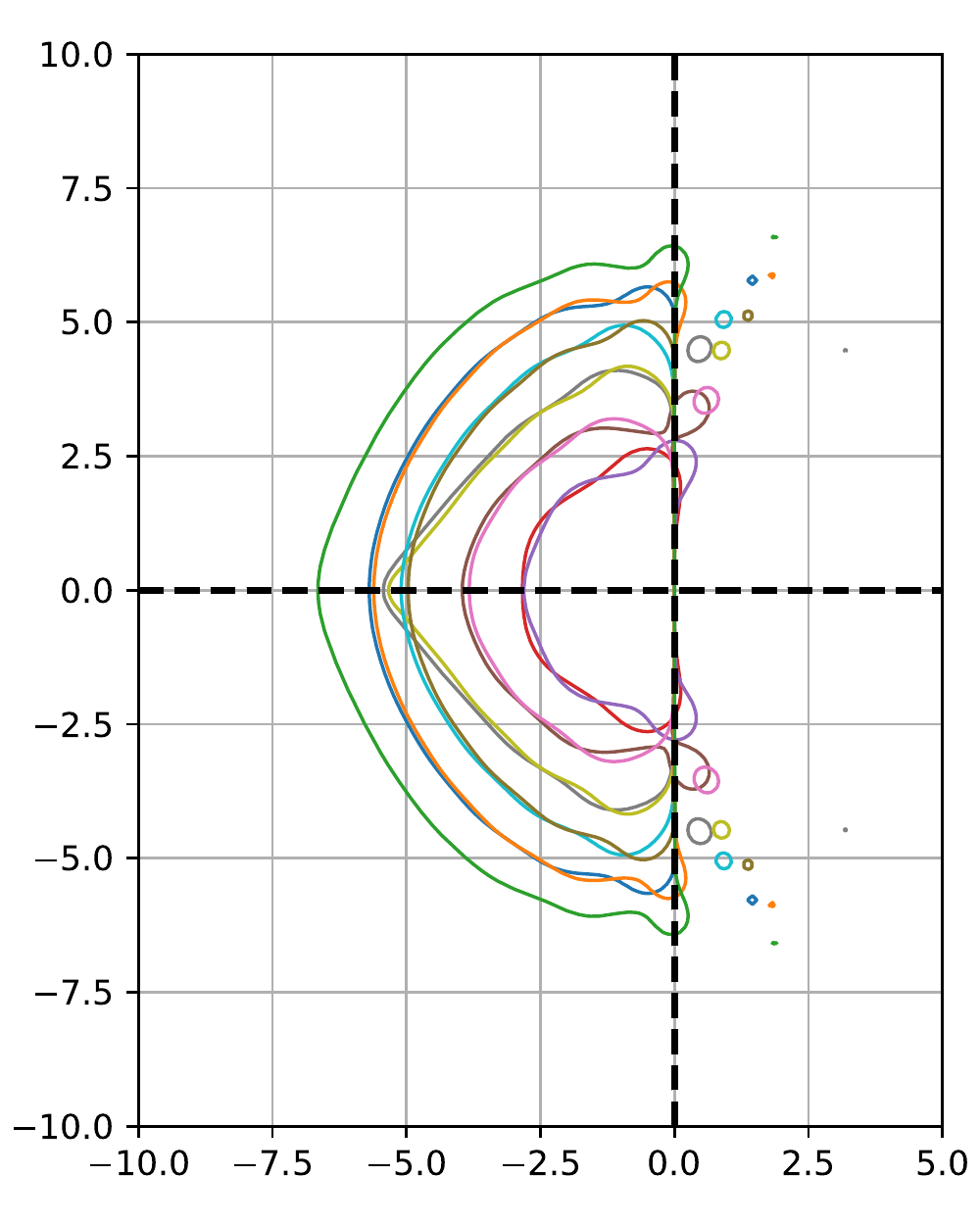}
		\includegraphics[height=0.25\textheight]{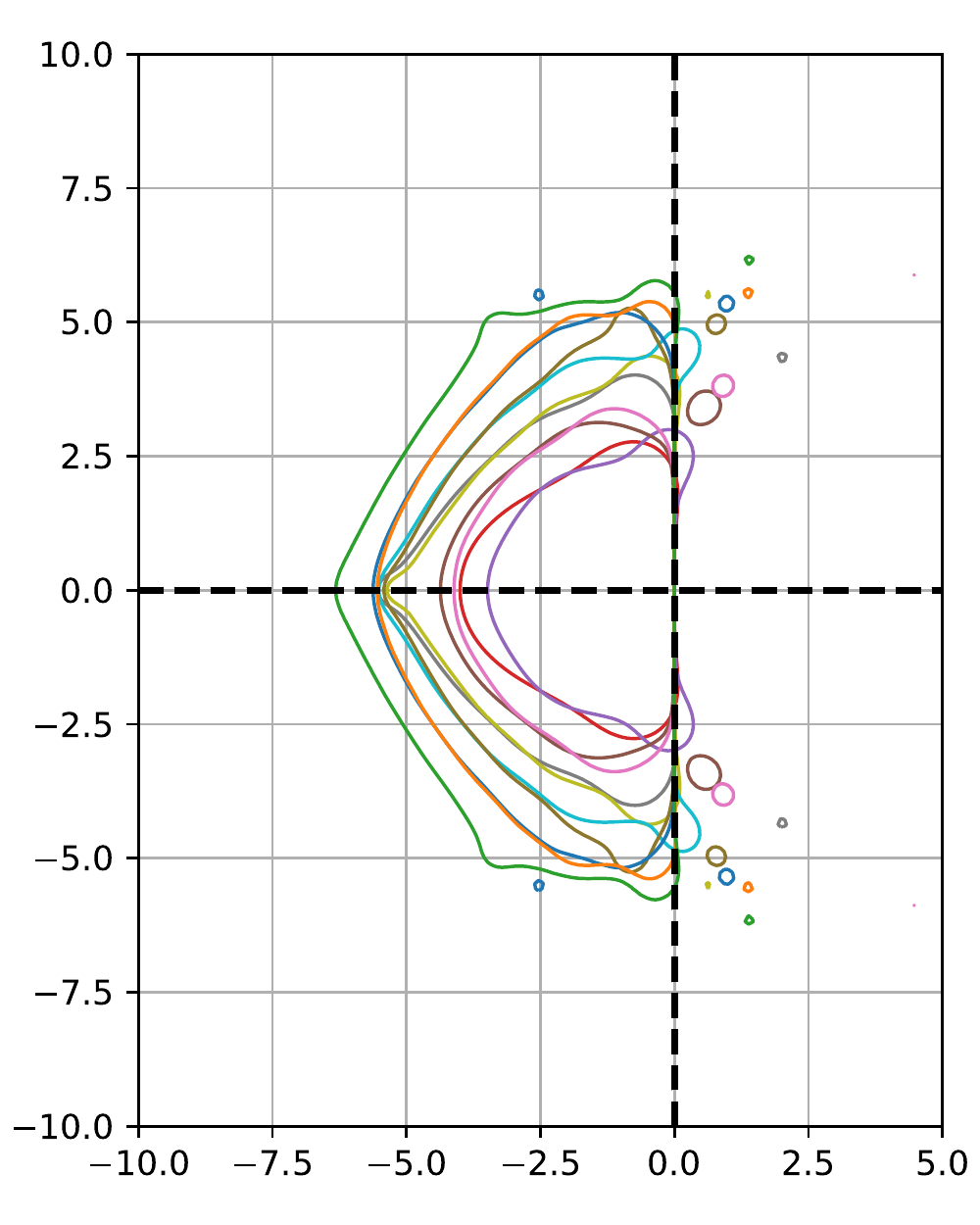} 
		\includegraphics[height=0.25\textheight, trim={0 80 10 65}, clip]{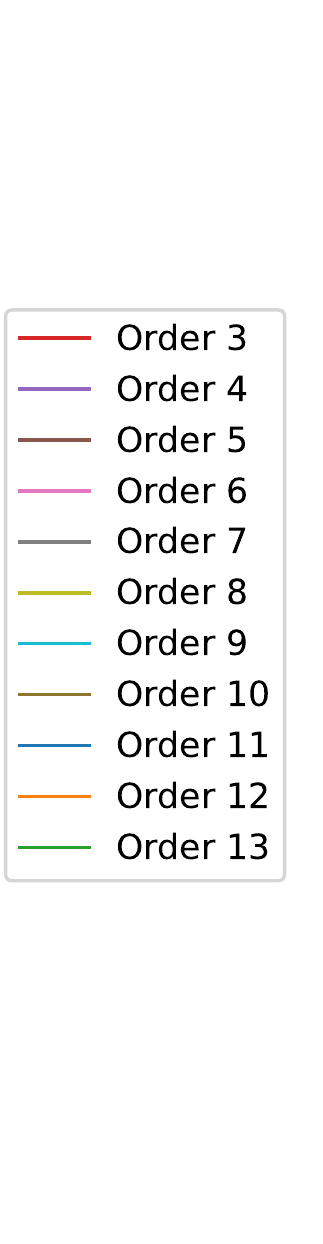}
	\end{center}
	\caption{Stability regions for various schemes with order from 3 to 13: bDeC, bDeCu and bDeCdu (equivalent) for any distribution of subtimenodes (top left), sDeC for equispaced subtimenodes (top center), sDeCu and sDeCdu (equivalent) for equispaced subtimenodes (top right), sDeC for \GL~subtimenodes (bottom left), sDeCu and sDeCdu (equivalent) (bottom center), legend (bottom right)}\label{fig:stability_all}
\end{figure}

In Figure~\ref{fig:stability_all}, we depict the stability region of all the presented methods from order 3 to 13. We remark that there is no difference in terms of stability between bDeC, bDeCu and bDeCdu, nor dependence on the distribution of the subtimenodes, as well as sDeCu and sDeCdu have the same stability regions for fixed subtimenodes. 

\section{Application to hyperbolic PDEs}\label{chapDECRD}
In this section, we apply the novel explicit efficient DeC techniques to hyperbolic PDEs. 
We will focus on the CG framework, which is particularly challenging with respect to FV and DG formulations, due to the presence of a global sparse mass matrix.
In particular, we will consider two strategies that allow to avoid the related issues.
We will describe the operators $\lopdt^1$ and $\lopdt^2$ for the two strategies in the bDeC formulation and see how to apply the bDeCu efficient modification. 
The proofs of the properties of the operators are provided in the supplementary material.
\subsection{Continuous Galerkin FEM}

The general form of a hyperbolic system of balance laws reads 
\begin{equation}
\label{sys}
\frac{\partial}{\partial t}\uvec{u}(\uvec{x},t)+\textrm{div}_{\uvec{x}} \boldsymbol{F}(\uvec{u}(\uvec{x},t))=\boldsymbol{S}(\uvec{x},\uvec{u}(\uvec{x},t)), \quad (\uvec{x},t) \in \Omega\times \mathbb{R}^+_0,
\end{equation}
where $\uvec{u}:\Omega \times \R^+_0\to \R^Q$, with some initial condition $\uvec{u}(\uvec{x},0)=\uvec{u}_0(\uvec{x})$ on the space domain $\Omega\subseteq \mathbb{R}^D$, and boundary conditions on $\partial \Omega$. 
We consider a tessellation $\tess$ of $\overline{\Omega}$ with characteristic length $h$, made by convex closed polytopals $K$, 
and we introduce the space of continuous piecewise polynomial functions $V_h:=\lbrace g \in C^0(\overline{\Omega})~ s.t.~g\vert_K \in \mathbb{P}_M(K)~\forall K \in \tess\rbrace$.
We choose a basis $\lbrace \varphi_i \rbrace_{i=1,\dots,I}$ of $V_h$, e.g., the Lagrange polynomials or the Bernstein polynomials, which is such that each basis function $\varphi_i$ can be associated to a  degree of freedom (DoF) $\uvec{x}_i\in \overline{\Omega}$ and such that $supp\left\lbrace\varphi_i\right\rbrace=\cup_{K\in K_i}K$ with $K_i:=\lbrace K\in \tess ~ s.t. ~ \uvec{x}_i\in K \rbrace$. 
Further, we assume a normalization of the basis functions yielding $\sum_{i=1}^I\varphi_i \equiv 1.$ 
Then, we project the weak formulation in space of the PDE \eqref{sys} over $V_h$, i.e., we look for $\uapp(\uvec{x},t)=\sum_{j=1}^I\uvec{c}_j(t)\varphi_j(\uvec{x}) \in V_h^Q$ such that for any $i=1,\dots,I$
\begin{equation}
\int_{\Omega} \left(\frac{\partial}{\partial t}\uapp(\uvec{x},t)+\textrm{div}_{\uvec{x}} \boldsymbol{F}(\uapp(\uvec{x},t))-\boldsymbol{S}(\uvec{x},\uapp(\uvec{x},t))\right)\varphi_i(\uvec{x}) d\uvec{x}+\uvec{ST}_i(\uapp)=\uvec{0}  ,
\label{eq:systosolve0}
\end{equation}
where the stabilization term $\uvec{ST}_i(\uapp)$ is added to avoid the instabilities associated to central schemes. 
Thanks to the assumption on the support of the basis functions, it is possible to recast \eqref{eq:systosolve0} as 
\begin{equation}
\sum_{K\in K_i}\sum_{\uvec{x}_j \in K} \left(\int_K \varphi_i(\uvec{x}) \varphi_j(\uvec{x}) d \uvec{x}\right)\frac{d}{d t} \uvec{c}_j(t)+ \spacestuff(\uvec{c}(t))=\uvec{0}, \quad i=1,\dots,I,
\label{systosolve}
\end{equation}
where $\uvec{c}$ is the vector of all $\uvec{c}_i$ and the space residuals $\spacestuff(\uvec{c}(t))$ are defined as
\begin{equation}
\spacestuff(\uvec{c}(t))\!=\!\!\sum_{K\in K_i} \!\int_{K}\!\! \left( \textrm{div}_{\uvec{x}} \boldsymbol{F}(\uapp(\uvec{x},t))\!-\!\boldsymbol{S}(\uvec{x},\uapp(\uvec{x},t))\right)\varphi_i(\uvec{x}) d\uvec{x}+\uvec{ST}_i(\uapp).
\label{spst}
\end{equation}
We would like to solve this system of ODEs in time without solving any linear system at each iteration nor inverting the huge mass matrix.

The first possibility consists in adopting particular basis functions, which, combined with the adoption of the induced quadrature formulas, allow to achieve a high order lumping of the mass matrix. This leads to a system of ODEs like the one described in the previous section and, hence, the novel methods can be applied in a straightforward way. 
Examples of such basis functions are given by the Lagrange polynomials associated to the \GL~points in one-dimensional domains and the Cubature elements in two-dimensional domains, introduced in \cite{cohen:hal-01010373} and studied in \cite{jund2007arbitrary,pasquetti2018cubature,michel2021spectral,michel2022spectral}. 
The second strategy, introduced by Abgrall in \cite{Decremi} and based on the concept of residual \cite{abgrall2006residual,ricchiuto2010explicit,abgrall2019high,abgrall2020high}, exploits the abstract DeC formulation presented in Section~\ref{chapDECAbstract}, introducing a first order lumping in the mass matrix of the operator $\lopdt^1$, resulting in a fully explicit scheme, as we will explain in detail in the following.


\subsection{DeC for CG}
In this section, we will define the operators $\lopdt^1$ and $\lopdt^2$ of the DeC formulation for CG FEM discretizations proposed by Abgrall in \cite{Decremi}. 
In this context, the parameter $\Delta$ of the DeC is the mesh parameter $h$ of the space discretization. We assume CFL conditions of the type $\Delta \approx \Delta t \approx h$.

The definition of the high order implicit operator $\lopdt^2$ is not very different from the one seen in the context of the bDeC method for ODEs. 
We denote by $\uvec{c}(t^m)$ the exact solution of the ODE \eqref{systosolve} in the subtimenode $t^m$ and by $\uvec{c}^m$ its approximation, containing, respectively, all components $\uvec{c}_i(t^m)$ and $\uvec{c}_i^m$.
As usual, for the first subtimenode we set $\uvec{c}^0=\uvec{c}(t^0)=\uvec{c}(t_n)=\uvec{c}_n$. %
Starting from the exact integration of \eqref{systosolve} over $[t^0,t^m]$ and replacing $\spacestuff(\uvec{c}(t))$ by its $M$-th order interpolation in time associated to the $M+1$ subtimenodes, 
we get the definition of the operator $\lopdt^2:\mathbb{R}^{(I \times Q \times M)}\rightarrow \mathbb{R}^{(I \times Q \times M)}$ as  
\begin{equation}
\lopdt^2(\cund)=\left(\mathcal{L}^2_{\Delta,1}(\cund),\mathcal{L}^2_{\Delta,2}(\cund),\dots,\mathcal{L}^2_{\Delta,I}(\cund)\right), \quad \forall \cund \in \mathbb{R}^{(I \times Q \times M)},
\label{l2DeC}
\end{equation}
where, for any $i=1,\dots,I$ and $m=1,\dots,M$, we have 
\begin{equation}
\lopdi^{2,m}(\cund)=
                \sum\limits_{K\in K_i}\sum\limits_{\uvec {x}_j \in K} \left(\int_K \varphi_i(\uvec{x}) \varphi_j(\uvec{x}) d \uvec{x}\right)\left( \uvec{c}_j^m-\uvec{c}_j^0\right)+\Delta t \sum\limits_{\ell=0}^{M} \theta^m_\ell\spacestuff(\uvec{c}^\ell).
\label{l2iDeC}
\end{equation}
The solution $\cund_\Delta$ to $\lopdt^2(\cund)=\uvec{0}$ is $(M+1)$-th order accurate.
Unfortunately, such problem is a huge nonlinear system difficult to directly solve. 
According to the DeC philosophy, we introduce the operator $\lopdt^1$ making use of low order approximations of \eqref{systosolve} in order to achieve an explicit formulation.
In particular, we use the forward Euler time discretization and a first order mass lumping, obtaining $\lopdt^1:\mathbb{R}^{(I \times Q \times M)}\rightarrow \mathbb{R}^{(I \times Q \times M)}$  
\begin{equation}
	\lopdt^1(\cund)=\left(\mathcal{L}^1_{\Delta,1}(\cund),\mathcal{L}^1_{\Delta,2}(\cund),\dots,\mathcal{L}^1_{\Delta,I}(\cund)\right), \quad \forall \cund \in \mathbb{R}^{(I \times Q \times M)},
	\label{l1DeC}
\end{equation}
whose components, for any $i=1,\dots,I$ and $m=1,\dots,M$, are defined as
\begin{equation}
\lopdi^{1,m}(\cund) := C_i \left( \uvec{c}_i^m-\uvec{c}_i^0\right)+
\Delta t\beta^{m}\spacestuff(\uvec{c}^0),
\label{l1iDeC}
\end{equation}
with $C_i:=\int_\Omega \varphi_i(\uvec{x})d\uvec{x}$. 


\begin{remark}[Choice of the basis functions]\label{obs:choice_basis}
For any $m$ and $i$, we can explicitly compute $\uvec{c}_i^m$ from $\lopdi^{1,m}(\cund)=\uvec{0}$ if and only if $C_i\neq 0$. This means that the construction of the operator $\lopdt^1$ is not always well-posed for any arbitrary basis of polynomials. 
For example, with Lagrange polynomials of degree $2$ on triangular meshes, we have $\int_\Omega \varphi_i(\uvec{x})d\uvec{x}=0$ for some $i$. 
However, the construction is always well-posed with Bernstein bases, which verify $C_i>0$ $\forall i$. 
\end{remark}

Let us characterize the iterative formula \eqref{DeC} in this context.
We have
\begin{equation}
\label{DeCRD}
\lopd^1(\underline{\uvec{c}}^{(p)})=\lopd^1(\underline{\uvec{c}}^{(p-1)})-\lopd^2(\underline{\uvec{c}}^{(p-1)}), \quad p=1,\dots,P, 
\end{equation}
where $\underline{\uvec{c}}^{(p)}\in \mathbb{R}^{(I \times Q \times M)}$ 
consists of $M$ subtimenodes components $\uvec{c}^{m,(p)}$, each of them containing $I$ DoF components $\uvec{c}_i^{m,(p)}$. 
Just like in the ODE case, procedure \eqref{DeCRD} results in an explicit iterative algorithm due to the fact that the operator $\lopdt^1$ is explicit.
After a direct computation, the update of the component  associated to the general DoF $i$ in the $m$-th subtimenode at the $p$-th iteration reads
\begin{align}
\begin{split}
   \uvec{c}_i^{m,(p)}=\uvec{c}_i^{m,(p-1)}\!\!&-\frac{1}{C_i}\Bigg[\sum_{K\in K_i}\sum_{\uvec {x}_j \in K}   \left( \uvec{c}_j^{m,(p-1)}\!\!-\uvec{c}_j^0\right)     \int_K \varphi_i(\uvec{x}) \varphi_j(\uvec{x}) d \uvec{x}\\
   &+\Delta t \sum_{\ell=0}^{M} \theta^m_\ell\spacestuff(\uvec{c}^{\ell,(p-1)})\Bigg].
\end{split}
  \label{eq:DeCfinalformulaPDE}
\end{align}
We remark that also in this case we assume $\uvec{c}_i^{m,(p)}=\uvec{c}_i(t_n)$ whenever $p$ or $m$ are equal to $0$. For what concerns the optimal number of iterations, analogous considerations to the ones made in the ODE case hold.
Finally, it is worth observing that the resulting DeC schemes cannot be written in RK form due to the difference between the mass matrices in $\lopdt^1$ and $\lopdt^2$.
In fact, such DeC formulation is not obtained via a trivial application of the method of lines.
\subsection{bDeCu for CG}
As for ODEs, it is possible to modify the original DeC for hyperbolic problems to get a new more efficient method by introducing interpolation processes between the iterations. The underlying idea is the same, we increase the number of subtimenodes as the accuracy of the approximation increases. 
At the general iteration $p$, the interpolation process allows to get $\underline{\uvec{c}}^{*(p-1)}$ from $\underline{\uvec{c}}^{(p-1)}$ and then we perform the iteration via \eqref{eq:DeCfinalformulaPDE} getting
\begin{align}
\begin{split}
   \uvec{c}_i^{m,(p)}=\uvec{c}_i^{*m,(p-1)}&-\frac{1}{C_i}\Bigg[\sum_{K\in K_i}\sum_{\uvec {x}_j \in K} \left( \uvec{c}_j^{*m,(p-1)}-\uvec{c}_j^0\right) \int_K \varphi_i(\uvec{x}) \varphi_j(\uvec{x}) d \uvec{x}\\
   &+\Delta t \sum_{\ell=0}^{M} \theta^m_\ell\spacestuff(\uvec{c}^{*l,(p-1)})\Bigg].
\end{split}
  \label{eq:DeCfinalformulaPDESTggered}
\end{align}

\section{Application to adaptivity}\label{chapADAPT}
In this section, we will see how to exploit the interpolation processes in the new schemes, $\alpha$DeCu and $\alpha$DeCdu, to design adaptive methods. 
In the context of an original $\alpha$DeC method with a fixed number of subtimenodes, iteration by iteration, we increase the order of accuracy with respect to the solution $\usol$ of the operator $\lopdt^2$. For this reason, performing a number of iterations higher than the order of accuracy of the discretization adopted in the construction of the operator $\lopdt^2$ is formally useless, as we have already pointed out in Section~\ref{chapDECAbstract}. 
Instead, in the context of an $\alpha$DeCu or $\alpha$DeCdu method, we could in principle keep adding subtimenodes, through interpolation, always improving the accuracy of the approximation with respect to the exact solution of \eqref{ODE}, until a convergence condition on the final component of $\underline{\uvec{u}}^{(p)}$ (always associated to $t_{n+1}$) is met, e.g.,
\begin{equation}
\frac{\norm{ \underline{\uvec{u}}^{p,(p)} - \underline{\uvec{u}}^{p-1,(p-1)}}}{\norm{\underline{\uvec{u}}^{p,(p)}}}\leq \varepsilon
\label{eq:tolerance}
\end{equation}
with $\varepsilon$ a desired tolerance. This leads to a p-adaptive version of the presented algorithms.

\section{Numerical results}\label{chapNum}
In this section, we will numerically investigate the new methods, showing the computational advantage with respect to the original ones.
Since the $\alpha$DeC, $\alpha$DeCu and $\alpha$DeCdu methods of order $2$ coincide, we will focus on methods from order $3$ on.  

\subsection{ODE tests}
We will assess here the properties of the new methods on different ODEs tests, checking their computational costs, their errors and their adaptive versions. 
We will focus on the methods got for $\alpha=0$ (bDeC) and $\alpha=1$ (sDeC).

\subsubsection{Linear system}
\newcommand{\test}{linear_system2}
The first test is a very simple $2\times 2$ system of equations
\begin{equation}
\begin{cases}
	u' = -5u+v,\\
	v' = 5u-v,
\end{cases} \qquad \begin{pmatrix}
u_0\\v_0
\end{pmatrix} = \begin{pmatrix}
0.9\\0.1
\end{pmatrix},
\end{equation}
with exact solution $u(t)= u_0 + (1-e^{-6t}) (-5u_0+v_0)$ and $v(t)=1-u(t)$. 
We assume a final time $T=1$. 
\begin{figure}[h!]
  \begin{center}
    bDeC\\\vspace{2mm}
    \includegraphics[width=0.48\textwidth]{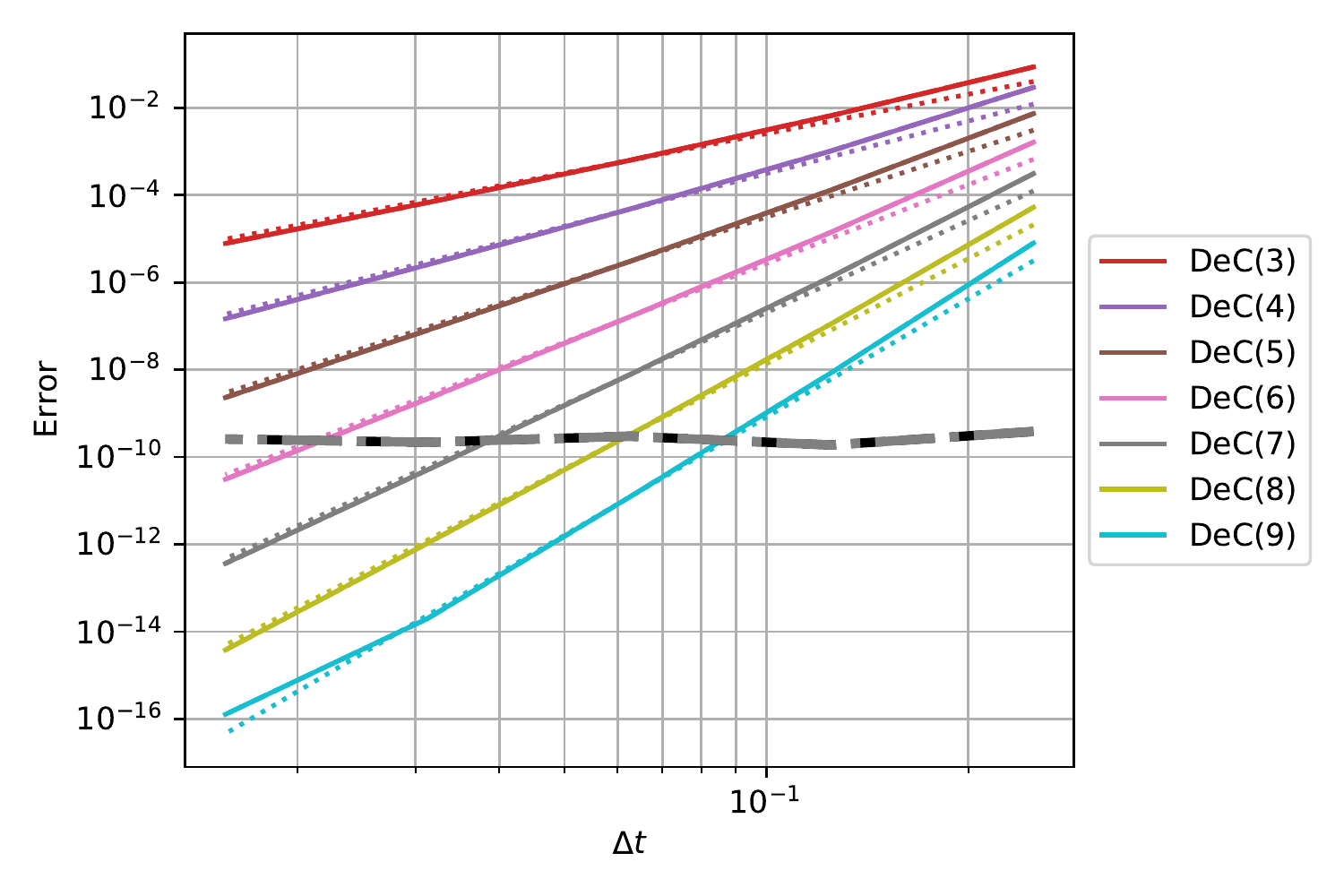} \hfill
    \includegraphics[width=0.48\textwidth]{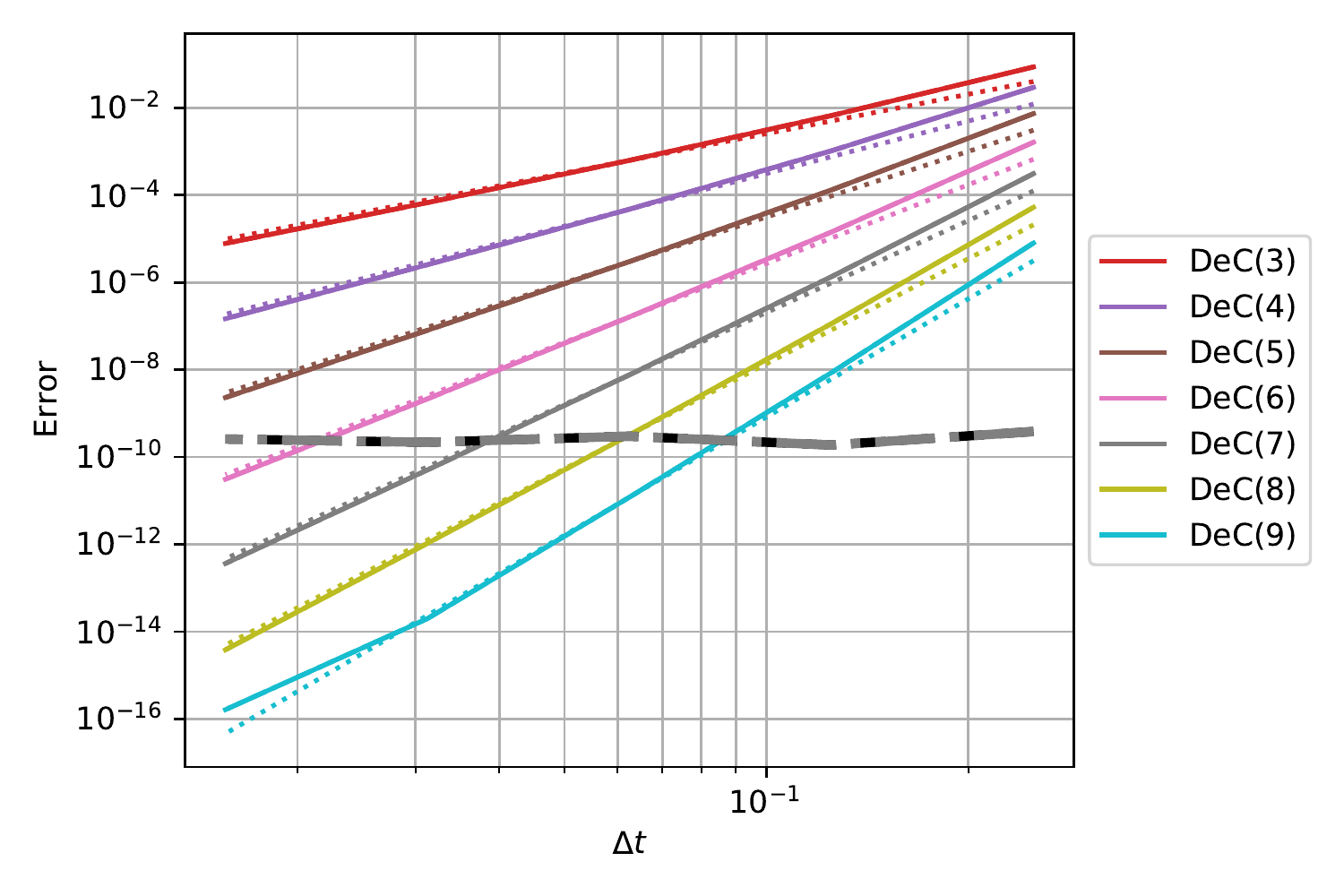}\\\vspace{2mm} 
    sDeC\\ \vspace{2mm}
    \includegraphics[width=0.48\textwidth]{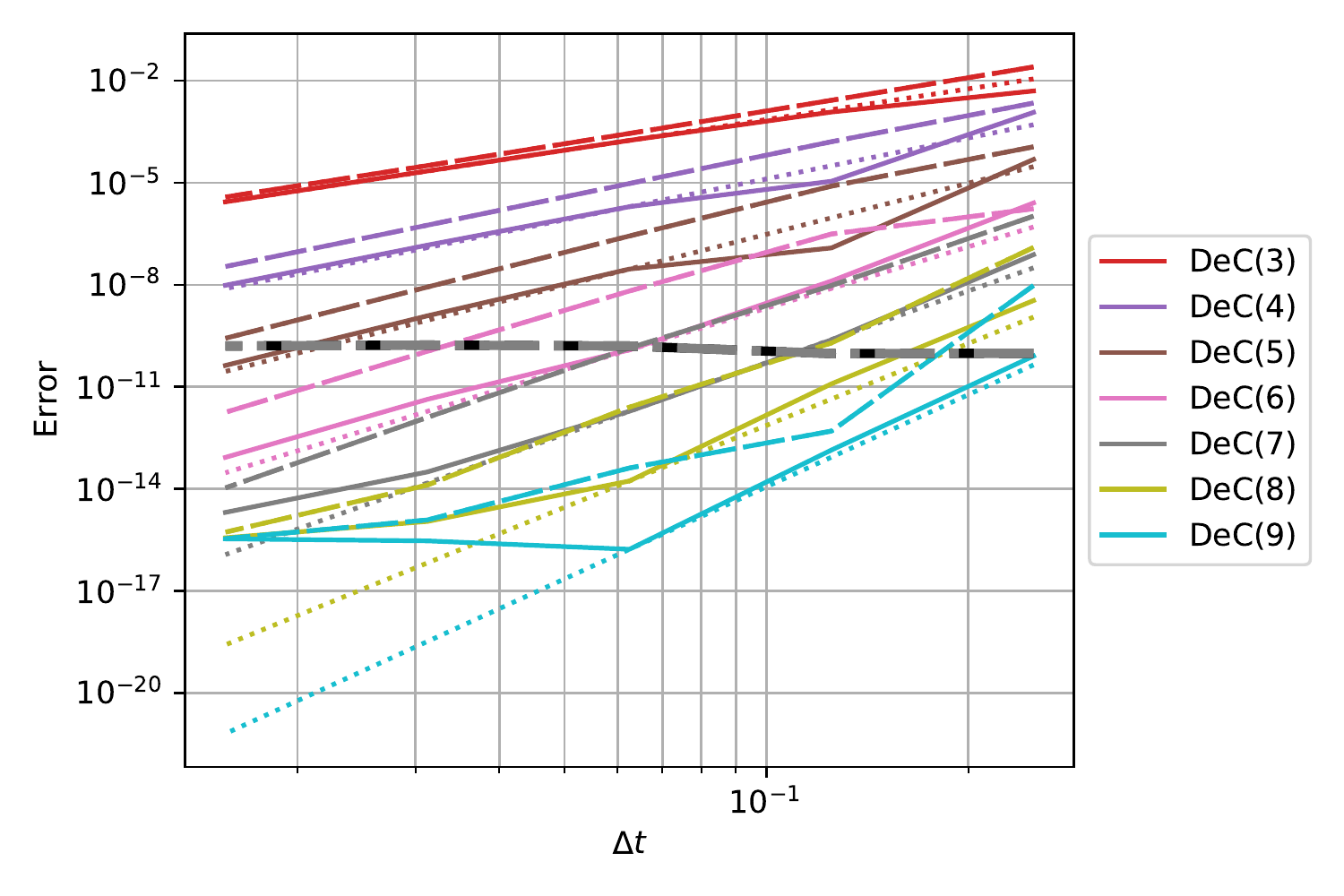}\hfill
    \includegraphics[width=0.48\textwidth]{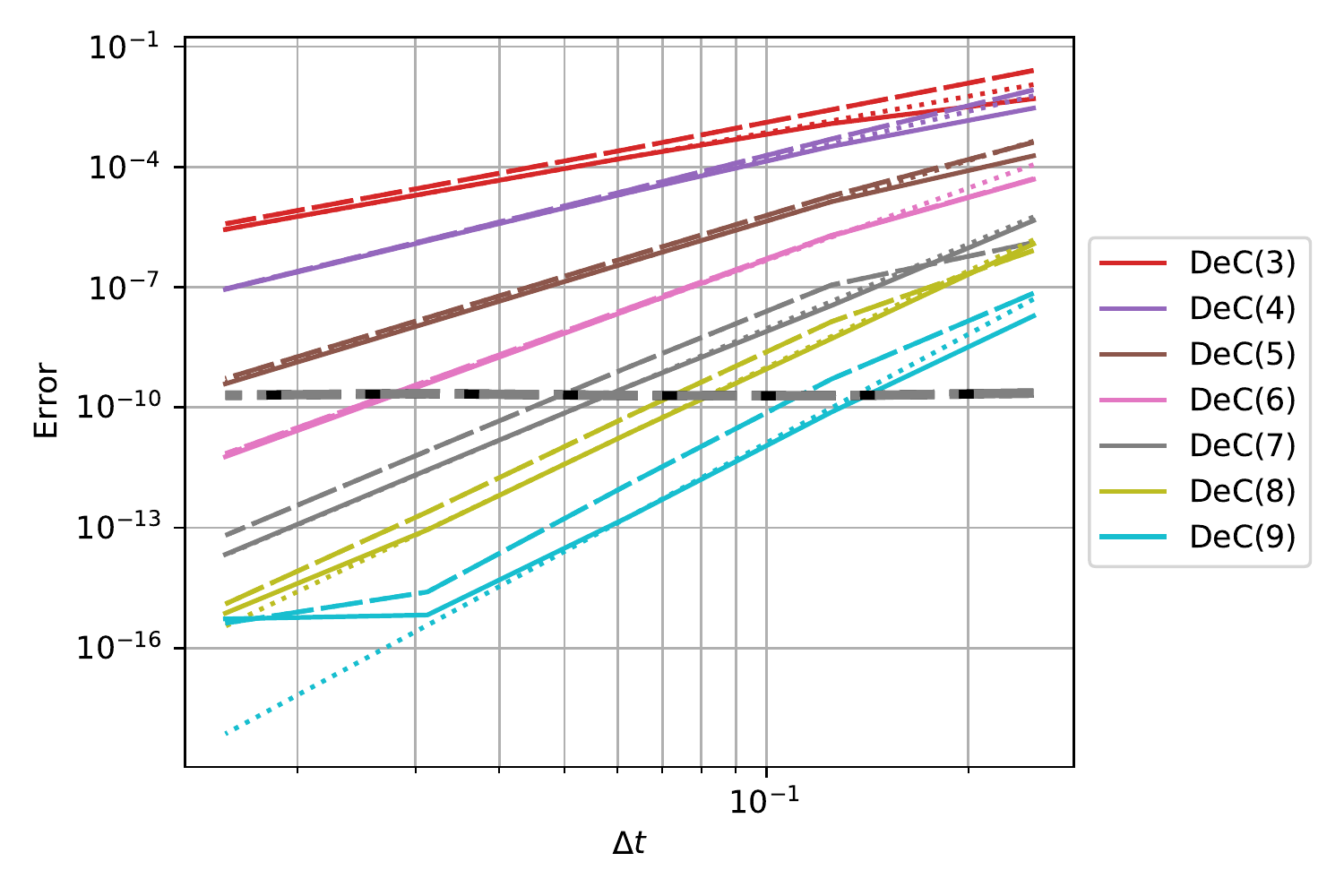}
  \end{center}
\caption{Linear system: Error decay for DeC with continuous line, DeCu with dashed line, DeCdu with dash-dotted line, reference order with dotted line, adaptive DeCu with dashed black line, adaptive DeCdu with dash-dotted gray line. Equispaced subtimenodes on the left and \GL~on the right}\label{fig:convergence_linear_system}
\end{figure}
\begin{figure}[h!]
  \begin{center}
    bDeC\\\vspace{2mm}
    \includegraphics[width=0.48\textwidth]{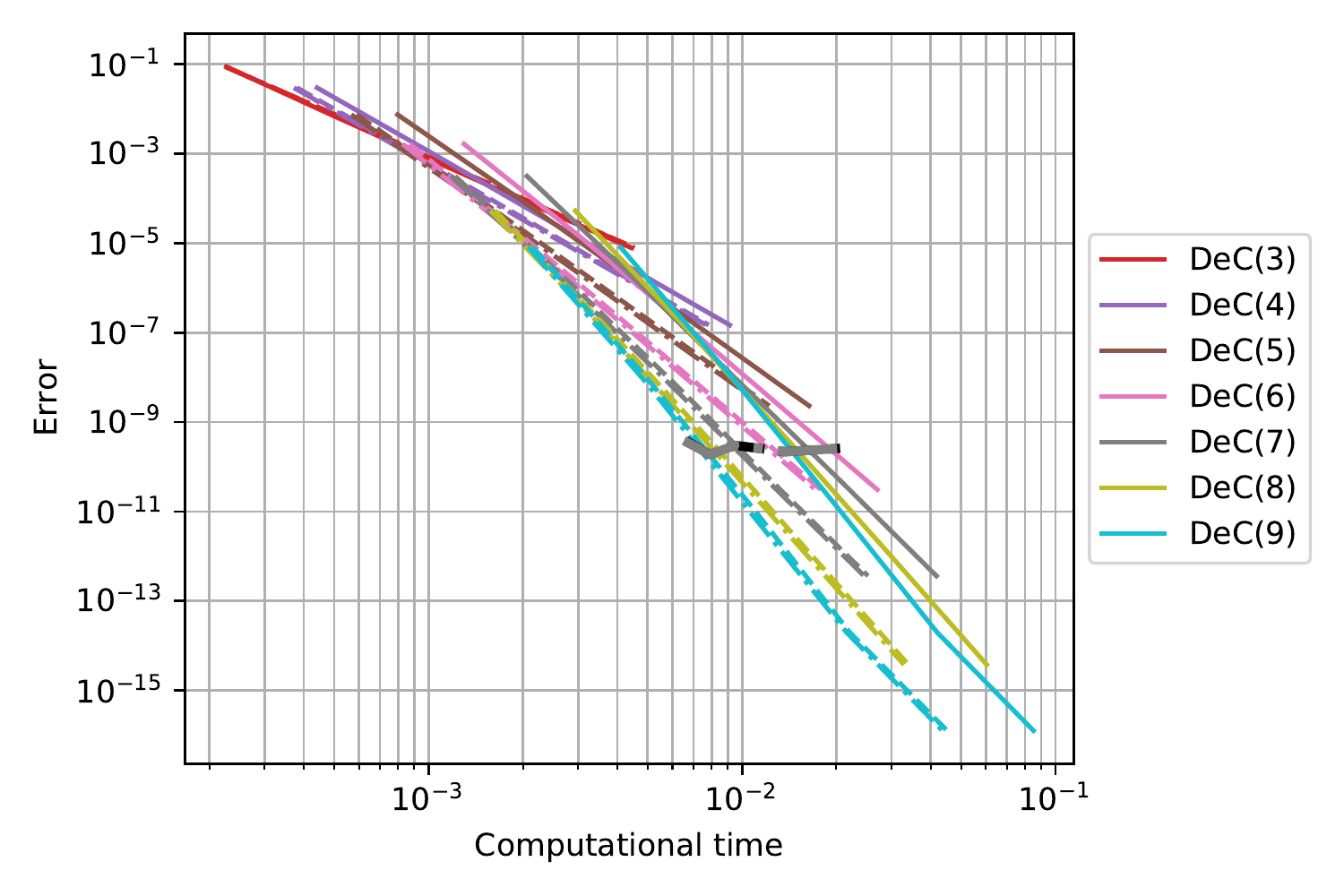} \hfill
    \includegraphics[width=0.48\textwidth]{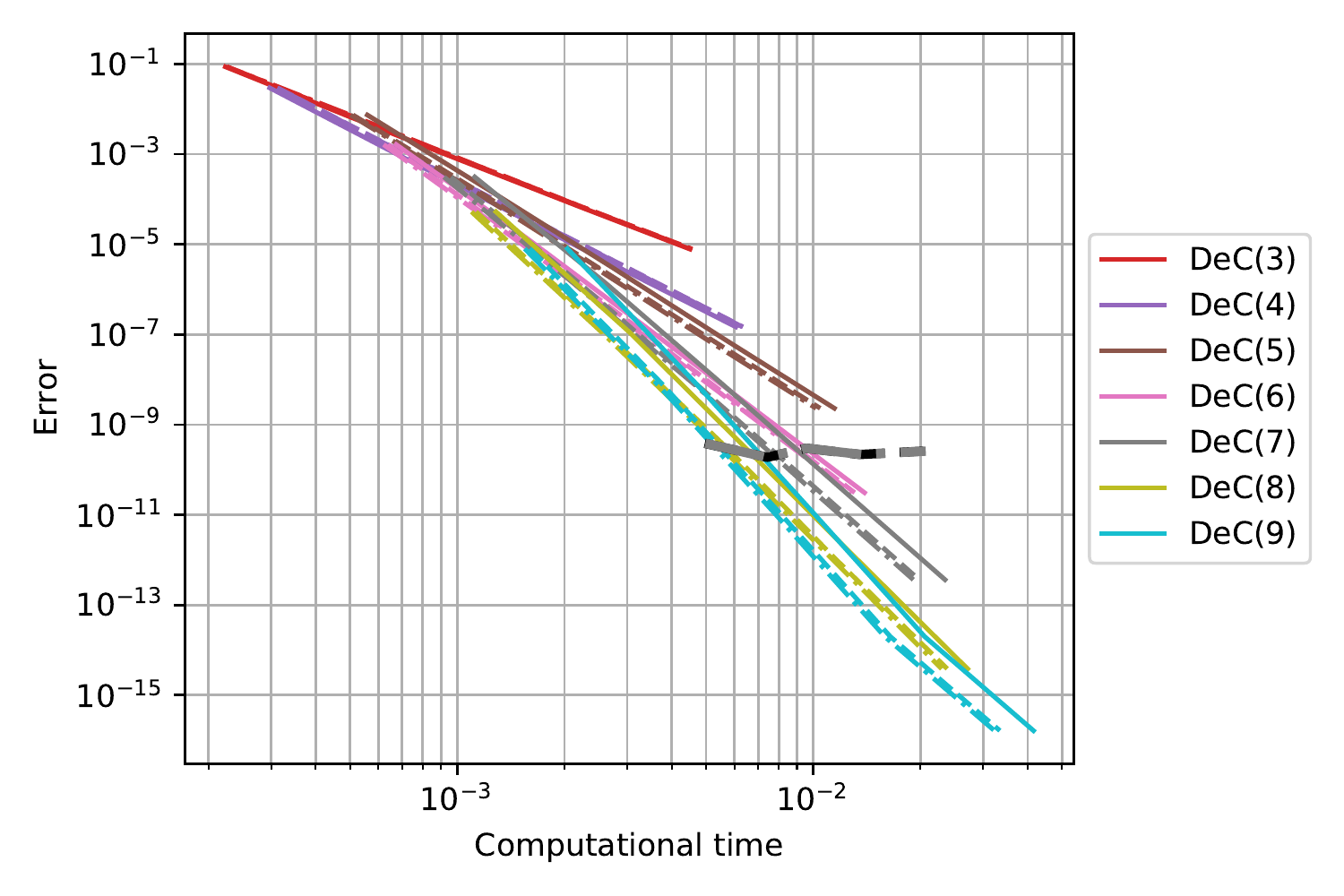}\\\vspace{2mm} 
    sDeC\\ \vspace{2mm}
    \includegraphics[width=0.48\textwidth]{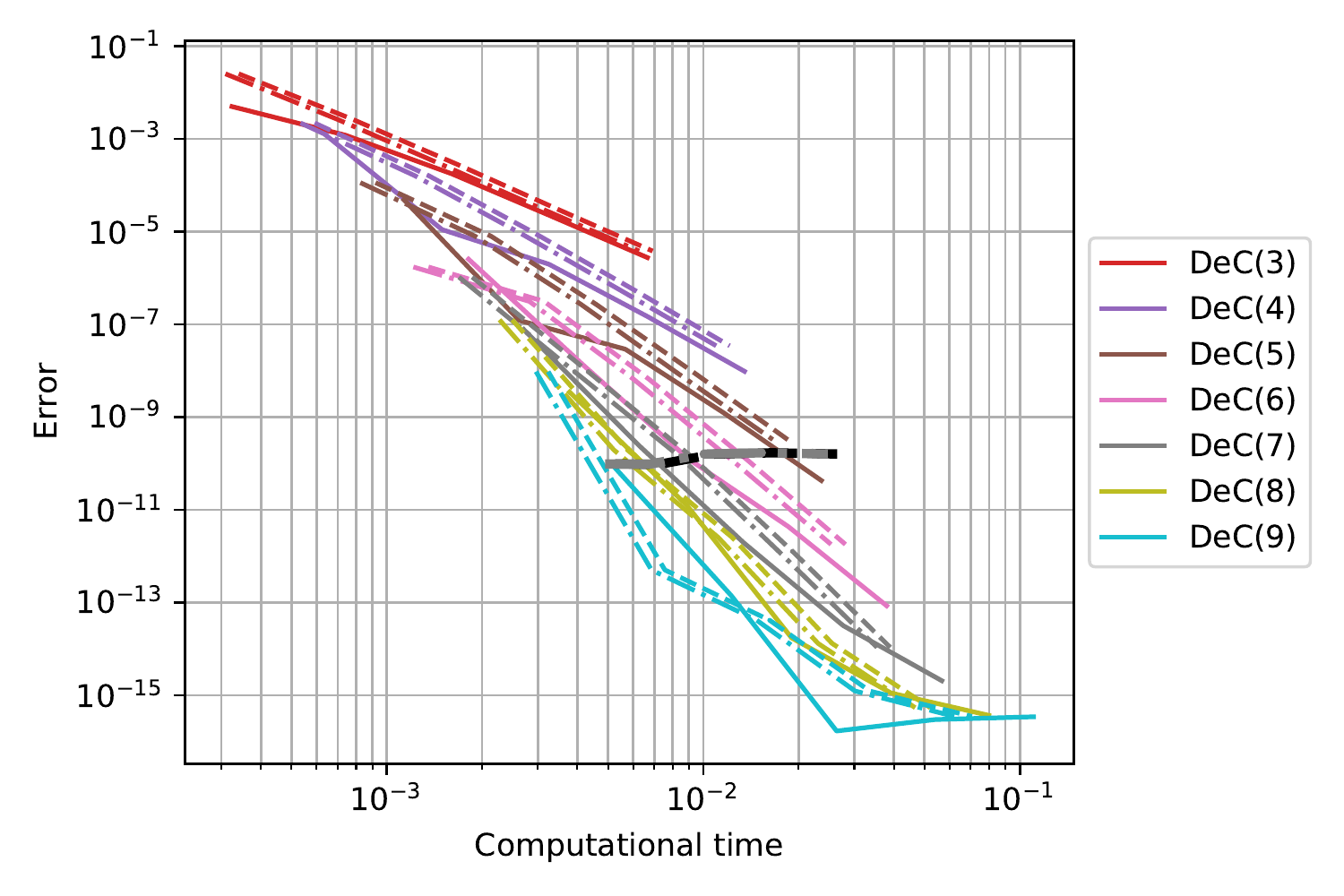}\hfill
    \includegraphics[width=0.48\textwidth]{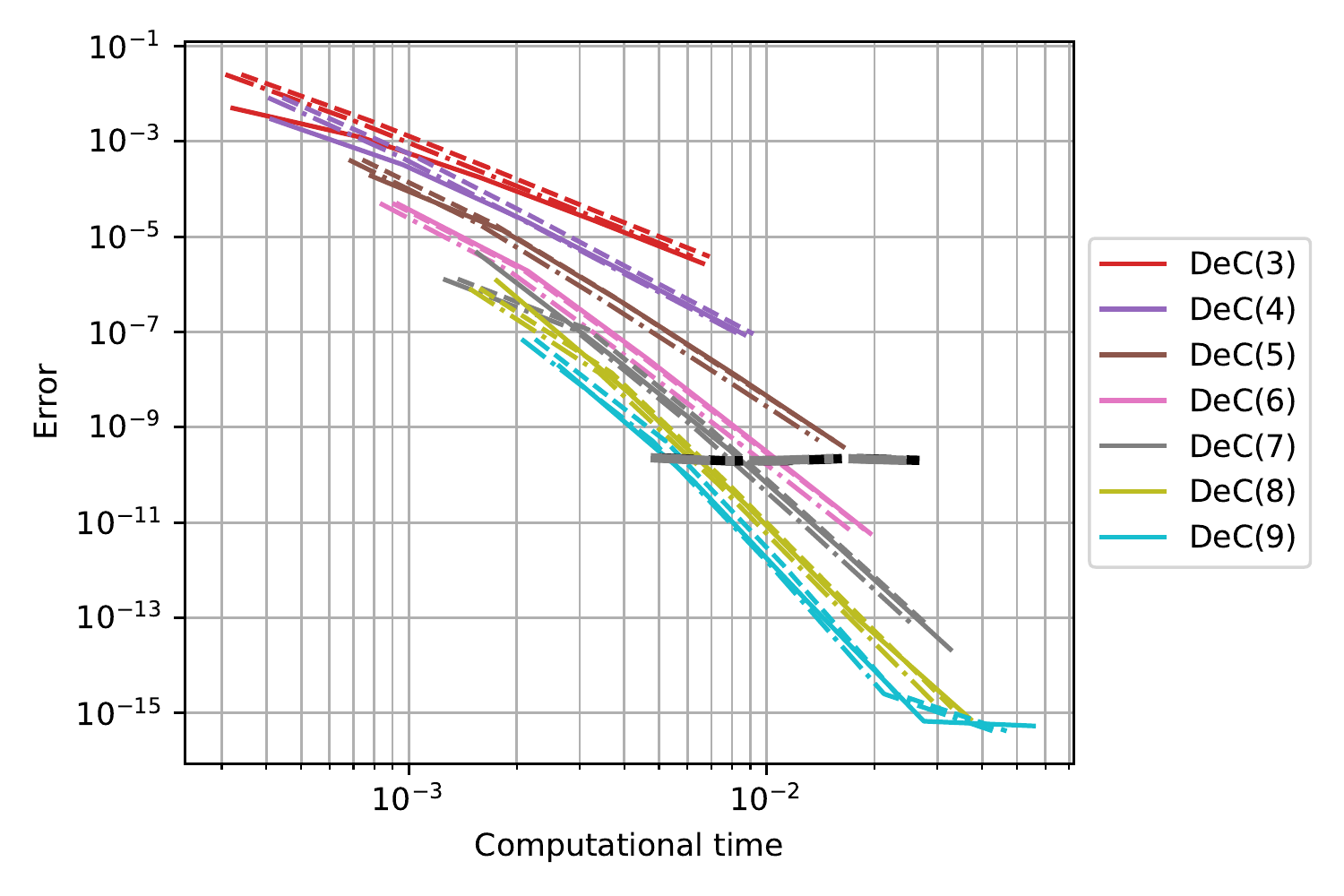}
  \end{center}
	\caption{Linear system: Error with respect to computational time for DeC with continuous line, DeCu with dashed line, DeCdu with dash-dotted line, adaptive DeCu with dashed black line, adaptive DeCdu with dash-dotted gray line. Equispaced subtimenodes on the left and \GL~on the right}\label{fig:conv_time_linear_system}
\end{figure}
In Figure~\ref{fig:convergence_linear_system}, we plot the error decay for all methods with respect to $\Delta t$  for all orders from 3 to 9 and the expected order of convergence is achieved in all cases.
We can see that the bDeC, bDeCu and bDeCdu methods have the same error, since they coincide on linear problems, as shown in Theorem~\ref{th:bDeC_equivalence}.
The sDeC methods show a more irregular behavior and, on average, the errors with the sDeCu and sDeCdu, which coincide due to Theorem~\ref{th:equivalent_linear}, are slightly larger than the one of sDeC for a fixed $\Delta t$. 
In Figure~\ref{fig:conv_time_linear_system}, we plot the error against the computational time of the methods. %
For bDeC methods there is a huge advantage in using the novel methods: %
the Pareto front is composed only by the novel methods. In particular, for equispaced subtimenodes there is a larger reduction in computational cost than for \GL~ones, as predicted by theory. 
For sDeC methods the situation is not as clear as in the bDeC case.
We can systematically see a difference between sDeCu and sDeCdu, being the latter more efficient than the former. In the context of \GL~subtimenodes, the sDeCdu is slightly better than the original sDeC method from order 5 on in the mesh refinement. 
We also tested the adaptive versions of the methods, characterized by the convergence criterion \eqref{eq:tolerance} with a tolerance $\varepsilon=10^{-8}$.
As we observe in Figure~\ref{fig:convergence_linear_system}, the error of these methods (in black and gray) is constant and independent of $\Delta t$. The required computational time, see Figure~\ref{fig:conv_time_linear_system}, is comparable to the one of very high order schemes. 
In Figure~\ref{fig:p_adaptive_linear}, we report the average number of iterations $\pm $ half standard deviation for different adaptive methods with respect to the time discretization. 
As expected, the smaller the timestep, the smaller is the number of iterations necessary to reach the expected accuracy. 
In Figure~\ref{fig:speed_up_linear_system}, we display, for different $\Delta t$, the speed up factor of the bDeCdu method with respect to the bDeC method computed as the ratio between the computational times required by the bDeCdu and the bDeC method. 
For equispaced subtimenodes we see that, as the order increases, the interpolation process reduces the computational time by an increasing factor, which is almost 2 for order 9.
For \GL~subtimenodes the reduction is smaller but still remarkable, close to $\frac{4}{3}$ in the asymptotic limit. 

\begin{figure}
	\centering
	\includegraphics[width=0.4\textwidth]{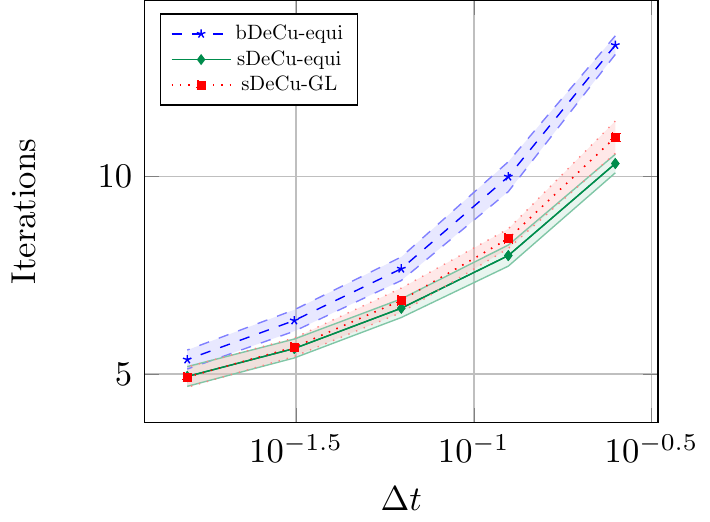}
    \caption{Linear test: Average number of iterations ($\pm$ half standard deviation) of some adaptive DeC for different time steps \label{fig:p_adaptive_linear}}
\end{figure}

\begin{figure}[h!]
  \begin{center}
    \includegraphics[height=0.18\textheight]{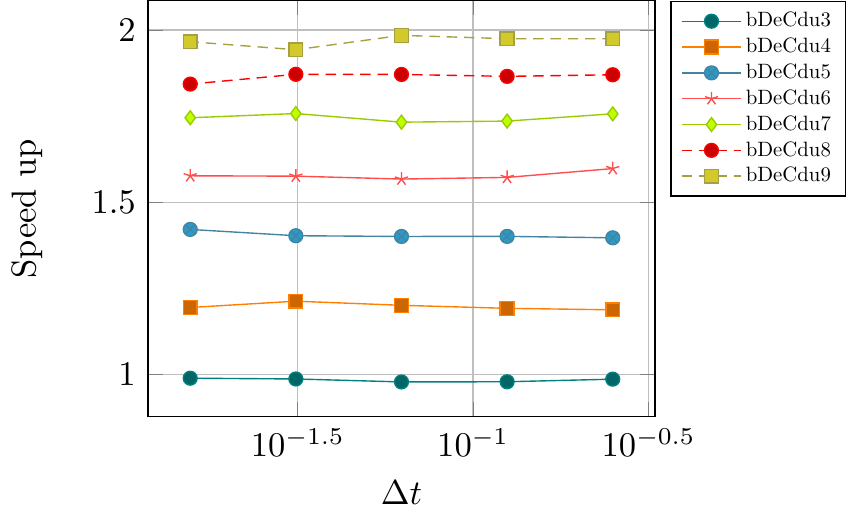}
    \includegraphics[height=0.18\textheight]{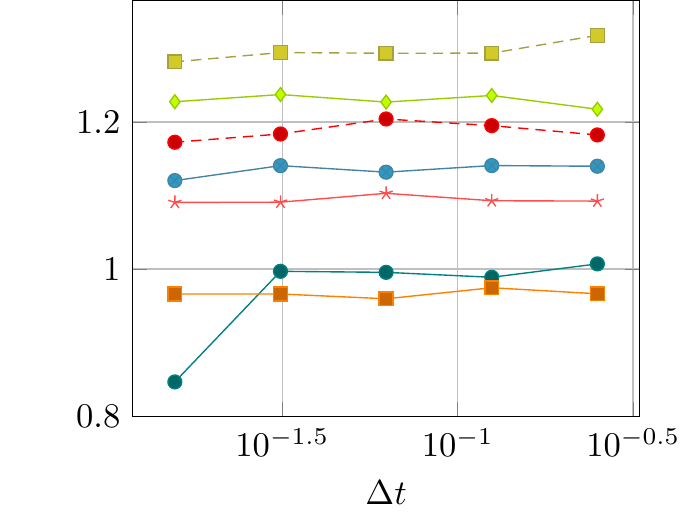}
  \end{center}
	\caption{Linear system test: Speed up factor for the bDeCdu method. Equispaced subtimenodes on the left and \GL~on the right\label{fig:speed_up_linear_system}}
\end{figure}

\subsubsection{Vibrating system}
Let us consider a vibrating system defined by the following ODE
\begin{align}
\begin{cases}
& my''+ry'+ky=F\cos(\Omega t +\varphi),\quad t \in \mathbb{R}^+_0,\\
& y(0)=A ,\\
& y'(0)=B,
\end{cases}
\label{eq:1DoFNEW}
\end{align}
with $m,k,\Omega >0$, $r,F,\psi \geq 0$.
Its exact solution \cite{cheli2015advanced} reads $y^{ex}(t)=y_h(t)+y_p(t)$ with $y_p(t)=Y_p\cos(\Omega t + \psi)$ particular solution of the whole equation characterized by
\begin{align}
Y_p &=\frac{F}{\sqrt{(-m\Omega^2+k)^2+\Omega^2 r^2}},\qquad
\psi &=\varphi-\arg{(-m\Omega^2+k+i\Omega r)}
\label{eq:solpart2NEW}
\end{align}
and $y_h(t)$ general solution of the homogeneous equation
\begin{align}
y_h(t)=\begin{cases}
C_1e^{\lambda_1 t}+C_2e^{\lambda_2 t}, & \text{if }r> 2\sqrt{km},\\ 
C_1e^{\lambda t}+C_2te^{\lambda t}, & \text{if } r= 2\sqrt{km},\\
e^{-\frac{r}{2m} t}\left(C_1\cos(\omega t)+C_2\sin(\omega t)\right), & \text{if } r< 2\sqrt{km},
\end{cases} 
\label{eq:solhomoNEW}
\end{align}
where $\omega=\frac{1}{2m}\sqrt{4km-r^2}$, $\lambda_1$ and $\lambda_2$ are the real roots of the characteristics polynomial associated to \eqref{eq:1DoFNEW}, which are equal to $\lambda$ when $r= 2\sqrt{km}$. 
$C_1$ and $C_2$ are two constants computed by imposing the initial conditions $y(0)=A$ and $y'(0)=B$. 
The mathematical steps needed to get the solution are reported in the supplementary material. 
The second order scalar ODE \eqref{eq:1DoFNEW} can be rewritten in a standard way as a vectorial first order ODE. 
In the test, we have set $m=5$, $r=2$, $k=5$, $F=1$, $\Omega=2$, $\varphi=0.1$, $A=0.5$ and $B=0.25$ with a final time $T=4$.         
\renewcommand{\test}{vibratingDamped}
\begin{figure}[h!]
  \begin{center}
    bDeC\\\vspace{2mm}
    \includegraphics[width=0.48\textwidth]{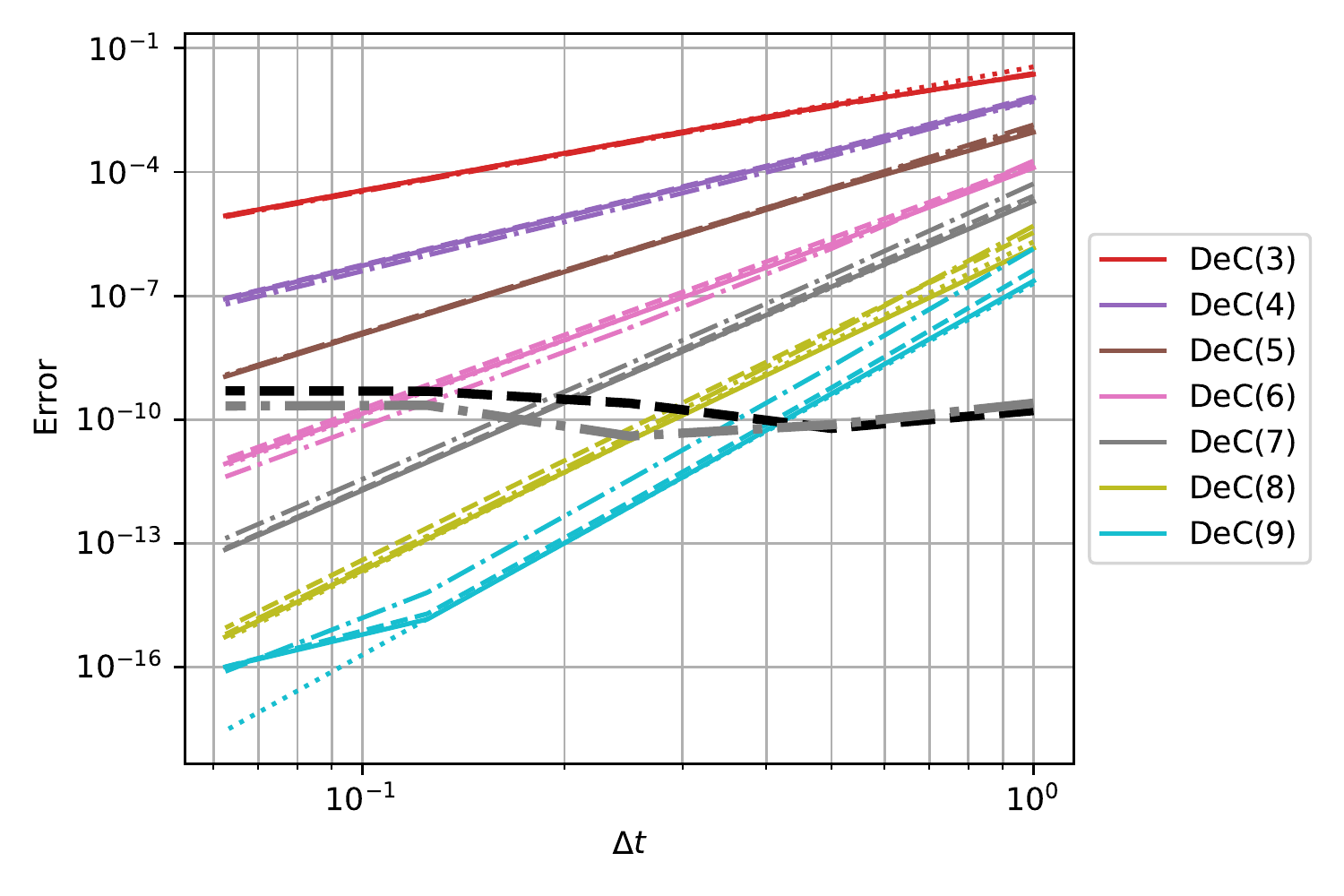} \hfill
    \includegraphics[width=0.48\textwidth]{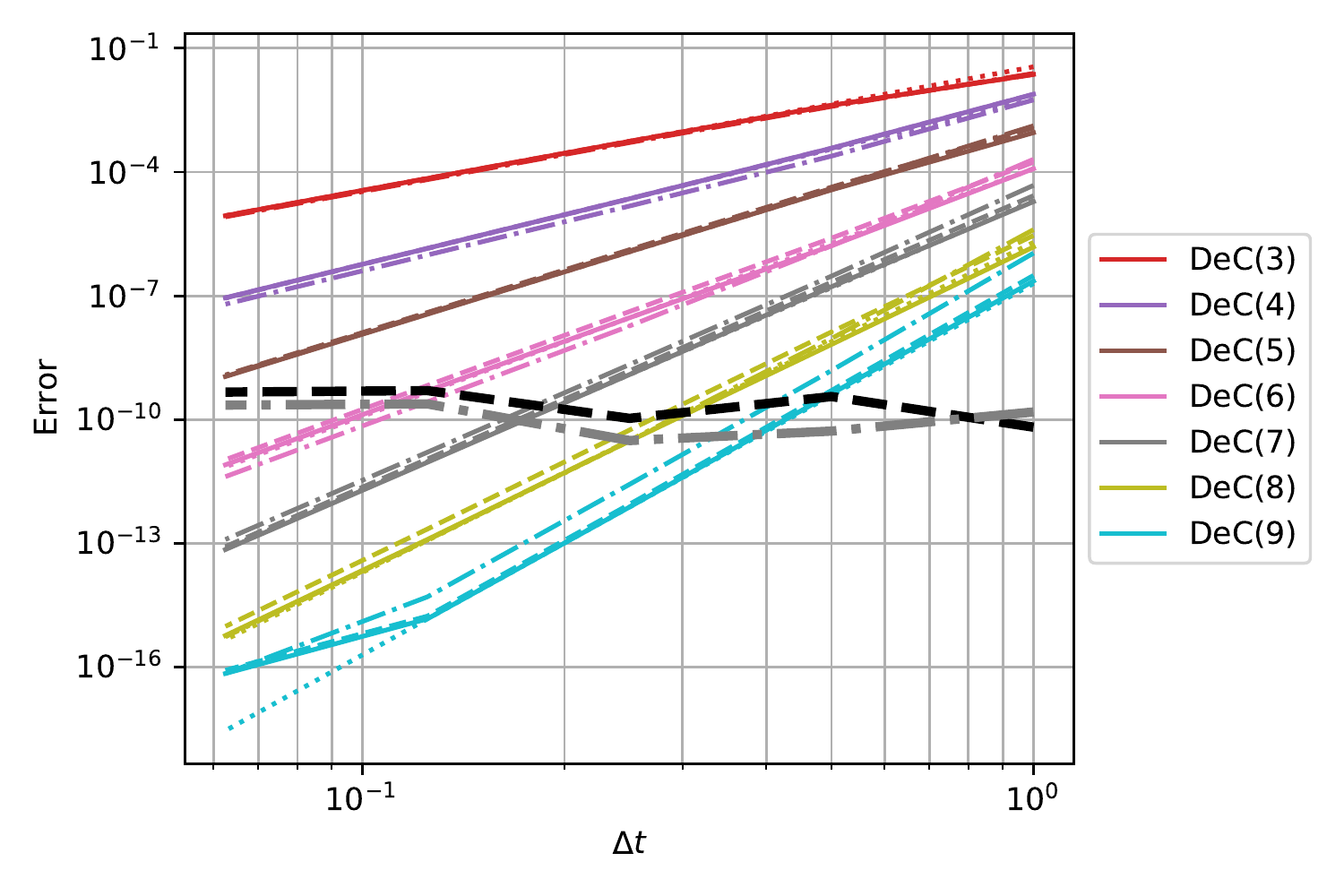}\\\vspace{2mm} 
    sDeC\\ \vspace{2mm}
    \includegraphics[width=0.48\textwidth]{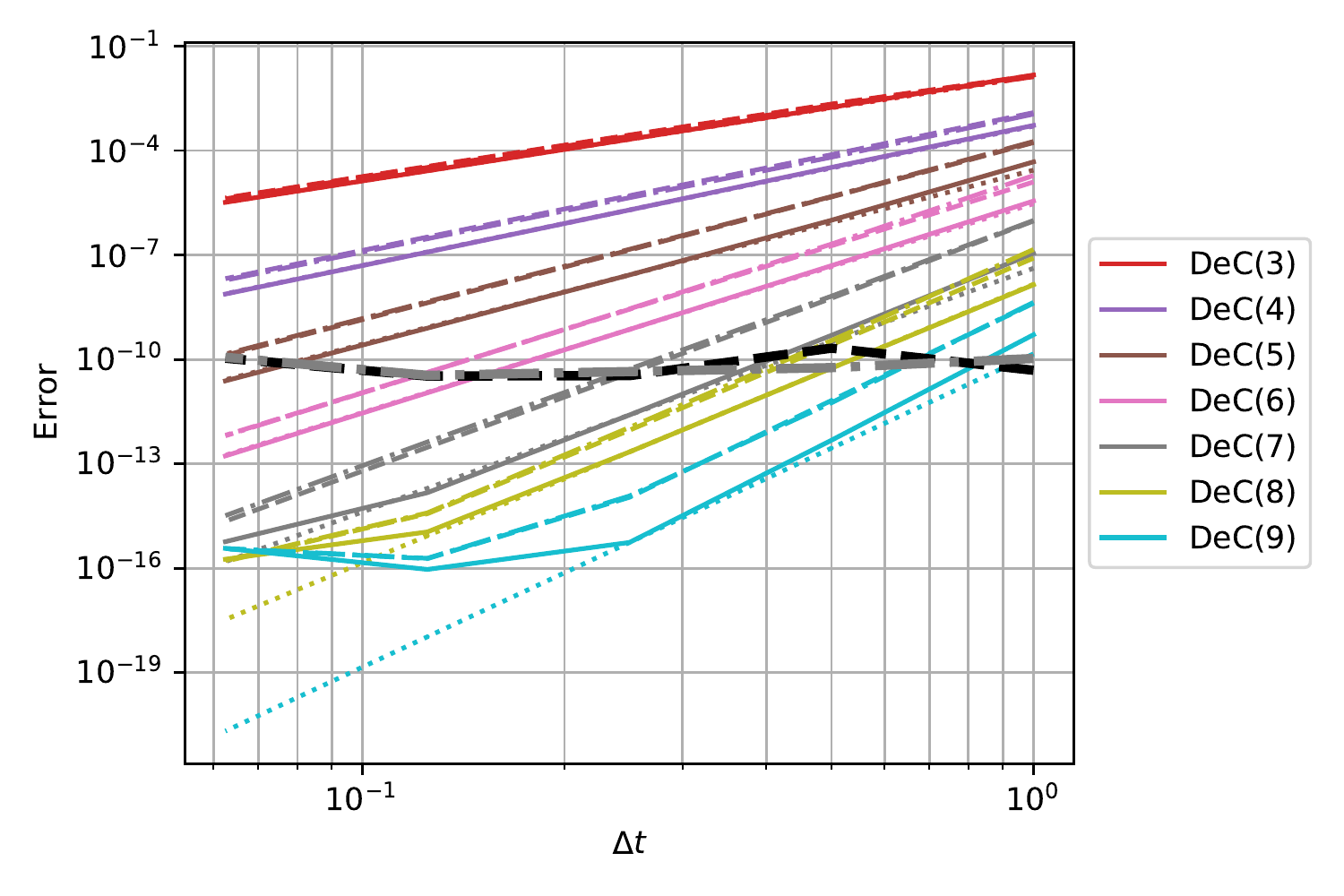}\hfill
    \includegraphics[width=0.48\textwidth]{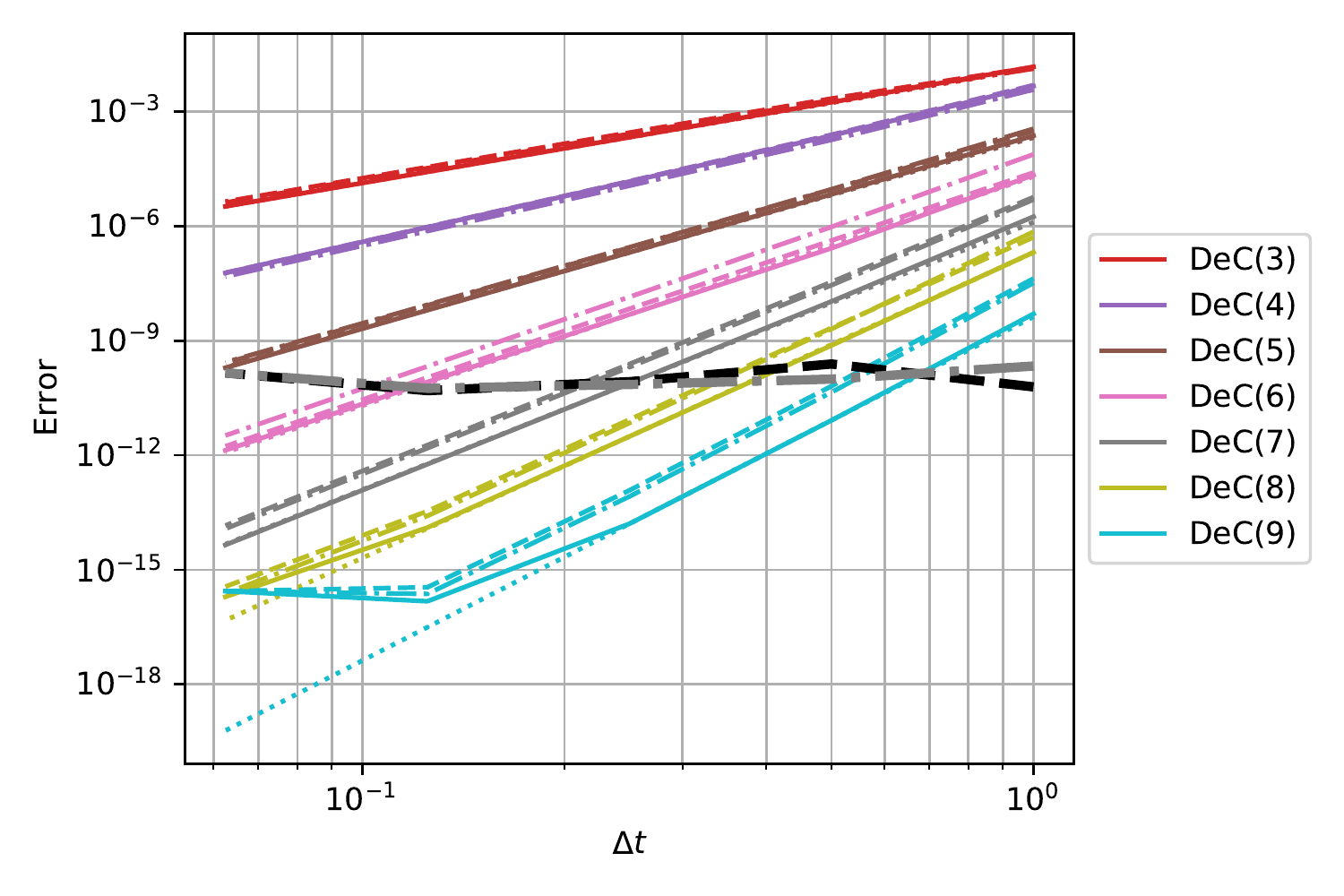}
  \end{center}
\caption{Vibrating system: Error decay for DeC with continuous line, DeCu with dashed line, DeCdu with dash-dotted line, reference order with dotted line, adaptive DeCu with dashed black line, adaptive DeCdu with dash-dotted gray line. Equispaced subtimenodes on the left and \GL~on the right}\label{fig:convergence_vibratingDamped}
\end{figure}
\begin{figure}[h!]
  \begin{center}
    bDeC\\\vspace{2mm}
    \includegraphics[width=0.48\textwidth]{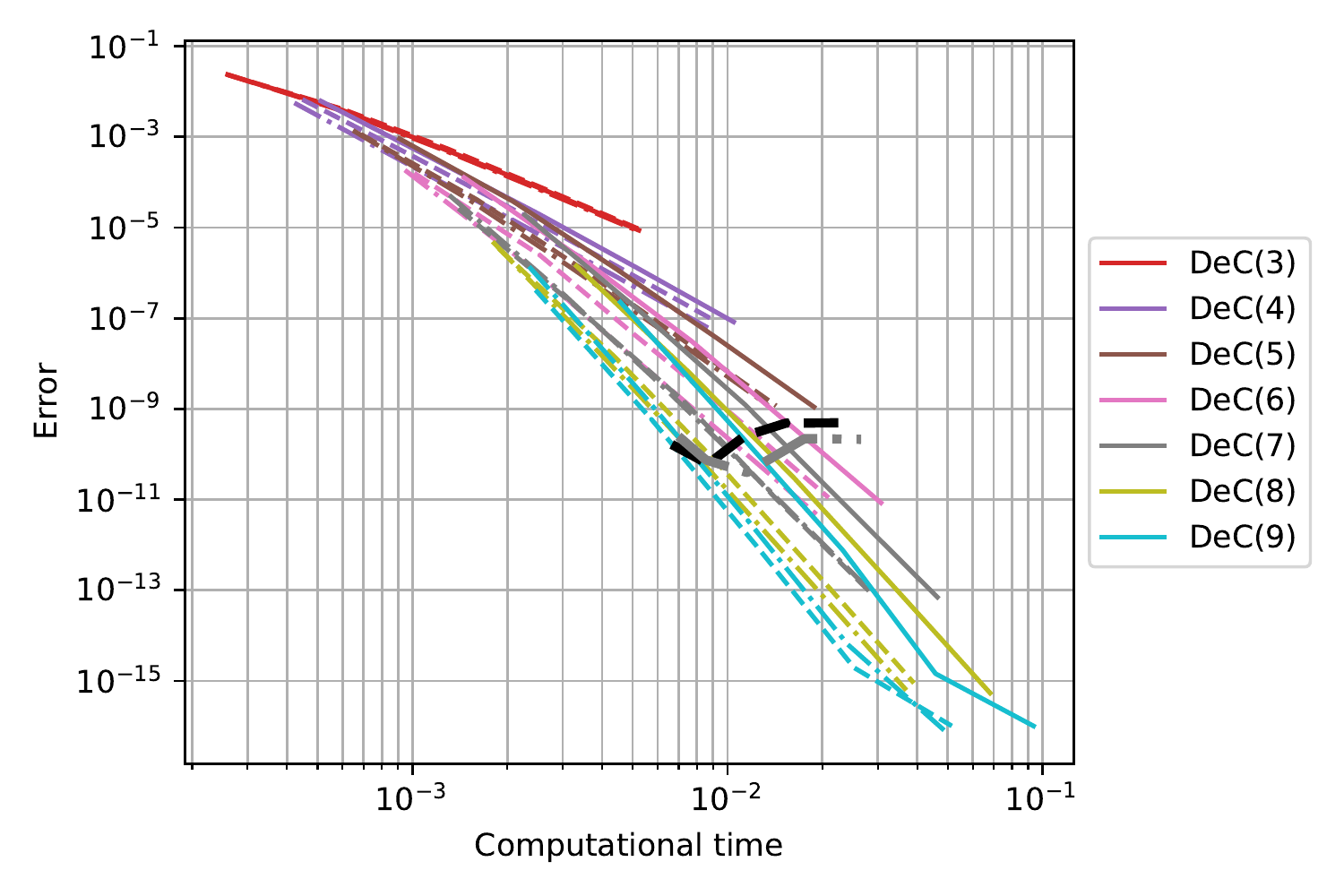} \hfill
    \includegraphics[width=0.48\textwidth]{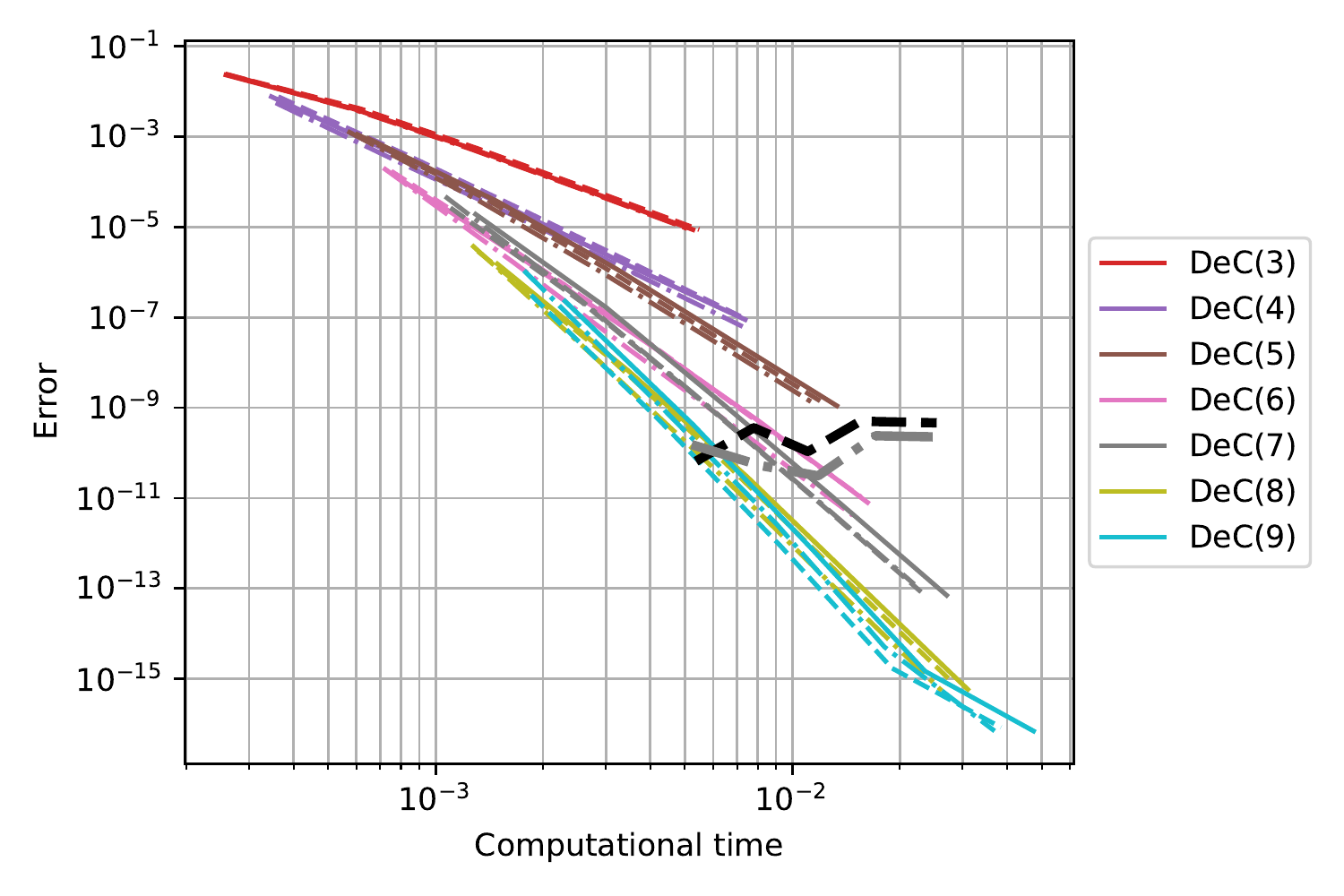}\\\vspace{2mm} 
    sDeC\\ \vspace{2mm}
    \includegraphics[width=0.48\textwidth]{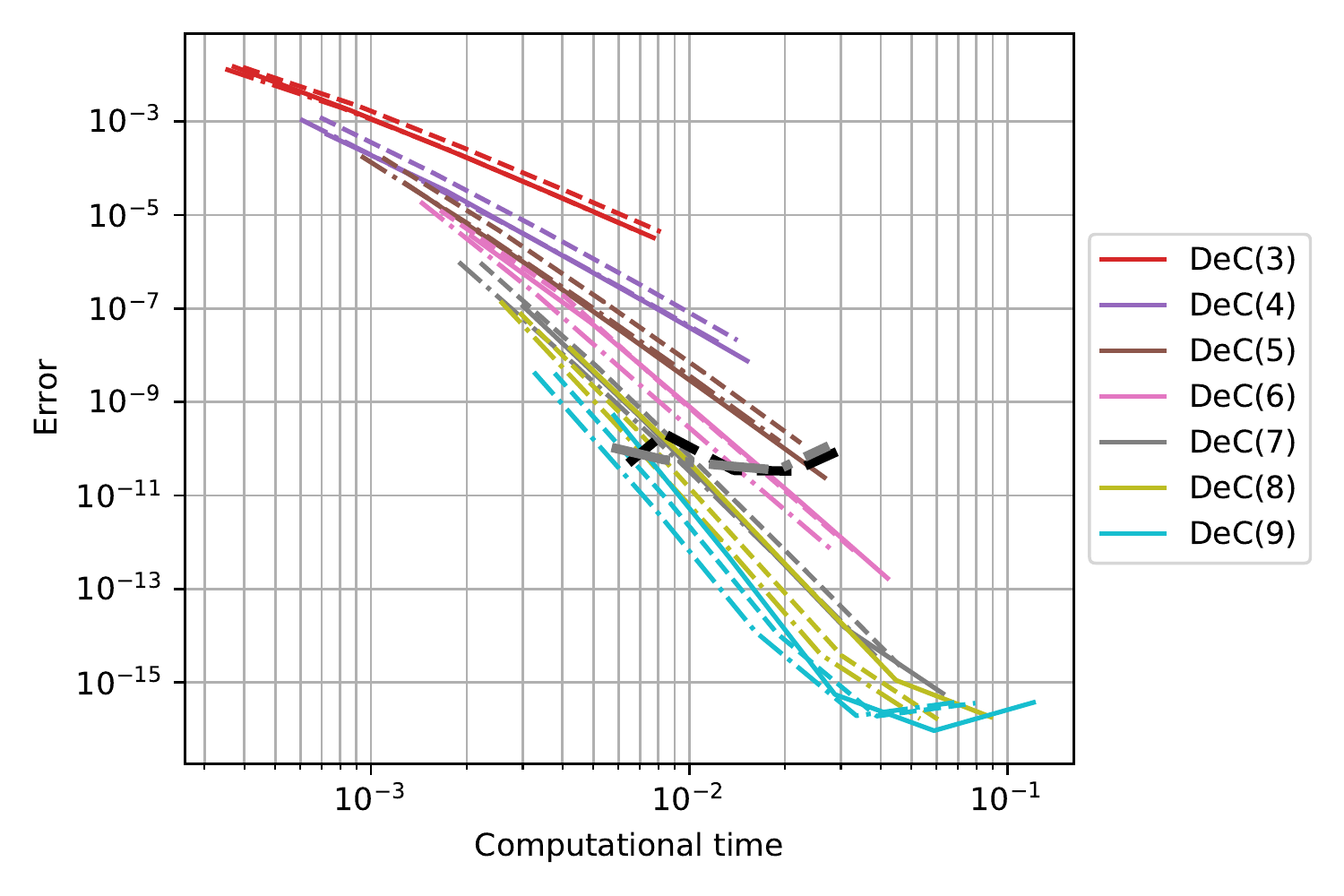}\hfill
    \includegraphics[width=0.48\textwidth]{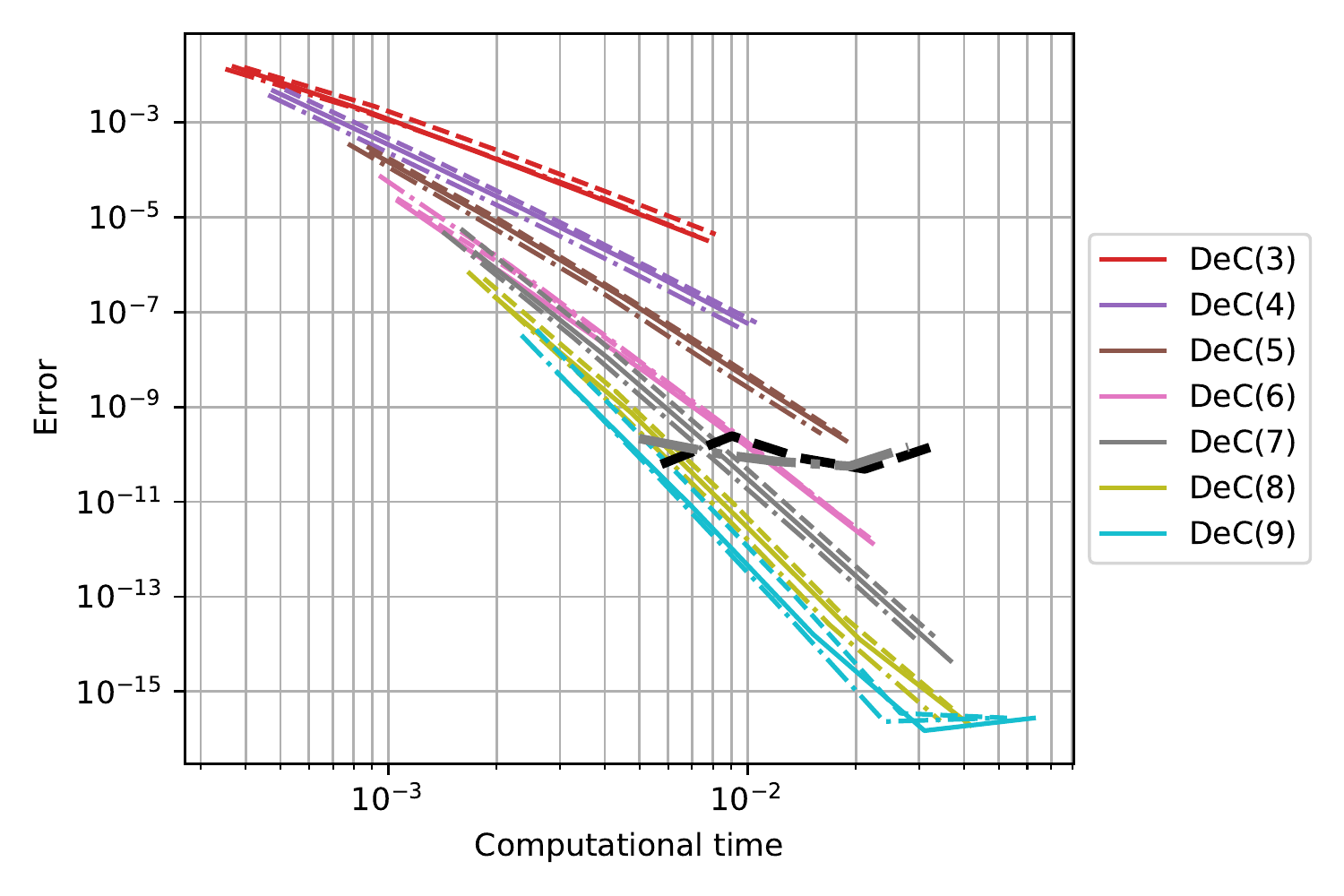}
  \end{center}
	\caption{Vibrating system: Error with respect to computational time for DeC with continuous line, DeCu with dashed line, DeCdu with dash-dotted line, adaptive DeCu with dashed black line, adaptive DeCdu with dash-dotted gray line. Equispaced subtimenodes on the left and \GL~on the right}\label{fig:conv_time_vibratingDamped}
\end{figure}
In Figure~\ref{fig:convergence_vibratingDamped}, we show the error decay for all methods. Differently from the linear case, here bDeC, bDeCu and bDeCdu are not equivalent. 
Nevertheless, in terms of errors, they behave in a similar way and, also comparing equispaced and \GL~subtimenodes, we do not observe large deviations. On average the novel schemes are slightly less accurate for a fixed $\Delta t$, even if this is not true for all orders of accuracy. 
For the sDeC, there is a larger difference in the errors between sDeC and sDeCu or sDeCdu, though being the order of accuracy always correct.
These effects are visible also in Figure~\ref{fig:conv_time_vibratingDamped}. For bDeC with equispaced subtimenodes, the advantages of using the novel methods are evident: the error is almost the same and the computational time reduces by almost half for high order schemes. 
For bDeC methods with \GL~subtimenodes the computational advantage of the novel methods is not as big as the one registered in the previous case as expected from theory, see Table~\ref{tab:stagesNumbersEquispaced} and Table~\ref{tab:stagesNumbersGL}, but still pretty visible.
For what concerns the sDeC methods with equispaced subtimenodes, the performance of sDeCdu is similar to the one of sDeC until order $5$, while, from order $6$ on, the novel method is definitely more convenient. The sDeCu method is always less efficient than the sDeCdu one; in particular, only for very high orders it appears to be preferable to the standard method. 
The general trend of the sDeC methods with \GL~subtimenodes is that the sDeCdu and the sDeCu always perform, respectively, slightly better and slightly worse than the original sDeC. 
The results of the adaptive methods for this test are qualitatively similar to the ones seen in the context of the previous test: the methods produce a constant error for any $\Delta t$.
Also in this case, the threshold for the relative error has been chosen equal to $10^{-8}$.  
%
%
Finally, in Figure~\ref{fig:speed_up_vibrating}, we display the speed up factor of the new bDeCdu methods with respect to the original bDeC: as expected from theory, it increases with the order of accuracy. 
\begin{figure}[h!]
  \begin{center}
    \includegraphics[height=0.18\textheight]{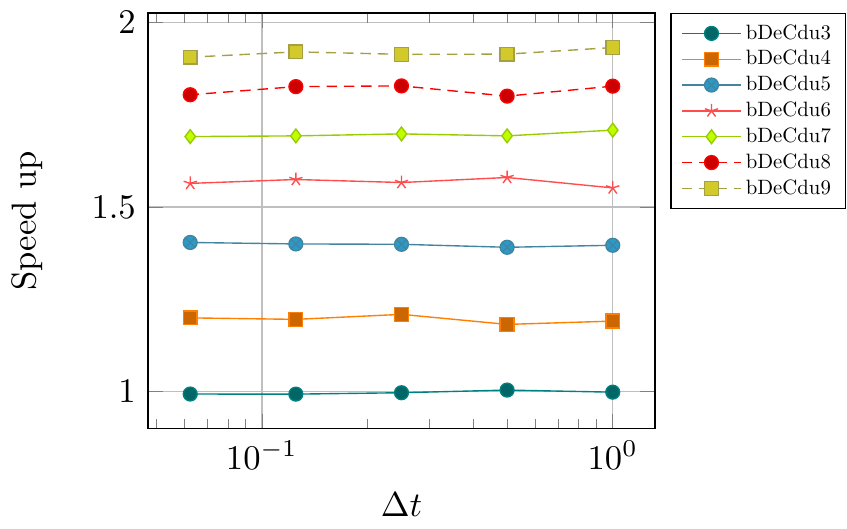}
    \includegraphics[height=0.18\textheight]{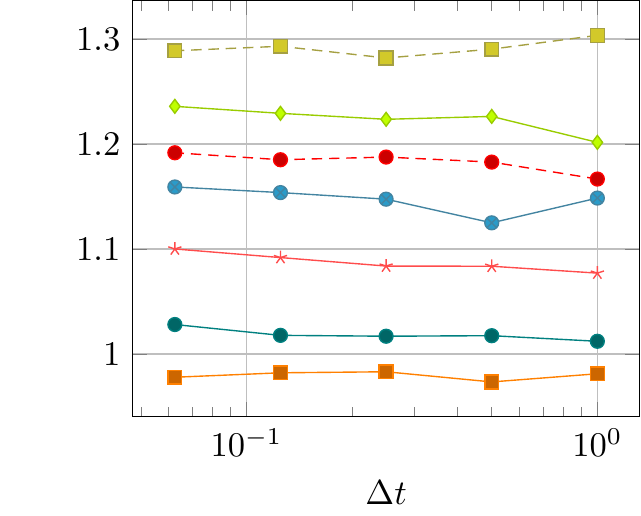}
  \end{center}
	\caption{Vibrating system test: Speed up factor for the bDeCdu method. Equispaced subtimenodes on the left and \GL~on the right\label{fig:speed_up_vibrating}}
\end{figure}

\subsection{Hyperbolic PDE tests}
For hyperbolic PDEs, we will focus on the bDeC and the bDeCu methods with equispaced subtimenodes. The order of the DeC will be chosen to match the spatial discretization one. 
We will use two stabilizations discussed in \cite{michel2021spectral,michel2022spectral}: continuous interior penalty (CIP) and orthogonal subscale stabilization (OSS).
The CIP stabilization is defined as
\begin{equation}
\uvec{ST}_i(\uapp)=\sum_{f \in \mathcal{F}_h} \alpha^{\CIP}_{f} \int_f [\![ \nabla_{\nu_f} \varphi_i ]\!] \cdot [\![ \nabla_{\nu_f} \uvec{u}_h ]\!]  \textrm{d}\sigma(\uvec{x}),     
\label{eq:CIP}
\end{equation}
where $\alpha^{\CIP}_{f}=\delta^{\CIP} \bar{\rho}_f h_f^{2}$, $\mathcal{F}_h$ is the set of the $(D-1)$-dimensional faces shared by two elements of $\mathcal{T}_h$, $[\![ \cdot ]\!]$ is the jump across the face $f$, $\nabla_{\nu_f}$ is the partial derivative in the direction $\nu_f$ normal to the face $f$, $\bar{\rho}_f$ is a local reference value for the spectral radius of the normal Jacobian of the flux, $h_f$ is the diameter of $f$ and $\delta^{\CIP}$ is a parameter that must be tuned. 

The OSS stabilization is given by
\begin{equation}
\uvec{ST}_i(\uapp)=\sum_{K \in \tess} \alpha^{\OSS}_K \int_K  \nabla_{\uvec{x}}\varphi_i \left( \nabla_{\uvec{x}} \uvec{u}_h -\uvec{w}_h\right)   \textrm{d}\uvec{x},
\label{eq:OSS}
\end{equation}
where $\alpha^{\OSS}_K = \delta^{\OSS} \bar{\rho}_K h_K $, $\uvec{w}_h$ is the $L^2$ projection of $\nabla_{\uvec{x}} \uapp$ onto $V_h^{Q\times D}$, $\bar{\rho}_K$ is a local reference value for the spectral radius of the normal Jacobian of the flux, $h_K$ is the diameter of $K$ and $\delta^{\OSS}$ is a parameter that must be tuned.

\subsubsection{1D Linear Advection Equation}
We consider the linear advection equation (LAE), $
u_t+u_x=0,$ with periodic boundary conditions on the domain $\Omega=[0,1]$, initial condition $u_0(x)=\cos(2\pi x)$ and final time $T=1$. The exact solution is given by $u(x,t)=u_0(x-t)$.
\begin{figure}[h!]
  \begin{center}
    \includegraphics[width=0.4\textwidth]{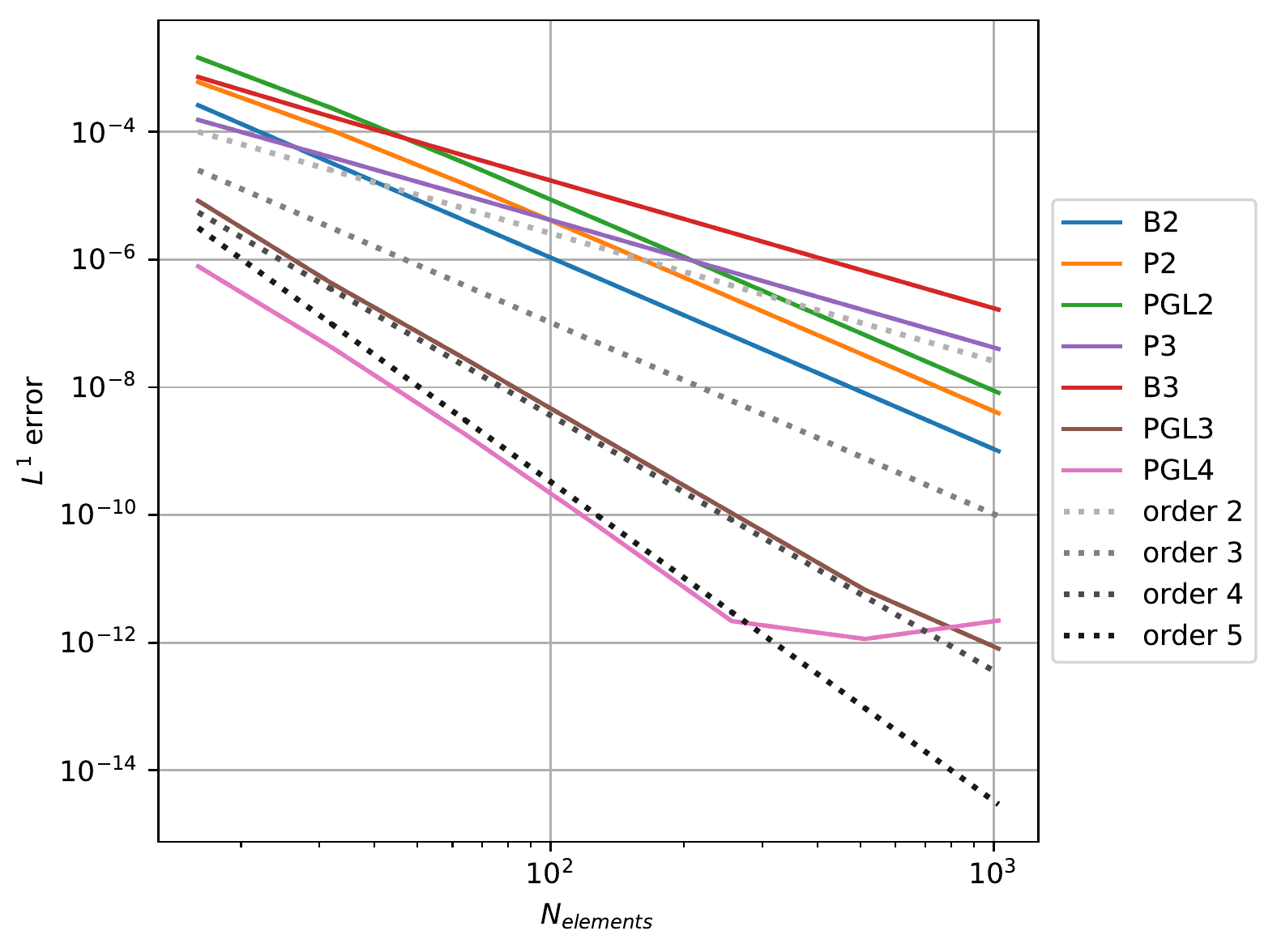}
    \includegraphics[width=0.4\textwidth]{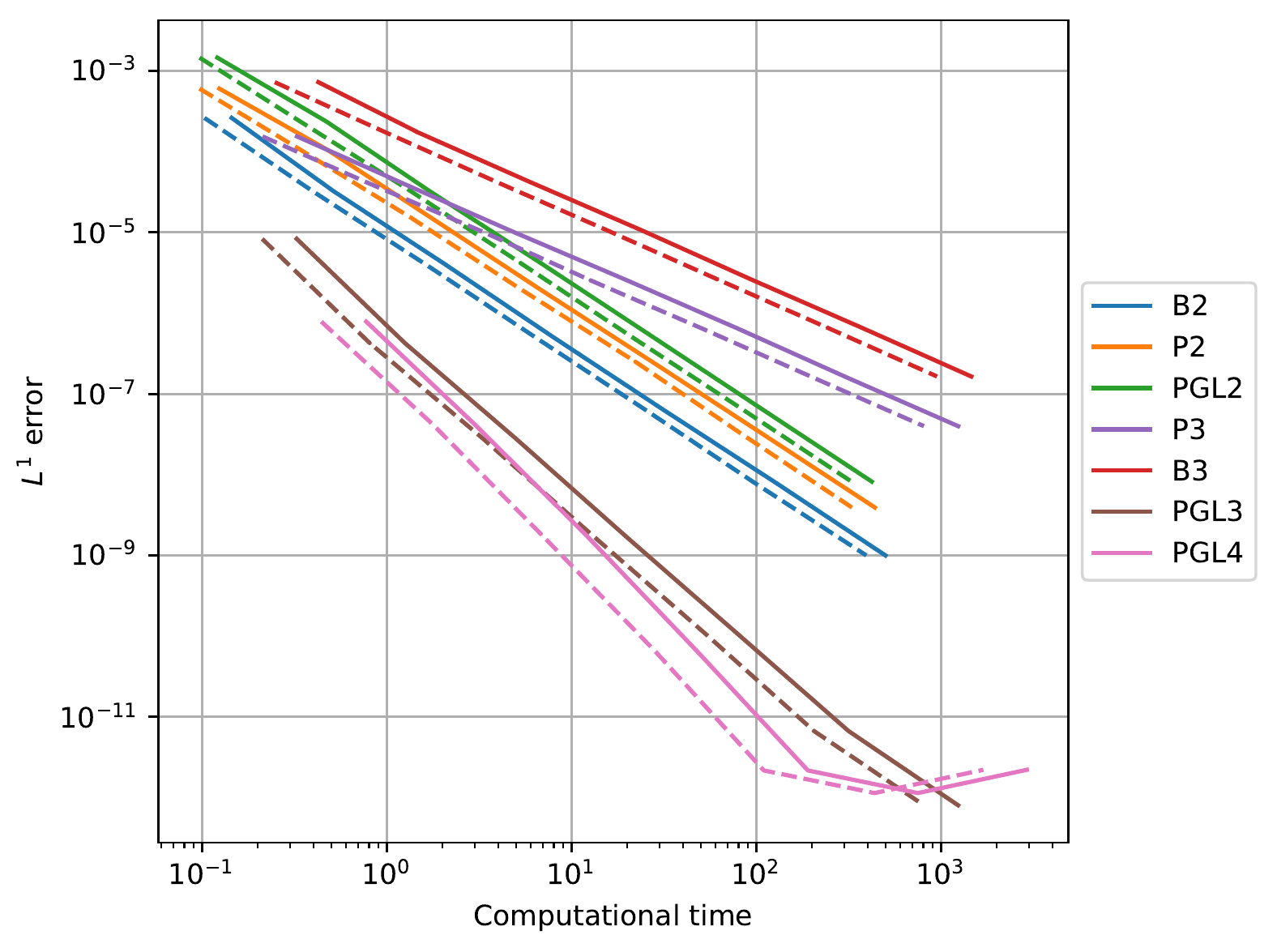}
  \end{center}
	\caption{1D LAE: bDeC with continuous line, bDeCu with dashed line, reference order with dotted line. Convergence analysis on the left and error with respect to computational time on the right\label{fig:1dtests}}
\end{figure}
For the spatial discretization, we considered three families of polynomial basis functions with degree $n$: B$n$, the Bernstein polynomials \cite{abgrall2019high,Decremi}; P$n$, the Lagrange polynomials associated to equispaced nodes; PGL$n$, the Lagrange polynomials associated to the \GL~nodes \cite{michel2021spectral}. 
For B$n$ and P$n$, we used the bDeC version for hyperbolic PDEs \eqref{eq:DeCfinalformulaPDE} introduced by Abgrall; for PGL$n$, we adopted the bDeC formulation for ODEs \eqref{eq:DeCODE_Remi}, as, in this case, the adopted quadrature formula associated to the Lagrangian nodes leads to a high order mass lumping.  
For all of them, we used the CIP stabilization \eqref{eq:CIP} with the coefficients $\delta^{\CIP}$ reported in Table~\ref{tab:coeff_CIP_1D} found in~\cite{michel2021spectral} to minimize the dispersion error, even if, differently from there, we assumed here a constant CFL $=0.1$. 
In particular, since the coefficients for P3 and PGL4 were not provided, we used for the former the same coefficient as for B3, while, for the latter the same coefficient as for PGL3.
\begin{table}
	\centering
	\begin{tabular}{|c|c|c|c|c|c|c|c|}\hline
	 &	B2 & P2 & PGL2 & B3 & P3$^{*}$ & PGL3 & PGL4$^{*}$\\\hline
  $\delta^{\CIP}$  &  0.016  &  0.00242  &  0.00346 &  0.00702  &  0.00702 &  0.000113 & 0.000113\\ \hline		
	\end{tabular}
	\caption{Coefficients $\delta^{\CIP}$ used for LAE in one dimension. $^{*}$The coefficients adopted for P3 and PGL4 are not provided in \cite{michel2021spectral}. \label{tab:coeff_CIP_1D}}
\end{table}
The results of the convergence analysis and of the computational cost analysis are displayed in Figure~\ref{fig:1dtests}. 
For a fixed number of elements, the errors of the bDeC and of the bDeCu methods are essentially identical, leading to a remarkable computational advantage of the novel method with respect to the original bDeC, visible in the plot on the right, where the error against the computational time is depicted. The formal order of accuracy is recovered in all the cases but for B3 and P3 for which we get only second order for both bDeC and bDeCu. 
\begin{remark}[Issues with the DeC for PDEs]
The loss of accuracy for bDeC4 and B3 elements has been registered in other works, e.g. \cite{michel2021spectral,michel2022spectral,abgrall2019high}.
Even in the original paper \cite{Decremi}, the author underlines the necessity to perform more iterations than theoretically expected for orders greater than $3$ to recover the formal order of accuracy.
According to authors' opinion the problem deserves a particular attention, for this reason, the results related to B3 and P3 have not been omitted.
The pathology seems to have effect only in the context of unsteady tests and it is maybe due to a \textit{high order weak instability}. The phenomenon is currently under investigation; more details can be found in the supplementary material. However, we remark that this issue does not occur for elements that allow a proper mass lumping like PGL (or Cubature in 2D).
\end{remark}
The speed up factor of the novel bDeCu with respect to the original method is reported in Figure~\ref{fig:hyperbolic_tests}. 
The obtained speed up factors are higher than ODE ones, because in the implementation of the DeC for PDEs the major cost is not given by the flux evaluation of previously computed stages, but by the evolution of the new stages. This slightly changes the expected and the observed speed up, providing even larger computational advantages.

\subsubsection{2D Shallow Water Equations}
We consider the Shallow Water (SW) equations onto $\Omega=(0,3)\times(0,3)\in \R^2$, defined, in the form \eqref{sys}, by
\begin{equation}
\uvec{u}=\begin{pmatrix}
H\\
H\uvec{v}
\end{pmatrix},\quad\uvec{F}(\uvec{u})=\begin{pmatrix}
H\uvec{v}\\
H\uvec{v}\otimes \uvec{v}+g\frac{H^2}{2}\mathbb{I}
\end{pmatrix},\quad 
\boldsymbol{S}(\uvec{x},\uvec{u})=0,
\label{sw}
\end{equation}
where $H$ is the water height, $\uvec{v}=(v_1,v_2)^T\in\mathbb{R}^2$ is the vertically averaged speed of the flow, $g$ is the gravitational constant, $\mathbb{I}\in\mathbb{R}^{D\times D}$ is the identity matrix and $D=2$ is the number of physical dimensions. The test is a $C^6(\Omega)$ compactly supported unsteady vortex from the collection presented in \cite{vort} given by 
\begin{equation}
\uvec{u} = \uvec{u}^\infty + \begin{cases}
	\uvec{u}_{r_0}(r), & \text{ if }r=||\uvec{x}-\uvec{x}_m(t)||_2 <r_0,\\
	0, & \text{ else,}
\end{cases}
\end{equation}
where $\uvec{u}^{\infty} =(1,1,1)^T$, $\uvec{x}_m(t) = \uvec{x}_c + t\cdot(1,1)^T$ and 
\begin{align}
	&\uvec{u}_{r_0}(r) = \begin{pmatrix}
		\frac{1}{g}\left(\frac{\Gamma}{\omega}\right)^2 \left(\lambda(\omega r)-\lambda(\pi)\right)\\
		\Gamma\left(1+\cos(\omega r)\right)^2(x_2-x_{m,2})\\
		- \Gamma\left(1+\cos(\omega r)\right)^2(x_1-x_{m,1}) 
	\end{pmatrix},\; \Gamma = \frac{12\pi\sqrt{g\Delta H}}{r_0\sqrt{315\pi^2 - 2048}}
\end{align}
with $\omega = \frac{\pi}{r_0}$ and the function $\lambda$ defined by
\begin{equation}
\begin{split}
\lambda(s)& = \frac{20}{3}\cos(s) + \frac{27}{16}\cos(s)^2 + \frac{4}{9}\cos(s)^3+ \frac{1}{16}\cos(s)^4 +\frac{20}{3}s\sin(s) \\
& + \frac{35}{16}s^2 + \frac{27}{8}s\cos(s)\sin(s) +\frac{4}{3}s\cos(s)^2\sin(s) + \frac{1}{4}s\cos(s)^3\sin(s).
\label{eq:supp_vort}	
\end{split}
\end{equation}
We set $g=9.81$, $r_0=1$, $\Delta H = 0.1$, $\uvec{x}_c=(1,1)^T$ with a final time $T=1$ and Dirichlet boundary conditions.
For the spatial discretization, we considered two basis functions: B$n$, the Bernstein polynomials; C$n$, the Cubature elements introduced in \cite{cohen:hal-01010373}. %
As they allow a high order mass lumping, for C$n$ elements we used the  bDeC \eqref{eq:DeCODE_Remi} for ODEs and OSS stabilization \eqref{eq:OSS}, instead, for B$n$ we considered the PDE formulation \eqref{eq:DeCfinalformulaPDE} and CIP stabilization \eqref{eq:CIP}.
The tests with B2 have been run with $\CFL=0.1$ and $\delta^{\CIP}=0.04$; for C2 elements we have set $\CFL=0.12$ and $\delta^{\OSS}=0.07$, the optimal coefficients minimizing the dispersion error of the original bDeC according to the linear analysis performed in \cite{michel2022spectral}; for C3 we adopted $\CFL=0.015$ and $\delta^{\OSS}=0.2$.
\begin{figure}[h!]
  \begin{center}
    \includegraphics[height=0.18\textheight]{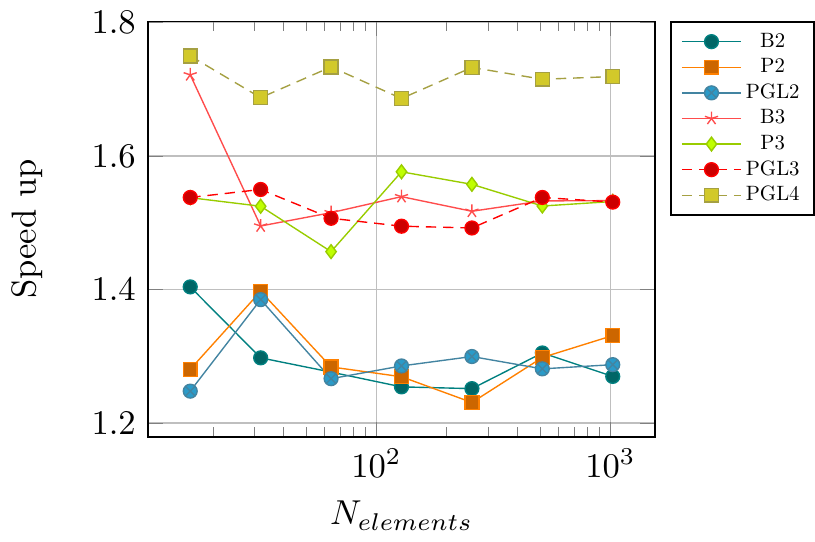}
    \includegraphics[height=0.18\textheight]{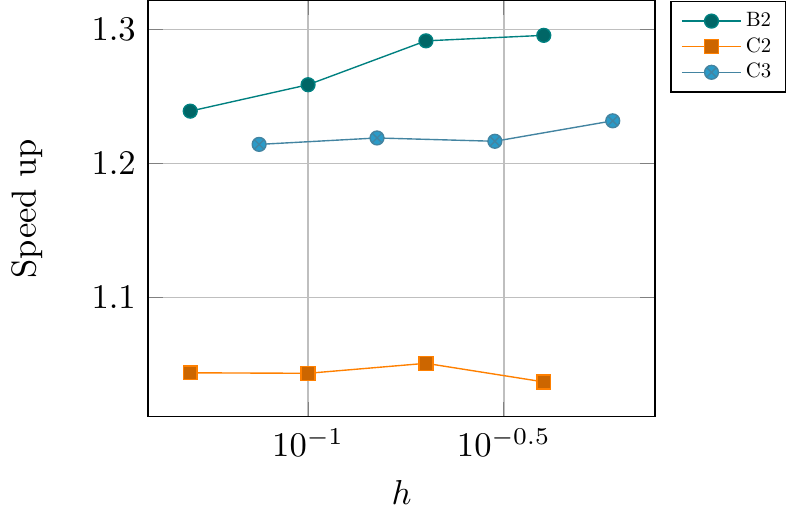}
  \end{center}
	\caption{Speed up in the hyperbolic tests of bDeCu with respect to bDeC. 1D LAE on the left and 2D SW on the right\label{fig:hyperbolic_tests}}
\end{figure}
\begin{figure}[h!]
  \begin{center}
    \includegraphics[width=0.4\textwidth]{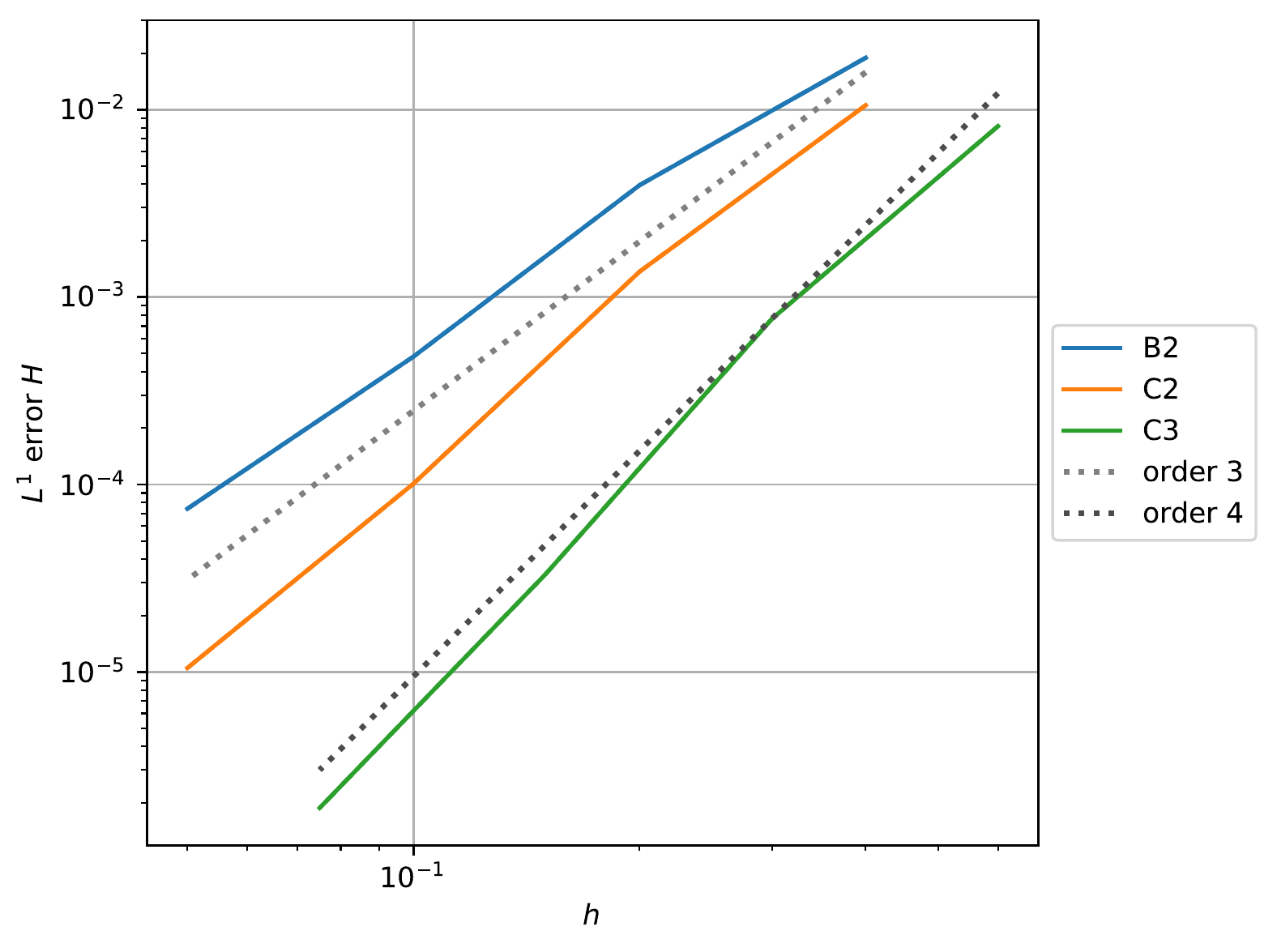}
    \includegraphics[width=0.4\textwidth]{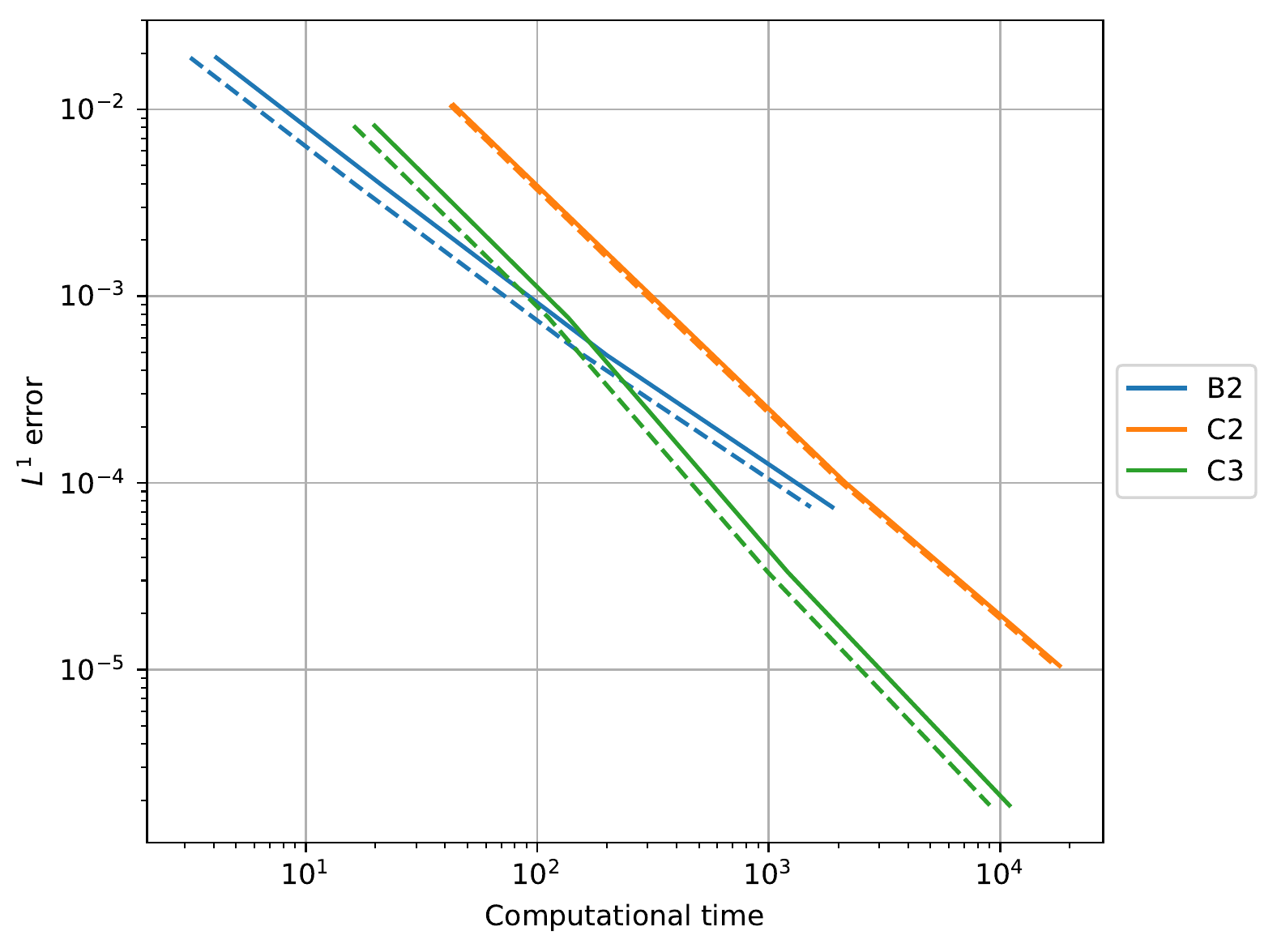}
  \end{center}
	\caption{2D SW: bDeC with continuous line, bDeCu with dashed line, reference order with dotted line. Convergence analysis on the left and error with respect to computational time on the right\label{fig:2dtests}}
\end{figure}
The results of the convergence analysis and of the computational cost analysis are displayed in Figure~\ref{fig:2dtests}. 
The errors produced by the novel and the original bDeC method are so close that the lines coincide. 
The resulting computational advantage can be seen in the plot on the right. The formal order of accuracy is recovered in all the cases and the speed up factor, in Figure~\ref{fig:hyperbolic_tests}, proves the convenience in using the novel bDeCu formulation instead of the original bDeC.
Let us observe that, according to Table~\ref{tab:stagesNumbersEquispaced}, the number of stages of bDeC3 and bDeCu3 is identical, nevertheless, as observed in Remark~\ref{remark_on_alphadecu}, the number of stages does not strictly correspond to the computational time. If we do not consider the ``cheap'' stages computed via interpolation, we get the theoretical speed up factor $\frac{5}{4}=1.25$, which is what we obtained in the numerical test for B2.
%
We conclude this section with one last observation: the computational advantage registered with B2 is much higher with respect to C2 and C3 ones, because we have run the simulations with different codes: the results obtained with B2 are obtained with a Fortran implementation, while, for C2 and C3 we have used Parasol, a Python implementation developed by Sixtine Michel \cite{michel2022spectral} and kindly provided to us. A more careful implementation would increase further the speed up factors associated to such elements.

\section{Conclusions and further developments}\label{chapConc}
In this work, we have investigated analytical and numerical aspects of two novel families of efficient explicit DeC methods. 
The novel methods are constructed by introducing interpolation processes between the iterations, which increase the degree of the discretization in order to match the accuracy of the approximation associated to the iterations. 
In particular, we proved that for some of the novel methods the stability region coincides with the one of the original methods. 
The novel methods have been tested on classical benchmarks in the ODE context revealing, in most of the cases, a remarkable computational advantage with respect to the original ones. 
Furthermore, the interpolation strategies have been used to design adaptive schemes.
Finally, we successfully proved the good performances of the novel methods in the context of the numerical solution of hyperbolic PDEs with continuous space discretizations.
Overall, we believe that the approach proposed in this work can alleviate the computational costs not only of DeC methods but also of other schemes with a similar structure.
For this reason, investigations of other numerical frameworks are planned and, in particular, we are working on applications to hyperbolic PDEs (with FV and ADER schemes), in which also the order of the space reconstruction is gradually increased iteration by iteration.
We hope to spread broadly this technique in the community in order to save computational time and resources in the numerical solution of differential problems, as only little effort is required to embed the novel modification in an existing DeC code.
\subsubsection*{Supplementary information}
The interested reader is referred to the supplementary material for all the proofs omitted in this document for the sake of compactness.
\subsubsection*{Acknowledgments}
L. Micalizzi has been funded by the SNF grant 200020\_204917 ``Structure preserving and fast methods for hyperbolic systems of conservation laws'' and by the Forschungskredit grant FK-21-098.
D. Torlo has been funded by a SISSA Mathematical Fellowship.
The authors warmly acknowledge Sixtine Michel for providing the code Parasol.
\subsubsection*{Conflict of interest}
On behalf of all authors, the corresponding author states that there is no conflict of interest. 
\subsubsection*{Compliance with Ethical Standards}
On behalf of all authors, the corresponding author is available to collect documentation of compliance with ethical standards and send upon request.
\subsubsection*{Funding}
This work received financial support by the Swiss National Foundation (Switzerland) and Scuola Internazionale Superiore di Studi Avanzati (Italy).

\appendix
\section{Residual formulations}\label{app:residual_formulations}
Here, we report the residual formulations of the original bDeC and sDeC methods presented in Section~\ref{chapDECODE}. 
In particular, we will present the spectral DeC formulation in terms of residuals introduced in \cite{Decoriginal} and prove that it is equivalent to the sDeC method. 
Then, we will see how to get, with a little modification of the presented spectral DeC formulation, the residual formulation of the bDeC method. 
\subsection{Link between spectral DeC and sDeC}
\begin{subequations}
We want to solve system \eqref{ODE} in the interval $[t_n,t_{n+1}]$ getting $\uvec{u}_{n+1}\approx\uvec{u}(t_{n+1})$ from $\uvec{u}_{n}=\uvec{u}(t_{n})$. 
%
%
Also in this case, we consider an iterative procedure on the approximated values $\uvec{u}^{m,(p)}$ of the solution in the subtimenodes $m=1,\dots,M$, collected in the vector $\undu^{(p)}$, with $\uvec{u}^{0,(p)}:=\uvec{u}_{n}$ fixed. 
Given $\undu^{(p-1)}$, we consider the interpolation polynomial $\uvec{\mathcal{I}}(\undu^{(p-1)},t):=\sum_{m=0}^M \uvec{u}^{m,(p-1)}\psi^m(t)$.
The spectral DeC relies on the definition at each iteration $p$ of two support variables, namely the error function $\uvec{e}^{(p)}(t)$ with respect to the exact solution and the residual function $\uvec{r}^{(p-1)}(t)$ respectively given by
\begin{align}
\uvec{e}^{(p)}(t)&:=\uvec{u}(t)-\uvec{\mathcal{I}}(\undu^{(p-1)},t),\\
\uvec{r}^{(p-1)}(t)&:=\uvec{u}_n+\int_{t_n}^{t}\uvec{G}(s,\uvec{\mathcal{I}}(\undu^{(p-1)},s))ds -\uvec{\mathcal{I}}(\undu^{(p-1)},t).
\end{align}
By integrating the original ODE \eqref{ODE}, making use of the definitions of the error function $\uvec{e}^{(p)}(t)$ and of the residual function $\uvec{r}^{(p-1)}(t)$ and differentiating again, we get that the error function satisfies the ODE 
\begin{align}
\begin{cases}
\frac{d}{dt}\uvec{e}^{(p)}(t)=\uvec{G}(t,\uvec{\mathcal{I}}(\undu^{(p-1)},t)+\uvec{e}^{(p)}(t))-\uvec{G}(t,\uvec{\mathcal{I}}(\undu^{(p-1)},t))+\frac{d}{dt}\uvec{r}^{(p-1)}(t),\\
\uvec{e}^{(p)}(t_n)=0.
\end{cases}
\label{eq:ODEerror}
\end{align}
We can numerically solve such ODE in each subinterval $[t^{m-1},t^{m}]$ through the explicit Euler method starting from $m=1$ on, thus getting 
\begin{align}
\begin{split}
\uvec{e}^{m,(p)}=\uvec{e}^{m-1,(p)} &+ \Delta t \gamma^m \left[ \right. \uvec{G}(t^{m-1},\uvec{u}^{m-1,(p-1)}+\uvec{e}^{m-1,(p)})\\
&\left .-\uvec{G}(t^{m-1},\uvec{u}^{m-1,(p-1)}) \right]+\uvec{r}^{m,(p-1)}-\uvec{r}^{m-1,(p-1)},
\end{split}
\label{eq:almostsdec}
\end{align}
with the integrals in the residual function approximated through a spectral integration, i.e., $\uvec{r}^{m,(p-1)}:=\uvec{u}_n+\int_{t_n}^{t^m}\sum_{\ell=0}^M \uvec{G}(t^{\ell},\uvec{u}^{\ell,(p-1)}) \psi^{\ell}(t) dt-\uvec{u}^{m,(p-1)}$. We have used for $\uvec{e}^{m,(p)}$ and $\uvec{r}^{m,(p-1)}$ the usual convention adopted throughout the manuscript, with $m$ standing for the subtimenode $t^m$ to which such quantities are associated. Indeed, we have $\uvec{e}^{0,(p)}=\uvec{0}$ and $\uvec{r}^{0,(p-1)}=\uvec{0}$. 
The computed errors are then used to get new approximated values of the solution $\uvec{u}^{m,(p)}:=\uvec{u}^{m,(p-1)}+\uvec{e}^{m,(p)}$, allowing to repeat the described process with new error and residual functions, $\uvec{e}^{(p+1)}(t)$ and $\uvec{r}^{(p)}(t)$, analogously defined.
The procedure gains one order of accuracy at each iteration until the accuracy of the discretization is saturated and, at the end of the iteration process with $P$ iterations, one can set $\uvec{u}_{n+1}:=\uvec{u}^{M,(P)}$.
By explicit computation, we have that \eqref{eq:almostsdec} is equivalent to the sDeC updating formula \eqref{secondformulationintermediate}. In fact, recalling the definition of $\uvec{u}^{m,(p)}$ and $\uvec{r}^{m,(p-1)}$, we get
\begin{align}
\begin{split}
\uvec{e}^{m,(p)}=\uvec{e}^{m-1,(p)} &+ \Delta t \gamma^m \left[\uvec{G}(t^{m-1},\uvec{u}^{m-1,(p)})-\uvec{G}(t^{m-1},\uvec{u}^{m-1,(p-1)}) \right]\\
&+\int_{t^{m-1}}^{t^m}\sum_{\ell=0}^M \uvec{G}(t^{\ell},\uvec{u}^{\ell,(p-1)}) \psi^{\ell}(t) dt-\uvec{u}^{m,(p-1)}+\uvec{u}^{m-1,(p-1)},
\end{split}
\end{align}
from which, recalling the definition of $\delta_{\ell}^m$, follows
\begin{align}
\begin{split}
\uvec{u}^{m,(p)}=\uvec{u}^{m-1,(p)} &+ \Delta t \gamma^m \left[\uvec{G}(t^{m-1},\uvec{u}^{m-1,(p)})-\uvec{G}(t^{m-1},\uvec{u}^{m-1,(p-1)}) \right]\\
&+\Delta t\sum_{\ell=0}^M \delta_\ell^m \uvec{G}(t^{\ell},\uvec{u}^{\ell,(p-1)}).
\end{split}
\end{align}
\end{subequations}

\subsection{bDeC}
\begin{subequations}
The residual formulation of the bDeC method is obtained in a similar way.
Keeping the same definitions of $\uvec{\mathcal{I}}(\undu^{(p-1)},t)$, $\uvec{e}^{(p)}(t)$ and $\uvec{r}^{(p-1)}(t)$, we have that \eqref{eq:ODEerror} still holds.
We solve it through the explicit Euler method in each subinterval $[t^{0},t^{m}]$ obtaining
\begin{align}
\begin{split}
\uvec{e}^{m,(p)}=\uvec{e}^{0,(p)} &+ \Delta t \beta^m \left[\uvec{G}(t^{0},\uvec{u}^{0,(p-1)}+\uvec{e}^{0,(p)})-\uvec{G}(t^{0},\uvec{u}^{0,(p-1)}) \right]\\
&+\uvec{r}^{m,(p-1)}-\uvec{r}^{0,(p-1)},
\end{split}
\label{eq:almostbdec}
\end{align}
with the same definition for $\uvec{r}^{m,(p-1)}$ through spectral integration.
This is the residual formulation of the bDeC method.
Recalling that $\uvec{e}^{0,(p)}=\uvec{r}^{0,(p-1)}=\uvec{0}$, we get
\begin{align}
\begin{split}
\uvec{e}^{m,(p)}=\uvec{r}^{m,(p-1)}=\uvec{u}_n+\int_{t_n}^{t^m}\sum_{\ell=0}^M \uvec{G}(t^{\ell},\uvec{u}^{\ell,(p-1)}) \psi^{\ell}(t) dt-\uvec{u}^{m,(p-1)},
\end{split}
\end{align}
from which, recalling the definition of $\uvec{u}^{m,(p)}$ and of $\theta_{\ell}^m$, finally follows
\begin{align}
\begin{split}
\uvec{u}^{m,(p)}=\uvec{u}_n+\Delta t\sum_{\ell=0}^M  \theta_{\ell}^m \uvec{G}(t^{\ell},\uvec{u}^{\ell,(p-1)}),
\end{split}
\end{align}
which is nothing but \eqref{eq:DeCODE_Remi}.
\end{subequations}

\bibliography{sn-bibliography}
\bibliographystyle{plain}

\end{document}


\author{L. Micalizzi\footnote{Corresponding author. Affiliation: Institute of Mathematics, University of Zurich, Winterthurerstrasse 190, Zurich, 8057, Switzerland. Email: lorenzo.micalizzi@math.uzh.ch.} and D. Torlo\footnote{Affiliation: SISSA mathLab, SISSA, via Bonomea 265, Trieste, 34136, Italy. Email: davide.torlo@sissa.it.}}
\title{A new efficient explicit Deferred Correction framework: \\
analysis and applications to hyperbolic PDEs and adaptivity\\ Supplementary Material\footnote{Main document submitted to Communications on Applied Mathematics and Computation.}}

\maketitle

\section*{Introduction}
In this supplementary material, we show the proofs and the details that were too lengthy to be put in the principal manuscript. We show the proof of the Deferred Correction procedure in a general framework in \cref{chapDECAbstract}. In \cref{chapDECODE}, we provide the proofs of the accuracy and of the properties of the operators $\lopdt^2$ and $\lopdt^1$ of the bDeC method in the context of ODEs, and we show how the sDeC method can be seen as a perturbation of the bDeC.  In \cref{sec:CG}, we prove the properties of the operators $\lopdt^2$ and $\lopdt^1$ of the bDeC formulation for the continuous Galerkin (CG) finite element framework and we investigate the issues experienced in many works with such formulation. Finally, in \cref{sec:vibrating}, we show how to find the analytical solution to the ODE modeling a monodimensional vibrating system.

For each section, we recall the basic notions of the main document needed for the discussion, in order to make this document as much self-contained as possible, and sometimes deepened, in order to increase the understandability. 

\section{Abstract DeC formulation}\label{chapDECAbstract}
Assume that we have two operators, depending on the same parameter $\Delta$, between two normed vector spaces $\left(X, \norm{\cdot}_X \right)$ and $\left(Y, \norm{\cdot}_Y \right)$
\begin{equation}
\lopd^1,\lopd^2:X \longrightarrow Y,
\end{equation}
associated to two discretizations of the same problem.
Then, the following theorem holds.
\begin{theorem}[Deferred Correction accuracy]\label{th:DeC}
Let the following hypotheses hold
\begin{enumerate}
\item \textbf{Existence of a unique solution to $\lopd^2$} \\
$\exists ! \,\usol \in X$ solution of $\lopd^2$ such that $\lopd^2(\usol)=\uvec{0}_Y$;
\item \textbf{Coercivity-like property of $\lopd^1$} \\
$\exists \,\alpha_1 \geq 0$ independent of $\Delta$ s.t. 
\begin{equation}
	\norm{\lopd^1(\underline{\uvec{v}})-\lopd^1(\underline{\uvec{w}})}_Y\geq \alpha_1\norm{\underline{\uvec{v}}-\underline{\uvec{w}}}_X, \quad \forall \underline{\uvec{v}},\underline{\uvec{w}}\in X;
\end{equation}
\item \textbf{Lipschitz-continuity-like property of $\lopd^1-\lopd^2$} \\
$\exists\, \alpha_2 \geq 0$ independent of $\Delta$ s.t. 
\begin{equation}
	\norm{\left(\lopd^1(\underline{\uvec{v}})-\lopd^2(\underline{\uvec{v}})\right)-\left(\lopd^1(\underline{\uvec{w}})-\lopd^2(\underline{\uvec{w}})\right)}_Y\leq \alpha_2 \Delta \norm{\underline{\uvec{v}}-\underline{\uvec{w}}}_X, \quad\forall \underline{\uvec{v}},\underline{\uvec{w}}\in X.
\end{equation}
\end{enumerate}
Then, if we iteratively define $\underline{\uvec{u}}^{(p)}$ as the solution of
\begin{equation}
\label{DeC}
\lopd^1(\underline{\uvec{u}}^{(p)})=\lopd^1(\underline{\uvec{u}}^{(p-1)})-\lopd^2(\underline{\uvec{u}}^{(p-1)}), \quad p=1,\dots,P,
\end{equation}
we have that
\begin{equation}
\label{accuracyestimate}
\norm{\underline{\uvec{u}}^{(P)}-\usol}_X \leq \left( \Delta \frac{\alpha_2}{\alpha_1} \right)^P\norm{\underline{\uvec{u}}^{(0)}-\usol}_X.
\end{equation}
\end{theorem}
\begin{proof}
By using the coercivity-like property of $\lopd^1$ and the definition of $\lopd^1(\underline{\uvec{u}}^{(p)})$ in \eqref{DeC}, we have
\begin{equation}
\label{eq:2}
\norm{\underline{\uvec{u}}^{(P)}-\usol}_X\leq \frac{1}{\alpha_1}  \norm{\lopd^1(\underline{\uvec{u}}^{(P)})-\lopd^1(\usol)}_Y= \frac{1}{\alpha_1}  \norm{\lopd^1(\underline{\uvec{u}}^{(P-1)})-\lopd^2(\underline{\uvec{u}}^{(P-1)})-\lopd^1(\usol)}_Y.
\end{equation}
Since $\usol$ is the solution of $\lopd^2$, we have that $\lopd^2(\usol)=\uvec{0}_Y$ and we can add it inside the norm on the right hand side of the equality in \eqref{eq:2} and we get
\begin{equation}
\label{eq:3}
\norm{\underline{\uvec{u}}^{(P)}-\usol}_X\leq \frac{1}{\alpha_1}  \norm{\left[\lopd^1(\underline{\uvec{u}}^{(P-1)})-\lopd^2(\underline{\uvec{u}}^{(P-1)})\right]-\left[\lopd^1(\usol)-\lopd^2(\usol)\right]}_Y.
\end{equation}
Now, by applying the Lipschitz-continuity-like property we get
\begin{equation}
\label{eq:4}
\norm{\underline{\uvec{u}}^{(P)}-\usol}_X\leq \Delta \frac{\alpha_2}{\alpha_1}\norm{\underline{\uvec{u}}^{(P-1)}-\usol}_X.
\end{equation}
By repeating these calculations recursively we get the thesis.
\end{proof}

\section[DeC for ODEs]{The Deferred Correction for systems of ODEs}\label{chapDECODE}

We will focus on the numerical solution of the general Cauchy problem
\begin{equation}
\label{ODE}
\begin{cases}
\frac{d}{dt}\uvec{u}(t) = \uvec{G}(t,\uvec{u}(t)),\quad t\in[0,T], \\
\uvec{u}(0)=\uvec{z},
\end{cases}
\end{equation}
with $\uvec{u}(t) \in \mathbb{R}^Q$, $\uvec{z} \in \R^Q$ and $\uvec{G}: \R^+_0 \times \R^Q \to \R^Q$ a continuous map Lipschitz continuous with respect to $\uvec{u}$ uniformly with respect to $t$ with a Lipschitz constant $L$. This ensures the existence of a unique solution for the system of ODEs \eqref{ODE}.

We will assume here a classical one-step method setting: we discretize the time domain $[0,T]$ by introducing $N+1$ time nodes $t_n$, which are such that $0=t_0<t_1<\dots<t_N=T$ and therefore inducing $N$ intervals $[t_n,t_{n+1}]$, we denote by $\uvec{u}_{n}$ an approximation of the exact solution $\uvec{u}(t_n)$ at the time $t_n$ and we look for a recipe to compute $\uvec{u}_{n+1}$ by knowing $\uvec{u}_{n}$ for each $n=0,1,\dots,N-1$. We will focus on the generic time interval $[t_n,t_{n+1}]$ with $\Delta t=t_{n+1}-t_{n}$ and, as in the context of a general consistency analysis, we will assume $\uvec{u}_{n}=\uvec{u}(t_{n})$.

\subsection{bDeC}
In the general time step $[t_n,t_n+\Delta t]$ we introduce $M+1$ subtimenodes $t^0, \dots , t^M$ such that $t_n=t^0<t^1<\dots<t^M=t_n+\Delta t$, which are assumed here to be equispaced. We will refer to $\uvec{u}(t^m)$ as the exact solution in the node $t^m$ and to $\uvec{u}^m$ as the approximation of the solution in the same node.
Just for the first node, we set $\uvec{u}^0:=\uvec{u}_n$ and, in the accuracy study, we will consider it to be exact, i.e., $\uvec{u}^0=\uvec{u}(t^0)=\uvec{u}(t_n)=\uvec{u}_n$.

\subsubsection{Definition of $\lopdt^2$}

An exact integration of the system of ODEs over $[t^0,t^m]$ would result in
\begin{equation}
\label{exint}
\uvec{u}(t^m)-\uvec{u}^0-\int_{t^0}^{t^m}\uvec{G}(t,\uvec{u}(t))dt=\uvec{0}, \quad \forall m=1,\dots,M,
\end{equation}
from which we would have the exact solution $\uvec{u}(t^m)$.

Unfortunately, we cannot perform in general the exact integration 
and we need to make some approximations. We replace $\uvec{G}(t,\uvec{u}(t))$ by the Lagrange interpolating polynomial of degree $M$ associated to the $M+1$ nodes $t^m$ with $m=0,1,\dots,M,$ getting
\begin{equation}
\label{almosttruebefore}
\uvec{u}^m-\uvec{u}^0-\int_{t^0}^{t^m}\sum_{\ell=0}^{M} \uvec{G}(t^\ell,\uvec{u}(t^\ell))\psi^\ell(t)dt=\uvec{0}, \quad \forall m=1,\dots,M.
\end{equation}
Moving the finite sum and the vectors $\uvec{G}(t^\ell,\uvec{u}(t^\ell))$ outside of the integral, \eqref{almosttruebefore} can be recast as 
\begin{equation}
	\uvec{u}^m-\uvec{u}^0-\Delta t \sum_{\ell=0}^{M} \theta^m_\ell \uvec{G}(t^\ell,\uvec{u}(t^\ell))=\uvec{0} ,\quad \forall m=1,\dots,M,
	\label{almosttrue}
\end{equation}
where the coefficients $\theta^m_\ell$ are the normalized integrals of the Lagrange basis functions
and do not depend on $\dt$.

\begin{proposition}
$\uvec{u}^m$ satisfying \eqref{almosttrue} is an $(M+1)$-th order accurate approximation of $\uvec{u}(t^m)$.
\end{proposition}
\begin{proof}
For the proof, we will focus on the original equivalent formulation \eqref{almosttruebefore}.
Let us compute $\uvec{u}(t^m)-\uvec{u}^m$ with $\uvec{u}^m$ got by \eqref{almosttruebefore}. From \eqref{exint}, \eqref{almosttruebefore} and the $M$-th order accuracy on the approximation of $\uvec{G}(t,\uvec{u}(t))$ due to the interpolation with Lagrange polynomials of degree $M$ we have
\begin{align}
\begin{split}
	\uvec{u}(t^m)-\uvec{u}^m&=\uvec{u}^0+\int_{t^0}^{t^m}\uvec{G}(t,\uvec{u}(t))dt-\uvec{u}^0-\int_{t^0}^{t^m}\sum_{\ell=0}^{M} \uvec{G}(t^\ell,\uvec{u}(t^\ell))\psi^\ell(t)dt\\
	&=\int_{t^0}^{t^m} \left[ \uvec{G}(t,\uvec{u}(t))-\sum_{\ell=0}^{M} \uvec{G}(t^\ell,\uvec{u}(t^\ell))\psi^\ell(t) \right]dt\\
	&=\int_{t^0}^{t^m}O(\Delta t^{M+1})dt=O(\Delta t^{M+2}).
\end{split}
\end{align}

\end{proof}

Despite this result, the previous formula cannot be used in practice because the exact solution $\uvec{u}(t^\ell)$ in the nodes $t^\ell$ with $\ell=1,\dots,M$ is not available. 

We use the approximated values $\uvec{u}^\ell$ in place of them, thus getting the following implicit formulation
\begin{equation}
\uvec{u}^m-\uvec{u}^0-\Delta t \sum_{\ell=0}^{M} \theta^m_\ell \uvec{G}(t^\ell,\uvec{u}^\ell)=\uvec{0} \quad \forall m=1,\dots,M,
\label{true1}
\end{equation}
which leads to the definition of our $\lopdt^2$ operator
\begin{equation}
\label{l2ODE}
\lopdt^2(\underline{\uvec{u}}) = \begin{pmatrix}
        \uvec{u}^1-\uvec{u}^0-\Delta t \sum_{\ell=0}^{M} \theta^1_\ell \uvec{G}(t^\ell,\uvec{u}^\ell)\\
        \vdots\\
        \uvec{u}^m-\uvec{u}^0-\Delta t \sum_{\ell=0}^{M} \theta^m_\ell \uvec{G}(t^\ell,\uvec{u}^\ell)\\
        \vdots\\
        \uvec{u}^M-\uvec{u}^0-\Delta t \sum_{\ell=0}^{M} \theta^M_\ell \uvec{G}(t^\ell,\uvec{u}^\ell)\\
\end{pmatrix}
\text{ with }
\underline{\uvec{u}}=\left(
   \begin{array}{ccc}
   \uvec{u}^1\\
   \vdots\\
   \uvec{u}^m\\
   \vdots\\
   \uvec{u}^M
   \end{array}
\right).
\end{equation}

\begin{proposition}
Let $\uvec{u}^m$ be the $m$-th component of the solution of $\lopdt^2(\underline{\uvec{u}}) =0$. Then, $\uvec{u}^m$ is an $(M+1)$-th order accurate approximation of $\uvec{u}(t^m)$.
\end{proposition}
\begin{proof}
Let us consider the following operator $\mathcal{J}: \mathbb{R}^{(M \times Q)} \rightarrow \mathbb{R}^{(M \times Q)}$ defined as
\begin{equation}
\label{hoperator}
\underline{\uvec{y}}=\mathcal{J}(\underline{\uvec{u}}) = \begin{pmatrix}
        \uvec{u}^0+\Delta t \sum_{\ell=0}^{M} \theta^1_\ell \uvec{G}(t^\ell,\uvec{u}^\ell)\\
        \vdots\\
        \uvec{u}^0+\Delta t \sum_{\ell=0}^{M} \theta^m_\ell \uvec{G}(t^\ell,\uvec{u}^\ell)\\
        \vdots\\
        \uvec{u}^0+\Delta t \sum_{\ell=0}^{M} \theta^M_\ell \uvec{G}(t^\ell,\uvec{u}^\ell)\\
\end{pmatrix}
\text{ with }
\underline{\uvec{y}}=\left(
   \begin{array}{ccc}
   \uvec{y}^1\\
   \vdots\\
   \uvec{y}^m\\
   \vdots\\
   \uvec{y}^M
   \end{array}
\right).\end{equation}
Again, we remark that $\uvec{u}^0$, the vector corresponding to the initial subtimenode, is always fixed.
The proof consists of two parts.
We will first show that, for $\Delta t$ small enough, $\mathcal{J}$ is a contraction over $\mathbb{R}^{(M \times Q)}$, which is a finite dimensional space (and so complete with respect to the distance induced by any norm). This will ensure, thanks to the Banach fixed-point theorem, that there exists a fixed point $\tilde{\underline{\uvec{u}}}$ such that $\tilde{\underline{\uvec{u}}}=\mathcal{J}(\tilde{\underline{\uvec{u}}})$ and that it is unique. It is very easy to see that this fixed point is the (unique) solution to the operator $\lopdt^2$.
Then, by iteratively applying the operator, we will generate a sequence of vectors converging to this fixed point and  we will show that this limit is an $(M+1)$-th order accurate approximation of the exact solution to the system of ODEs.

Let us first prove that $\mathcal{J}$ is a contraction for $\Delta t$ small enough.
We recall that $\theta^m_\ell$ are constant coefficients independent on $\dt$ and bounded by $C_\theta= \max |\theta_{\ell}^m|$ and that $\uvec{G}(t,\uvec{u})$ is Lipschitz-continuous with respect to $\uvec{u}$ uniformly with respect to $t$ with constant $L$. Now, using the triangular inequality, we have
\begin{align}
\begin{split}
\norm{ \mathcal{J}(\underline{\uvec{v}})-\mathcal{J}(\underline{\uvec{w}}) }_\infty=&\Delta t \norm{\sum_{\ell=0}^{M}\begin{pmatrix} 
		\theta^1_\ell \left[ \uvec{G}(t^\ell,\uvec{v}^\ell)-\uvec{G}(t^\ell,\uvec{w}^\ell) \right]\\
		\vdots\\
		\theta^m_\ell \left[ \uvec{G}(t^\ell,\uvec{v}^\ell)-\uvec{G}(t^\ell,\uvec{w}^\ell) \right]\\
		\vdots\\
		\theta^M_\ell \left[ \uvec{G}(t^\ell,\uvec{v}^\ell)-\uvec{G}(t^\ell,\uvec{w}^\ell) \right]
\end{pmatrix}}_\infty  \\
 \leq & \Delta t C_\theta \sum_{\ell=0}^{M} \norm{\uvec{G}(t^\ell,\uvec{v}^\ell)-\uvec{G}(t^\ell,\uvec{w}^\ell)}_{\infty, Q}   \\ 
 \leq & \Delta t C_\theta \sum_{\ell=0}^{M} L \norm{\uvec{v}^\ell-\uvec{w}^\ell}_{\infty, Q} \leq \Delta t C_\theta L M\norm{\underline{\uvec{v}}-\underline{\uvec{w}}}_\infty.
\end{split}
\end{align}

The last inequality follows from the fact that $\underline{\uvec{v}}-\underline{\uvec{w}}$ contains as components all the vectors $\uvec{v}^\ell-\uvec{w}^\ell$ for all $\ell=1,\dots,M$ and from the fact that $\uvec{v}^0=\uvec{w}^0=\uvec{u}^0$ and so
\begin{align}
\norm{\uvec{v}^\ell-\uvec{w}^\ell}_{\infty, Q} \leq\norm{\underline{\uvec{v}}-\underline{\uvec{w}}}_\infty, \quad \forall \ell=1,\dots,M,
\end{align}
where $\norm{\cdot}_{\infty, Q}$ is the infinity norm over $\mathbb{R}^{Q}$, while $\norm{\cdot}_{\infty}$ is the infinity norm over $\R^{M\times Q}$. For $\Delta t < \frac{1}{C_\theta L M}$, we have 
\begin{align}
\norm{ \mathcal{J}(\underline{\uvec{v}})-\mathcal{J}(\underline{\uvec{w}}) }_\infty<\delta\norm{\underline{\uvec{v}}-\underline{\uvec{w}}}_\infty
\end{align}
with $\delta<1$ and so $\mathcal{J}$ is a contraction.
As anticipated, there exists a unique fixed point $\tilde{\underline{\uvec{u}}}$, solution of $\lopdt^2$.

For the second part, we will prove the accuracy of the iteration of the fixed point procedure. 
We consider the sequence $\lbrace\underline{\uvec{y}}^{(k)}\rbrace_{k\in \mathbb{N}}$ given by the following recursive definition
\begin{align}
\underline{\uvec{y}}^{(k)}=\mathcal{J}(\underline{\uvec{y}}^{(k-1)})
\end{align}
with its general element being
\begin{align}
\underline{\uvec{y}}^{(k)}=\left(
   \begin{array}{ccc}
   \uvec{y}^{1,(k)}\\
   \vdots\\
   \uvec{y}^{m,(k)}\\
   \vdots\\
   \uvec{y}^{M,(k)}
   \end{array}
\right) \text{, with } \underline{\uvec{y}}^{(0)}=\left(
   \begin{array}{ccc}
   \uvec{y}^{1,(0)}\\
   \vdots\\
   \uvec{y}^{m,(0)}\\
   \vdots\\
   \uvec{y}^{M,(0)}
   \end{array}
\right)=\left(
   \begin{array}{ccc}
   \uvec{u}^{0}\\
   \vdots\\
   \uvec{u}^{0}\\
   \vdots\\
   \uvec{u}^{0}
   \end{array}
\right).
\end{align}

The general component $\uvec{y}^{m,(k)}$ of $\underline{\uvec{y}}^{(k)}$ is a 
$Q$-dimensional vector. The first index $m$ is referred to the subtimenode, the second is the index of the sequence. In order to have a more compact notation, we will not write $\uvec{G}(t^0,\uvec{u}^0)$ as a separate term but we set $\uvec{y}^{0,(k)}=\uvec{u}^0$ $\forall k\geq 0$, because the value of the solution at the first subtimenode is known.
From theory, we know that this sequence converges to the fixed point of $\mathcal{J}$ and so to the solution of the operator $\lopdt^2$.

Let us prove by induction on $k$ that  for all $m=1,\dots, M$, we have
\begin{equation}
	\uvec{y}^{m,(k)} = \uvec{u}(t^m) + O(\Delta t ^{\min(k+1,M+2)}).
\end{equation}
The base case, for $k=0$, is clearly true as a simple Taylor expansion gives
\begin{equation}
	\uvec{u}(t^m)  = \uvec{u}(t^0) + \Delta t \uvec{G}(t^0,\uvec{u}(t^0)) (t^m-t^0)+O(\Delta t^2) = \uvec{y}^{m,(0)} + O(\Delta t),
\end{equation}
reminding that $\frac{d}{dt}\uvec{u}(t) = \uvec{G}(t,\uvec{u}(t))$.

For the induction step, we assume that $\uvec{y}^{m,(k)}= \uvec{u}(t^m)+ O(\dt^{\min(k+1,M+2)})$ and we will prove that $\uvec{y}^{m,(k+1)}= \uvec{u}(t^m)+ O(\dt^{\min(k+2,M+2)})$.
By exploiting the Lipschitz-continuity of $\uvec{G}$, we have that
\begin{align}
\begin{split}
	\uvec{G}(t^\ell,\uvec{u}(t^\ell)) &=\uvec{G}(t^\ell,\uvec{y}^{\ell,(k)}) +\nabla_{\uvec{u}}\uvec{G}(t^\ell,\uvec{y}^{\ell,(k)})(\uvec{u}(t^{\ell})-\uvec{y}^{\ell,(k)})+O\left( \norm{\uvec{u}(t^{\ell})-\uvec{y}^{\ell,(k)}}^2_{\infty,Q} \right)\\ 
 & =\uvec{G}(t^\ell,\uvec{y}^{\ell,(k)}) + O(\dt^{\min(k+1,M+2)}),
 \end{split}
\end{align}
where $\nabla_{\uvec{u}}\uvec{G}(t^\ell,\uvec{y}^{\ell,(k)})$ is bounded in some norm by $L$.
We are then able to prove that
\begin{align}
\begin{split}
	\uvec{y}^{m,(k+1)}&=\uvec{u}(t^0) + \dt \sum_{\ell=0}^{M} \theta^m_\ell \uvec{G}(t^\ell,\uvec{y}^{\ell, (k)})\\
	& = \uvec{u}(t^0) + \dt \sum_{\ell=0}^{M} \theta^m_\ell \uvec{G}(t^\ell,\uvec{u}(t^\ell)) + O(\dt^{1+\min(k+1,M+2)}).
\end{split}
\end{align}
Now, thanks to the $(M+1)$-th order accuracy of \eqref{almosttrue}, we have that 
\begin{align}
\begin{split}
	\uvec{y}^{m,(k+1)}&= \uvec{u}(t^0) + \dt \sum_{\ell=0}^{M} \theta^m_\ell \uvec{G}(t^\ell,\uvec{u}(t^\ell)) + O(\dt^{1+\min(k+1,M+2)})\\
	& = \uvec{u}(t^m) + O(\dt^{M+2}) + O(\dt^{1+\min(k+1,M+2)}) =  \uvec{u}(t^m)  +O(\dt^{\min(k+2,M+2)}).
\end{split}
\end{align}
Hence, for $k>M$ the components $\uvec{y}^{(k),m}$ are an $(M+1)$ accurate solution of $u(t^m)$ and their limit for $k\to \infty$, i.e., the solutions of $\lopdt^2$,  is as well an $(M+1)$ approximation of the exact solution.
\end{proof}

\subsubsection{Definition of $\lopdt^1$}
If we apply the Euler method to get the approximate solution $\uvec{u}^m$ in the node $t^m$ we have
\begin{equation}
\label{Euler}
\uvec{u}^m-\uvec{u}^0-\Delta t \beta^m \uvec{G}(t^0,\uvec{u}^0)=\uvec{0},
\end{equation}
where $\beta^m=\frac{t^m-t^0}{\Delta t}$.
\begin{proposition}
Let $\uvec{u}^m$ be the solution of \eqref{Euler}, then $\uvec{u}^m$ is first order accurate, i.e., $\uvec{u}(t^m)-\uvec{u}^m=O(\dt^2)$. 
\end{proposition} 
\begin{proof} 
We consider the difference between the exact solution $\uvec{u}(t^m)$ to our ODEs system and $\uvec{u}^m$ got from \eqref{Euler}. Through a first order Taylor expansion of $\uvec{u}(t)$ and from the fact that $\frac{d}{dt}\uvec{u}(t) = \uvec{G}(t,\uvec{u}(t))$, we have
\begin{equation}
\uvec{u}(t^m)-\uvec{u}^m=\uvec{u}^0+\uvec{G}(t^0,\uvec{u}^0)(t^m-t^0)+O(\Delta t^2)-\uvec{u}^0-\Delta t \beta^m \uvec{G}(t^0,\uvec{u}^0)=O(\Delta t^2),
\end{equation}
because $\uvec{u}^0=\uvec{u}(t^0)=\uvec{u}(t_n)=\uvec{u}_n$ and $\beta^m=\frac{t^m-t^0}{\Delta t}$.
\end{proof}
Directly from \eqref{Euler}, we get our explicit, low order operator
$\lopdt^1: \mathbb{R}^{(M \times Q)} \rightarrow \mathbb{R}^{(M \times Q)}$ defined as
\begin{equation}
\label{l1ODE}
\lopdt^1(\underline{\uvec{u}}) = \begin{pmatrix}

        \uvec{u}^1-\uvec{u}^0-\Delta t \beta^1 \uvec{G}(t^0,\uvec{u}^0)\\
        \vdots\\
        \uvec{u}^m-\uvec{u}^0-\Delta t \beta^m \uvec{G}(t^0,\uvec{u}^0)\\
        \vdots\\
        \uvec{u}^M-\uvec{u}^0-\Delta t \beta^M \uvec{G}(t^0,\uvec{u}^0)\\
        \end{pmatrix}
\text{ with }
\underline{\uvec{u}}=\left(
   \begin{array}{ccc}
   \uvec{u}^1\\
   \vdots\\
   \uvec{u}^m\\
   \vdots\\
   \uvec{u}^M
   \end{array}
\right).
\end{equation}

\subsubsection{Proof of the properties of $\lopdt^1$ and $\lopdt^2$}
We equip $X=Y=\mathbb{R}^{(M \times Q)}$ with the infinity norm $\norm{\cdot}_\infty$ and we recall here the hypotheses that are needed to apply the Deferred Correction method from the abstract formulation but characterizing them to our case.
\begin{itemize}
\item[i)] \textbf{Existence of a solution to $\lopdt^2$} \\
$\exists ! \usoldt \in \mathbb{R}^{(M \times Q)}$ solution of $\lopdt^2$, i.e. such that $\lopdt^2(\usoldt)=\uvec{0}$;
\item[ii)] \textbf{Coercivity-like property of $\lopdt^1$} \\
$\exists \alpha_1 \geq 0$ independent of $\Delta t$ s.t. 
\begin{equation}
\norm{\lopdt^1(\underline{\uvec{v}})-\lopdt^1(\underline{\uvec{w}})}_\infty \geq \alpha_1\norm{\underline{\uvec{v}}-\underline{\uvec{w}}}_\infty, \quad \forall \underline{\uvec{v}},\underline{\uvec{w}} \in \mathbb{R}^{(M \times Q)};
\end{equation}
\item[iii)] \textbf{Lipschitz-continuity-like condition of $\lopdt^1-\lopdt^2$} \\
$\exists \alpha_2 \geq 0$ independent of $\Delta t$ s.t. 
\begin{equation}
\norm{\left[\lopdt^1(\underline{\uvec{v}})-\lopdt^2(\underline{\uvec{v}})\right]-\left[\lopdt^1(\underline{\uvec{w}})-\lopdt^2(\underline{\uvec{w}})\right]}_\infty \leq \alpha_2 \Delta t \norm{\underline{\uvec{v}}-\underline{\uvec{w}}}_\infty, \quad \forall \underline{\uvec{v}},\underline{\uvec{w}}\in \mathbb{R}^{(M \times Q)}.
\label{eq:lip_like_bdec}
\end{equation}
\end{itemize}

\begin{proof} We prove in order the three properties.
\begin{itemize}
\item[i)] \textbf{Existence of a solution to $\lopdt^2$} \\
The first property, i.e., the existence of a unique solution to $\lopdt^2$, has already been shown in the proof of its $(M+1)$-th order accuracy by introducing the operator $\mathcal{J}: \mathbb{R}^{(M \times Q)} \rightarrow \mathbb{R}^{(M \times Q)}$ defined by \eqref{hoperator}. 
We showed that for $\Delta t$ small enough it is a contraction over the space $\mathbb{R}^{(M \times Q)}$ equipped with the infinity norm, so, there exists a unique fixed point of $\mathcal{J}$, which is the unique solution to $\lopdt^2$.

\item[ii)] \textbf{Coercivity-like property of $\lopdt^1$} \\
Let us now consider two generic vectors $\underline{\uvec{v}},\underline{\uvec{w}}\in \mathbb{R}^{(M \times Q)}$
\begin{align}
	\underline{\uvec{v}}=\left(
	\begin{array}{ccc}
		\uvec{v}^1\\
		\vdots\\
		\uvec{v}^m\\
		\vdots\\
		\uvec{v}^M
	\end{array}
	\right), \quad
	\underline{\uvec{w}}=\left(
	\begin{array}{ccc}
		\uvec{w}^1\\
		\vdots\\
		\uvec{w}^m\\
		\vdots\\
		\uvec{w}^M
	\end{array}
	\right),
\end{align}
with $\uvec{v}^m$ and $\uvec{w}^m$ for $m=1,\dots,M$ generic $Q$-dimensional vectors. 
From a direct computation, we have
\begin{align}
	\begin{split}
&\lopdt^1(\underline{\uvec{v}})-\lopdt^1(\underline{\uvec{w}})\\
&=             \begin{pmatrix}
        \uvec{v}^1-\uvec{u}^0-\Delta t \beta^1 \uvec{G}(t^0,\uvec{u}^0)\\
        \vdots\\
        \uvec{v}^m-\uvec{u}^0-\Delta t \beta^m \uvec{G}(t^0,\uvec{u}^0)\\
        \vdots\\
        \uvec{v}^M-\uvec{u}^0-\Delta t \beta^M \uvec{G}(t^0,\uvec{u}^0)\\
        \end{pmatrix}-\begin{pmatrix}
        \uvec{w}^1-\uvec{u}^0-\Delta t \beta^1 \uvec{G}(t^0,\uvec{u}^0)\\
        \vdots\\
        \uvec{w}^m-\uvec{u}^0-\Delta t \beta^m \uvec{G}(t^0,\uvec{u}^0)\\
        \vdots\\
        \uvec{w}^M-\uvec{u}^0-\Delta t \beta^M \uvec{G}(t^0,\uvec{u}^0)\\
        \end{pmatrix}=\begin{pmatrix}
        \uvec{v}^1-\uvec{w}^1\\
        \vdots\\
        \uvec{v}^m-\uvec{w}^m\\
        \vdots\\
        \uvec{v}^M-\uvec{w}^M
        \end{pmatrix},
	\end{split}
\end{align}
i.e., $\lopdt^1(\underline{\uvec{v}})-\lopdt^1(\underline{\uvec{w}}) =\underline{\uvec{v}}-\underline{\uvec{w}}$. Then,
\begin{align}
\norm{\lopdt^1(\underline{\uvec{v}})-\lopdt^1(\underline{\uvec{w}})}_\infty=\norm{\underline{\uvec{v}}-\underline{\uvec{w}}}_\infty
\end{align}
and thus the coercivity-like property of $\lopdt^1$ is verified with $\alpha_1=1$ and results in an equality.
Again, we remark that $\uvec{u}^0$ is given, it is part of the problem and embedded in the operators $\lopdt^1$ and $\lopdt^2$.
\item[iii)] \textbf{Lipschitz-continuity-like condition of $\lopdt^1-\lopdt^2$} \\
Again, we consider a direct computation but focusing, for the sake of compactness, on the $Q$-dimensional component got for a general $m$
\begin{equation}
	\begin{split}
&\left[\lopdt^{1,m}(\underline{\uvec{v}})-\lopdt^{2,m}(\underline{\uvec{v}})\right]-\left[\lopdt^{1,m}(\underline{\uvec{w}})-\lopdt^{2,m}(\underline{\uvec{w}})\right]\\
=&\uvec{v}^m-\uvec{u}^0-\Delta t \beta^m \uvec{G}(t^0,\uvec{u}^0) - \uvec{v}^m+\uvec{u}^0+ \Delta t \sum_{\ell=0}^{M} \theta^m_\ell \uvec{G}(t^\ell,\uvec{v}^\ell)  \\
-& \left[\uvec{w}^m-\uvec{u}^0-\Delta t \beta^m \uvec{G}(t^0,\uvec{u}^0)  -\uvec{w}^m+\uvec{u}^0+\Delta t \sum_{\ell=0}^{M} \theta^m_\ell \uvec{G}(t^\ell,\uvec{w}^\ell)\right] \\
=&\dt \sum_{\ell=0}^M  \theta^m_\ell \left( \uvec{G}(t^\ell,\uvec{v}^\ell)- \uvec{G}(t^\ell,\uvec{w}^\ell) \right),
\label{intermezzo}
	\end{split}
\end{equation}
where clearly $\uvec{v}^0=\uvec{w}^0=\uvec{u}^0$. As we pointed out several times, $\uvec{u}^0$ is not an unknown, it is a given vector, it is ``part'' of the problem and is embedded in the operators. We use $\uvec{v}^0$ and $\uvec{w}^0$ instead of $\uvec{u}^0$ for the sake of compactness. 
Let us recall that $\theta^m_\ell$, for $m=1,\dots,M$ and $\ell=0,1,\dots,M$, are fixed constant coefficients independent of $\Delta t$, thus bounded in absolute value by a positive constant $C_\theta$, and that $\uvec{G}(t,\uvec{u})$ is Lipschitz-continuous with respect to $\uvec{u}$ uniformly with respect to $t$ with a Lipschitz constant $L$.
By applying the triangular inequality, we have
\begin{align}
\begin{split}
&\norm{\left[\lopdt^1(\underline{\uvec{v}})-\lopdt^2(\underline{\uvec{v}})\right]-\left[\lopdt^1(\underline{\uvec{w}})-\lopdt^2(\underline{\uvec{w}})\right]}_\infty\\
=&\Delta t \norm{\sum_{\ell=0}^{M}\begin{pmatrix} 
        \theta^1_\ell \left[ \uvec{G}(t^\ell,\uvec{v}^\ell)-\uvec{G}(t^\ell,\uvec{w}^\ell) \right]\\
        \vdots\\
        \theta^m_\ell \left[ \uvec{G}(t^\ell,\uvec{v}^\ell)-\uvec{G}(t^\ell,\uvec{w}^\ell) \right]\\
        \vdots\\
        \theta^M_\ell \left[ \uvec{G}(t^\ell,\uvec{v}^\ell)-\uvec{G}(t^\ell,\uvec{w}^\ell) \right]
        \end{pmatrix}}_\infty 
        \leq  \Delta t C_\theta \sum_{\ell=0}^{M} \norm{\begin{pmatrix} 
        \uvec{G}(t^\ell,\uvec{v}^\ell)-\uvec{G}(t^\ell,\uvec{w}^\ell)\\
        \vdots\\
        \uvec{G}(t^\ell,\uvec{v}^\ell)-\uvec{G}(t^\ell,\uvec{w}^\ell)\\
        \vdots\\
        \uvec{G}(t^\ell,\uvec{v}^\ell)-\uvec{G}(t^\ell,\uvec{w}^\ell)
        \end{pmatrix}}_\infty \\
    =& \Delta t C_\theta \sum_{\ell=0}^{M} \norm{\uvec{G}(t^\ell,\uvec{v}^\ell)-\uvec{G}(t^\ell,\uvec{w}^\ell)}_{\infty, Q} 
        \!\!\leq \Delta t C_\theta \sum_{\ell=0}^{M} L \norm{\uvec{v}^\ell-\uvec{w}^\ell}_{\infty, Q} \!\!\leq \Delta t C_\theta L M\norm{\underline{\uvec{v}}-\underline{\uvec{w}}}_\infty,
\end{split}
\end{align}
where the last inequality follows from the fact that $\underline{\uvec{v}}-\underline{\uvec{w}}$ contains as components all the vectors $\uvec{v}^\ell-\uvec{w}^\ell$ for $\ell=1,\dots,M$ and from the fact that $\uvec{v}^0=\uvec{w}^0=\uvec{u}^0$. 
This proves the Lipschitz-continuity-like condition of $\lopdt^1-\lopdt^2$ with $\alpha_2=C_\theta L M$.
For more clarity, we underline that the infinity norm $\norm{\cdot}_{\infty,Q}$ is applied to $Q$-dimensional vectors (and not to $(M\times Q)$-dimensional vectors like $\norm{\cdot}_\infty$).
This completes the analysis of the Deferred Correction applied to the context of the systems of ordinary differential equations.
\end{itemize}
\end{proof}

\begin{observation}[On the optimal value of $\alpha_2$]
The constant $C_\theta L M$ is not the sharpest estimate for $\alpha_2$ in \eqref{eq:lip_like_bdec}. Introducing the support structures
\begin{align}
\Theta=\begin{pmatrix}
0 & 0 & \dots & 0 \\
\theta^1_0 & \theta^1_1 & \dots & \theta^1_M \\
\theta^2_0 & \theta^2_1 & \dots & \theta^2_M \\
\vdots & \vdots & \ddots & \vdots \\
\theta^M_0 & \theta^M_1 & \dots & \theta^M_M \\
\end{pmatrix}, \quad \underline{\uvec{G}}(\underline{\uvec{v}})=\left(
   \begin{array}{ccc}
      \uvec{G}(t^{1},\uvec{v}^{1})\\
      \vdots \\
      \uvec{G}(t^{m},\uvec{v}^{m})\\
      \vdots \\
      \uvec{G}(t^{M},\uvec{v}^{M})
   \end{array}
\right), \quad \underline{\uvec{G}}(\underline{\uvec{w}})=\left(
   \begin{array}{ccc}
      \uvec{G}(t^{1},\uvec{w}^{1})\\
      \vdots \\
      \uvec{G}(t^{m},\uvec{w}^{m})\\
      \vdots \\
      \uvec{G}(t^{M},\uvec{w}^{M})
   \end{array}
\right),
\end{align}
and recalling that $\uvec{v}^0=\uvec{w}^0=\uvec{u}^0$, one can easily verify that
\begin{align}
\begin{split}
\sum_{\ell=0}^{M}\begin{pmatrix}
        \theta^1_\ell \left[ \uvec{G}(t^\ell,\uvec{v}^\ell)-\uvec{G}(t^\ell,\uvec{w}^\ell) \right]\\
        \vdots\\
        \theta^m_\ell \left[ \uvec{G}(t^\ell,\uvec{v}^\ell)-\uvec{G}(t^\ell,\uvec{w}^\ell) \right]\\
        \vdots\\
        \theta^M_\ell \left[ \uvec{G}(t^\ell,\uvec{v}^\ell)-\uvec{G}(t^\ell,\uvec{w}^\ell) \right]
        \end{pmatrix}=\Theta_{1:,1:} \left[ \underline{\uvec{G}}(\underline{\uvec{v}}) - \underline{\uvec{G}}(\underline{\uvec{w}})\right],
\end{split}
\end{align}
where by $\Theta_{1:,1:}$ we mean the submatrix extracted from $\Theta$ with row and column indices from $1$ on, assuming a zero-based numeration.

Therefore, we have that
\begin{align}
\begin{split}
\norm{\sum_{\ell=0}^{M}\begin{pmatrix} 
        \theta^1_\ell \left[ \uvec{G}(t^\ell,\uvec{v}^\ell)-\uvec{G}(t^\ell,\uvec{w}^\ell) \right]\\
        \vdots\\
        \theta^m_\ell \left[ \uvec{G}(t^\ell,\uvec{v}^\ell)-\uvec{G}(t^\ell,\uvec{w}^\ell) \right]\\
        \vdots\\
        \theta^M_\ell \left[ \uvec{G}(t^\ell,\uvec{v}^\ell)-\uvec{G}(t^\ell,\uvec{w}^\ell) \right]
        \end{pmatrix}}_\infty&=\norm{\Theta_{1:,1:} \left[ \underline{\uvec{G}}(\underline{\uvec{v}}) - \underline{\uvec{G}}(\underline{\uvec{w}})\right]}_{\infty},
\end{split}
\end{align}
By basic linear algebra and thanks to the Lipschitz-continuity of $\uvec{G}(t,\uvec{u})$ with respect to $\uvec{u}$ uniformly with respect to $t$, we get
\begin{align}
\begin{split}
\norm{\Theta_{1:,1:} \left[ \underline{\uvec{G}}(\underline{\uvec{v}}) - \underline{\uvec{G}}(\underline{\uvec{w}})\right]}_{\infty}
        &\leq \norm{\Theta_{1:,1:}}_{\infty}\norm{\underline{\uvec{G}}(\underline{\uvec{v}}) - \underline{\uvec{G}}(\underline{\uvec{w}})}_{\infty}\\
        &\leq \norm{\Theta_{1:,1:}}_{\infty}L\norm{\underline{\uvec{v}}- \underline{\uvec{w}}}_{\infty}
\end{split}
\end{align}
where $\norm{\cdot}_{\infty}$ applied to the matrix $\Theta_{1:,1:}$ is the matrix norm induced by the corresponding vector norm and hence $\norm{\Theta_{1:,1:}}_{\infty}=\max\limits_{m=1,\dots,M}\sum\limits_{\ell=1}^M\abs{\theta_\ell^m}$.
Thus, one gets
\begin{equation}
\norm{\left[\lopdt^1(\underline{\uvec{v}})-\lopdt^2(\underline{\uvec{v}})\right]-\left[\lopdt^1(\underline{\uvec{w}})-\lopdt^2(\underline{\uvec{w}})\right]}_\infty\leq \alpha_2\Delta t\norm{\underline{\uvec{v}}- \underline{\uvec{w}}}_{\infty}
\end{equation}
with $\alpha_2=\norm{\Theta_{1:,1:}}_{\infty}L$, which constitutes a sharper estimate with respect to $C_\theta L M$.
Indeed, the reported matrix $\Theta$ is referred to a scalar problem and it must be block-expanded in the context of a vectorial problem, however, this does not influence the estimate. 
\end{observation}

\subsection{sDeC} 
The construction of this DeC method makes use of the definition of the subtimenodes introduced for the bDeC method. The main difference is that here we focus on the integration of the system of ODEs in the intervals $[t^{m-1},t^{m}]$ rather than $[t^{0},t^{m}].$

\subsubsection{Definition of $\lopdt^2$}
We start from the exact integration of the system of ODEs in the interval $[t^{m-1},t^m]$, which would result in
\begin{equation}
\label{exint_2}
\uvec{u}(t^m)-\uvec{u}(t^{m-1})-\int_{t^{m-1}}^{t^m}\uvec{G}(t,\uvec{u}(t))dt=\uvec{0}, \quad \forall m=1,\dots,M.
\end{equation}
Again, in order to get an expression that can actually be used, we replace $\uvec{G}(t,\uvec{u}(t))$ with its $M$-th order accurate Lagrange interpolant of degree $M$ associated to the $M+1$ subtimenodes $t^m$ and replace $\uvec{u}(t^\ell)$ by $\uvec{u}^\ell$ thus getting 
\begin{equation}
\label{true_2}
\uvec{u}^m-\uvec{u}^{m-1}-\int_{t^{m-1}}^{t^m}\sum_{\ell=0}^{M} \uvec{G}(t^\ell,\uvec{u}^\ell)\psi^\ell(t)dt=\uvec{0}, \quad \forall m=1,\dots,M.
\end{equation}
Moving the finite sum and the vectors $\uvec{G}(t^\ell,\uvec{u}^\ell)$ outside of the integral and performing the exact integration of the Lagrangian polynomial functions $\psi^\ell(t)$ in the subinterval $[t^{m-1},t^m]$ we get
\begin{equation}
\uvec{u}^m-\uvec{u}^{m-1}-\Delta t \sum_{\ell=0}^{M} \delta^m_\ell \uvec{G}(t^\ell,\uvec{u}^\ell)=\uvec{0}, \quad \forall m=1,\dots,M,
\label{true1_2}
\end{equation}
where, just like in the previous case, coefficients $\delta^m_\ell$ are normalized integrals of the Lagrange basis functions independent of $\Delta t$.

Our implicit $(M+1)$-th order accurate operator $\lopdt^2: \mathbb{R}^{(M \times Q)} \rightarrow \mathbb{R}^{(M \times Q)}$ is therefore defined as
\begin{equation}
\label{l2ODE_2}
\lopdt^2(\underline{\uvec{u}}) = \begin{pmatrix}
        \uvec{u}^1-\uvec{u}^{0}-\Delta t \sum_{\ell=0}^{M} \delta^1_\ell \uvec{G}(t^\ell,\uvec{u}^\ell)\\
        \vdots\\
        \uvec{u}^m-\uvec{u}^{m-1}-\Delta t \sum_{\ell=0}^{M} \delta^m_\ell \uvec{G}(t^\ell,\uvec{u}^\ell)\\
        \vdots\\
        \uvec{u}^M-\uvec{u}^{M-1}-\Delta t \sum_{\ell=0}^{M} \delta^M_\ell \uvec{G}(t^\ell,\uvec{u}^\ell)\\
\end{pmatrix}
\text{ with }
\underline{\uvec{u}}=\left(
   \begin{array}{ccc}
   \uvec{u}^1\\
   \vdots\\
   \uvec{u}^m\\
   \vdots\\
   \uvec{u}^M
   \end{array}
\right).
\end{equation}

\subsubsection{Definition of $\lopdt^1$}
Also in this case the operator $\lopdt^1$ is obtained by a first order approximation in the integration of our initial system of ODEs. Applying the Euler method in the subinterval $[t^{m-1},t^m]$, we get
\begin{equation}
\label{Euler_2}
\uvec{u}^m-\uvec{u}^{m-1}-\Delta t \gamma^m \uvec{G}(t^{m-1},\uvec{u}^{m-1})=\uvec{0}
\end{equation}
where $\gamma^m=\frac{t^m-t^{m-1}}{\Delta t}$ are normalized coefficients.
The explicit, first order order operator
$\lopdt^1: \mathbb{R}^{(M \times Q)} \rightarrow \mathbb{R}^{(M \times Q)}$ is defined as
\begin{equation}
\label{l1ODE_2}
\lopdt^1(\underline{\uvec{u}}) = \begin{pmatrix}

        \uvec{u}^1-\uvec{u}^{0}-\Delta t \gamma^1 \uvec{G}(t^{0},\uvec{u}^{0})\\
        \vdots\\
        \uvec{u}^m-\uvec{u}^{m-1}-\Delta t \gamma^m \uvec{G}(t^{m-1},\uvec{u}^{m-1})\\
        \vdots\\
        \uvec{u}^M-\uvec{u}^{M-1} \Delta t \gamma^M \uvec{G}(t^{M-1},\uvec{u}^{M-1})\\
        \end{pmatrix}
        \text{ with }
\underline{\uvec{u}}=\left(
   \begin{array}{ccc}
   \uvec{u}^1\\
   \vdots\\
   \uvec{u}^m\\
   \vdots\\
   \uvec{u}^M
   \end{array}
\right).
\end{equation}

\subsubsection{sDeC as a perturbation of bDeC}
The proofs seen for the previous formulation cannot be extended to this case in a straightforward way, but it is possible to show that the second formulation is actually a perturbation of the first one with no impact on the accuracy.
Let us recall here, for more clarity, the updating formulas of the bDeC and of the sDeC methods for the computation of $\uvec{u}^{m,(p)}$, $m$-th component of the approximated solution at the iteration $p$, 
\begin{itemize}
\item \textbf{bDeC}
\begin{align}\label{eq:bDeC}
\uvec{u}^{m,(p)}_b=\uvec{u}^0+\Delta t \sum_{\ell=0}^{M} \theta^m_\ell \uvec{G}(t^\ell,\uvec{u}_b^{\ell,(p-1)})
\end{align}
\item \textbf{sDeC}
\begin{align}\label{eq:sDeC}
\uvec{u}^{m,(p)}_s=\uvec{u}^{0}+\Delta t \sum_{\ell=0}^{m-1} \gamma^{\ell+1} \left( \uvec{G}(t^{\ell},\uvec{u}_s^{\ell,(p)})-\uvec{G}(t^{\ell},\uvec{u}_s^{\ell,(p-1)}) \right)+\Delta t \sum_{\ell=0}^{M} \theta^m_\ell \uvec{G}(t^\ell,\uvec{u}_s^{\ell,(p-1)}).
\end{align}
\end{itemize}
The difference lies in the term
\begin{align}
\Delta t \sum_{\ell=0}^{m-1} \gamma^\ell \left( \uvec{G}(t^{\ell},\uvec{u}_s^{\ell,(p)})-\uvec{G}(t^{\ell},\uvec{u}_s^{\ell,(p-1)}) \right),
\end{align}
which consists in a sum of differences of evaluations of the function $\uvec{G}$ multiplied by $\Delta t$.
We will show now why this term can be seen as a perturbation of the updating formula of the first formulation with no impact on the accuracy. This actually depends on the fact that $\uvec{u}^{\ell,(p)}$ and $\uvec{u}^{\ell,(p-1)}$ are approximations of the same quantity.

\begin{proposition}[sDeC accuracy]
	The approximation $\uvec{u}^{m,(p)}_s$ provided by the sDeC \eqref{eq:sDeC} is an $O(\dt^{p+1})$ perturbation of $\uvec{u}^{m,(p)}_b$ obtained through the bDeC \eqref{eq:bDeC}.
\end{proposition}
\begin{proof}
We will prove it by induction over $p$ and $m$. The base case of the induction is clearly true as $\uvec{u}^{m,(p)}_s=\uvec{u}^{m,(p)}_b=\uvec{u}^{0}$ whenever $p$ or $m$ are equal to $0$. We focus now on the induction step. We select $p,m\geq 1$ and assume
\begin{align}
\uvec{u}^{\ell,(k)}_s=\uvec{u}^{\ell,(k)}_b+O(\dt^{k+1}),  \text{  for  } \begin{cases}
k<p, \quad \forall \ell=1,\dots,M, \text{ or}\\
k=p, \quad \forall \ell \leq m-1
\end{cases}
\end{align}
and we will prove that $\uvec{u}^{m,(p)}_s=\uvec{u}^{m,(p)}_b+O(\dt^{p+1})$.
We start from \eqref{eq:sDeC} and, thanks to the induction hypothesis, to the Lipschitz-continuity of $\uvec{G}$ and by definition of $\uvec{u}^{m,(p)}_b$ in \eqref{eq:bDeC}, we have that
\begin{align}
\begin{split}
\uvec{u}^{m,(p)}_s&=\uvec{u}^{0}+\Delta t \sum_{\ell=0}^{m-1} \gamma^{\ell+1} \left( \uvec{G}(t^{\ell},\uvec{u}_s^{\ell,(p)})-\uvec{G}(t^{\ell},\uvec{u}_s^{\ell,(p-1)}) \right)+\Delta t \sum_{\ell=0}^{M} \theta^m_\ell \uvec{G}(t^\ell,\uvec{u}_s^{\ell,(p-1)})\\
\begin{split}
&=\uvec{u}^{0}+\Delta t \sum_{\ell=0}^{m-1} \gamma^{\ell+1} \left( \uvec{G}(t^{\ell},\uvec{u}_b^{\ell,(p)})-\uvec{G}(t^{\ell},\uvec{u}_b^{\ell,(p-1)})+O(\dt^{p}) \right)\\
&\quad\quad\quad\quad\quad\quad\quad\quad\quad\quad\quad\quad\quad\quad +\Delta t \left(\sum_{\ell=0}^{M} \theta^m_\ell \uvec{G}(t^\ell,\uvec{u}_b^{\ell,(p-1)})+O(\dt^{p}) \right)
\end{split}\\
&=\uvec{u}^{m,(p)}_b+\Delta t \sum_{\ell=0}^{m-1} \gamma^{\ell+1}  \left( \uvec{G}(t^{\ell},\uvec{u}_b^{\ell,(p)})-\uvec{G}(t^{\ell},\uvec{u}_b^{\ell,(p-1)}) \right)+O(\dt^{p+1}).
\end{split}
\label{eq:sDeC2}
\end{align}
Thanks again to the Lipschitz-continuity of $\uvec{G}$ and to the results on the accuracy of the bDeC method, for each $\ell=1,\dots, m-1$, we can write 
\begin{align}
\begin{split}
\norm{\uvec{G}(t^{\ell},\uvec{u}_b^{\ell,(p)})-\uvec{G}(t^{\ell},\uvec{u}_b^{\ell,(p-1)})}_{\infty,Q}&\leq L\norm{\uvec{u}_b^{\ell,(p)}-\uvec{u}_b^{\ell,(p-1)}}_{\infty,Q}\\
&\leq L\norm{\uvec{u}_\Delta^{\ell}-\uvec{u}_\Delta^{\ell} +O(\dt^{p}) }_{\infty,Q}=O(\dt^{p}),
\end{split}
\label{eq:neededbutnoname}
\end{align}
where $\uvec{u}_\Delta^{\ell}$ is the $\ell$-th component of $\usoldt$, solution to $\lopdt^2$; further, for  $\ell=0$ we have $\uvec{G}(t^{\ell},\uvec{u}_b^{\ell,(p)})-\uvec{G}(t^{\ell},\uvec{u}_b^{\ell,(p-1)})=\uvec{0}$ as the component at the initial subtimenode is always equal to $\uvec{u}^0$.
By the previous fact, coming back to \eqref{eq:sDeC2}, we get the thesis
\begin{equation}
\uvec{u}^{m,(p)}_s=\uvec{u}^{m,(p)}_b+O(\dt^{p+1}).
\end{equation}

\end{proof}

\section{Continuous Galerkin FEM}\label{sec:CG}
Let $\Omega\subset \R^D$ an open regular bounded domain. The general form of a hyperbolic system of balance laws reads 
\begin{equation}
\label{sys}
\frac{\partial}{\partial t}\uvec{u}(\uvec{x},t)+\textrm{div}_{\uvec{x}} \boldsymbol{F}(\uvec{u}(\uvec{x},t))=\boldsymbol{S}(\uvec{x},\uvec{u}(\uvec{x},t)), \qquad (\uvec{x},t) \in \Omega\times \mathbb{R}^+_0,
\end{equation}
provided with some initial condition $\uvec{u}(\uvec{x},0)=\uvec{u}_0(\uvec{x})$ on $\Omega$ and some boundary conditions on $\partial \Omega$.

Let us define $\mathcal{T}_h$ a triangulation of $\overline{\Omega}$ and denote with $K$ the general element, which we assume to be convex and closed. Consider the continuous finite element space $V_h = \lbrace g \in C^0(\overline{\Omega}) : g|_K \in \mathbb{P}_M(K) \,~ \forall K \in \mathcal{T}_h \rbrace$, let $\lbrace \varphi_i \rbrace_{i=1,\dots,I}$ be a basis of $V_h$ such that each $\varphi_i$ can be associated to a degree of freedom  $\uvec{x}_i\in \overline{\Omega}$ and has support contained in $\mathcal{K}_i:=\cup_{K \in K_i} K$, where $K_i:=\lbrace K \in \mathcal{T}_h : \uvec{x}_i \in K \rbrace$. Further, we assume the basis functions normalized in such a way that $\sum_{i=1}^I\varphi_i\equiv 1$. The general form of the semidiscrete formulation of a continuous Galerkin FEM scheme consists in finding a solution $\uvec{u}_h(\uvec{x}) = \sum_{i} \uvec{c}_i(t) \varphi_i(\uvec{x})$, with $\uvec{c}_i(t) \in \R^{Q}$ at any time $t$, such that
\begin{equation}
\sum_{K\in K_i}\sum_{\uvec {x}_j \in K} \left(\int_K \varphi_i(\uvec{x}) \varphi_j(\uvec{x}) d \uvec{x}\right)\frac{d}{d t} \uvec{c}_j(t)+ \spacestuff(\uvec{c}(t))=\uvec{0}, \qquad \forall i=1,\dots,I,
\label{systosolve}
\end{equation}
where $\uvec{ST}_i(\uapp)$ are some stabilization terms and the space residuals $\spacestuff(\uvec{c}(t))$ are defined as
\begin{equation}
\spacestuff(\uvec{c}(t))=\sum_{K\in K_i} \int_{K} \left( \div_{\uvec{x}} \boldsymbol{F}(\uapp(\uvec{x},t))-\boldsymbol{S}(\uvec{x},\uapp(\uvec{x},t))\right)\varphi_i(\uvec{x}) d\uvec{x}+\uvec{ST}_i(\uapp),
\label{spst}
\end{equation}
with $\uvec{c}(t)\in \mathbb{R}^{I\times Q}$ containing as components all the $Q$-dimensional vectors $\uvec{c}_i(t)$ associated to the DoFs.

\subsection{DeC for CG}
In this context, the parameter $\Delta$ of the Deferred Correction is the mesh parameter $h$ of the space discretization. We assume CFL conditions on the temporal step size, i.e.,  $\Delta t \leq C h$ for some fixed constant $C>0$.
We will implicitly assume the Bernstein polynomials as basis functions; nevertheless, the method can be extended also to other basis functions provided that some constraints concerning the construction of the operator $\lopdt^1$, specified in the following, are fulfilled.

\subsubsection{Preliminary results}
Here, we will present some useful preliminary results that will be used later to prove the first order accuracy of $\lopdt^1$ and the Lipschitz-continuity-like condition of $\lopdt^1-\lopdt^2$.
In particular, we will prove two propositions, via some intermediate lemmas. We will focus on the Bernstein polynomials; nevertheless the results can be easily extended to other polynomial bases.

Let us consider a general element $K$, the vector space $\mathbb{P}_M(K)$ of the scalar polynomial functions of degree $M$ defined on it and $u\in \mathbb{P}_M(K)$.
We can express $u$ as a linear combination of the Bernstein polynomials $\left\lbrace \varphi_r \right\rbrace_{r=1,\dots,R}$ of degree $M$ defined on the element because they are a basis of $\mathbb{P}_M(K)$.
We have thus 
\begin{equation}
u(\uvec{x})=\sum_{r=1}^R c_r \varphi_r(\uvec{x}), \quad \forall \uvec{x}\in K,
\label{uBernstein}
\end{equation}
where the scalar coefficients $c_r$ are the Bernstein coefficients associated to the DoFs $\uvec{x}_r \in K$.
Another possibility is to express $u$ in terms of the Lagrange basis functions $\left\lbrace \hat{\varphi}_r \right\rbrace_{r=1,\dots,R}$ defined on $K$ which constitute another basis of $\mathbb{P}_M(K)$.
Therefore, we can also write
\begin{equation}
u(\uvec{x})=\sum_{r=1}^R v_r \hat{\varphi}_r(\uvec{x}), \quad \forall \uvec{x}\in K,
\label{uLagrange}
\end{equation}
where the scalar coefficients $v_r$ are the values of $u$ in the DoFs $\uvec{x}_r\in K$.
We define the vector $\uvec{c}\in \mathbb{R}^R$ of the coefficients of $u\in \mathbb{P}_M(K)$ with respect to the Bernstein basis and the vector $\uvec{v}\in \mathbb{R}^R$ of the values of $u$ in all the DoFs of $K$, i.e., the coefficients with respect to the Lagrange basis.

It is always possible to pass from the Bernstein coefficients to the values in the DoFs through the transition matrix $T$ defined as

\begin{equation}
T=\begin{pmatrix}
\varphi_1(\uvec{x}_1) & \varphi_2(\uvec{x}_1) & \dots & \varphi_R(\uvec{x}_1)\\
\varphi_1(\uvec{x}_2) & \varphi_2(\uvec{x}_2) & \dots & \varphi_R(\uvec{x}_2)\\
\vdots & \vdots & \ddots & \vdots \\
\varphi_1(\uvec{x}_R) & \varphi_2(\uvec{x}_R) & \dots & \varphi_R(\uvec{x}_R)
\end{pmatrix}.
\label{transitionmatrix}
\end{equation}
The general element of $T=(T_{ij})_{i,j=1,\dots,R}$ with row index $i$ and column index $j$ is $T_{ij}=\varphi_j(\uvec{x}_i)$ and we have $\uvec{v}=T\uvec{c}$ and $\uvec{c}=T^{-1}\uvec{v}$.
\begin{observation}[Independence of the mesh parameter.]
\label{independece}
Neither the matrix $T$ nor its inverse $T^{-1}$ depend on the size of the element $K$.
They just depend on the spatial dimension $D$ and on the degree $M$.
Once we fix $D$ and $M$, for any specific type of elements, for example the simplices, we have a fixed $T$ and $T^{-1}$.
\end{observation}
It is clear that the sum of the elements of each row of $T$ is equal to $1$, in fact
\begin{equation}
\sum_{j=1}^R T_{ij}=\sum_{j=1}^R \varphi_j(\uvec{x}_i)=1, \quad \forall i=1,\dots,R.
\label{sumrowT}
\end{equation}
This is due to the assumption on the basis functions, which are normalized in such a way that that
\begin{equation}
\sum_{j=1}^R \varphi_j(\uvec{x})\equiv 1, \quad \forall \uvec{x} \in K.
\label{normalization}
\end{equation}
Also its inverse $T^{-1}$ enjoys the same property as we will prove in the next lemma.

\begin{lemma}\label{prop:sum_row_inverse}
The sum of the elements of each row of $T^{-1}$, inverse of the transition matrix defined in \eqref{transitionmatrix}, is equal to $1$.
\end{lemma}
\begin{proof}
Let us observe that proving the thesis is equivalent to prove that $T^{-1}\uvec{1}=\uvec{1}$ where $\uvec{1}\in \mathbb{R}^{R}$ is a vector with all the entries equal to $1$.
From \eqref{sumrowT} we have that $T\uvec{1}=\uvec{1}$.
Thanks to the previous equality, we have that
\begin{equation}
T^{-1}\uvec{1}=T^{-1}T\uvec{1}=\uvec{1}
\label{eq:final_theorem_T}
\end{equation}
which is the thesis.
\end{proof}

The previous result will be used to prove the following lemma.

\begin{lemma}\label{prop:difference_coefficients}
For any polynomial $u\in \mathbb{P}_M(K)$ such that
\begin{equation}
u(\uvec{x})=\sum_{r=1}^R c_r \varphi_r(\uvec{x})=\sum_{r=1}^R v_r \hat{\varphi}_r(\uvec{x}), \quad \forall \uvec{x}\in K,
\label{uBernsteinLagrange}
\end{equation}
where $\varphi_r$ are the Bernstein polynomials of $\mathbb{P}_M(K)$, $c_r$ the Bernstein coefficients, $\hat{\varphi}_r$ the Lagrange polynomials of $\mathbb{P}_M(K)$ and $v_r$ the Lagrange coefficients, it holds that
\begin{equation}
\sup_{i,j=1,\dots,R}{\vert c_{i}-c_{j} \vert}\leq \tilde{C}\sup_{i,j=1,\dots,R}{\vert v_{i}-v_{j} \vert},
\label{fromcoefficientstovalues}
\end{equation}
where $\tilde{C}>0$ is independent of the size and aspect ratio of $K$.
\end{lemma}
\begin{proof}
The proof is a straightforward consequence of \cref{prop:sum_row_inverse}. From the fact that $\uvec{c}=T^{-1}\uvec{v}$ we know that every Bernstein coefficient $c_r$ can be expressed as a linear combination of the values $v_k$ in the DoFs through the coefficients of the row $r$ of the matrix $T^{-1}$
\begin{equation}
c_{i}= \sum_{k=1}^R (T^{-1})_{ik}v_k, \quad c_{j}=\sum_{k=1}^R (T^{-1})_{jk}v_k
\label{eq:c_combination_of_v}
\end{equation}
and therefore
\begin{equation}
\vert c_i-c_j \vert=\left\vert \sum_{k=1}^R (T^{-1})_{ik}v_k- \sum_{k=1}^R (T^{-1})_{jk}v_k \right\vert.
\label{combinationvalues}
\end{equation}
Now, from \cref{prop:sum_row_inverse}, we know that the coefficients $(T^{-1})_{rk}$ are such that
\begin{equation}
\sum_{k=1}^R(T^{-1})_{rk}=1 \quad \forall r=1,\dots, R.
\end{equation}
This is in particular true for $r=i$ and $r=j$ and so there exist some coefficients $\lambda_{k,\ell}^{i,j}$, depending on $i$ and $j$, such that \eqref{combinationvalues} can be written as
\begin{equation}
\vert c_{i}-c_{j}\vert=\left\vert \sum_{k=1}^R (T^{-1})_{ik}v_k-\sum_{k=1}^R (T^{-1})_{jk}v_k \right\vert=\left\vert \sum_{k,\ell=1}^R \lambda_{k,\ell}^{i,j}(v_k-v_\ell) \right\vert.
\label{combinationvaluespro}
\end{equation}
One simple choice of these coefficients is given by $\lambda_{k,\ell}^{i,j} = \frac{(T^{-1})_{ik}-(T^{-1})_{jk}}{R}$ and a simple computation can be used to prove it. This might lead to suboptimal values of the estimations.
The coefficients $\lambda_{k,\ell}^{i,j}$, like the coefficients $T_{ij}$ and $(T^{-1})_{ij},$ do not depend on the size of $K$, and, thus, they can be bounded by a positive constant $C_\lambda$, which depends just on the type of the element considered. Then, thanks to the triangular inequality, \eqref{combinationvaluespro} gives
\begin{equation}
\vert c_{i}-c_{j}\vert=\left\vert \sum_{k,\ell=1}^R \lambda_{k,\ell}^{i,j}(v_k-v_\ell) \right\vert\leq \sum_{k,\ell=1}^R \vert \lambda_{k,\ell}^{i,j} \vert \vert v_k-v_\ell \vert \leq C_\lambda \sum_{k,\ell=1}^R \vert v_k-v_\ell \vert.
\label{combinationvaluesprofinal}
\end{equation}
Since the number of dimensions $D$ and the degree $M$ are fixed, also $R$ is fixed and so the number of terms in the sum. Therefore, from \eqref{combinationvaluesprofinal} we get
\begin{equation}
\vert c_{i}-c_{j}\vert\leq C_\lambda \sum_{k,\ell=1}^R \vert v_k-v_\ell \vert \leq \tilde{C} \sup_{i,j=1,\dots,R}{\vert v_{i}-v_{j} \vert}
\label{eq:not_needed1}
\end{equation}
for some $\tilde{C}=C_\lambda R^2$ independent of the size of $K$.
\end{proof}

This allows to prove the following result.
\begin{lemma}\label{prop:bound_H1}
For any polynomial $u\in \mathbb{P}_M(K)$ such that
$
u(\uvec{x})=\sum_{r=1}^R c_r \varphi_r(\uvec{x}),\, \forall \uvec{x}\in K,
$
where $\varphi_r$ are the Bernstein polynomials of $\mathbb{P}_M(K)$ and $c_r$ the Bernstein coefficients, then
\begin{equation}
\sup_{i,j=1,\dots,R}{\vert c_{i}-c_{j} \vert}\leq \tilde{C}h\norm{ \norm{\nabla_{\uvec{x}}u}_1 }_{L^\infty(K)}
\label{crucialinequality}
\end{equation}
where $\tilde{C}$ is the positive constant in \eqref{fromcoefficientstovalues} (and thus independent of the size of $K$, dependent just on the number of dimensions $D$, on the degree $M$ and on the type of the element) and $h$ is such that $diam(K) \leq h$.
The norm $\norm{ \cdot }_1$ is the $1$-norm in $\mathbb{R}^D$, the norm $\norm{ \cdot }_{L^\infty(K)}$ is the $L^\infty$ norm over $K$. 
\end{lemma}
\begin{proof}
This is a consequence of \cref{prop:difference_coefficients}, in fact, from basic analysis, we know that for any smooth scalar function $f\in C^1(K)$  
\begin{equation}
\sup_{\uvec{x},\uvec{y}\in K} \vert f(\uvec{x})-f(\uvec{y}) \vert \leq h \norm{ \norm{\nabla_{\uvec{x}}f}_1 }_{L^\infty(K)},
\label{not_needed3}
\end{equation}
where we remark that $K$ is assumed to be closed.
Thus for the polynomial $u$, thanks to the inequality \eqref{fromcoefficientstovalues}, we have
\begin{equation}
\sup_{i,j=1,\dots,R}{\vert c_{i}-c_{j} \vert}\leq \tilde{C}\sup_{i,j=1,\dots,R}{\vert v_{i}-v_{j} \vert} \leq \tilde{C}h\norm{ \norm{\nabla_{\uvec{x}}u}_1 }_{L^\infty(K)},
\label{not_needed4}
\end{equation}
because $v_{r}$ are the values of $u$ in the DoFs of $K$.
\end{proof}

We will continue now with the first proposition of this section, which will be used later in the proofs of the first order accuracy of $\lopdt^1$ and of the Lipschitz-continuity-like condition of $\lopdt^1-\lopdt^2$.
\begin{proposition}[Mass lumping accuracy]
\label{finalresult}
Let us consider a scalar continuous piecewise polynomial function $u\in V_h$.
We can write $u$ as a linear combination of the Bernstein polynomials $\left\lbrace \varphi_i \right\rbrace_{i=1,\dots,I}$ associated to the tessellation which constitute a basis of $V_h$, i.e., $
u(\uvec{x})=\sum_{i=1}^I c_i \varphi_i(\uvec{x}) \quad \forall \uvec{x}\in \overline{\Omega}
$ with $c_i$ scalar coefficients.
Then, we have  $\forall i=1,\dots,I$ that 
\begin{align}
\label{finalresultscalar}
& \left\vert \sum_{K\in K_i} c_i \int_K\!\! \varphi_i(\uvec{x})d\uvec{x}-\sum_{K\in K_i}\sum_{\uvec{x}_j\in K}   c_j\int_K\!\! \varphi_i(\uvec{x})\varphi_j(\uvec{x})d\uvec{x}  \right\vert  \leq \hat{C}h\norm{ \norm{\nabla_{\uvec{x}}u}_1 }_{L^\infty(\mathcal{K}_i)}  \int_{\mathcal{K}_i}\!\! \left\vert \varphi_i(\uvec{x}) \right\vert d\uvec{x}, 
\end{align}
with $h=\max_{K\in\mathcal{T}_h} diam(K)$  and  $\hat{C}$ being a constant independent of $h$, dependent just on the dimension $D$, on the degree $M$ and on the type of the elements in the mesh.
\end{proposition}
\begin{proof}
We will assume at first all the elements of the tessellation to be of the same type but this hypothesis can be relaxed to the general case with different types of elements.

Let us focus on the left-hand side of \eqref{finalresultscalar}. Thanks to the normalization \eqref{normalization} of the basis functions and to the fact that the only basis functions that are not identically zero in the element $K$ are the ones associated to the DoFs contained in that element, we can write
\begin{equation}
	\begin{split}
&\left\vert \sum_{K\in K_i}c_i \int_K\!\! \varphi_i(\uvec{x})d\uvec{x} -\sum_{K\in K_i} \sum_{\uvec{x}_j\in K} c_j \int_K\!\! \varphi_i(\uvec{x})\varphi_j(\uvec{x})d\uvec{x}   \right\vert
=\left\vert \sum_{K\in K_i}\sum_{\uvec{x}_j\in K}(c_i- c_j) \int_K \!\!\varphi_i(\uvec{x}) \varphi_j(\uvec{x}) d\uvec{x}  \right\vert. 
	\end{split}
\end{equation}

Now, thanks to the triangular inequality, to the fact that the absolute value of the basis functions $\varphi_j$ can be bounded by a constant 
$C_0$, independent of the size of $K$, dependent just on the dimension $D$, on the degree $M$ and on the type of the elements in the tessellation and also to the fact that the number $R$ of DoFs $\uvec{x}_j$ in each element $K$ is fixed since $D$ and $M$ are fixed, we can write
\begin{align}
\begin{split}
&\left\vert \sum_{K\in K_i}\sum_{\uvec{x}_j\in K} (c_i- c_j) \int_K \varphi_i(\uvec{x}) \varphi_j(\uvec{x}) d\uvec{x} \right\vert \leq  \sum_{K\in K_i}\sum_{\uvec{x}_j\in K}  \vert c_i- c_j \vert \left\vert \int_K \varphi_i(\uvec{x}) \varphi_j(\uvec{x}) d\uvec{x} \right\vert \\
\leq & \sum_{K\in K_i}\sum_{\uvec{x}_j\in K}  \sup_{\uvec{x}_\ell \in K} \vert c_i- c_\ell \vert \int_K \vert \varphi_i(\uvec{x}) \vert \vert \varphi_j(\uvec{x}) \vert d\uvec{x}    \leq 
\sum_{K\in K_i}\sum_{\uvec{x}_j\in K} C_0 \sup_{\uvec{x}_\ell \in K} \vert c_i- c_\ell \vert \int_K \vert \varphi_i(\uvec{x}) \vert  d\uvec{x}  \\
\leq &
\sum_{K\in K_i}R C_0  \sup_{\uvec{x}_\ell \in K} \vert c_i- c_\ell \vert \int_K \vert \varphi_i(\uvec{x}) \vert  d\uvec{x}.
\end{split}
\end{align}

By applying the previous proposition \eqref{crucialinequality} and from the fact that by definition $\mathcal{K}_i=\cup_{K\in K_i}K$, we can continue the sequence of inequalities and get
\begin{align}
\begin{split}
&\sum_{K\in K_i}R C_0 \sup_{\uvec{x}_\ell \in K} \vert c_i- c_\ell \vert  \int_K \vert \varphi_i(\uvec{x}) \vert  d\uvec{x} \leq
\sum_{K\in K_i}R C_0 \, \tilde{C}h\norm{ \norm{\nabla_{\uvec{x}}u}_1 }_{L^\infty(K)} \int_K \vert \varphi_i(\uvec{x}) \vert  d\uvec{x}\\
\leq& R C_0 \tilde{C}h \norm{ \norm{\nabla_{\uvec{x}}u}_1 }_{L^\infty(\mathcal{K}_i)}\sum_{K\in K_i}  \int_K \vert \varphi_i(\uvec{x}) \vert  d\uvec{x}=R C_0 \tilde{C}h\norm{ \norm{\nabla_{\uvec{x}}u}_1 }_{L^\infty(\mathcal{K}_i)} \int_{\mathcal{K}_i} \left\vert \varphi_i(\uvec{x}) \right\vert d\uvec{x}. \label{eq:last_inequality}
\end{split}
\end{align}

We take $\hat{C}=R C_0 \tilde{C}$ and we have the thesis, in fact, none of $R$, $C_0$ and $\tilde{C}$ depend on $h$, but they just depend on the dimension $D$, on the degree $M$ and on the type of the elements in the tessellation.

We remark that we assumed that all the elements of the tessellation were of the same type. To deal with the general case in which we have different types of elements we suffice to take $\tilde{C}$ as the maximum of the coefficients $\tilde{C}$ of \cref{prop:bound_H1} associated to the different types of elements and $R$ as the highest number of degrees of freedom in a single element.

\end{proof}
Before going ahead let us make some useful observations.
\begin{observation}
Since the Bernstein basis functions are not negative, we can actually remove the absolute value inside the integral in \eqref{finalresultscalar}.
We left it on purpose to be more general.
In fact, it is easy to see that what is proved in this section is actually not limited to the specific case of Bernstein polynomials; the results can be easily extended to other polynomial bases, like for example the Lagrange polynomials (for which the matrix $T$ is the identity and the constant $\tilde{C}=1$) provided that the normalization \eqref{normalization} holds, i.e. $\sum_{i=1}^I\varphi(\uvec{x})\equiv 1$.
\end{observation}
\begin{observation}
The final result \eqref{finalresultscalar}, which has been proven for a scalar polynomial $u \in V_h$, can be easily extended to the vectorial case by applying it componentwise.
If $\uvec{u}\in V_h^Q$, then we have
\begin{align}
\uvec{u}(\uvec{x})=\sum_{i=1}^I\uvec{c}_i\varphi_i(\uvec{x}) \quad \forall \uvec{x}\in \overline{\Omega}
\end{align}
with $\uvec{c}_i\in \mathbb{R}^Q$ $\forall i=1,\dots,I$ being $Q$-dimensional vectors of coefficients and $\left\lbrace \varphi_i\right\rbrace_{i=1,\dots,I}$ the Bernstein basis and it holds that
\begin{align}
&\norm{ \sum_{K\in K_i}\left(\int_K \varphi_i(\uvec{x})d\uvec{x}\right) \uvec{c}_i-\sum_{K\in K_i}\sum_{\uvec{x}_j\in K} \left(\int_K \varphi_i(\uvec{x})\varphi_j(\uvec{x})d\uvec{x} \right) \uvec{c}_j }_\infty \nonumber \\
&\leq \hat{C}h\norm{\norm{ \norm{\nabla_{\uvec{x}}\uvec{u}}_1 }_{L^\infty(\mathcal{K}_i)}}_\infty  \int_{\mathcal{K}_i} \left\vert \varphi_i(\uvec{x}) \right\vert d\uvec{x}  \quad \forall i=1,\dots,I
\label{finalresultvectorial}
\end{align}
where the norms $\norm{\cdot}_1$ and $\norm{\cdot}_{L^\infty(\mathcal{K}_i)}$ are applied to each scalar component while the norm $\norm{\cdot}_\infty$ is on $\mathbb{R}^Q$.

The key point is that the result \eqref{finalresultscalar} is uniform with respect to all the components of $\uvec{u}$ and so we can easily take the infinity norm of both sides to pass from the scalar to the vectorial case.

\end{observation}

We focus now on another intermediate lemma before proving the second and final proposition of this section.

\begin{lemma}\label{lem:equivalence_1_inf_element}
	Let $z\in C^1(K)$ and assume that its gradient is bounded in such a way that $\norm{ \norm{\nabla_{\uvec{x}}z}_1}_{L^\infty(K)}\leq C_g$. Then, for $K$ small enough it holds
	\begin{equation}\label{eq:equivalence_1_inf_element}
		 \norm{z}_{L^1(K)}  \geq C^* \norm{z}_{L^\infty(K)} |K|,
	\end{equation}
with $|K|$ measure of $K$ and $C^*$ a constant dependent on $C_g$ and on $\norm{z}_{L^\infty(K)}$ but independent of the size of $K$.
\end{lemma}
\begin{proof}
As $K$ is closed and $z\in C^1(K)$, then 
\begin{equation}
\exists \uvec{x}^*\in K~\text{s.t.}~ \abs{z(\uvec{x}^*)}=\norm{z}_{L^\infty(K)}<+\infty.
\end{equation}
Further, due to the continuity of $z$, the set $B$ of the points in $K$ for which the absolute value of the function is larger or equal than $\frac{\abs{z(\uvec{x}^*)}}{2}$ is non-empty and has a strictly positive measure, i.e. $\abs{B}>0$ with 
\begin{equation}
B:=\left\lbrace \uvec{x}\in K ~\text{s.t.}~ \abs{z(\uvec{x})} \geq \frac{\abs{z(\uvec{x}^*)}}{2}  \right\rbrace.
\end{equation}
We try now to find a lower bound for $\abs{B}$ by defining a set $B^*\subseteq B$ whose measure is known; in particular we define  
\begin{equation}
B^*:=\left\lbrace \uvec{x}\in K ~\text{s.t.}~d(\uvec{x},\uvec{x}^*)\leq \frac{\abs{z(\uvec{x}^*)}}{2C_g} \right\rbrace.
\end{equation}
where $d(\cdot,\cdot)$ is the Euclidean distance. Indeed, we have that $B^*\subseteq B$. Let $\tilde{\uvec{x}}\in B^*$, then by a simple Taylor expansion we get
\begin{align}
\abs{z(\tilde{\uvec{x}})}=\abs{z(\uvec{x}^*)+\nabla_{\uvec{x}}z(\uvec{s}) (\uvec{s}-\uvec{x}^*)}
\end{align}
with $\uvec{s}$ being a point, dependent on $\tilde{\uvec{x}}$, contained in the segment $S(\tilde{\uvec{x}},\uvec{x}^*)$ connecting $\tilde{\uvec{x}}$ and $\uvec{x}^*$.
The triangle inequality gives 
\begin{align}
\abs{z(\tilde{\uvec{x}})}&=\abs{z(\uvec{x}^*)+\nabla_{\uvec{x}}z(\uvec{s}) (\uvec{s}-\uvec{x}^*)}\geq \abs{z(\uvec{x}^*)}-\abs{\nabla_{\uvec{x}}z(\uvec{s}) (\uvec{s}-\uvec{x}^*)}.
                \label{eq:bound_x_in_B}
\end{align}
Now, we have that $\abs{\nabla_{\uvec{x}}z(\uvec{s}) (\uvec{s}-\uvec{x}^*)}\leq \frac{\abs{z(\uvec{x}^*)}}{2}$ because of the regularity assumption on the gradient of $z$ and because $d(\uvec{s},\uvec{x}^*)\leq d(\tilde{\uvec{x}},\uvec{x}^*)$ as $\uvec{s}$ belongs to the segment $S(\tilde{\uvec{x}},\uvec{x}^*)$. This can be seen by simple computations:
\begin{align}
\abs{\nabla_{\uvec{x}}z(\uvec{s}) (\uvec{s}-\uvec{x}^*)}&\leq \norm{ \norm{\nabla_{\uvec{x}}z}_1}_{L^\infty(K)}d(\uvec{s},\uvec{x}^*)\leq C_g\frac{\abs{z(\uvec{x}^*)}}{2C_g}=\frac{\abs{z(\uvec{x}^*)}}{2}.
\end{align}
Coming back to \eqref{eq:bound_x_in_B} with this information, we can write
\begin{align}
\abs{z(\tilde{\uvec{x}})}&\geq \abs{z(\uvec{x}^*)}-\abs{\nabla_{\uvec{x}}z(\uvec{s}) (\uvec{s}-\uvec{x}^*)}\geq\abs{z(\uvec{x}^*)}-\frac{\abs{z(\uvec{x}^*)}}{2}=\frac{\abs{z(\uvec{x}^*)}}{2}
\end{align} 
and hence $\tilde{\uvec{x}}\in B$ and $B^*\subseteq B$. 

We are able to estimate the measure of $B^*$ providing therefore a lower bound for $\abs{B}$, indeed, by definition, such set is the intersection between $K$ and the ball $B_{\rho}(\uvec{x}^*)$ centered in $\uvec{x}^*$ with radius $\rho:=\frac{\abs{z(\uvec{x}^*)}}{2C_g}$. If the ball $B_{\rho}(\uvec{x}^*)$ is entirely contained in $K$ then $B^*=B_{\rho}(\uvec{x}^*)$ and its measure is given by $\abs{B^*}=\abs{B_{\rho}(\uvec{x}^*)}=C_s\rho^D$ where $C_s$ is the measure of the unitary ball in $\mathbb{R}^D.$ If this does not hold, it is anyway always possible to find a lower bound for the measure of $B^*$ of the type 
\begin{equation}
\abs{B^*}\geq \min{(C_\alpha \rho^D,\abs{K})}
\end{equation}
with $C_\alpha$ constant dependent only on the aspect ratio of $K$ but not on its size.
Therefore, from the definition of $B$ and from $\abs{B}\geq\abs{B^*}\geq \min{(C_\alpha \rho^D,\abs{K})}$, we get
\begin{align}
\norm{z}_{L^1(K)}=\int_K \abs{z(\uvec{x})}d \uvec{x} \geq \int_{B} \abs{z(\uvec{x})}d \uvec{x} \geq \frac{\abs{z(\uvec{x}^*)}}{2} \abs{B}\geq \frac{\abs{z(\uvec{x}^*)}}{2} \abs{B^*}  \geq \frac{\abs{z(\uvec{x}^*)}}{2} \min{(C_\alpha \rho^D,\abs{K})}. 
\end{align}
Now, recalling that $\abs{z(\uvec{x}^*)}=\norm{z}_{L^\infty(K)}$, we have 
\begin{align}
\norm{z}_{L^1(K)}\geq  \frac{\norm{z}_{L^\infty(K)}}{2} \abs{K}\min{\left(\frac{C_\alpha \rho^D}{\abs{K}},1 \right)} 
\end{align}
We define thus $C^*:=\frac{1}{2}\min{\left(\frac{C_\alpha \rho^D}{\abs{K}},1 \right)}$ and we observe that, since $C_\alpha$ only depends on geometrical properties of $K$ and $\rho$ only depends on $z$, for $K$ small enough  $C^*=\frac{1}{2}$ and we get the thesis.

\end{proof}

Now, let us generalize this result to the whole domain for piecewise $C^1$ functions, even discontinuous, by proving the last result of this section.

\begin{proposition}[Relation between $L^\infty$ and $L^1$ norms]\label{lem:equivalence_norm1_norminf}
	Let $z \in \left\{ z \in L^1(\Omega) ~\text{s.t.}~ z|_K \in C^1(K), \, \forall K \in \mathcal{T}_h \right\}$ satisfying locally in each element the hypotheses of the previous lemma, i.e. $\norm{ \norm{\nabla_{\uvec{x}}z}_1}_{L^\infty(K)}\leq C_g$ and $K$ small enough. 
	Assume the mesh to be regular in the sense that for any $i=1,\dots,I$ it holds that
	\begin{equation}\label{eq:mesh_regularity}
		\int_{K} \abs{\varphi_i(\x)} d\x \leq C_\mathcal{M} \int_{\tilde{K}} \abs{\varphi_i(\x)} d\x,\qquad \forall K,\tilde{K} \in K_i,
	\end{equation} 
where $\left\lbrace \varphi_i \right\rbrace_{i=1,\dots,I}$ is the basis of $V_h$ given by Bernstein polynomials.
Then,
	\begin{equation}\label{eq:equivalence_norm1_norminf}
		\sum_{i=1}^I \norm{z}_{L^\infty(\mathcal{K}_i)} \sum_{K \in K_i} \int_{K} \abs{\varphi_i(\x)} d\x \leq \tilde{C}^* \norm{z}_{L^1(\Omega)},
	\end{equation}
	where $\tilde{C}^*$ is a positive constant independent of the mesh parameter.
\end{proposition}
\begin{proof}
Let $K^i \in K_i$ be the element such that $\norm{z}_{L^\infty(\mathcal{K}_i)} = \norm{z}_{L^\infty(K^i)} $; then, using the mesh regularity assumption \eqref{eq:mesh_regularity} and the fact that the basis functions are bounded in absolute value by a constant $C_0$ independent of the mesh parameter, we have 
\begin{align}
\begin{split}
 \sum_{i=1}^I  \norm{z}_{L^\infty(\mathcal{K}_i)} \sum_{K \in K_i} \int_{K} \abs{\varphi_i(\x)} d\x &=\sum_{i=1}^I \sum_{K \in K_i}  \norm{z}_{L^\infty(\mathcal{K}_i)} \int_{K} \abs{\varphi_i(\x)} d\x\\
 &= \sum_{i=1}^I \sum_{K \in K_i}  \norm{z}_{L^\infty(K^i)} \int_{K} \abs{\varphi_i(\x)} d\x \\
	&\leq\sum_{i=1}^I \sum_{K \in K_i}  C_\mathcal{M}\norm{z}_{L^\infty(K^i)} \int_{K^i} \abs{\varphi_i(\x)} d\x\\
	&\leq\sum_{i=1}^I \sum_{K \in K_i}  C_\mathcal{M} C_0 \norm{z}_{L^\infty(K^i)} \abs{K^i}. 
\end{split}
\end{align} 
We apply now the previous \cref{lem:equivalence_1_inf_element} and, switching the sums over the elements and the DoFs, we get
\begin{align}
\begin{split}
\sum_{i=1}^I \sum_{K \in K_i}  C_\mathcal{M} C_0 \norm{z}_{L^\infty(K^i)} \abs{K^i} &\leq \sum_{i=1}^I \sum_{K \in K_i}  \frac{C_\mathcal{M} C_0}{C^*}  \norm{z}_{L^1(K^i)} \\
&=  \frac{C_\mathcal{M} C_0}{C^*} \sum_{K \in \tess} \sum_{\uvec{x}_i\in K}     \norm{z}_{L^1(K^i)}
\end{split}
\end{align}
where $C^*$ is the minimal coefficient of \cref{lem:equivalence_1_inf_element} among the ones associated to all the elements $K^i$.
If $R$ is the maximal number of DoFs in a single element in the whole mesh, we can continue and write
\begin{align}
\frac{C_\mathcal{M} C_0}{C^*} \sum_{K \in \tess} \sum_{\uvec{x}_i\in K}   \norm{z}_{L^1(K^i)} \leq \frac{R C_\mathcal{M} C_0}{C^*} \sum_{K \in \tess} \sup_{\uvec{x}_i\in K}   \norm{z}_{L^1(K^i)}.
\label{eq:counting}
\end{align}
Now, in \eqref{eq:counting}, each element $K$ in the tessellation is contributing to the sum with the $L^1$ norm of $z$ over one element $K^i$ among the ones associated to the DoFs $\uvec{x}_i \in K.$
The generic element $K^i$ can be present in the sum at most a number of times equal to $M_{n}+1$ where $M_{n}$ represents the maximal number of neighbors that an element can have in the tessellation. Hence, we get
\begin{align}
\frac{R C_\mathcal{M} C_0}{C^*} \sum_{K \in \tess} \sup_{\uvec{x}_i\in K}   \norm{z}_{L^1(K^i)}\leq \frac{R C_\mathcal{M} C_0}{C^*} (M_n+1) \norm{z}_{L^1(\Omega)}.
\end{align}
Observe that none of the coefficients $R$, $C_\mathcal{M}$, $C_0$, $C^*$ or $M_n$ depend on the mesh parameter, therefore, by setting $\tilde{C}^*=\frac{R C_\mathcal{M} C_0}{C^*} (M_n+1)$, we get the thesis.
\end{proof}
Also in this case, we remark that, in the context of Bernstein polynomials, which are non-negative, the absolute value on $\varphi_i$ is not necessary. We kept it just to be more general. Indeed, all the results can be generalized to other basis functions like the Lagrange polynomials.
 
\subsubsection{Definition of $\lopdt^2$}
The operator $\lopdt^2$ is the high order implicit operator that we would like to solve. Its definition is not very different from the one seen in the context of the bDeC for ODEs. 
We introduce the $M+1$ subtimenodes $t^m$ with $m=0,\dots,M$ in the interval $[t_n,t_n+\Delta t]$ in which we will consider the approximations of the values of the solution to our system of ODEs.
We refer to $\uvec{c}(t^m)$ as the exact solution in the node $t^m$ and to $\uvec{c}^m$ as the approximation of the solution in the same node. Clearly, in this case $\uvec{c}(t^m)$ and $\uvec{c}^m$ contain as components all the coefficients corresponding to the spatial DoFs, i.e., respectively the vectors $\uvec{c}_i(t^m)$ of the exact coefficients in the DoFs at the time $t^m$ and the vectors $\uvec{c}_i^m$ of the approximated ones. As usual, for the first subtimenode we set $\uvec{c}^0=\uvec{c}(t^0)=\uvec{c}(t_n)=\uvec{c}_n$ without any approximation. Starting from the exact integration of \eqref{systosolve} over $[t^0,t^m]$ and substituting $\spacestuff(\uvec{c}(t))$ with its $M$-th order interpolation in time associated to the $M+1$ subtimenodes, we get 
\begin{equation}
\sum_{K\in K_i}\sum_{\uvec {x}_j \in K} \left(\int_K \varphi_i(\uvec{x}) \varphi_j(\uvec{x}) d \uvec{x}\right)\left( \uvec{c}_j^m-\uvec{c}_j^0\right)+\Delta t \sum_{\ell=0}^{M} \theta^m_\ell\spacestuff(\uvec{c}^\ell)=\uvec{0}, ~ \forall i=1,\dots,I ~ \forall m=1,\dots,M.
\label{trueDeC1}
\end{equation}

Therefore, we can define the operator $\lopdt^2:\mathbb{R}^{(I \times Q \times M)}\rightarrow \mathbb{R}^{(I \times Q \times M)}$ as  

\begin{equation}
\lopdt^2(\cund)=\left(\mathcal{L}^2_{\Delta,1}(\cund),\mathcal{L}^2_{\Delta,2}(\cund),\dots,\mathcal{L}^2_{\Delta,I}(\cund)\right), \qquad \forall \cund \in \mathbb{R}^{(I \times Q \times M)},
\label{l2DeC}
\end{equation}
where for any $i$ we have 
\begin{equation}
\lopdi^2(\cund)=\begin{pmatrix}
                \sum_{K\in K_i}\sum_{\uvec {x}_j \in K} \left(\int_K \varphi_i(\uvec{x}) \varphi_j(\uvec{x}) d \uvec{x}\right)\left( \uvec{c}_j^1-\uvec{c}_j^0\right)+\Delta t \sum_{\ell=0}^{M} \theta^1_\ell\spacestuff(\uvec{c}^\ell)\\
        \vdots\\
        \sum_{K\in K_i}\sum_{\uvec {x}_j \in K} \left(\int_K \varphi_i(\uvec{x}) \varphi_j(\uvec{x}) d \uvec{x}\right)\left( \uvec{c}_j^m-\uvec{c}_j^0\right)+\Delta t \sum_{\ell=0}^{M} \theta^m_\ell\spacestuff(\uvec{c}^\ell)\\
        \vdots\\
                \sum_{K\in K_i}\sum_{\uvec {x}_j \in K} \left(\int_K \varphi_i(\uvec{x}) \varphi_j(\uvec{x}) d \uvec{x}\right)\left( \uvec{c}_j^M-\uvec{c}_j^0\right)+\Delta t \sum_{\ell=0}^{M} \theta^M_\ell\spacestuff(\uvec{c}^\ell)
\end{pmatrix}.
\label{l2iDeC}
\end{equation}
with the general argument $\cund\in \mathbb{R}^{(I \times Q \times M)}$ having $M$ components $\uvec{c}^m\in \mathbb{R}^{(I \times Q)}$ each one associated to a subtimenode and having $I$ components $\uvec{c}_i^m$ each one associated to a DoF.

The solution $\cund_\Delta$ to $\lopdt^2(\cund_\Delta)=\uvec{0}$ is $(M+1)$-th order accurate in the sense that would contain as components $(M+1)$-th order accurate approximations of the coefficients which represent the exact solution to \eqref{systosolve} in all the subtimenodes $t^m$ $m=1,\dots,M$.  Unfortunately, the problem $\lopd^2(\cund)=\uvec{0}$ is a huge nonlinear system.

\subsubsection{Definition of $\lopdt^1$}
Performing an Euler approximation in time to numerically solve \eqref{systosolve} in $[t^0,t^m]$ we get
\begin{equation}
\sum_{K\in K_i}\sum_{\uvec {x}_j \in K} \left(\int_K \varphi_i(\uvec{x}) \varphi_j(\uvec{x}) d \uvec{x}\right)\left( \uvec{c}_j^m-\uvec{c}_j^0\right)+
\Delta t\beta^{m}\spacestuff(\uvec{c}^0)=\uvec{0}, \quad\forall i=1,\dots,I, \quad \forall m=1,\dots,M.
\label{EulerDeC0}
\end{equation}
Further, we perform a first order mass lumping in space to get a fully explicit approximation formula for $\uvec{c}_i^m$ 
\begin{equation}
C_i \left( \uvec{c}_i^m-\uvec{c}_i^0\right)+
\Delta t\beta^{m}\spacestuff(\uvec{c}^0)=\uvec{0}, \quad \forall i=1,\dots,I \quad \forall m=1,\dots,M
\label{EulerDeC}
\end{equation}
where $C_i$ are constant quantities defined as  
\begin{equation}
C_i:=\int_\Omega \varphi_i(\uvec{x})d\uvec{x}=\sum_{K\in K_i} \int_{K} \varphi_i(\uvec{x})d\uvec{x}, \quad \forall i=1,\dots,I.
\label{dualcell}
\end{equation}
We assume a choice of the basis functions such that $C_i\neq 0$ $\forall i$ so that \eqref{EulerDeC} is well-posed. For example, if we choose the Bernstein polynomials, we have $C_i> 0$ $\forall i$ as the basis functions $\varphi_i$ are nonnegative.
Indeed, $\uvec{c}_i^m$ got from \eqref{EulerDeC} is a first order approximation of the exact coefficient $\uvec{c}_i(t^m)$, as proved in the next proposition.

\begin{proposition}[First order accuracy of \eqref{EulerDeC}]
The solution to \eqref{EulerDeC} is first order accurate with respect to the exact solution $\uvec{c}(t)$ to \eqref{systosolve} 
evaluated in all the subtimenodes $t^m$ for $m=1,\dots,M$.
\end{proposition}
\begin{proof}
We can equivalently show that if we insert the exact solution to \eqref{systosolve} evaluated in all the subtimenodes $t^m$ $m=1,\dots,M$ into the left-hand side of \eqref{EulerDeC} we get an error $O(\Delta^{D+2})$ where $D$ is the number of spatial dimensions and the parameter $\Delta$ is the mesh parameter $h$ of the space discretization. 
Therefore, we want to prove that
\begin{equation}
C_i \left( \uvec{c}_i(t^m)-\uvec{c}_i^0\right)+
\Delta t\beta^{m}\spacestuff(\uvec{c}^0)=O(\Delta^{D+2}), \quad \forall i=1,\dots,I, \quad \forall m=1,\dots,M.
\label{firstorderaccuracy}
\end{equation}

We know that by plugging the exact solution $\uvec{c}(t)$ in \eqref{EulerDeC0} we get an error $O(\Delta^{D+2})$:
\begin{equation}
\sum_{K\in K_i}\sum_{\uvec {x}_j \in K} \left(\int_K \varphi_i(\uvec{x}) \varphi_j(\uvec{x}) d \uvec{x}\right)\left( \uvec{c}_j(t^m)-\uvec{c}_j^0\right)+
\Delta t\beta^{m}\spacestuff(\uvec{c}^0)=O(\Delta^{D+2}). 
\label{intermediateaccuracy}
\end{equation}
Hence, instead of \eqref{firstorderaccuracy}, we can show that the difference of \eqref{firstorderaccuracy} and \eqref{intermediateaccuracy} is an $O(\Delta ^{D+2})$, i.e.,
\begin{align}
C_i \left( \uvec{c}_i(t^m)-\uvec{c}_i^0\right)-\sum_{K\in K_i}\sum_{\uvec {x}_j \in K} \left(\int_K \varphi_i(\uvec{x}) \varphi_j(\uvec{x}) d \uvec{x}\right)\left( \uvec{c}_j(t^m)-\uvec{c}_j^0\right) = O(\Delta^{D+2}).
\label{termtoshowthatisfirstorderaccurate}
\end{align}
By definition of the coefficients $C_i$ in \eqref{dualcell} and the preliminary result \eqref{finalresultvectorial}, we can write 
\begin{align}
\begin{split}
& \norm{C_i \left( \uvec{c}_i(t^m)-\uvec{c}_i^0\right)-\sum_{K\in K_i}\sum_{\uvec {x}_j \in K} \left(\int_K \varphi_i(\uvec{x}) \varphi_j(\uvec{x}) d \uvec{x}\right)\left( \uvec{c}_j(t^m)-\uvec{c}_j^0\right)}_\infty \\
\leq& \hat{C}h\norm{\norm{ \norm{\nabla_{\uvec{x}}(\uvec{u}_{h}(\uvec{x},t^m)-\uvec{u}_{h}(\uvec{x},t^0))}_1 }_{L^\infty(\mathcal{K}_i)}}_\infty  \int_{\mathcal{K}_i} \left\vert \varphi_i(\uvec{x}) \right\vert d\uvec{x},
\end{split}
\label{idontneed0}
\end{align}
where we remark that the internal norms $\norm{\cdot}_1$ and $\norm{\cdot}_{L^\infty(\mathcal{K}_i)}$ are applied componentwise while the external one, $\norm{\cdot}_{\infty}$, is on $\mathbb{R}^Q$.
From a Taylor expansion it is easy to see that
\begin{align}
\nabla_{\uvec{x}}(\uvec{u}_{h}(\uvec{x},t^m)-\uvec{u}_{h}(\uvec{x},t^0))=O(\Delta t).
\end{align}
Moreover, $\int_{\mathcal{K}_i} |\varphi_i(\uvec{x})| d \uvec{x} = O(\Delta^D)$, hence, we have 
\begin{align}
& \norm{C_i \left( \uvec{c}_i(t^m)-\uvec{c}_i^0\right)-\sum_{K\in K_i}\sum_{\uvec {x}_j \in K} \left(\int_K \varphi_i(\uvec{x}) \varphi_j(\uvec{x}) d \uvec{x}\right)\left( \uvec{c}_j(t^m)-\uvec{c}_j^0\right)}_\infty = O(\Delta ^{D+2}).
\end{align}
\end{proof}

Directly from \eqref{EulerDeC}, we can define the explicit low order operator $\lopdt^1:\mathbb{R}^{(I \times Q \times M)}\rightarrow \mathbb{R}^{(I \times Q \times M)}$ as 

\begin{equation}
\lopdt^1(\cund)=\left(\mathcal{L}^1_{\Delta,1}(\cund),\mathcal{L}^1_{\Delta,2}(\cund),\dots,\mathcal{L}^1_{\Delta,I}(\cund)\right), \quad \forall \cund \in \mathbb{R}^{(I \times Q \times M)},
\label{l1DeC}
\end{equation}
where for any $i$ we have
\begin{equation}
\lopdi^1(\cund)=\begin{pmatrix}
                C_i\left( \uvec{c}_i^1-\uvec{c}_i^0\right)+\Delta t \beta^1\spacestuff(\uvec{c}^0)\\
        \vdots\\
        C_i\left( \uvec{c}_i^m-\uvec{c}_i^0\right)+\Delta t \beta^m\spacestuff(\uvec{c}^0)\\
        \vdots\\
                C_i\left( \uvec{c}_i^M-\uvec{c}_i^0\right)+\Delta t \beta^M\spacestuff(\uvec{c}^0)\\
\end{pmatrix}.
\label{l1iDeC}
\end{equation}
in which the convention on the indices of the components of the general argument $\cund\in \mathbb{R}^{(I \times Q \times M)}$ is the same that we had for the operator $\lopdt^2$.

\subsubsection{Proof of the properties of $\lopdt^1$ and $\lopdt^2$}
The operators $\lopdt^1$ and $\lopdt^2$ act from
$X$ to $Y$ with $X=Y=\mathbb{R}^{(I \times Q \times M)}$.
Let us recall again the hypotheses that are needed in order to apply the Deferred Correction method
\begin{itemize}
\item[i)] \textbf{Existence of a solution to $\lopdt^2$} \\
$\exists ! \usoldt \in \mathbb{R}^{(I \times Q \times M)}$ solution of $\lopdt^2$, i.e. such that $\lopdt^2(\usoldt)=\uvec{0}$;
\item[ii)] \textbf{Coercivity-like property of $\lopdt^1$} \\
$\exists \alpha_1 \geq 0$ independent of $\Delta $ s.t. 
\begin{equation}
\norm{\lopdt^1(\underline{\uvec{v}})-\lopdt^1(\underline{\uvec{w}})}_Y \geq \alpha_1\norm{\underline{\uvec{v}}-\underline{\uvec{w}}}_X, \quad \forall \underline{\uvec{v}},\underline{\uvec{w}} \in \mathbb{R}^{(I \times Q \times M)};
\end{equation}
\item[iii)] \textbf{Lipschitz-continuity-like condition of $\lopdt^1-\lopdt^2$} \\
$\exists \alpha_2 \geq 0$ independent of $\Delta $ s.t. 
\begin{equation}
\norm{\left[\lopdt^1(\underline{\uvec{v}})-\lopdt^2(\underline{\uvec{v}})\right]-\left[\lopdt^1(\underline{\uvec{w}})-\lopdt^2(\underline{\uvec{w}})\right]}_Y \leq \alpha_2 \Delta  \norm{\underline{\uvec{v}}-\underline{\uvec{w}}}_X,\quad \forall \underline{\uvec{v}},\underline{\uvec{w}}\in \mathbb{R}^{(I \times Q \times M)}.
\end{equation}
\end{itemize}

We remark that in this context the parameter $\Delta$ is the mesh parameter $h$ and that we assume the temporal step size $\Delta t \leq Ch$ for some fixed constant $C$.

We will not prove the first property, i.e., the existence of a unique solution to $\lopdt^2$, because the proof is identical to the one we had in the ODE case up to the inversion of the mass matrix: from $\lopdt^2$ we can define an operator $\mathcal{J}: \mathbb{R}^{(I \times Q \times M)} \rightarrow \mathbb{R}^{(I \times Q \times M)}$ whose fixed points (if any) are solutions to $\lopdt^2$; further, we can show that for $\Delta$ small enough the operator is a contraction over the space $\mathbb{R}^{(I \times Q \times M)}$ equipped with the infinity norm and, hence, there exists a unique fixed point of $\mathcal{J}$ which is the unique solution to $\lopdt^2$. 

Before going to the proofs of the other two properties, we need to define the norms adopted on the spaces $X$ and $Y$. 
Despite having $X=Y=\mathbb{R}^{(I \times Q \times M)}$ we will equip $X$ and $Y$ with two different norms, differently from what we have done in the ODE case.
We will specify the norms after the following useful observation.
\begin{observation}[Remark on the indices]\label{obs:indices}
The main complication of the proofs is that we have to deal with many indices.
We remind that
\begin{itemize}
\item[•] $i=1,\dots,I$ is referred to the DoFs;
\item[•] $q=1,\dots,Q$ is referred to the components of the approximated solution $\uapp$ to the system of balance laws \eqref{sys};
\item[•] $m=1,\dots,M$ is referred to the subtimenodes $t^m$, even if we remark that we also have an initial subtimenode $t^0=t_n$ in which the quantities are not unknown. 
\end{itemize}
We are already used to the fact that the general element $\underline{\uvec{c}}\in\mathbb{R}^{(I \times Q \times M)}$ must be thought as a collection of $M$ components $\uvec{c}^m\in \mathbb{R}^{(I \times Q)}$ $m=1,\dots,M$. Each component $\uvec{c}^m$ can be thought as the vector of the coefficients of a vectorial continuous piecewise polynomial function $\uvec{u}_h(\uvec{x},t)=\sum_{i=1}^I\uvec{c}_i(t)\varphi_i (\uvec{x})$ evaluated in the subtimenode $t^m$.
In fact, each $\uvec{c}^m$ is made by $I$ components $\uvec{c}^m_i\in \mathbb{R}^{Q}$ with  $i=1,\dots,I$ associated to the DoFs.
Finally, each $\uvec{c}^m_i$ is made by $Q$ components $c^{q,m}_i$ $q=1,\dots,Q$, scalar coefficients associated to the components of the solution to the system of PDEs that we would like to solve, i.e.,
\begin{equation}
\underline{\uvec{c}}=\left(
   \begin{array}{ccc}
   \uvec{c}^1\\
   \vdots\\
   \uvec{c}^m\\
   \vdots\\
   \uvec{c}^M
   \end{array}
\right)\in\mathbb{R}^{(I \times Q \times M)}, \quad
\uvec{c}^m=\left(
   \begin{array}{ccc}
   \uvec{c}^m_1\\
   \vdots\\
   \uvec{c}^m_i\\
   \vdots\\
   \uvec{c}^m_I
   \end{array}
\right)\in\mathbb{R}^{(I \times Q)}, \quad 
\uvec{c}^m_i=\left(
   \begin{array}{ccc}
   c^{1,m}_i\\
   \vdots\\
   c^{q,m}_i\\
   \vdots\\
   c^{Q,m}_i
   \end{array}
\right)\in\mathbb{R}^Q.
\label{eq:c_coefficients_sketch}
\end{equation}

\end{observation}
In the proofs, we are going to focus on a single scalar component $q=1,\dots,Q$ of a single subtimenode $m=1,\dots,M$ and our results will be uniform with respect to the indices $q$ and $m$, so we will be able to pass from the scalar results to the desired vectorial results through an infinity norm $\norm{\cdot}_\infty$ on $\mathbb{R}^{(Q \times M)}$, similarly to what we did when we passed from \eqref{finalresultscalar} to \eqref{finalresultvectorial} in the preliminary results.
Therefore, the norm that we choose for the single component of $\underline{\uvec{c}}\in X=\mathbb{R}^{(I \times Q \times M)}$ with fixed indices $q$ and $m$, denoted by  ${\uvec{c}}^{q,m}\in \mathbb{R}^{I}$, is the $W^{1,1}_I(\Omega)$-norm, a discrete version of the classical $W^{1,1}(\Omega)$-norm. In particular, on a scalar function $u:\Omega \to \R$ the $W^{1,1}(\Omega)$-norm is defined as
\begin{equation}\label{eq:normW11}
	\norm{u}_{W^{1,1}(\Omega)}:=\norm{u}_{L^1(\Omega)}+\sum_{d=1}^D\norm{\frac{\partial}{\partial x_d}u}_{L^1(\Omega)}=\norm{u}_{L^1(\Omega)}+\norm{\norm{\nabla_{\uvec{x}}u}_1}_{L^1(\Omega)},
\end{equation}
from which we define the corresponding discrete norm on $\R^{I}$, defined by $\norm{\cdot}_{W^{1,1}_I(\Omega)} : \R^I\to \R^+_0$ as 
\begin{equation}
\norm{{\uvec{c}}^{q,m}}_{W^{1,1}_I(\Omega)}:= \norm{ \sum_{i=1}^I c_i^{q,m} \varphi_i }_{W^{1,1}(\Omega)}.
\label{abuseofnotationnorm}
\end{equation}
Using then the classical infinity norm on the space $\R^{Q\times M}$ defined by $\norm{\cdot}_{\infty,Q,M}:\R^{Q\times M} \to \R^+_0$, we introduce the $X$ norm $\norm{\cdot}_X:\R^{I\times Q \times M}\to \R^+_0$ as 
\begin{equation}
	\norm{\underline{\uvec{c}}}_X:=\norm{\left\lbrace \norm{{\uvec{c}}^{q,m}}_{W^{1,1}_I(\Omega)} \right\rbrace_{\substack{q=1,\dots,Q\\m=1,\dots,M}} }_{\infty,Q,M}.
	\label{Xnorm}
\end{equation}

Instead, we equip $Y$ with a different norm; we choose for the single component ${\tilde{\uvec{c}}}^{q,m}\in \mathbb{R}^I$ with fixed indices $q$ and $m$ of $\underline{\tilde{\uvec{c}}} \in Y=\mathbb{R}^{(I \times Q \times M)}$ the 1-norm  $\norm{\cdot}_{1,I}:\R^{I}\to \R^+_0$ defined as 
\begin{equation}
	\norm{{\tilde{\uvec{c}}}^{q,m}}_{1,I}:=\sum_{i=1}^I |\tilde{c}_i^{q,m}|,
\end{equation}
then the norm on the whole space $Y=\mathbb{R}^{(I \times Q \times M)}$, $\norm{\cdot}_Y:\R^{(I \times Q \times M)}\to \R^+_0$, is defined by
\begin{equation} 
\norm{\underline{\uvec{\tilde{c}}}}_Y=\norm{\left\lbrace\norm{{\tilde{\uvec{c}}}^{q,m}}_{1,I} \right \rbrace_{\substack{q=1,\dots,Q\\m=1,\dots,M}} }_{\infty,Q,M}.
\label{Ynorm}
\end{equation}
\begin{observation}
We remark that the initial subtimenode $m=0$ is not kept into account in the norms \eqref{Xnorm} and \eqref{Ynorm} as it is a datum of the problem.
\end{observation}

\begin{observation}[On the choice of the norms]\label{obs:norms}
The reason of the difference in the norms assumed on $X$ and $Y$ is intuitively due to the following fact.
Practically speaking, the elements of $X$, the arguments of $\lopdt^2$ and $\lopdt^1$ given respectively by \eqref{l2DeC} and \eqref{l1DeC} (and so by \eqref{l2iDeC} and \eqref{l1iDeC}), are the coefficients associated to a vectorial continuous piecewise polynomial function evaluated in the subtimenodes $t^m$ $m=1,\dots,M$.
Therefore, on the space $X$ we take an integral norm for ``functions''.
Instead, the elements of the space $Y$, the images of $\lopdt^2$ and $\lopdt^1$, are consistent with integrals of the mentioned function associated to the coefficients.
In order to guarantee the consistency of the terms in the inequalities to prove and to compare $\norm{\cdot}_X$ and $\norm{\cdot}_Y$, we must take for $Y$ a norm which does not modify the integral ``character'' of the components of the elements of the space.
\end{observation}
It is straightforward to prove that \eqref{Xnorm} and \eqref{Ynorm} are norms but we will not do it for the sake of brevity. In the context of the proofs of the properties of $\lopdt^2$ and $\lopdt^1$, we are going to make use of the two following regularity assumptions.
\begin{assumption}[Poincaré-like inequality] \label{ass:poincare}
	We assume that we are working with coefficients regular enough to guarantee that the associated functions $\uvec{g}_h$, for some $C_p\geq 0$ independent of $\Delta$, are such that
	\begin{equation}
		\norm{\uvec{g}_h}_{W^{1,1}(\Omega)}\leq C_p \norm{\uvec{g}_h}_{L^{1}(\Omega)},
		\label{poincare}
	\end{equation}
	i.e., we assume that we can control the norm of the gradient of all functions that we will consider with the norm of the functions.
\end{assumption}
\begin{assumption}[Smoothness of the space residuals] \label{ass:smooth_residual}
We assume the functions $\spacestuff$ defined in \eqref{spst} to be smooth.
\end{assumption}

Finally, the notation in \cref{eq:c_coefficients_sketch} will hold for two generic vectors $\underline{\uvec{v}},\underline{\uvec{w}}\in \mathbb{R}^{(I \times Q \times M)}$ that will be used in the proof.

In order to deal with the single component got for fixed $m$ and $q$, as we are going to do in a few lines, it is very useful to define here the scalar continuous piecewise polynomial functions 
\begin{equation}
v^{q,m}_{h}(\uvec{x})=\sum_{i=1}^I v^{q,m}_{i}\varphi_i(\uvec{x}), \quad w^{q,m}_{h}(\uvec{x})=\sum_{i=1}^I w^{q,m}_{i}\varphi_i(\uvec{x})
\label{scalarfunctions}
\end{equation}
associated to the scalar coefficients $v^{q,m}_{i}$ and $w^{q,m}_{i}$ $i=1,\dots,I$.

Now, we have all the elements that we need in order to handle the proofs of the properties of the two operators.

\begin{proposition}[Coercivity-like property of $\lopdt^1$]
Let $\lopdt^1:X\to Y$ be the operator defined in \eqref{l1DeC} and \eqref{l1iDeC}, $\underline{\uvec{v}},\underline{\uvec{w}}\in  X$ and suppose that \cref{ass:poincare} holds, then $\exists \alpha_1 > 0$ independent of $\Delta $ s.t. 
\begin{equation}\label{eq:coercivity}
		\norm{\lopdt^1(\underline{\uvec{v}})-\lopdt^1(\underline{\uvec{w}})}_Y \geq \alpha_1\norm{\underline{\uvec{v}}-\underline{\uvec{w}}}_X, \quad \forall \underline{\uvec{v}},\underline{\uvec{w}} \in \mathbb{R}^{(I \times Q \times M)}.
	\end{equation}
\end{proposition}
\begin{proof}
From a direct computation we have, for every $i=1,\dots ,I$, $m=1,\dots,M$ and $q=1,\dots, Q$, that
\begin{equation}
	\begin{split}
\lopdi^{1,q,m}(\underline{\uvec{v}})-\lopdi^{1,q,m}(\underline{\uvec{w}})
&=   C_i\left( v_i^{q,m}\!-c_i^{0,q}\right)+\Delta t \beta^m\spacestuff^q(\uvec{c}^0)
-      C_i\left( w_i^{q,m}\!-c_i^{0,q}\right)-\Delta t \beta^m\spacestuff^q(\uvec{c}^0)\\
&=
        C_i\left(v_i^{q,m}\!-w_i^{q,m}\right).
	\end{split}\label{l1decsub}
\end{equation}

We remark again that $\uvec{c}^0$ is known and so also $\uvec{c}_i^0$.
We will start by proving the coercivity-like property for a fixed component $q$ and a fixed subtimenode $m$, i.e., we will prove that the $1$-norm of \eqref{l1decsub} over the indexes $i=1,\dots,I$ is such that
\begin{align}
\norm{\lopdt^{1,q,m}(\underline{\uvec{v}})-\lopdt^{1,q,m}(\underline{\uvec{w}})}_{1,I}\geq \alpha_1 \norm{v^{q,m}_h-w^{q,m}_h}_{W^{1,1}(\Omega)}
\end{align}
for some $\alpha_1$ independent of $\Delta$ for all $m$ and $q$.
Recalling the definition \eqref{dualcell} of the coefficients $C_i=\int_{\Omega} \varphi_i(\uvec{x}) d \uvec{x}$ and the fact that the Bernstein basis functions are nonnegative, we have 
\begin{align}
\norm{\lopdt^{1,q,m}(\underline{\uvec{v}})-\lopdt^{1,q,m}(\underline{\uvec{w}})}_{1,I}&=\sum_{i=1}^I \vert C_i\left( v_i^{q,m}-w_i^{q,m}\right) \vert=\sum_{i=1}^I  \int_\Omega \left \vert (v_i^{q,m}-w_i^{q,m}) \varphi_i(\uvec{x})\right \vert d\uvec{x}.
\label{suburra1}
\end{align}
Using the triangular inequality and recalling the definition \eqref{scalarfunctions} of the scalar continuous piecewise polynomial functions $v^{q,m}_h$ and $w^{q,m}_h$, from the previous equation we get 
\begin{align}
\begin{split}
\norm{\lopdt^{1,q,m}(\underline{\uvec{v}})-\lopdt^{1,q,m}(\underline{\uvec{w}})}_{1,I}
\geq &\int_\Omega  \left\lvert \sum_{i=1}^I (v_i^{q,m}-w_i^{q,m})\varphi_i(\uvec{x}) \right\rvert d\uvec{x}\\
=& \int_\Omega \left\lvert v^{q,m}_h(\uvec{x})-w^{q,m}_h(\uvec{x}) \right\rvert d\uvec{x}=\norm{v^{q,m}_h-w^{q,m}_h }_{L^1(\Omega)}\\
\geq &\frac{1}{C_p} \norm{v^{q,m}_h-w^{q,m}_h}_{W^{1,1}(\Omega)}=\alpha_1 \norm{v^{q,m}_h-w^{q,m}_h}_{W^{1,1}(\Omega)},
\end{split}
\label{eq:poincare_step}
\end{align}
where, in the last inequality, we used the Poincaré-like inequality \eqref{poincare} and
$\alpha_1=\frac{1}{C_p}$ with $C_p$ independent of $\Delta$, which is the intermediate result that we wanted to show.

In order to get the final result, it suffices to observe that the previous inequality is uniform with respect to the indices $q$ and $m$, so, we can take the infinity norm on these indices of both the sides and get
\begin{align}
\norm{\lopdt^{1}(\underline{\uvec{v}})-\lopdt^{1}(\underline{\uvec{w}})}_{Y} \geq \alpha_1  \norm{\underline{\uvec{v}}-\underline{\uvec{w}}}_{X}
\end{align}
using the definitions \eqref{Xnorm} and \eqref{Ynorm}.
\end{proof}

\begin{proposition}[Lipschitz-continuity-like condition of $\lopdt^1-\lopdt^2$]
Let $\lopdt^1,\lopdt^2:X\to Y$ the operators defined in \eqref{l1DeC} and \eqref{l2DeC}. Consider $\underline{\uvec{v}},\underline{\uvec{w}}\in  X$ regular enough and suppose that \cref{ass:smooth_residual} holds. Then, $\exists \alpha_2 > 0$ independent of $\Delta $ s.t.
	\begin{align}
		\norm{\left[\lopdt^1(\underline{\uvec{v}})-\lopdt^2(\underline{\uvec{v}})\right]-\left[\lopdt^1(\underline{\uvec{w}})-\lopdt^2(\underline{\uvec{w}})\right]}_Y \leq \alpha_2 \Delta  \norm{\underline{\uvec{v}}-\underline{\uvec{w}}}_X.
	\end{align}
\end{proposition}
\begin{proof}
Focusing on one DoF $i\in \{1,\dots,I\}$ and on one subtimenode $m\in\{1,\dots,M\}$, we have 
\begin{equation}
	\begin{split}
&\left[ \mathcal{L}^{1,m}_{\Delta,i}(\underline{\uvec{v}})- \mathcal{L}^{2,m}_{\Delta,i}(\underline{\uvec{v}}) \right]- \left[ \mathcal{L}^{1,m}_{\Delta,i}(\underline{\uvec{w}})- \mathcal{L}^{2,m}_{\Delta,i}(\underline{\uvec{w}}) \right]=\\
&  C_i\left( \uvec{v}_i^m-\uvec{w}_i^m\right)-\sum_{K\in K_i}\sum_{\uvec {x}_j \in K} \left( \uvec{v}_j^m-\uvec{w}_j^m\right) \int_K \varphi_i(\uvec{x}) \varphi_j(\uvec{x}) d \uvec{x}-\Delta t \sum_{\ell=0}^{M} \theta^m_\ell \left[\spacestuff(\uvec{v}^\ell)-\spacestuff(\uvec{w}^\ell) \right]. 
	\end{split}\label{l2declip}
\end{equation}

Just like we did when we proved the coercivity-like property of $\lopdt^1$, we will work on the single component of \eqref{l2declip} for fixed $q=1,\dots,Q$ and $m=1,\dots,M$, then we will derive the final result on the norms of $X$ and $Y$ by considering the $\infty$-norm over the indices $q$ and $m$.

Let us thus focus on
\begin{align}
	\begin{split}
&\left[ \mathcal{L}^{1,q,m}_{\Delta,i}(\underline{\uvec{v}})- \mathcal{L}^{2,q,m}_{\Delta,i}(\underline{\uvec{v}}) \right]- \left[ \mathcal{L}^{1,q,m}_{\Delta,i}(\underline{\uvec{w}})- \mathcal{L}^{2,q,m}_{\Delta,i}(\underline{\uvec{w}}) \right] =C_i\left( v_i^{q,m}-w_i^{q,m}\right)\\
-&\sum_{K\in K_i}\sum_{\uvec {x}_j \in K} \left( v_j^{q,m}-w_j^{q,m}\right) \int_K \varphi_i(\uvec{x}) \varphi_j(\uvec{x}) d \uvec{x} -\Delta t \sum_{\ell=0}^{M} \theta^m_\ell \left[\phi_i^q(\uvec{v}^\ell)-\phi_i^q(\uvec{w}^\ell) \right]
	\end{split}
	\label{nomeacase}
\end{align}
where $\phi_i^q(\cdot)$ represents the $q$-th component of the space residual $\spacestuff(\cdot).$
We want to show now that the $1$-norm, over all the indices $i$, of \eqref{nomeacase}, for fixed $q$ and $m$, is such that
\begin{align}\label{eq:lipschitz}
&\norm{\left[ \mathcal{L}^{1,q,m}_{\Delta}(\underline{\uvec{v}})- \mathcal{L}^{2,q,m}_{\Delta}(\underline{\uvec{v}}) \right]- \left[ \mathcal{L}^{1,q,m}_{\Delta}(\underline{\uvec{w}})- \mathcal{L}^{2,q,m}_{\Delta}(\underline{\uvec{w}}) \right]}_{1,I}\leq \alpha_2 \Delta  \norm{\underline{\uvec{v}}-\underline{\uvec{w}}}_{X},
\end{align}
for some $\alpha_2$ independent of $\Delta$, from which we will get the final result by taking the infinity norm of the left hand side with respect to the indices $q$ and $m$.
Thanks to the triangular inequality we have
\begin{align}
&\norm{\left[ \mathcal{L}^{1,q,m}_{\Delta}(\underline{\uvec{v}})- \mathcal{L}^{2,q,m}_{\Delta}(\underline{\uvec{v}}) \right]- \left[ \mathcal{L}^{1,q,m}_{\Delta}(\underline{\uvec{w}})- \mathcal{L}^{2,q,m}_{\Delta}(\underline{\uvec{w}}) \right]}_{1,I}\\
&=\sum_{i=1}^I \left\vert \left[ \mathcal{L}^{1,q,m}_{\Delta,i}(\underline{\uvec{v}})- \mathcal{L}^{2,q,m}_{\Delta,i}(\underline{\uvec{v}}) \right]- \left[ \mathcal{L}^{1,q,m}_{\Delta,i}(\underline{\uvec{w}})- \mathcal{L}^{2,q,m}_{\Delta,i}(\underline{\uvec{w}}) \right] \right\vert\\
&\leq \underbrace{\sum_{i=1}^I \left\vert\sum_{K\in K_i}\left( v_i^{q,m}-w_i^{q,m}\right) \int_K \varphi_i(\uvec{x})d\uvec{x}-\sum_{K\in K_i}\sum_{\uvec {x}_j \in K} \left( v_j^{q,m}-w_j^{q,m}\right)\int_K \varphi_i(\uvec{x}) \varphi_j(\uvec{x}) d \uvec{x}\right\vert}_{=:E_1} \label{useful1}\\
&+\underbrace{\sum_{i=1}^I \left\vert\Delta t \sum_{\ell=0}^{M} \theta^m_\ell \left[\phi^q_i(\uvec{v}^\ell)-\phi^q_i(\uvec{w}^\ell) \right]\right\vert}_{=:E_2}
\label{useful2}
\end{align}
recalling the definition of $C_i=\sum_{K \in K_i} \int_K \varphi(\uvec{x}) d\uvec{x}$ in \eqref{dualcell}.

Thanks to the previous inequality, we can deal separately with the two terms of the right hand side, the first one \eqref{useful1} concerning the mass matrix and the second one \eqref{useful2} involving the space residuals, and show that they can be bounded in the following way
\begin{align}
&E_1\leq C_a \Delta \norm{\underline{\uvec{v}}-\underline{\uvec{w}}}_{X},
\label{final1}\\
&
E_2\leq C_b \Delta \norm{\underline{\uvec{v}}-\underline{\uvec{w}}}_{X},
\label{final2}
\end{align}

with $C_a$ and $C_b$ independent of $\Delta$ which would give us the desired result. \\[1em]
$\bullet$ \textbf{First term concerning the mass matrix}\\
In order to bound this term, we can directly apply the preliminary result in \cref{finalresult} and we get
\begin{align}
&\left\vert \sum_{K\in K_i}  \left( v_i^{q,m}-w_i^{q,m}\right)\int_K \varphi_i(\uvec{x})d\uvec{x} -\sum_{K\in K_i}\sum_{\uvec{x}_j\in K} \left( v_j^{q,m}-w_j^{q,m}\right) \int_K \varphi_i(\uvec{x})\varphi_j(\uvec{x})d\uvec{x}   \right\vert \nonumber\\
&\leq \hat{C}\Delta C_i \norm{ \norm{\nabla_{\uvec{x}}\left( v_h^{q,m}-w_h^{q,m}\right)}_1 }_{L^\infty(\mathcal{K}_i)}, \quad \forall i=1,\dots,I,
\label{randomname1}
\end{align}
with $\hat{C}$ independent of the mesh parameter $\Delta=h$, dependent just on the number of dimensions $D$, on the degree $M$ and on the type of the elements in the mesh.
From \eqref{randomname1} we have
\begin{align}
E_1=&\sum_{i=1}^I \left\vert \sum_{K\in K_i}  \left( v_i^{q,m}-w_i^{q,m}\right)\int_K \varphi_i(\uvec{x})d\uvec{x} -\sum_{K\in K_i}\sum_{\uvec{x}_j\in K} \left( v_j^{q,m}-w_j^{q,m}\right) \int_K \varphi_i(\uvec{x})\varphi_j(\uvec{x})d\uvec{x}   \right\vert \nonumber \\
&\leq \hat{C}\Delta  \sum_{i=1}^I \norm{ \norm{\nabla_{\uvec{x}}\left( v_h^{q,m}-w_h^{q,m}\right)}_1 }_{L^\infty(\mathcal{K}_i)} C_i .
\label{randomname2}
\end{align}

Thanks to \cref{lem:equivalence_norm1_norminf} taking $z=\norm{v_h^{q,m}-w_h^{q,m}}_1$, then \eqref{randomname2} can be bounded in the following way
\begin{equation}
\hat{C}\Delta \sum_{i=1}^I\norm{ \norm{\nabla_{\uvec{x}}\left( v_h^{q,m}-w_h^{q,m}\right)}_1 }_{L^\infty(\mathcal{K}_i)}  C_i\leq \hat{C}\Delta \tilde{C}^* \norm{ \norm{\nabla_{\uvec{x}}\left( v_h^{q,m}-w_h^{q,m}\right)}_1 }_{L^{1}(\Omega)}.
\end{equation}
Hence, by definition of the $W^{1,1}(\Omega)$-norm \eqref{eq:normW11}, of the $W^{1,1}_I(\Omega)$-norm \eqref{abuseofnotationnorm} and of the $X$ norm \eqref{Xnorm}, we have
\begin{align}
\begin{split}
	E_1 &\leq 
\hat{C}\Delta \tilde{C}^* \norm{ \norm{\nabla_{\uvec{x}}\left( v_h^{q,m}-w_h^{q,m}\right)}_1 }_{L^{1}(\Omega)}\leq \hat{C}\tilde{C}^*\Delta  \norm{ v_h^{q,m}-w_h^{q,m} }_{W^{1,1}(\Omega)} \\
&\leq  C_a \Delta  \norm{ v_h^{q,m}-w_h^{q,m} }_{W^{1,1}(\Omega)} \leq C_a \Delta \norm{ \underline{\uvec{v}}-\underline{\uvec{w}} }_X,
\end{split}
\end{align}
with $C_a=\hat{C}\tilde{C}^*$ independent of $\Delta$.\\[1em]

$\bullet$ \textbf{Second term involving the space residuals}\\
By applying the triangular inequality, recalling that $\theta^m_\ell$ are fixed normalized constant coefficients, thus, bounded in absolute value by a positive constant $C_\theta$, and that $\Delta t\leq Ch=C\Delta$ for some fixed constant $C$, we have
\begin{align}
E_2=\sum_{i=1}^I \left\vert\Delta t \sum_{\ell=0}^{M} \theta^m_\ell \left[\phi^q_i(\uvec{v}^\ell)-\phi^q_i(\uvec{w}^\ell) \right]\right\vert\leq \Delta C C_\theta \sum_{i=1}^I  \sum_{\ell=0}^{M} \abs{\phi^q_i(\uvec{v}^\ell)-\phi^q_i(\uvec{w}^\ell) }.
\end{align}

From the fact that $\uvec{v}^0=\uvec{w}^0=\uvec{c}^0$, we have
\begin{align}
\Delta C C_\theta \sum_{i=1}^I  \sum_{\ell=0}^{M} \abs{\phi^q_i(\uvec{v}^\ell)-\phi^q_i(\uvec{w}^\ell) }
\leq
\Delta C C_\theta M \sum_{i=1}^I  \norm{ \left\lbrace \phi^q_i(\uvec{v}^m)-\phi^q_i(\uvec{w}^m)   \right\rbrace_{\substack{q=1,\dots,Q\\m=1,\dots,M}} }_{\infty,Q,M}.
\label{tobecontinued1}
\end{align}
Then, we use the assumption of smoothness of the space residuals $\spacestuff(\cdot)$. In particular, we assume the following Lipschitz-continuity-like condition
\begin{equation}
\sum_{i=1}^I  \norm{ \left\lbrace \phi^q_i(\uvec{v}^m)-\phi^q_i(\uvec{w}^m)   \right\rbrace_{\substack{q=1,\dots,Q\\m=1,\dots,M}} }_{\infty,Q,M}\leq C_\phi\norm{ \norm{ \underline{\uvec{v}}_h-\underline{\uvec{w}}_h }_{W^{1,1}(\Omega)} }_{\infty,Q,M}=C_\phi\norm{\underline{\uvec{v}} -\underline{\uvec{w}} }_X
\label{otherassumption}
\end{equation}
with $C_\phi$ independent of $\Delta$.
Using this, from \eqref{tobecontinued1} we get
\begin{align}
E_2 \leq C_b \Delta \norm{\underline{\uvec{v}} -\underline{\uvec{w}} }_X
\end{align}
with $C_b=C C_\theta M C_\phi$ independent of $\Delta$, obtaining \eqref{final2}.

Now, that we have proven \eqref{final1} and \eqref{final2}, the Lipschitz inequality \eqref{eq:lipschitz} is proven with $\alpha_2=C_a+C_b$ independent of $\Delta$.
Finally, we get the final result by observing that what we have proved holds for any component with fixed indices $q=1,\dots,Q$ and $m=1,\dots,M$. So, applying the infinity norm of the left hand side with respect to these indices, we get
\begin{align}
\begin{split}
\max_{q,m} \norm{\left[ \mathcal{L}^{1,q,m}_{\Delta}(\underline{\uvec{v}})- \mathcal{L}^{2,q,m}_{\Delta}(\underline{\uvec{v}}) \right]- \left[ \mathcal{L}^{1,q,m}_{\Delta}(\underline{\uvec{w}})- \mathcal{L}^{2,q,m}_{\Delta}(\underline{\uvec{w}}) \right]}_{1,I}&\\
=\norm{\left[ \mathcal{L}^{1}_{\Delta}(\underline{\uvec{v}})- \mathcal{L}^{2}_{\Delta}(\underline{\uvec{v}}) \right]- \left[ \mathcal{L}^{1}_{\Delta}(\underline{\uvec{w}})- \mathcal{L}^{2}_{\Delta}(\underline{\uvec{w}}) \right]}_{Y}
&\leq\alpha_2 \Delta \norm{\underline{\uvec{v}} -\underline{\uvec{w}} }_X,
\end{split}
\end{align}
which is the thesis.
\end{proof}

\subsection{Issues with the DeC for CG}
We discuss here a negative result seen in the numerical tests even on the monodimensional linear advection equation (LAE) reported in the main document and in many other works, e.g. \cite{Decremi,michel2022spectral,abgrall2019high}. 
The DeC formulation for PDEs with the lumping of the mass matrix does not give the expected formal order of accuracy for space discretizations of order higher than or equal to $4$ if one performs the theoretical optimal number of iterations. 
In this section, we will try to investigate the problem by numerically assessing the impact of the number of iterations $P$, of the CFL and of the CIP stabilization on higher order derivatives.
Before starting, we remark that the loss in the accuracy is not registered in the context of steady problems, indeed, in \cite{abgrall2020high} the expected order of accuracy is obtained with B3 on a nontrivial steady test for the bidimensional Euler equations. 
Further, one of the authors is involved in a project \cite{nonsubmitted} on some novel CIP stablizations for the monodimensional SW equations, soon to be submitted, in which the right order of accuracy is obtained for P3, B3 and B4 with the theoretical optimal number of iterations on all the considered steady tests. 
Therefore, we will focus on the same unsteady test for the monodimensional LAE presented in the main document and, in particular, we will consider P3, B3 and B4 as basis functions and the original formulation of the bDeC for PDEs without interpolations between the iterations as timestepping method. For P3 and B3 we will use, in the context of the CIP stabilization, the same coefficients adopted in the main document, $\delta^{\CIP}=0.00702$. As the optimal coefficient for B4 is not provided in \cite{michel2021spectral}, we will adopt the same coefficient as for B3 and P3.
The reference CFL adopted for the tests with B3 and P3 is 0.1, instead, with B4 it is 0.05. Where not specified, such values have been adopted.

\subsubsection{Impact of the number of iterations}
The numerical results for different number of iterations are reported in  \cref{fig:ITER}.
In all the cases we can see the same trend: the optimal number of iterations gives order 2, increasing the number of iterations improves the accuracy allowing to reach the formal order. Nevertheless, it is important to notice that many more iterations, with respect to the optimal number, are needed in order to achieve the right order of convergence: 10 for P3, 80 for B3, 320 for B4.

\begin{figure}
\centering
	\begin{subfigure}{0.49\textwidth}
		\includegraphics[width=\textwidth]{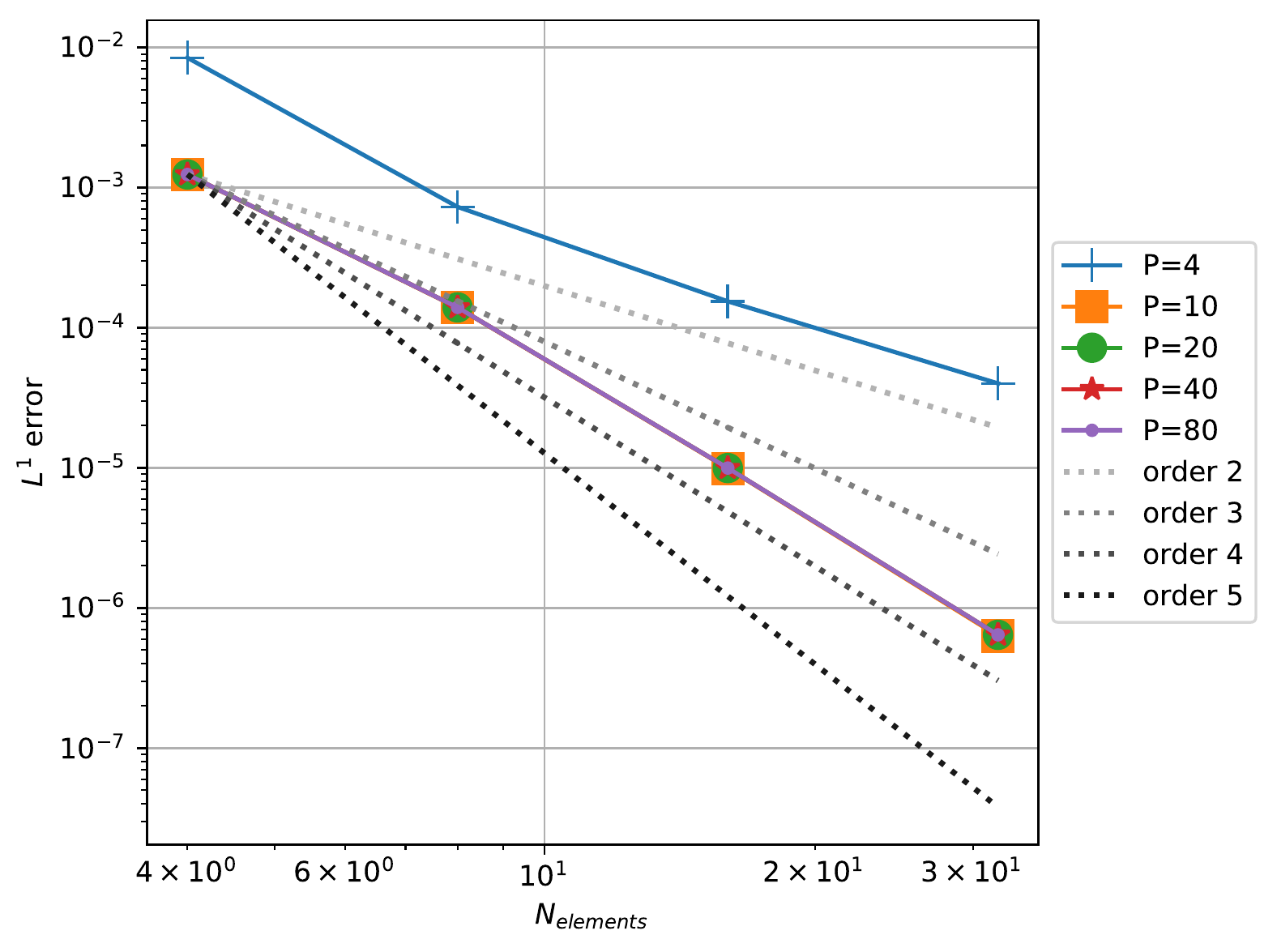}\caption{P3}\label{subfig:P3ITER}
	\end{subfigure}
	\begin{subfigure}{0.49\textwidth}
	\includegraphics[width=\textwidth]{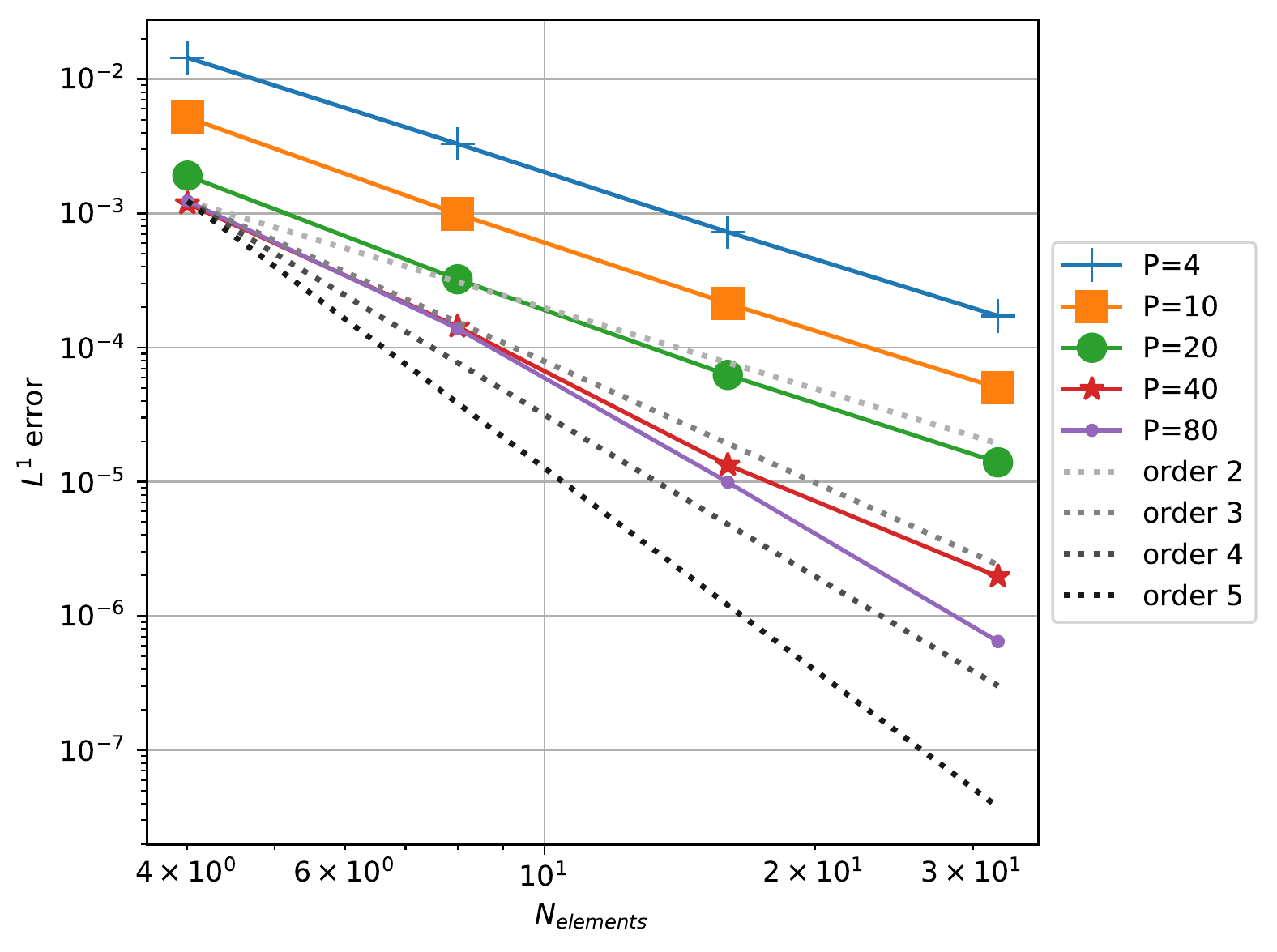}\caption{B3}\label{subfig:B3ITER}
\end{subfigure}\\
	\begin{subfigure}{0.49\textwidth}
		\includegraphics[width=\textwidth]{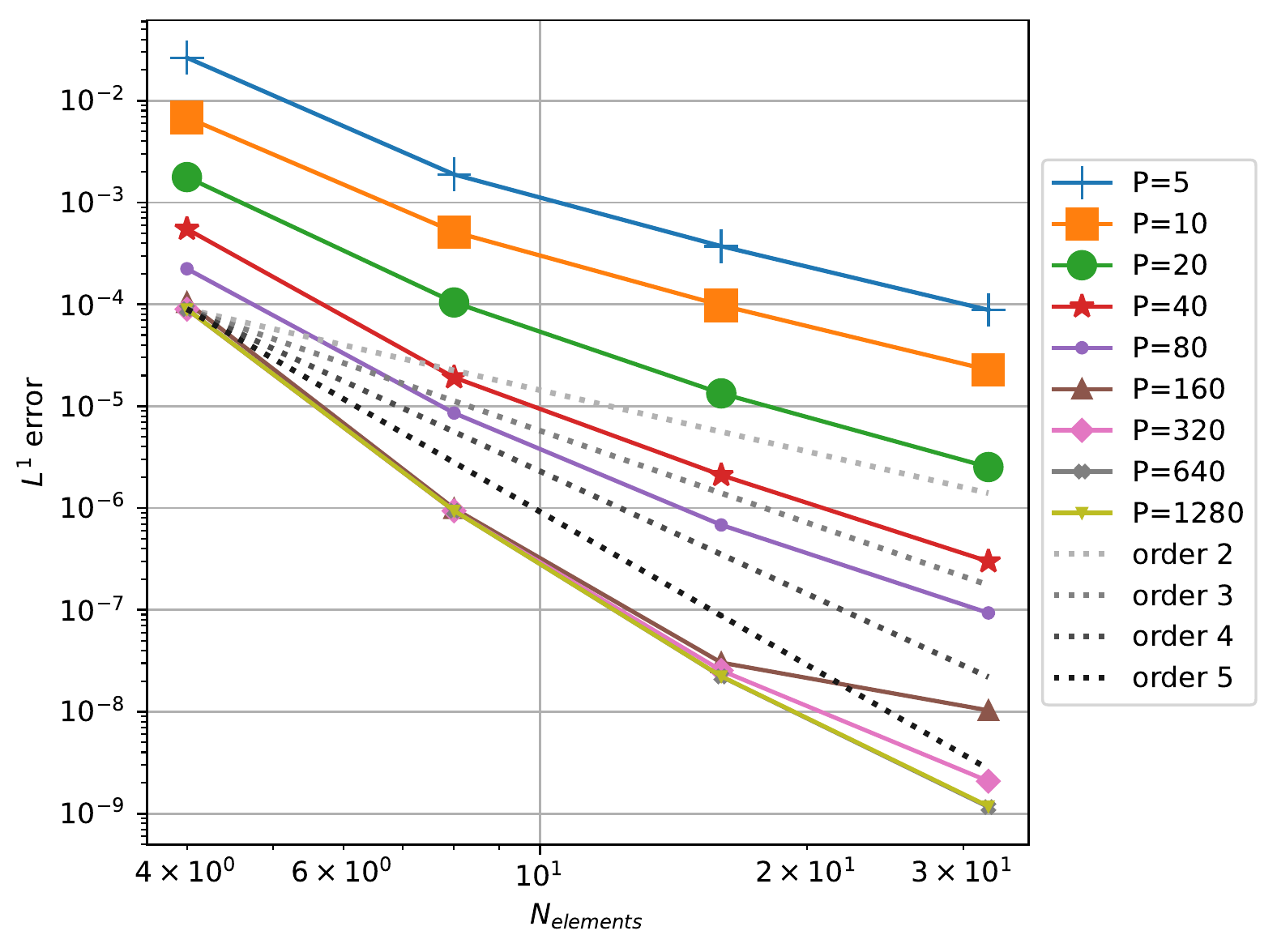}\caption{B4}\label{subfig:B4ITER}
	\end{subfigure}
\caption{1D LAE: tests with different numbers of iterations}\label{fig:ITER}
\end{figure}

\subsubsection{Impact of the CFL}
The numerical results for different values of the CFL are reported in  \cref{fig:CFL}.
Such parameter seems not to have impact on the order. For P3, $\CFL=0.1$ performs better than $\CFL=0.01$ and $\CFL=0.001$; for the other basis functions one gets similar results for the different values of the $\CFL$ meaning that spatial error is dominating with respect to the error in time.

\begin{figure}
\centering
	\begin{subfigure}{0.49\textwidth}
		\includegraphics[width=\textwidth]{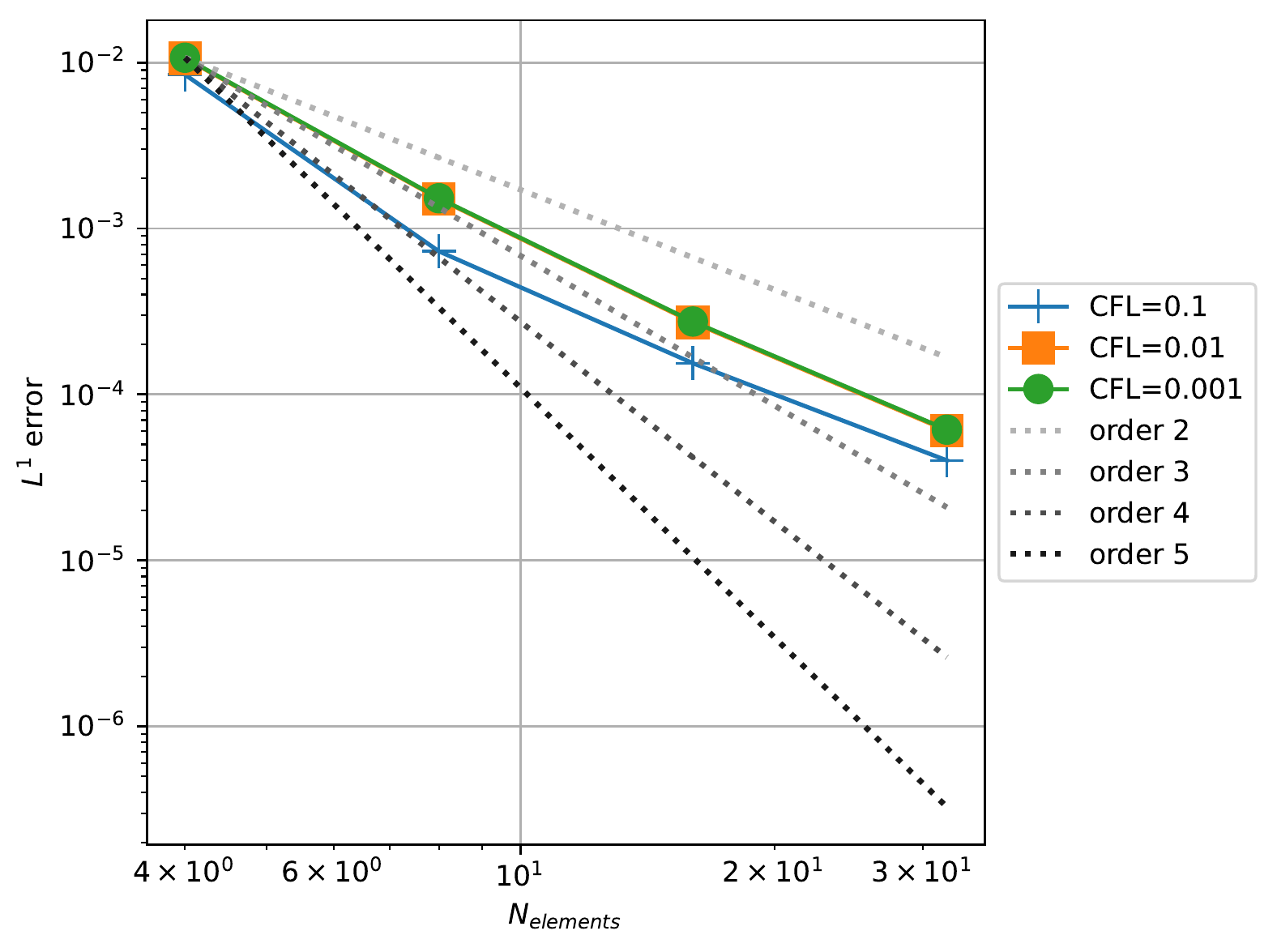}\caption{P3}\label{subfig:P3CFL}
	\end{subfigure}
	\begin{subfigure}{0.49\textwidth}
	\includegraphics[width=\textwidth]{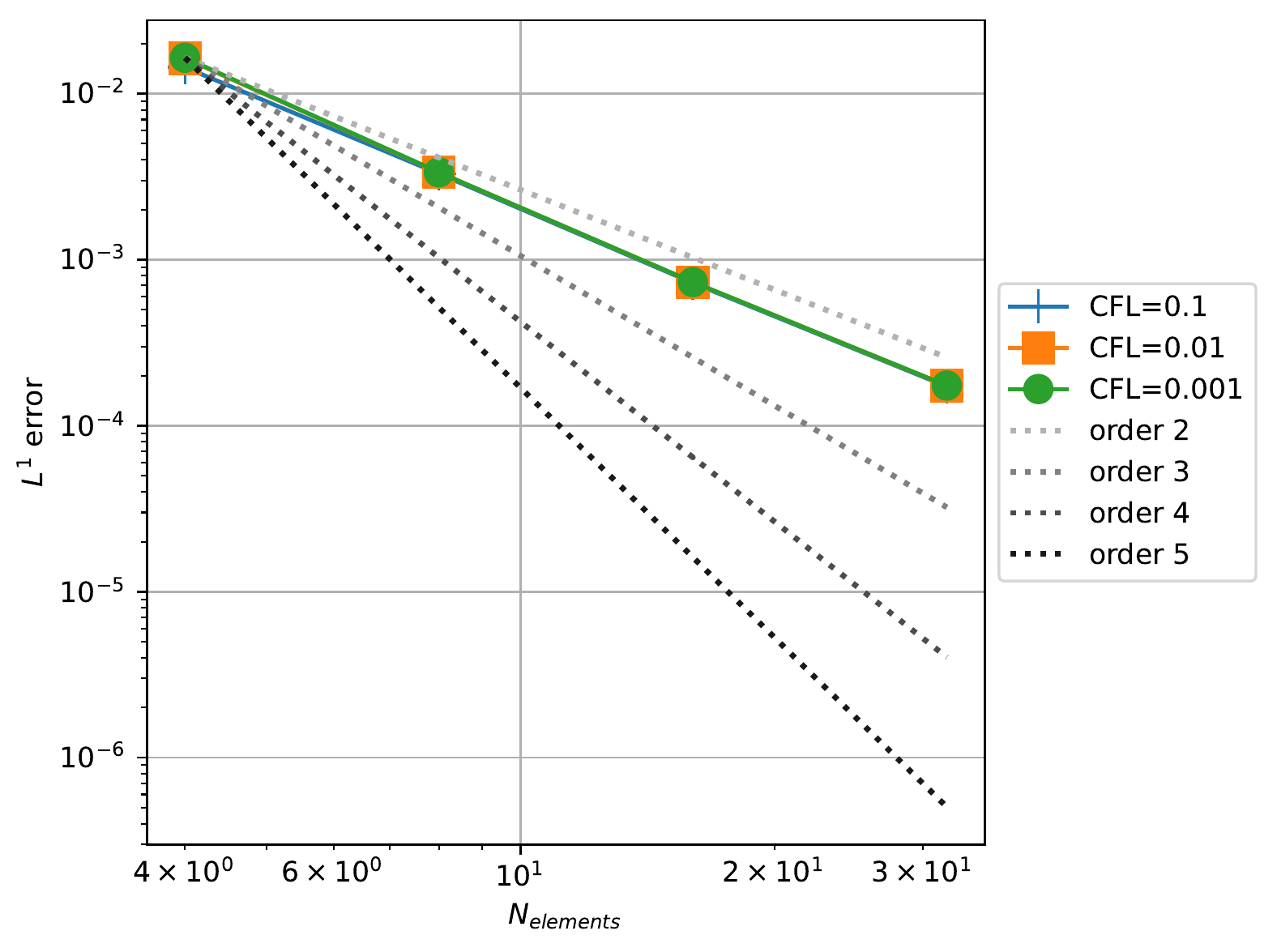}\caption{B3}\label{subfig:B3CFL}
\end{subfigure}\\
	\begin{subfigure}{0.49\textwidth}
		\includegraphics[width=\textwidth]{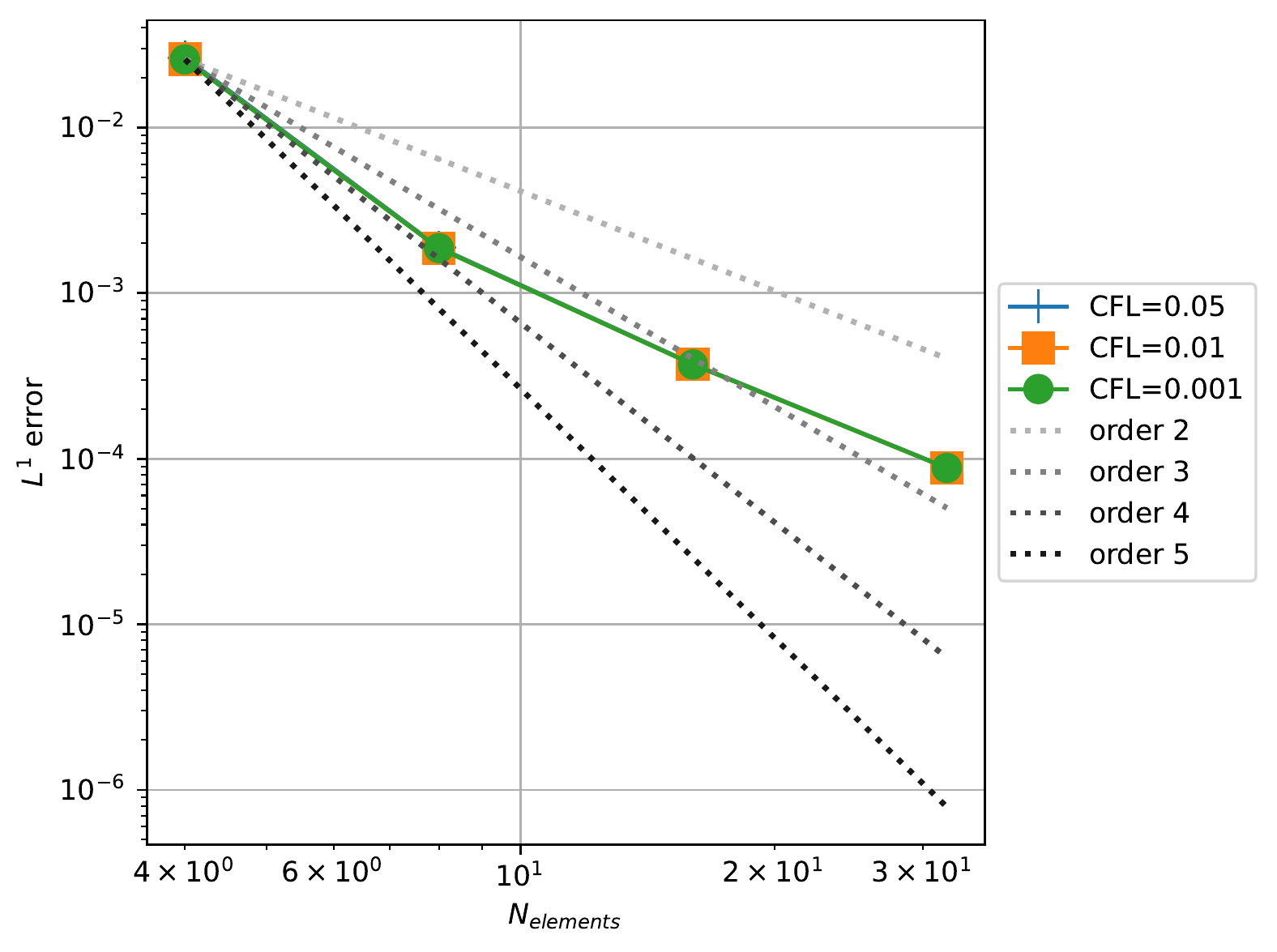}\caption{B4}\label{subfig:B4CFL}
	\end{subfigure}
\caption{1D LAE: tests with different CFLs}\label{fig:CFL}
\end{figure}

\subsubsection{Impact of the stabilization on higher order derivatives}

The CIP stabilization on the first derivative that we have presented can be actually generalized to keep into account higher order derivatives as in \cite{larson2020stabilization}

\begin{equation}
\uvec{ST}_i(\uapp)=\sum_{f\in \mathcal{F}_h}\sum_{r=1}^R \alpha^{\CIP}_{f,r} \int_f \Big\llbracket \nabla^r_{\nu_f} \varphi_i \Big\rrbracket \cdot \Big\llbracket \nabla^r_{\nu_f} \uvec{u}_h \Big\rrbracket  d\sigma(\uvec{x}),\quad \alpha^{\CIP}_{f,r}=\delta_r^{\CIP} \bar{\rho}_f h_f^{2r}     
\label{eq:CIP_HO}
\end{equation}
where $\mathcal{F}_h$ is the set of the $(D-1)$-dimensional faces shared by two elements of $\mathcal{T}_h$, $\nabla^r_{\nu_f}$ is the $r$-th partial derivative in the direction $\nu_f$ normal to the face $f$ and $\delta_r^{\CIP}$ are constant parameters which must be tuned.
We will focus on the stabilization of the first and second derivatives only, $R=2$.
The results obtained with $\delta^{\CIP}_1=0.00702$ and different values of $\delta^{\CIP}_2$ are displayed in \cref{fig:2ndDER}. For B3 and B4, the extra stabilization seems to help in decreasing the errors but still it is not sufficient to achieve the right order of accuracy.

\begin{figure}
\centering
	\begin{subfigure}{0.49\textwidth}
		\includegraphics[width=\textwidth]{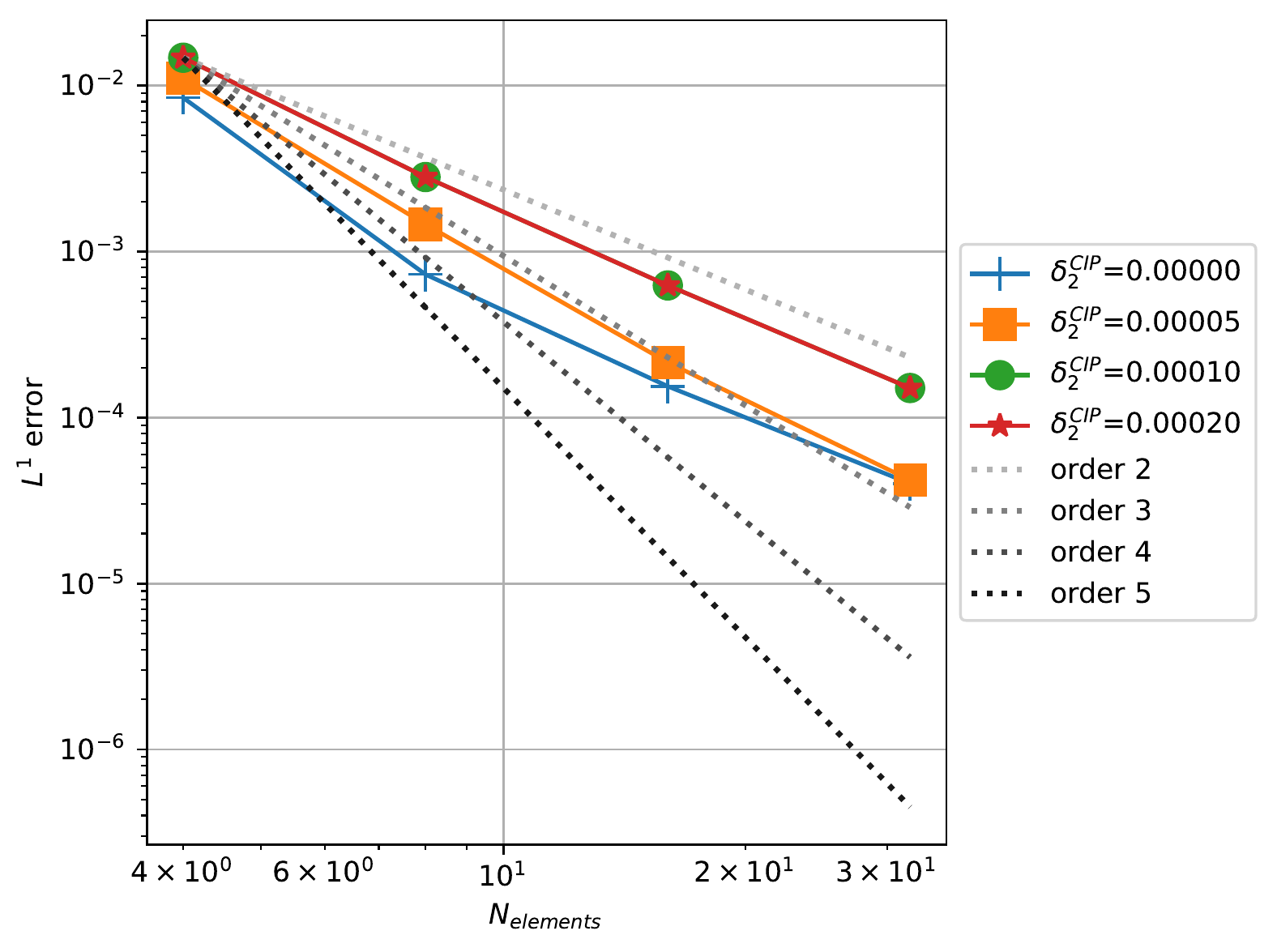}\caption{P3.}\label{subfig:P32ndDER}
	\end{subfigure}
	\begin{subfigure}{0.49\textwidth}
	\includegraphics[width=\textwidth]{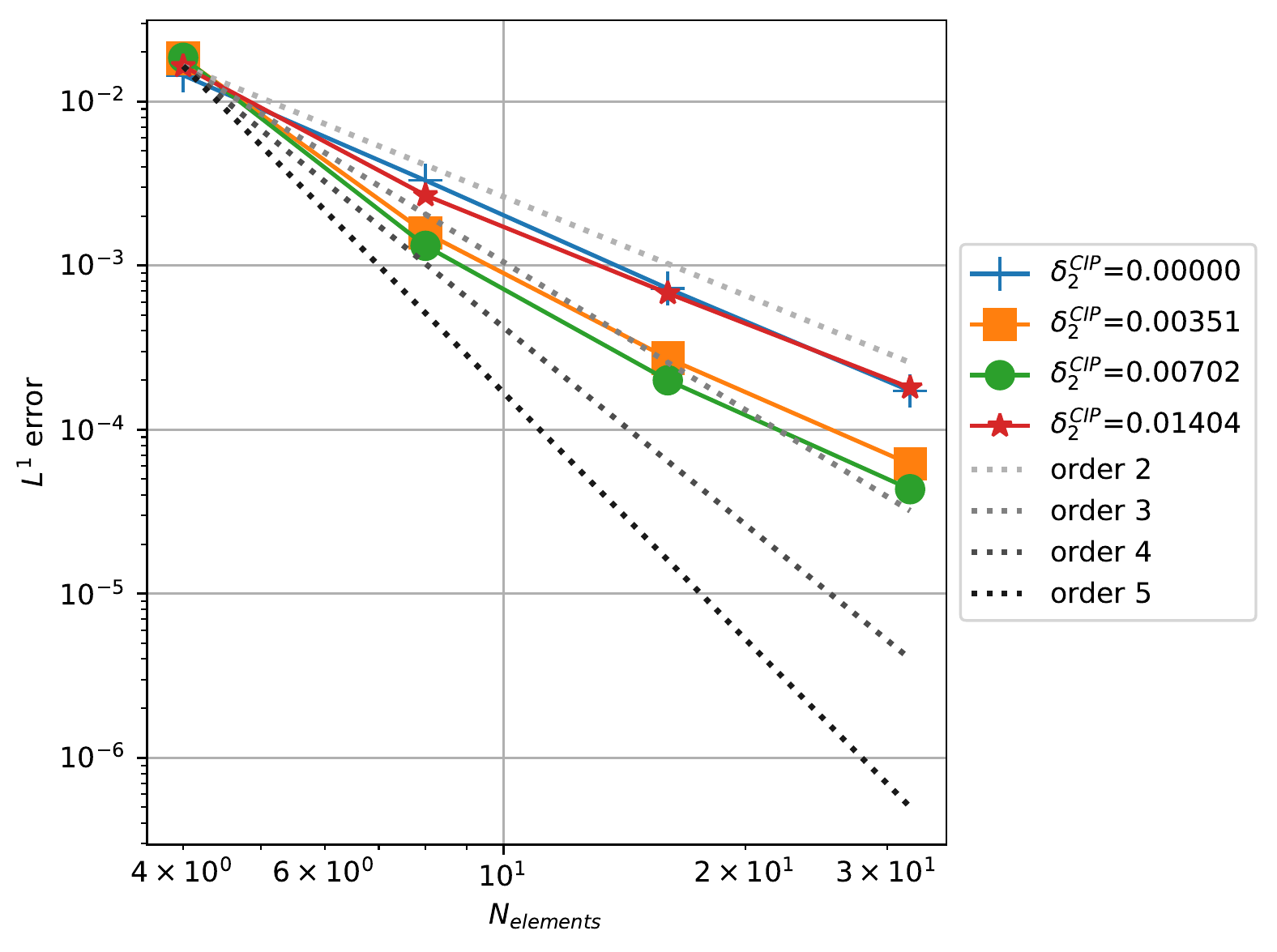}\caption{B3.}\label{subfig:B32ndDER}
\end{subfigure}\\
	\begin{subfigure}{0.49\textwidth}
		\includegraphics[width=\textwidth]{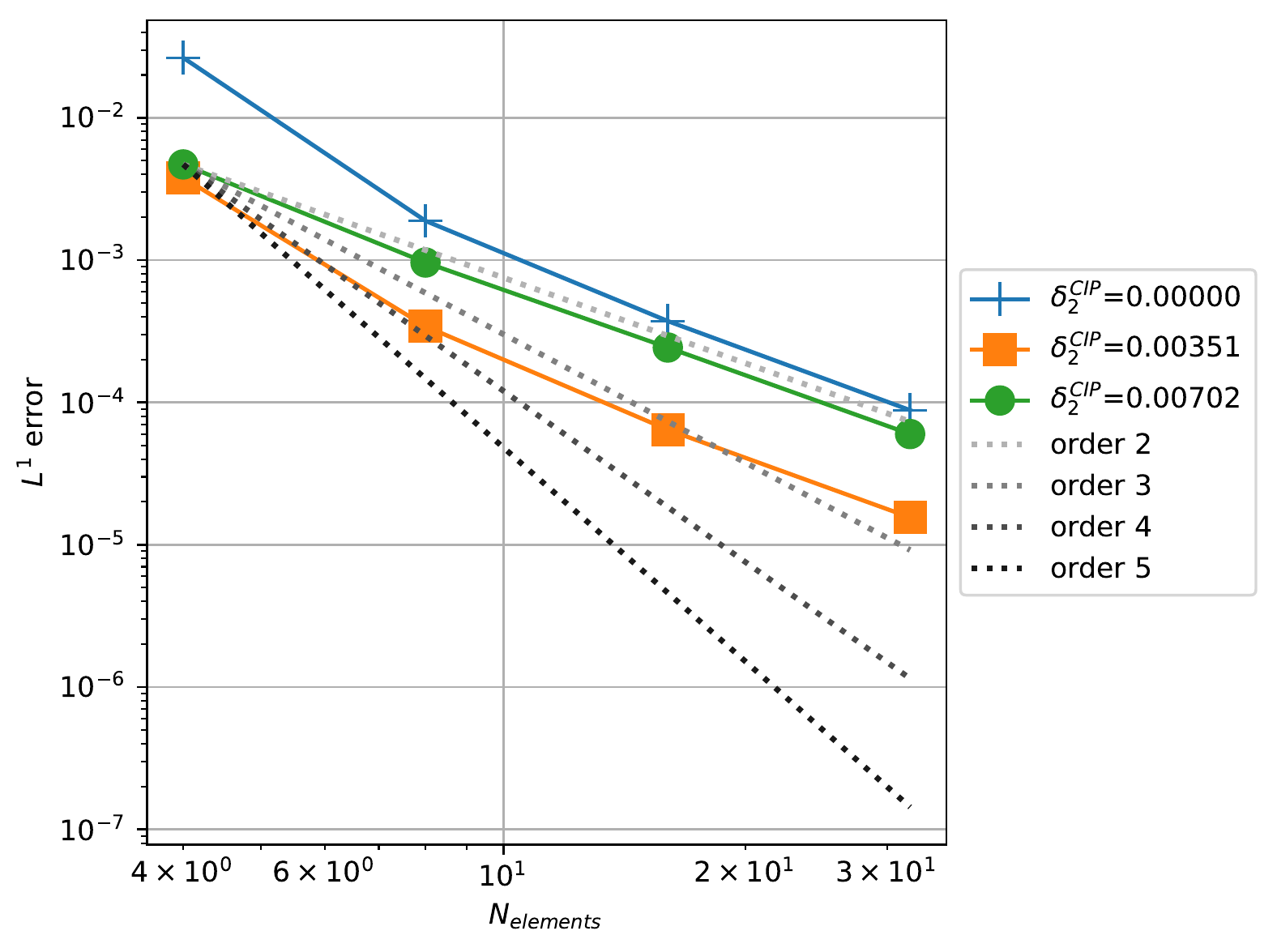}\caption{B4.}\label{subfig:B42ndDER}
	\end{subfigure}
\caption{1D LAE: tests with different stabilization parameters on the second derivative.}\label{fig:2ndDER}
\end{figure}

\subsubsection{Final remarks}
The conclusion of the previous analysis is that, in the context of unsteady problems, with the optimal number of iterations one obtains second order accuracy. 
Among the three aspects numerically analyzed, only the first one seems to have an effect on the order of accuracy; in particular, many more iterations than expected are needed to reach the formal order.
For the moment we do not further investigate this issue, but we have other ideas on how to proceed.
First of all, an analysis of the combination of the three parameters studied above could give better results and a linear stability/dispersion analysis, in the style of \cite{michel2021spectral,michel2022spectral}, can help in determining the optimal setting to achieve the best possible results. 
Further, higher order derivatives stabilization terms could be taken in consideration hoping for a better stabilization, this has been suggested also in \cite{abgrall2019high,michel2022spectral}. More in general, other stabilizations other than CIP and OSS could be considered.
Moreover, the authors suspect that \cref{ass:poincare} is not verified by the approximations and a stronger estimate on the $H^1$ norm of the solution of the discrete problem should be provided with weaker hypotheses to guarantee the accuracy results.

We conclude this section remarking that the mentioned problem does not occur with cubature elements also in the DeC framework, which provide accurate and fast results.

\section{Vibrating system}\label{sec:vibrating}
Let us consider a general sinusoidal function
\begin{equation}
x(t)=X\cos{(\Omega t +\varphi)},
\label{eq:singeneral}
\end{equation}
then we refer to $X\in \mathbb{R}^+_0$ as the amplitude, to $\Omega\in \mathbb{R}^+$ as the frequency and to $\varphi \in [0,2\pi[$ modulo $2\pi n$ with $n\in \mathbb{Z}$ as the phase.

Let us introduce two general sinusoidal functions
\begin{align}
	x_j(t)=X_j\cos{(\Omega t +\varphi_j)},\quad  \text{ for } j=1,2,
	\label{eq:sinf}
\end{align}
characterized by the same frequency $\Omega>0$, amplitudes $X_1,X_2\geq 0$ and phases $\varphi_1,\varphi_2 \in [0,2\pi[$ modulo $2n\pi$ with $n\in \mathbb{Z}$.

\begin{proposition}
The sum $x_s(t)=x_1(t)+x_2(t)$ between two sinusoidal functions with the same frequency $\Omega$ is another sinusoidal function with the same frequency.
\end{proposition}
\begin{proof}
If $x_1(t)+x_2(t) = 0$ or at least one between $X_1$ or $X_2$ is zero, then the proof is straightforward so let us focus on the case in which $x_1(t)+x_2(t) \neq 0$ and both $X_1$ and $X_2$ are different from $0.$

From basic trigonometry, we have
\begin{align}
x_j(t)=X_j\cos{(\Omega t +\varphi_j)}=X_j[\cos{(\Omega t)}\cos{(\varphi_j)}-\sin{(\Omega t)}\sin{(\varphi_j)}],\quad \text{for }j=1,2,
\label{eq:sinfnew}
\end{align}
then
\begin{align}
&x_s(t) =x_1(t)+x_2(t)
=A\cos{(\Omega t)}-B\sin{(\Omega t)},\label{eq:sum_2}\\
&\text{with }A:=X_1\cos{(\varphi_1)}+X_2\cos{(\varphi_2)},\qquad 
B:=X_1\sin{(\varphi_1)}+X_2\sin{(\varphi_2)}.
\label{eq:AB}
\end{align}
We consider now the point $(A,B)\in \mathbb{R}^2$, different from $(0,0)$ by assumption, and the induced vector of length $X_s = \sqrt{A^2+B^2} $ and phase $\varphi_s  =\angle(A,B)$, so that $A=X_s\cos{(\varphi_s)}$ and $B=X_s \sin{(\varphi_s)}$.
By definition of such vector, \eqref{eq:sum_2} can be recast as
\begin{align}
x_1(t)+x_2(t)=X_s \cos{(\varphi_s)} \cos{(\Omega t)}-X_s \sin{(\varphi_s)} \sin{(\Omega t)}=X_s \cos{(\Omega t+\varphi_s)},
\end{align}
which completes the proof.
\end{proof}

We introduce now a bijection $\mathcal{S}$ from the quotient set of the sinusoidal functions with a fixed frequency $\Omega$ defined by $(X,\varphi)$, in which we identify all the functions characterized by $X=0$, onto the complex plane
\begin{align}
\mathcal{S}(x(t)) = \mathcal{S}(X,\varphi) = \begin{cases} Xe^{i\varphi} \quad &\text{if} \quad  X\neq 0\\
0 \quad &\text{if} \quad X= 0
\end{cases}.
\end{align}
The complex number $\overline{X}:=\mathcal{S}(x(t))$ is called phasor associated to the sinusoidal function $x(t)$.
\begin{proposition}
If we have two sinusoidal functions $x_1(t),x_2(t)$ with the same frequency $\Omega$ then the phasor $\overline{X}_s$ associated to the sum $x_s(t)$ of the two sinusoidal functions is the sum of the phasors $\overline{X}_1,\overline{X}_2$ associated to the single sinusoidal functions.
\end{proposition}
\begin{proof}
The phasors related to the sinusoidal functions \eqref{eq:sinf} are given by
\begin{align}
\overline{X}_j &=X_je^{i\varphi_j}=X_j\left[\cos{(\varphi_j)}+i\sin{(\varphi_j)}\right],\text{ for }j=1,2.
\label{eq:sinfbis}
\end{align}
If one between $X_1$ or $X_2$ is zero then the proof is straightforward therefore we focus on the case in which they are both different from $0.$ Further, we assume for the moment that $x_1(t)+x_2(t)\neq 0$.
The sum of the phasors gives
\begin{align}
\begin{split}
\overline{X}_r &=\overline{X}_1+\overline{X}_2\\
               &=\left[X_1\cos{(\varphi_1)}+X_2\cos{(\varphi_2)}\right]+i\left[X_1\sin{(\varphi_1)}+X_2\sin{(\varphi_2)}\right] = A+iB
\end{split}
\label{eq:sumfas}
\end{align}
with $A$ and $B$ defined exactly as in \eqref{eq:AB} leading to
\begin{align}
\overline{X}_r=X_r e^{i\varphi_r}
\end{align}
with $X_r=X_s$ and $\varphi_r=\varphi_s$ with $X_s$ and $\varphi_s$ defined from the phasor associated to $x_s(t)$.

If $x_1(t)+x_2(t)= 0$, by simple considerations, we must have $X_2=X_1$ and $\varphi_2=\varphi_1+\pi$ modulo $2\pi$, which leads to
\begin{align}
\overline{X}_1 =X_1e^{i\varphi_1},\quad \overline{X}_2 =X_1e^{i\left(\varphi_1+\pi\right)}=-\overline{X}_1.
\label{eq:if0}
\end{align}
Then, we clearly have $\overline{X}_1+\overline{X}_2=0$. Indeed, also the phasor $\overline{X}_s$ associated to the sum is $0$ and this completes the proof.
\end{proof}

It is clear that if we have a sinusoidal function $x(t)=X\cos{(\Omega t +\varphi)}$ then its derivative in time is still a sinusoidal function with the same frequency 
\begin{align}
x'(t) &=-\Omega X\sin{(\Omega t +\varphi)} =\Omega X\cos{\left(\Omega t +\varphi+ \frac{\pi}{2}\right)}.
\label{eq:derivativesin}
\end{align}
Then the phasor $\overline{X'}$ associated to the derivative in time $x'(t)$ is 
\begin{align}
\overline{X'} &=\Omega Xe^{i\left(\varphi+ \frac{\pi}{2}\right)}=i\Omega Xe^{i\varphi}=i\Omega \overline{X}.
\label{eq:firstderivativesin}
\end{align}
By the same argument we have that the phasor $\overline{X''}$ associated to the second derivative in time $x''(t)$ is
\begin{align}
\overline{X''} &=i\Omega \overline{X'}=i\Omega(i\Omega \overline{X})=-\Omega^2 \overline{X}.
\label{eq:secondderivativesin}
\end{align}
We consider the scalar ODE
\begin{align}
\begin{cases}
& my''+ry'+ky=F\cos(\Omega t +\varphi),\quad t \in \mathbb{R}^+_0\\
& y(0)=A, \\
& y'(0)=B,
\end{cases}
\label{eq:1DoF}
\end{align}
with the real nonnegative constants $m,k,\Omega>0$ and $r,F \geq 0$ with $\varphi\in [0,2\pi[$ modulo $2\pi n$ with $n\in \mathbb{Z}$.
The solution to \eqref{eq:1DoF} is given by
\begin{align}
y(t)=y_h(t)+y_p(t)
\label{eq:1DoFsol}
\end{align}
where $y_h(t)$ is a solution to the homogeneus equation and $y_p(t)$ is a solution to the whole equation.

We first focus on the homogeneus problem
\begin{equation}
 my''+ry'+ky=0
\end{equation}
and we look for a solution in the form $y(t)=Ae^{\lambda t}$ which is nontrivial and so we assume $A\neq 0$. We substitute it in the homogeneus equation and we get 
\begin{equation}
\left(m\lambda^2+r\lambda+k\right)Ae^{\lambda t}=0
\end{equation}
and since $Ae^{\lambda t}\neq 0$ $\forall t \in \mathbb{R}^+_0$ because $A\neq 0$ then we get the characteristic equation
\begin{equation}
\lambda^2+\alpha\lambda+\beta=0
\label{eq:charequ}
\end{equation}
with $\alpha=\frac{r}{m}\geq 0$ and $\beta=\frac{k}{m}> 0.$ 
The roots are given by
\begin{equation}
\lambda_{1,2}=\frac{1}{2}\left(-\alpha\pm\sqrt{\alpha^2-4\beta}\right)
\end{equation}
and, depending on the parameters of the problem, we have three possibilities
\begin{enumerate}
\item $\lambda_{1}\neq \lambda_{2}$, real, negative and different if $\alpha>2\sqrt{\beta} \Leftrightarrow r> 2\sqrt{km};$
\item $\lambda_{1}=\lambda_{2}=\lambda$, real, negative and coincident if $\alpha=2\sqrt{\beta} \Leftrightarrow r= 2\sqrt{km};$
\item $\lambda_{1,2}=\alpha\pm i \omega$, complex and conjugate with negative real part if $\alpha<2\sqrt{\beta} \Leftrightarrow r< 2\sqrt{km}.$
\end{enumerate}
Thus, the solution to our homogeneous ODE is
\begin{align}
y_h(t)=\begin{cases}
C_1e^{\lambda_1 t}+C_2e^{\lambda_2 t}, & \text{if} \quad \alpha>2\sqrt{\beta} \Leftrightarrow r> 2\sqrt{km},\\ 
C_1e^{\lambda t}+C_2te^{\lambda t}, & \text{if} \quad \alpha=2\sqrt{\beta} \Leftrightarrow r= 2\sqrt{km},\\
e^{-\frac{\alpha}{2} t}\left(C_1cos(\omega t)+C_2sin(\omega t)\right), & \text{if} \quad \alpha<2\sqrt{\beta} \Leftrightarrow r< 2\sqrt{km}.
\end{cases} 
\label{eq:solhomo}
\end{align}
Now, we focus on the whole ODE \eqref{eq:1DoF}
and we assume a sinusoidal solution of the type $y_p=Y_p\cos(\Omega t + \psi)$, we substitute it in \eqref{eq:1DoF} and we solve the equation in the space of the phasors.
Recalling the expression of the phasors associated to the first and the second derivatives of a sinusoidal function given by \eqref{eq:firstderivativesin} and \eqref{eq:secondderivativesin} we have
\begin{equation}
-m\Omega^2\overline{Y}_p +i\Omega r \overline{Y}_p+k \overline{Y}_p=F e^{i\varphi}.
\label{eq:phasorscomplete1DoF}
\end{equation}
Then
\begin{equation}
\overline{Y}_p=\frac{F e^{i\varphi}}{-m\Omega^2+k+i\Omega r},
\end{equation}
from which we get
\begin{align}\label{eq:sol}
Y_p &=\frac{F}{\sqrt{(-m\Omega^2+k)^2+\Omega^2 r^2}} ,\qquad 
\psi =\varphi-\arg{(-m\Omega^2+k+i\Omega r)},
\end{align}
where by $\arg{(\cdot)}$ we denote the phase of the argument up to $2n\pi$ with $n\in \mathbb{Z}$.
Once we compute $\overline{Y}_p$, we automatically get the unique associated sinusoidal function $y_p(t)=Y_p \cos{(\Omega t +\psi)}$.

So, the final solution to our ODE \eqref{eq:1DoF} is $y(t)=y_h(t)+y_p(t)$, where $y_h(t)$ is given by \eqref{eq:solhomo} and $y_p(t)$ is a sinusoidal function whose amplitude and phase are given by  \eqref{eq:sol}.

The two constants $C_1$ and $C_2$ in $y_h(t)$ are computed by imposing the initial conditions $y(0)=A$ and $y'(0)=B$ and solving the resulting $2$ by $2$ linear system.

\bibliography{sn-bibliography}
\bibliographystyle{plain}